\documentclass[a4paper,10pt]{amsart}

\usepackage{amsfonts,amssymb}

\input{xy}
\xyoption{all}
\newtheorem{proposition}{Proposition}[section]
\newtheorem{lemma}[proposition]{Lemma}
\newtheorem{corollary}[proposition]{Corollary}
\newtheorem{theorem}[proposition]{Theorem}
\theoremstyle{definition}
\newtheorem{definition}[proposition]{Definition}

\theoremstyle{remark}
\newtheorem{remark}[proposition]{Remark}

\newcommand{\thlabel}[1]{\label{th:#1}}
\newcommand{\thref}[1]{Theorem~\ref{th:#1}}
\newcommand{\selabel}[1]{\label{se:#1}}
\newcommand{\seref}[1]{Section~\ref{se:#1}}
\newcommand{\lelabel}[1]{\label{le:#1}}
\newcommand{\leref}[1]{Lemma~\ref{le:#1}}
\newcommand{\prlabel}[1]{\label{pr:#1}}
\newcommand{\prref}[1]{Proposition~\ref{pr:#1}}
\newcommand{\colabel}[1]{\label{co:#1}}
\newcommand{\coref}[1]{Corollary~\ref{co:#1}}

\newcommand{\relabel}[1]{\label{re:#1}}
\newcommand{\reref}[1]{Remark~\ref{re:#1}}

\newcommand{\delabel}[1]{\label{de:#1}}
\newcommand{\deref}[1]{Definition~\ref{de:#1}}

\newcommand{\eqlabel}[1]{\label{eq:#1}}
\newcommand{\equref}[1]{(\ref{eq:#1})}
\def\equuref#1#2{(\ref{eq:#1}.#2)}

\def\ra{\rightarrow}

\def\Id{{\rm Id}}

\newcommand{\gbrbox}{\mbox{$\gbr\gvac{-1}\gnot{\hspace*{-4mm}\Box}$}}
\def\ot{\otimes}
\def\va{\varepsilon}

\def\un{\underline}

\def\mf{\mathfrak}

\def\va{\varepsilon}

\def\ra{\rightarrow}

\def\cal{\mathcal}
\def\un{\underline}

\newcommand{\Cc}{\cal C}

\def\equal#1{\smash{\mathop{=}\limits^{#1}}}

\def\equalupdown#1#2{\smash{\mathop{=}\limits^{#1}\limits_{#2}}}

 \newcount\grcalca
 \newcount\grcalcb
 \newcount\grcalcc
 \newcount\grcalcd
 \newcount\grcalce
 \newcount\grcalcf
 \newcount\grcalcg
 \newcount\grcalch
 \newcount\grcalci
 \newcount\grrow
 \newcount\grcolumn
 \newcount\grwidth
 \newcount\factor
 \newcount\hfactor
 \newcount\qfactor
 \newcount\tfactor
 \newcount\sfactor
 \newcount\dfactor
 \hfactor = \factor
 \divide    \hfactor by 2
 \qfactor = \hfactor
 \divide    \qfactor by 2
 \tfactor = \factor
 \divide    \tfactor by 3
 \sfactor = \factor
 \divide    \sfactor by 6
 \dfactor = \factor
 \divide    \dfactor by 12

 \newcommand{\gbeg}[2]{
   \unitlength=1pt
   \grrow = #2
   \grcolumn = 0
   \grcalca = #1
   \grcalcb = #2
   \multiply \grcalca by \factor
   \grwidth = \grcalca
   \multiply \grcalcb by \factor
   \begin{minipage}{\grcalca pt}
   \begin{picture}(\grcalca,\grcalcb)
   \advance \grcalcb by -\factor
   \put(0, \grcalcb){\line(1,0){\grwidth}} }

 \newcommand{\gend}{
   \put(0, \factor){\line(1,0){\grwidth}}
   \end{picture}
   {\vskip2.5ex}
   \end{minipage} }

 \newcommand{\gnl}{
   \advance \grrow by -1
   \grcolumn = 0}

 \newcommand{\gvac}[1]{       
   \advance \grcolumn by #1} 
 
 \newcommand{\gcl}[1]{
   \grcalca = \grcolumn
   \multiply \grcalca by \factor
   \advance \grcalca by \hfactor
   \grcalcb = \grrow
   \multiply \grcalcb by \factor
   \grcalcc = #1
   \multiply \grcalcc by \factor
   \put(\grcalca,\grcalcb) {\line(0,-1){\grcalcc}} 
   \advance \grcolumn by 1}

 \newcommand{\gcn}[4]{
   \grcalca = \grcolumn
   \multiply \grcalca by \factor
   \grcalci = #3
   \multiply \grcalci by \hfactor
   \advance \grcalca by \grcalci
   \grcalcb = \grcolumn
   \multiply \grcalcb by \factor 
   \grcalci = #3
   \advance \grcalci by #4
   \multiply \grcalci by \qfactor
   \advance \grcalcb by \grcalci
   \grcalcc = \grcolumn
   \multiply \grcalcc by \factor
   \grcalci = #4
   \multiply \grcalci by \hfactor
   \advance \grcalcc by \grcalci
   \grcalcd = \grrow
   \multiply \grcalcd by \factor 
   \grcalce = \grrow
   \multiply \grcalce by \factor 
   \grcalci = #2
   \multiply \grcalci by \tfactor
   \advance \grcalce by -\grcalci
   \grcalcf = \grrow
   \multiply \grcalcf by \factor 
   \grcalci = #2
   \multiply \grcalci by \hfactor
   \advance \grcalcf by -\grcalci
   \grcalcg = \grrow
   \multiply \grcalcg by \factor 
   \grcalci = #2
   \multiply \grcalci by \tfactor
   \multiply \grcalci by 2
   \advance \grcalcg by -\grcalci
   \grcalch = \grrow
   \advance \grcalch by -#2
   \multiply \grcalch by \factor 
   \qbezier(\grcalca,\grcalcd)(\grcalca,\grcalce)(\grcalcb,\grcalcf) 
   \qbezier(\grcalcb,\grcalcf)(\grcalcc,\grcalcg)(\grcalcc,\grcalch) 
   \advance \grcolumn by #1}

 \newcommand{\gnot}[1]{
   \grcalca = \grcolumn
   \multiply \grcalca by \factor
   \advance \grcalca by \hfactor
   \grcalcb = \grrow
   \multiply \grcalcb by \factor
   \advance \grcalcb by -\hfactor
   \put(\grcalca,\grcalcb) {\makebox(0,0){$\scriptstyle #1$}} }

 \newcommand{\got}[2]{
   \grcalca = \grcolumn
   \multiply \grcalca by \factor
   \grcalcc = #1
   \multiply \grcalcc by \hfactor
   \advance \grcalca by \grcalcc
   \grcalcb = \grrow
   \multiply \grcalcb by \factor
   \advance \grcalcb by -\tfactor
   \advance \grcalcb by -\tfactor
   \put(\grcalca,\grcalcb){\makebox(0,0)[b]{$#2$}}
   \advance \grcolumn by #1}

 \newcommand{\gob}[2]{
   \grcalca = \grcolumn
   \multiply \grcalca by \factor
   \grcalcc = #1
   \multiply \grcalcc by \hfactor
   \advance \grcalca by \grcalcc
   \put(\grcalca,0){\makebox(0,0)[b]{$#2$}}
   \advance \grcolumn by #1}

 \newcommand{\gmu}{  
   \grcalca = \grcolumn
   \advance \grcalca by 1
   \multiply \grcalca by \factor
   \grcalcb = \grrow
   \multiply \grcalcb by \factor
   \grcalcc = \factor
   \advance \grcalcc by \hfactor
   \put(\grcalca,\grcalcb){\oval(\factor,\grcalcc)[b]}
   \advance \grcalcb by -\hfactor
   \advance \grcalcb by -\qfactor
   \put(\grcalca,\grcalcb) {\line(0,-1){\qfactor}} 
   \advance \grcolumn by 2}

 \newcommand{\gcmu}{   
   \grcalca = \grcolumn
   \advance \grcalca by 1
   \multiply \grcalca by \factor
   \grcalcb = \grrow
   \advance \grcalcb by -1
   \multiply \grcalcb by \factor
   \grcalcc = \factor
   \advance \grcalcc by \hfactor
   \put(\grcalca,\grcalcb){\oval(\factor,\grcalcc)[t]}
   \advance \grcalcb by \factor
   \put(\grcalca,\grcalcb) {\line(0,-1){\qfactor}} 
   \advance \grcolumn by 2}

 \newcommand{\glm}{
   \grcalca = \grcolumn
   \multiply \grcalca by \factor
   \advance \grcalca by \hfactor
   \grcalcb = \grcalca
   \advance \grcalcb by \factor
   \grcalcc = \grrow
   \multiply \grcalcc by \factor
   \grcalcd = \grcalcc
   \advance \grcalcd by -\tfactor
   \grcalce = \grcalcd
   \advance \grcalce by -\tfactor
   \put(\grcalca, \grcalcc){\line(0,-1){\tfactor}}
   \put(\grcalca, \grcalcd){\line(1,0){\factor}}
   \put(\grcalca, \grcalcd){\line(3,-1){\factor}}
   \put(\grcalcb, \grcalcc){\line(0,-1){\factor}}
   \advance \grcolumn by 2}

 \newcommand{\grm}{
   \grcalcb = \grcolumn
   \multiply \grcalcb by \factor
   \advance \grcalcb by \hfactor
   \grcalca = \grcalcb
   \advance \grcalca by \factor
   \grcalcc = \grrow
   \multiply \grcalcc by \factor
   \grcalcd = \grcalcc
   \advance \grcalcd by -\tfactor
   \grcalce = \grcalcd
   \advance \grcalce by -\tfactor
   \put(\grcalca, \grcalcc){\line(0,-1){\tfactor}}
   \put(\grcalca, \grcalcd){\line(-1,0){\factor}}
   \put(\grcalca, \grcalcd){\line(-3,-1){\factor}}
   \put(\grcalcb, \grcalcc){\line(0,-1){\factor}}
   \advance \grcolumn by 2}

 \newcommand{\glcm}{
   \grcalca = \grcolumn
   \multiply \grcalca by \factor
   \advance \grcalca by \hfactor
   \grcalcb = \grcalca
   \advance \grcalcb by \factor
   \grcalcc = \grrow
   \advance \grcalcc by -1
   \multiply \grcalcc by \factor
   \grcalcd = \grcalcc
   \advance \grcalcd by \tfactor
   \grcalce = \grcalcd
   \advance \grcalce by \tfactor
   \put(\grcalca, \grcalcc){\line(0,1){\tfactor}}
   \put(\grcalca, \grcalcd){\line(1,0){\factor}}
   \put(\grcalca, \grcalcd){\line(3,1){\factor}}
   \put(\grcalcb, \grcalcc){\line(0,1){\factor}}
   \advance \grcolumn by 2}

 \newcommand{\grcm}{
   \grcalcb = \grcolumn
   \multiply \grcalcb by \factor
   \advance \grcalcb by \hfactor
   \grcalca = \grcalcb
   \advance \grcalca by \factor
   \grcalcc = \grrow
   \advance \grcalcc by -1
   \multiply \grcalcc by \factor
   \grcalcd = \grcalcc
   \advance \grcalcd by \tfactor
   \grcalce = \grcalcd
   \advance \grcalce by \tfactor
   \put(\grcalca, \grcalcc){\line(0,1){\tfactor}}
   \put(\grcalca, \grcalcd){\line(-1,0){\factor}}
   \put(\grcalca, \grcalcd){\line(-3,1){\factor}}
   \put(\grcalcb, \grcalcc){\line(0,1){\factor}}
   \advance \grcolumn by 2}

 \newcommand{\gwmu}[1]{    
   \grcalca = \grcolumn
   \multiply \grcalca by \factor
   \grcalcd = \hfactor
   \multiply \grcalcd by #1
   \advance \grcalca by \grcalcd
   \grcalcb = \grrow
   \multiply \grcalcb by \factor
   \grcalcc = \factor
   \advance \grcalcc by \hfactor
   \grcalcd = #1
   \advance \grcalcd by -1
   \multiply \grcalcd by \factor
   \put(\grcalca,\grcalcb){\oval(\grcalcd,\grcalcc)[b]}
   \advance \grcalcb by -\hfactor
   \advance \grcalcb by -\qfactor
   \put(\grcalca,\grcalcb) {\line(0,-1){\qfactor}} 
   \advance \grcolumn by #1}

 \newcommand{\gwcm}[1]{   
   \grcalca = \grcolumn
   \multiply \grcalca by \factor
   \grcalcd = \hfactor
   \multiply \grcalcd by #1
   \advance \grcalca by \grcalcd
   \grcalcb = \grrow
   \advance \grcalcb by -1
   \multiply \grcalcb by \factor
   \grcalcc = \factor
   \advance \grcalcc by \hfactor
   \grcalcd = #1
   \advance \grcalcd by -1
   \multiply \grcalcd by \factor
   \put(\grcalca,\grcalcb){\oval(\grcalcd,\grcalcc)[t]}
   \advance \grcalcb by \factor
   \put(\grcalca,\grcalcb) {\line(0,-1){\qfactor}} 
   \advance \grcolumn by #1}

 \newcommand{\gwmuc}[1]{    
   \grcalca = \grcolumn
   \multiply \grcalca by \factor
   \advance \grcalca by \hfactor
   \grcalcb = \grrow
   \multiply \grcalcb by \factor
   \grcalcc = #1
   \advance \grcalcc by -1
   \multiply \grcalcc by \factor
   \put(\grcalca,\grcalcb){\line(1,0){\grcalcc}}
   \advance \grcalca by -\hfactor
   \grcalcd = \hfactor
   \multiply \grcalcd by #1
   \advance \grcalca by \grcalcd
   \grcalcc = \factor
   \advance \grcalcc by \hfactor
   \grcalcd = #1
   \advance \grcalcd by -1
   \multiply \grcalcd by \factor
   \put(\grcalca,\grcalcb){\oval(\grcalcd,\grcalcc)[b]}
   \advance \grcalcb by -\hfactor
   \advance \grcalcb by -\qfactor
   \put(\grcalca,\grcalcb) {\line(0,-1){\qfactor}} 
   \advance \grcolumn by #1}

 \newcommand{\gwcmc}[1]{   
   \grcalca = \grcolumn
   \multiply \grcalca by \factor
   \advance \grcalca by \hfactor
   \grcalcb = \grrow
   \multiply \grcalcb by \factor
   \advance \grcalcb by -\factor
   \grcalcc = #1
   \advance \grcalcc by -1
   \multiply \grcalcc by \factor
   \put(\grcalca,\grcalcb){\line(1,0){\grcalcc}}
   \grcalcd = #1
   \advance \grcalcd by -1
   \multiply \grcalcd by \hfactor
   \advance \grcalca by \grcalcd
   \grcalcc = \factor
   \advance \grcalcc by \hfactor
   \grcalcd = #1
   \advance \grcalcd by -1
   \multiply \grcalcd by \factor
   \put(\grcalca,\grcalcb){\oval(\grcalcd,\grcalcc)[t]}
   \advance \grcalcb by \factor
   \put(\grcalca,\grcalcb) {\line(0,-1){\qfactor}} 
   \advance \grcolumn by #1}

 \newcommand{\gev}{  
   \grcalca = \grcolumn
   \advance \grcalca by 1
   \multiply \grcalca by \factor
   \grcalcb = \grrow
   \multiply \grcalcb by \factor
   \grcalcc = \factor
   \advance \grcalcc by \hfactor
   \put(\grcalca,\grcalcb){\oval(\factor,\grcalcc)[b]}
   \advance \grcolumn by 2}

 \newcommand{\gdb}{   
   \grcalca = \grcolumn
   \advance \grcalca by 1
   \multiply \grcalca by \factor
   \grcalcb = \grrow
   \advance \grcalcb by -1
   \multiply \grcalcb by \factor
   \grcalcc = \factor
   \advance \grcalcc by \hfactor
   \put(\grcalca,\grcalcb){\oval(\factor,\grcalcc)[t]}
   \advance \grcolumn by 2}

 \newcommand{\gwev}[1]{    
   \grcalca = \grcolumn
   \multiply \grcalca by \factor
   \grcalcd = \hfactor
   \multiply \grcalcd by #1
   \advance \grcalca by \grcalcd
   \grcalcb = \grrow
   \multiply \grcalcb by \factor
   \grcalcc = \factor
   \advance \grcalcc by \hfactor
   \grcalcd = #1
   \advance \grcalcd by -1
   \multiply \grcalcd by \factor
   \put(\grcalca,\grcalcb){\oval(\grcalcd,\grcalcc)[b]}
   \advance \grcolumn by #1}

 \newcommand{\gwdb}[1]{   
   \grcalca = \grcolumn
   \multiply \grcalca by \factor
   \grcalcd = \hfactor
   \multiply \grcalcd by #1
   \advance \grcalca by \grcalcd
   \grcalcb = \grrow
   \advance \grcalcb by -1
   \multiply \grcalcb by \factor
   \grcalcc = \factor
   \advance \grcalcc by \hfactor
   \grcalcd = #1
   \advance \grcalcd by -1
   \multiply \grcalcd by \factor
   \put(\grcalca,\grcalcb){\oval(\grcalcd,\grcalcc)[t]}
   \advance \grcolumn by #1}

 \newcommand{\gbr}{
   \grcalca = \grcolumn
   \multiply \grcalca by \factor
   \advance \grcalca by \hfactor
   \grcalcb = \grcalca
   \advance \grcalcb by \hfactor
   \grcalcc = \grcalca
   \advance \grcalcc by \factor
   \grcalcd = \grrow
   \multiply \grcalcd by \factor
   \grcalce = \grcalcd
   \advance \grcalce by -\tfactor
   \grcalcf = \grcalcd
   \advance \grcalcf by -\hfactor
   \grcalcg = \grcalce
   \advance \grcalcg by -\tfactor
   \grcalch = \grcalcd
   \advance \grcalch by -\factor
   \qbezier(\grcalca,\grcalcd)(\grcalca,\grcalce)(\grcalcb,\grcalcf) 
   \qbezier(\grcalcb,\grcalcf)(\grcalcc,\grcalcg)(\grcalcc,\grcalch) 
   \advance \grcalcf by -\dfactor
   \advance \grcalcb by -\sfactor
   \qbezier(\grcalca,\grcalch)(\grcalca,\grcalcg)(\grcalcb,\grcalcf) 
   \advance \grcalcf by \sfactor
   \advance \grcalcb by \tfactor
   \qbezier(\grcalcc,\grcalcd)(\grcalcc,\grcalce)(\grcalcb,\grcalcf) 
   \advance \grcolumn by 2}

 \newcommand{\gibr}{
   \grcalca = \grcolumn
   \multiply \grcalca by \factor
   \advance \grcalca by \hfactor
   \grcalcb = \grcalca
   \advance \grcalcb by \hfactor
   \grcalcc = \grcalca
   \advance \grcalcc by \factor
   \grcalcd = \grrow
   \multiply \grcalcd by \factor
   \grcalce = \grcalcd
   \advance \grcalce by -\tfactor
   \grcalcf = \grcalcd
   \advance \grcalcf by -\hfactor
   \grcalcg = \grcalce
   \advance \grcalcg by -\tfactor
   \grcalch = \grcalcd
   \advance \grcalch by -\factor
   \qbezier(\grcalcc,\grcalcd)(\grcalcc,\grcalce)(\grcalcb,\grcalcf) 
   \qbezier(\grcalcb,\grcalcf)(\grcalca,\grcalcg)(\grcalca,\grcalch) 
   \advance \grcalcf by -\dfactor
   \advance \grcalcb by \sfactor
   \qbezier(\grcalcc,\grcalch)(\grcalcc,\grcalcg)(\grcalcb,\grcalcf) 
   \advance \grcalcf by \sfactor
   \advance \grcalcb by -\tfactor
   \qbezier(\grcalca,\grcalcd)(\grcalca,\grcalce)(\grcalcb,\grcalcf) 
   \advance \grcolumn by 2}

 \newcommand{\gbrc}{
   \grcalca = \grcolumn
   \multiply \grcalca by \factor
   \advance \grcalca by \hfactor
   \grcalcb = \grcalca
   \advance \grcalcb by \hfactor
   \grcalcc = \grcalca
   \advance \grcalcc by \factor
   \grcalcd = \grrow
   \multiply \grcalcd by \factor
   \grcalce = \grcalcd
   \advance \grcalce by -\tfactor
   \grcalcf = \grcalcd
   \advance \grcalcf by -\hfactor
   \grcalcg = \grcalce
   \advance \grcalcg by -\tfactor
   \grcalch = \grcalcd
   \advance \grcalch by -\factor
   \put(\grcalcb,\grcalcf){\circle{\hfactor}}
   \qbezier(\grcalca,\grcalcd)(\grcalca,\grcalce)(\grcalcb,\grcalcf) 
   \qbezier(\grcalcb,\grcalcf)(\grcalcc,\grcalcg)(\grcalcc,\grcalch) 
   \advance \grcalcf by -\dfactor
   \advance \grcalcb by -\sfactor
   \qbezier(\grcalca,\grcalch)(\grcalca,\grcalcg)(\grcalcb,\grcalcf) 
   \advance \grcalcf by \sfactor
   \advance \grcalcb by \tfactor
   \qbezier(\grcalcc,\grcalcd)(\grcalcc,\grcalce)(\grcalcb,\grcalcf) 
   \advance \grcolumn by 2}

 \newcommand{\gibrc}{
   \grcalca = \grcolumn
   \multiply \grcalca by \factor
   \advance \grcalca by \hfactor
   \grcalcb = \grcalca
   \advance \grcalcb by \hfactor
   \grcalcc = \grcalca
   \advance \grcalcc by \factor
   \grcalcd = \grrow
   \multiply \grcalcd by \factor
   \grcalce = \grcalcd
   \advance \grcalce by -\tfactor
   \grcalcf = \grcalcd
   \advance \grcalcf by -\hfactor
   \grcalcg = \grcalce
   \advance \grcalcg by -\tfactor
   \grcalch = \grcalcd
   \advance \grcalch by -\factor
   \put(\grcalcb,\grcalcf){\circle{\hfactor}}
   \qbezier(\grcalcc,\grcalcd)(\grcalcc,\grcalce)(\grcalcb,\grcalcf) 
   \qbezier(\grcalcb,\grcalcf)(\grcalca,\grcalcg)(\grcalca,\grcalch) 
   \advance \grcalcf by -\dfactor
   \advance \grcalcb by \sfactor
   \qbezier(\grcalcc,\grcalch)(\grcalcc,\grcalcg)(\grcalcb,\grcalcf) 
   \advance \grcalcf by \sfactor
   \advance \grcalcb by -\tfactor
   \qbezier(\grcalca,\grcalcd)(\grcalca,\grcalce)(\grcalcb,\grcalcf) 
   \advance \grcolumn by 2} 

 \newcommand{\gu}[1]{
   \grcalca = \grcolumn
   \multiply \grcalca by \factor
   \grcalcd = \hfactor
   \multiply \grcalcd by #1
   \advance \grcalca by \grcalcd
   \grcalcb = \grrow
   \advance \grcalcb by -1
   \multiply \grcalcb by \factor
   \put(\grcalca,\grcalcb) {\line(0,1){\hfactor}} 
   \advance \grcalcb by \hfactor
   \put(\grcalca,\grcalcb) {\circle*{3}}
   \advance \grcolumn by #1}

 \newcommand{\gcu}[1]{
   \grcalca = \grcolumn
   \multiply \grcalca by \factor
   \grcalcd = \hfactor
   \multiply \grcalcd by #1
   \advance \grcalca by \grcalcd
   \grcalcb = \grrow
   \multiply \grcalcb by \factor
   \put(\grcalca,\grcalcb) {\line(0,-1){\hfactor}} 
   \advance \grcalcb by -\hfactor
   \put(\grcalca,\grcalcb) {\circle*{3}}
   \advance \grcolumn by #1}

 \newcommand{\gmp}[1]{
   \grcalca = \grcolumn
   \multiply \grcalca by \factor
   \advance \grcalca by \hfactor
   \grcalcb = \grrow
   \multiply \grcalcb by \factor
   \put(\grcalca,\grcalcb) {\line(0,-1){\dfactor}} 
   \advance \grcalcb by -\factor
   \put(\grcalca,\grcalcb) {\line(0,1){\dfactor}} 
   \advance \grcalcb by \hfactor
   \grcalcc = \factor
   \advance \grcalcc by -\qfactor
   \put(\grcalca,\grcalcb) {\circle{\grcalcc}}
   \put(\grcalca,\grcalcb) {\makebox(0,0){$\scriptstyle #1$}}
   \advance \grcolumn by 1}

 \newcommand{\gbmp}[1]{
   \grcalca = \grcolumn
   \multiply \grcalca by \factor
   \advance \grcalca by \hfactor
   \grcalcb = \grrow
   \multiply \grcalcb by \factor
   \put(\grcalca,\grcalcb) {\line(0,-1){\dfactor}} 
   \advance \grcalcb by -\factor
   \put(\grcalca,\grcalcb) {\line(0,1){\dfactor}} 
   \advance \grcalca by -\hfactor
   \advance \grcalca by \dfactor
   \advance \grcalcb by \dfactor
   \grcalcc = \factor
   \advance \grcalcc by -\sfactor
   \put(\grcalca,\grcalcb) {\framebox(\grcalcc,\grcalcc){$\scriptstyle #1$}}
   \advance \grcolumn by 1}

 \newcommand{\gbmpt}[1]{
   \grcalca = \grcolumn
   \multiply \grcalca by \factor
   \advance \grcalca by \hfactor
   \grcalcb = \grrow
   \multiply \grcalcb by \factor
   \put(\grcalca,\grcalcb) {\line(0,-1){\dfactor}} 
   \advance \grcalcb by -\factor
   \advance \grcalca by -\hfactor
   \advance \grcalca by \dfactor
   \advance \grcalcb by \dfactor
   \grcalcc = \factor
   \advance \grcalcc by -\sfactor
   \put(\grcalca,\grcalcb) {\framebox(\grcalcc,\grcalcc){$\scriptstyle #1$}}
   \advance \grcolumn by 1}

 \newcommand{\gbmpb}[1]{
   \grcalca = \grcolumn
   \multiply \grcalca by \factor
   \advance \grcalca by \hfactor
   \grcalcb = \grrow
   \multiply \grcalcb by \factor
   \advance \grcalcb by -\factor
   \put(\grcalca,\grcalcb) {\line(0,1){\dfactor}} 
   \advance \grcalca by -\hfactor
   \advance \grcalca by \dfactor
   \advance \grcalcb by \dfactor
   \grcalcc = \factor
   \advance \grcalcc by -\sfactor
   \put(\grcalca,\grcalcb) {\framebox(\grcalcc,\grcalcc){$\scriptstyle #1$}}
   \advance \grcolumn by 1}

 \newcommand{\gbmpn}[1]{
   \grcalca = \grcolumn
   \multiply \grcalca by \factor
   \advance \grcalca by \hfactor
   \grcalcb = \grrow
   \multiply \grcalcb by \factor
   \advance \grcalcb by -\factor
   \advance \grcalca by -\hfactor
   \advance \grcalca by \dfactor
   \advance \grcalcb by \dfactor
   \grcalcc = \factor
   \advance \grcalcc by -\sfactor
   \put(\grcalca,\grcalcb) {\framebox(\grcalcc,\grcalcc){$\scriptstyle #1$}}
   \advance \grcolumn by 1}

 \newcommand{\glmptb}{    
   \grcalca = \grcolumn
   \multiply \grcalca by \factor
   \advance \grcalca by \hfactor
   \grcalcb = \grrow
   \multiply \grcalcb by \factor
   \put(\grcalca,\grcalcb) {\line(0,-1){\dfactor}} 
   \advance \grcalcb by -\factor
   \put(\grcalca,\grcalcb) {\line(0,1){\dfactor}} 
   \advance \grcalca by -\hfactor
   \advance \grcalca by \dfactor
   \advance \grcalcb by \dfactor
   \put(\grcalca,\grcalcb) {\line(1,0){\factor}} 
   \advance \grcalcb by \factor
   \advance \grcalcb by -\sfactor
   \put(\grcalca,\grcalcb) {\line(1,0){\factor}} 
   \grcalcc = \factor
   \advance \grcalcc by -\sfactor
   \put(\grcalca,\grcalcb) {\line(0,-1){\grcalcc}} 
   \advance \grcolumn by 1}

 \newcommand{\glmpt}{    
   \grcalca = \grcolumn
   \multiply \grcalca by \factor
   \advance \grcalca by \hfactor
   \grcalcb = \grrow
   \multiply \grcalcb by \factor
   \put(\grcalca,\grcalcb) {\line(0,-1){\dfactor}} 
   \advance \grcalca by -\hfactor
   \advance \grcalca by \dfactor
   \advance \grcalcb by -\dfactor
   \put(\grcalca,\grcalcb) {\line(1,0){\factor}} 
   \advance \grcalcb by -\factor
   \advance \grcalcb by \sfactor
   \put(\grcalca,\grcalcb) {\line(1,0){\factor}} 
   \grcalcc = \factor
   \advance \grcalcc by -\sfactor
   \put(\grcalca,\grcalcb) {\line(0,1){\grcalcc}} 
   \advance \grcolumn by 1}

 \newcommand{\glmpb}{    
   \grcalca = \grcolumn
   \multiply \grcalca by \factor
   \advance \grcalca by \hfactor
   \grcalcb = \grrow
   \multiply \grcalcb by \factor
   \advance \grcalcb by -\factor
   \put(\grcalca,\grcalcb) {\line(0,1){\dfactor}} 
   \advance \grcalca by -\hfactor
   \advance \grcalca by \dfactor
   \advance \grcalcb by \dfactor
   \put(\grcalca,\grcalcb) {\line(1,0){\factor}} 
   \advance \grcalcb by \factor
   \advance \grcalcb by -\sfactor
   \put(\grcalca,\grcalcb) {\line(1,0){\factor}} 
   \grcalcc = \factor
   \advance \grcalcc by -\sfactor
   \put(\grcalca,\grcalcb) {\line(0,-1){\grcalcc}} 
   \advance \grcolumn by 1}

 \newcommand{\glmp}{    
   \grcalca = \grcolumn
   \multiply \grcalca by \factor
   \advance \grcalca by \dfactor
   \grcalcb = \grrow
   \multiply \grcalcb by \factor
   \advance \grcalcb by -\dfactor
   \put(\grcalca,\grcalcb) {\line(1,0){\factor}} 
   \advance \grcalcb by -\factor
   \advance \grcalcb by \sfactor
   \put(\grcalca,\grcalcb) {\line(1,0){\factor}} 
   \grcalcc = \factor
   \advance \grcalcc by -\sfactor
   \put(\grcalca,\grcalcb) {\line(0,1){\grcalcc}} 
   \advance \grcolumn by 1}

 \newcommand{\gcmptb}{    
   \grcalca = \grcolumn
   \multiply \grcalca by \factor
   \advance \grcalca by \hfactor
   \grcalcb = \grrow
   \multiply \grcalcb by \factor
   \put(\grcalca,\grcalcb) {\line(0,-1){\dfactor}} 
   \advance \grcalcb by -\factor
   \put(\grcalca,\grcalcb) {\line(0,1){\dfactor}} 
   \advance \grcalca by -\hfactor
   \advance \grcalcb by \dfactor
   \put(\grcalca,\grcalcb) {\line(1,0){\factor}} 
   \advance \grcalcb by \factor
   \advance \grcalcb by -\sfactor
   \put(\grcalca,\grcalcb) {\line(1,0){\factor}} 
   \advance \grcolumn by 1}

 \newcommand{\gcmpt}{    
   \grcalca = \grcolumn
   \multiply \grcalca by \factor
   \advance \grcalca by \hfactor
   \grcalcb = \grrow
   \multiply \grcalcb by \factor
   \put(\grcalca,\grcalcb) {\line(0,-1){\dfactor}} 
   \advance \grcalcb by -\factor
   \advance \grcalca by -\hfactor
   \advance \grcalcb by \dfactor
   \put(\grcalca,\grcalcb) {\line(1,0){\factor}} 
   \advance \grcalcb by \factor
   \advance \grcalcb by -\sfactor
   \put(\grcalca,\grcalcb) {\line(1,0){\factor}} 
   \advance \grcolumn by 1}

 \newcommand{\gcmpb}{    
   \grcalca = \grcolumn
   \multiply \grcalca by \factor
   \advance \grcalca by \hfactor
   \grcalcb = \grrow
   \multiply \grcalcb by \factor
   \advance \grcalcb by -\factor
   \put(\grcalca,\grcalcb) {\line(0,1){\dfactor}} 
   \advance \grcalca by -\hfactor
   \advance \grcalcb by \dfactor
   \put(\grcalca,\grcalcb) {\line(1,0){\factor}} 
   \advance \grcalcb by \factor
   \advance \grcalcb by -\sfactor
   \put(\grcalca,\grcalcb) {\line(1,0){\factor}} 
   \advance \grcolumn by 1}

 \newcommand{\gcmp}{    
   \grcalca = \grcolumn
   \multiply \grcalca by \factor
   \grcalcb = \grrow
   \multiply \grcalcb by \factor
   \advance \grcalcb by -\factor
   \advance \grcalcb by \dfactor
   \put(\grcalca,\grcalcb) {\line(1,0){\factor}} 
   \advance \grcalcb by \factor
   \advance \grcalcb by -\sfactor
   \put(\grcalca,\grcalcb) {\line(1,0){\factor}} 
   \advance \grcolumn by 1}

 \newcommand{\grmptb}{    
   \grcalca = \grcolumn
   \multiply \grcalca by \factor
   \advance \grcalca by \hfactor
   \grcalcb = \grrow
   \multiply \grcalcb by \factor
   \put(\grcalca,\grcalcb) {\line(0,-1){\dfactor}} 
   \advance \grcalcb by -\factor
   \put(\grcalca,\grcalcb) {\line(0,1){\dfactor}} 
   \advance \grcalca by \hfactor
   \advance \grcalca by -\dfactor
   \advance \grcalcb by \dfactor
   \put(\grcalca,\grcalcb) {\line(-1,0){\factor}} 
   \advance \grcalcb by \factor
   \advance \grcalcb by -\sfactor
   \put(\grcalca,\grcalcb) {\line(-1,0){\factor}} 
   \grcalcc = \factor
   \advance \grcalcc by -\sfactor
   \put(\grcalca,\grcalcb) {\line(0,-1){\grcalcc}} 
   \advance \grcolumn by 1}

 \newcommand{\grmpt}{    
   \grcalca = \grcolumn
   \multiply \grcalca by \factor
   \advance \grcalca by \hfactor
   \grcalcb = \grrow
   \multiply \grcalcb by \factor
   \put(\grcalca,\grcalcb) {\line(0,-1){\dfactor}} 
   \advance \grcalca by \hfactor
   \advance \grcalca by -\dfactor
   \advance \grcalcb by -\dfactor
   \put(\grcalca,\grcalcb) {\line(-1,0){\factor}} 
   \advance \grcalcb by -\factor
   \advance \grcalcb by \sfactor
   \put(\grcalca,\grcalcb) {\line(-1,0){\factor}} 
   \grcalcc = \factor
   \advance \grcalcc by -\sfactor
   \put(\grcalca,\grcalcb) {\line(0,1){\grcalcc}} 
   \advance \grcolumn by 1}

 \newcommand{\grmpb}{    
   \grcalca = \grcolumn
   \multiply \grcalca by \factor
   \advance \grcalca by \hfactor
   \grcalcb = \grrow
   \multiply \grcalcb by \factor
   \advance \grcalcb by -\factor
   \put(\grcalca,\grcalcb) {\line(0,1){\dfactor}} 
   \advance \grcalca by \hfactor
   \advance \grcalca by -\dfactor
   \advance \grcalcb by \dfactor
   \put(\grcalca,\grcalcb) {\line(-1,0){\factor}} 
   \advance \grcalcb by \factor
   \advance \grcalcb by -\sfactor
   \put(\grcalca,\grcalcb) {\line(-1,0){\factor}} 
   \grcalcc = \factor
   \advance \grcalcc by -\sfactor
   \put(\grcalca,\grcalcb) {\line(0,-1){\grcalcc}} 
   \advance \grcolumn by 1}

 \newcommand{\grmp}{    
   \grcalca = \grcolumn
   \multiply \grcalca by \factor
   \advance \grcalca by \factor
   \advance \grcalca by -\dfactor
   \grcalcb = \grrow
   \multiply \grcalcb by \factor
   \advance \grcalcb by -\dfactor
   \put(\grcalca,\grcalcb) {\line(-1,0){\factor}} 
   \advance \grcalcb by -\factor
   \advance \grcalcb by \sfactor
   \put(\grcalca,\grcalcb) {\line(-1,0){\factor}} 
   \grcalcc = \factor
   \advance \grcalcc by -\sfactor
   \put(\grcalca,\grcalcb) {\line(0,1){\grcalcc}} 
   \advance \grcolumn by 1}
\newcommand{\gsy}{
   \grcalca = \grcolumn
   \multiply \grcalca by \factor
   \advance \grcalca by \hfactor
   \grcalcb = \grcalca
   \advance \grcalcb by \hfactor
   \grcalcc = \grcalca
   \advance \grcalcc by \factor
   \grcalcd = \grrow
   \multiply \grcalcd by \factor
   \grcalce = \grcalcd
   \advance \grcalce by -\tfactor
   \grcalcf = \grcalcd
   \advance \grcalcf by -\hfactor
   \grcalcg = \grcalce
   \advance \grcalcg by -\tfactor
   \grcalch = \grcalcd
   \advance \grcalch by -\factor
   \qbezier(\grcalcc,\grcalcd)(\grcalcc,\grcalce)(\grcalcb,\grcalcf) 
   \qbezier(\grcalcb,\grcalcf)(\grcalca,\grcalcg)(\grcalca,\grcalch) 
   \advance \grcalcf by -\dfactor
   \advance \grcalcb by \sfactor
   \qbezier(\grcalcc,\grcalch)(\grcalcc,\grcalcg)(\grcalcb,\grcalcf) 
   \qbezier(\grcalca,\grcalcd)(\grcalca,\grcalce)(\grcalcb,\grcalcf) 
   \advance \grcolumn by 2}

 \newcommand{\gwmuh}[3]{    
   \grcalca = \grcolumn
   \multiply \grcalca by \factor
   \grcalcb = #2
   \advance \grcalcb by #3
   \multiply \grcalcb by \qfactor
   \advance \grcalca by \grcalcb
   \grcalcb = \grrow
   \multiply \grcalcb by \factor
   \grcalcc = #3
   \advance \grcalcc by -#2
   \multiply \grcalcc by \hfactor
   \grcalcd = \factor
   \advance \grcalcd by \hfactor
   \put(\grcalca,\grcalcb){\oval(\grcalcc,\grcalcd)[b]}
   \grcalca = \grcolumn
   \multiply \grcalca by \factor
   \grcalcc = #1
   \multiply \grcalcc by \hfactor
   \advance \grcalca by \grcalcc
   \advance \grcalcb by -\hfactor
   \advance \grcalcb by -\qfactor
   \put(\grcalca,\grcalcb) {\line(0,-1){\qfactor}} 
   \advance \grcolumn by #1}

 \newcommand{\gwcmh}[3]{   
   \grcalca = \grcolumn
   \multiply \grcalca by \factor
   \grcalcb = #2
   \advance \grcalcb by #3
   \multiply \grcalcb by \qfactor
   \advance \grcalca by \grcalcb
   \grcalcb = \grrow
   \advance \grcalcb by -1
   \multiply \grcalcb by \factor
   \grcalcc = #3
   \advance \grcalcc by -#2
   \multiply \grcalcc by \hfactor
   \grcalcd = \factor
   \advance \grcalcd by \hfactor
   \put(\grcalca,\grcalcb){\oval(\grcalcc,\grcalcd)[t]}
   \grcalca = \grcolumn
   \multiply \grcalca by \factor
   \grcalcc = #1
   \multiply \grcalcc by \hfactor
   \advance \grcalca by \grcalcc
   \advance \grcalcb by \factor
   \put(\grcalca,\grcalcb) {\line(0,-1){\qfactor}} 
   \advance \grcolumn by #1}

 \newcommand{\gsbox}[1]{
   \grcalca = \grcolumn
   \multiply \grcalca by \factor
   \grcalcb = \grrow
   \multiply \grcalcb by \factor
   \advance \grcalcb by -\factor
   \grcalcc = #1
   \multiply \grcalcc by \factor
   \grcalcd = \factor
   \put(\grcalca,\grcalcb){\framebox(\grcalcc,\grcalcd){}}}

\allowdisplaybreaks[4]

\begin{document}
\title[Cross product Hopf algebras]
{On cross product Hopf algebras}
\author{D. Bulacu}
\address{Faculty of Mathematics and Informatics, University
of Bucharest, Str. Academiei 14, RO-010014 Bucharest 1, Romania}
\email{daniel.bulacu@fmi.unibuc.ro}
\author{S. Caenepeel}
\address{Faculty of Engineering, 
Vrije Universiteit Brussel, B-1050 Brussels, Belgium}
\email{scaenepe@vub.ac.be}
\author{B. Torrecillas}
\address{Department of Algebra and Analysis\\
Universidad de Almer\'{\i}a\\
E-04071 Almer\'{\i}a, Spain}
\email{btorreci@ual.es}
\thanks{\rm 
The first author was supported by the strategic grant POSDRU/89/1.5/S/58852, Project 
``Postdoctoral program for training scientific researchers" cofinanced by the European Social Fund within the 
Sectorial Operational Program Human Resources Development 2007 - 2013.
The second author was supported by research project G.0117.10  
``Equivariant Brauer groups and Galois deformations'' from
FWO-Vlaanderen. The third author was partially supported by FQM 3128 from Junta 
Andaluc'a MTM2008-03339 from MCI.\\ 
The first author thanks the Vrije Universiteit Brussel and the Universidad 
de Almer\'{\i}a for their hospitality.
The authors also thank Bodo Pareigis for sharing his ``diagrams" program.}

\begin{abstract}
Let $A$ and $B$ be algebras and coalgebras in a braided monoidal category $\Cc$,
and suppose that we have a cross product algebra and a cross coproduct coalgebra
structure on $A\ot B$. We present necessary and sufficient conditions for $A\ot B$ to be
a bialgebra, and sufficient conditions for $A\ot B$ to be a Hopf algebra.
We discuss when such a cross product Hopf algebra is a double cross (co)product,
a biproduct, or, more generally, a smash (co)product Hopf algebra. In each of these cases,
we provide an explicit description of the associated Hopf algebra projection.       
\end{abstract}
\maketitle
\section*{Introduction}\selabel{intro}
\setcounter{equation}{0}
Given algebras $A$ and $B$ in a monoidal category, and a local braiding between them, this
is a morphism $\psi:\ B\ot A\to A\ot B$ satisfying four properties, we can construct a new algebra
$A\#_\psi B$ with underlying object $A\ot B$, called cross product algebra. If $\Cc$ is braided,
then the tensor product algebra and the smash product algebra are special cases.
A dual construction is possible: given two coalgebras $A$ and $B$, and a morphism
$\phi:\ A\ot B\to B\ot A$ satisfying appropriate conditions, we can form the cross product coalgebra
$A\#^\phi B$.\\
Cross product bialgebras where introduced independently in \cite{cimz} (in the category of
vector spaces) and in \cite{bespdrab1} (in a general braided monoidal category). The
construction generalizes biproduct bialgebras \cite{rad} and double cross (co)product bialgebras 
\cite{majbip, maj}. It can be summarized easily: given algebras and coalgebras $A$ and $B$,
and local braidings $\psi$ and $\phi$, we can consider $A\#_\psi^\phi B$, with underlying
algebra $A\#_\psi B$ and underlying coalgebra $A\#^\phi B$. If this is a bialgebra, then
we call $A\#_\psi^\phi B$ a cross product bialgebra. Cross product bialgebras can be
characterized using injections and projections, see \cite[Prop. 2.2]{bespdrab1},
\cite[Theorem 4.3]{cimz} or \prref{DrBespExtVers}.\\
If $A\#_\psi B$ is a cross product algebra, and $A$ and $B$ are augmented, then $A$
is a left $B$-module, and $B$ is a right $A$-module. Similarly, if $A\#^\phi B$ is a cross product algebra, and $A$ and $B$ are coaugmented, then $A$ is a left $B$-comodule, and $B$ is a right $A$-comodule, we will recall
these constructions in Lemmas \ref{le:action} and \ref{le:coaction}.
In \cite{bespdrab1}, an attempt was made to characterize cross product bialgebras in terms of these
actions and coactions. A Hopf datum consists of a pair of algebras and coalgebras $A$ and $B$
that act and coact on each other as above, satisfying a list of compatibility conditions, that we will
refer to as the {\sl Bespalov-Drabant} list \cite[Def. 2.5]{bespdrab1}. If $A\#_\psi^\phi B$ is a cross
product bialgebra, then $A$ and $B$ together with the actions and coactions from 
Lemmas \ref{le:action} and \ref{le:coaction} form a Hopf pair, \cite[Prop. 2.7]{bespdrab1}. Conversely, if we have a Hopf pair,
then we can find $\psi$ and $\phi$ such that $A\#_\psi^\phi B$ is cross product algebra and coalgebra,
but we are not able to show that it is a bialgebra, see  \cite[Prop. 2.6]{bespdrab1}. Roughly stated,
the Bespalov-Drabant list is a list of necessary conditions but we do not know whether it is also sufficient.\\
The main motivation of this paper was to fill in this gap: in Sections \ref{se:crossprodbialg}
and \ref{se:crossprodvsHopfdatum},
we will present some alternatives to the Bespalov-Drabant list, consisting of necessary
and sufficient conditions. Our first main result is \thref{firstsetequivcond}, in which we provide
a set of lists of necessary and sufficient conditions, in terms of the local braidings $\phi$ and $\psi$.
Another set, now in terms of the actions and coactions, will be given in \thref{crossprobialasactandcoact}.\\
As we have already
mentioned, smash product algebras are special cases of cross product algebras, and they can be
characterized, see \seref{smash}. In \seref{smashcross(co)prodbialgs}, we first show that 
a cross product bialgebra is a smash cross product bialgebra if and only if $\psi$ satisfies a (left) normality condition,
see \deref{normal}. In this situation, the necessary and sufficient conditions from
Theorems \ref{th:firstsetequivcond} and \ref{th:crossprobialasactandcoact} take a more elegant form,
see \thref{strsmashcrossprodHa}.
We have a dual version, characterizing smash cross coproduct
bialgebras (with cross product coalgebra as underlying coalgebra), and a combination of the
two versions yields
a characterization of Radford's biproducts, see \coref{Radford}: a cross product bialgebra is a
Radford biproduct if $\psi$ is conormal and $\phi$ is normal. In \thref{strsmashcrossprodHa},
we also present sufficient conditions for a smash cross product bialgebra to be a Hopf algebra.
All this results have a left and right version; combining the left and right version, we have the following
interesting application, see \coref{6.7}: a cross 
product bialgebra is a double cross product in the sense of Majid  if and only if $\phi$  is  left and right normal.
In this situation, $\phi$ coincides with the braiding of $A$ and $B$. Otherwise stated: Majid's
double cross product bialgebras are precisely the cross product bialgebras for which the underlying
coalgebra is the cotensor coalgebra. 
Consequently, in the category of sets any cross product Hopf algebra is a bicrossed product of groups 
in the sense of \cite{tak}, see \coref{6.8}.\\
We have already mentioned that cross product bialgebras can be
characterized using injections and projections. The aim of \seref{strHopfalgwithprof}
is to study this characterization in the case of smash cross product algebras.
The structure of Hopf algebras with a projection was described completely by Radford in \cite{rad}:
if $H$ and $B$ are Hopf algebras, and there exist Hopf algebra maps $i:\ B\to H$ and
$\pi:\ H\to B$ such that $\pi i=\Id_B$, then $H$ is isomorphic to a biproduct Hopf algebra.
Several generalizations of this result have appeared in the literature. 
In \cite{schwproj}, the condition on $\pi$ is relaxed:
if $\pi$ is a left $B$-linear coalgebra map then $H$ is isomorphic to a smash product coalgebra, with an
algebra structure given by a complicated formula that does not imply in general that $H$ is
isomorphic to a crossed product bialgebra. The situation where $\pi$ is a right $B$-linear coalgebra morphism
was studied with different methods in \cite{amst}. The situation where 
$\pi$ is a Hopf algebra morphism and $i$ is a coalgebra morphism is studied in \cite{bcm},
and the case where $\pi$ is a morphism of bialgebras and $i$ is a $B$-bicolinear algebra map
is studied in \cite{amstAMS}.\\
With these examples in mind, we have been looking for the appropriate projection context on a
Hopf algebra, that ensures that the Hopf algebra is isomorpic to a smash cross product Hopf algebra.
Here the idea is the following. If $H=A\#_\psi^\phi B$ is a cross product coalgebra,
then we have algebra morphisms $i,j$ and coalgebra morphisms $p,\pi$, as in
\cite[Prop. 2.2]{bespdrab1}, \prref{DrBespExtVers}. If $H$ is a smash cross product bialgebra,
then $\pi$ is a bialgebra morphism, and 
$(A,p,j)$ can be reconstructed from $(B,\pi,i)$: $(A,j)$ is
the equalizer of a certain pair of morphisms, see \leref{scpbextprop}. Conversely, if we have
a bialgebra $B$, and a bialgebra map $\pi:\ H\to B$ and an algebra map $i:\ B\to H$
such that $\pi$ is left inverse of $i$, then we can construct $(A,j)$ as an equalizer, and
show that $A$ is an algebra and a coalgebra, and $j$ has a left inverse $p$, see \leref{coalgstrA}.
The definition of the coalgebra structure on $A$ requires the fact that $B$ is a Hopf algebra. 
At this point, we can explain 
why we have to restrict attention to smash cross product Hopf algebra, that is, the case where
$B$ is a bialgebra. In the general case where $H$ is a cross product Hopf algebra, and 
$B$ is only an algebra and a coalgebra, one could simply
require the existence of a convolution inverse of ${\rm Id}_B$. But this does not work, as we need
in the construction that the antipode is an anti-algebra and an anti-coalgebra map. We
also show that $(A,p)$ is a coequalizer.
The main result is \thref{strofHopfwithcertproj}, characterizing smash cross product
Hopf algebras in terms of projections. As a special case, we recover Radford's result that
$H$ can be written as a biproduct Hopf algebra if and only if we have a split Hopf algebra
map $\pi:\ H\to B$, see \coref{biprodfromproj}. As another application, we characterize
double cross coproduct Hopf algebras in terms of projections, see \coref{wpdoublecrossprod}.
This improves \cite[Theorem 2.15]{amst}.
We end with a sketch of the dual theory, characterizing smash cross coproduct Hopf algebras.
\section{Preliminary results}\selabel{prelimres}
\setcounter{equation}{0}
We assume that the reader is familiar with the basic theory of braided monoidal categories, and refer
to \cite{bulacu,kas,maj} for more details. Throughout this paper, $\Cc$ will be a braided monoidal category
with tensor product $\ot : \Cc\times \Cc\ra \Cc$, 
unit object $\un{1}$ and braiding $c: \ot\ra \ot\circ \tau$. Here $\tau: \Cc\times \Cc\ra \Cc\times \Cc$ is the twist functor.
We will assume implicitly that the monoidal category $\Cc$ is strict, that is, the associativity and unit constraints
are all identity morphisms in $\Cc$. Our results will remain valid in arbitrary monoidal categories, since
every monoidal category is monoidal equivalent to a strict one, see for example \cite{bulacu, kas}.\\
For $X,Y\in \Cc$, we write $c_{X, Y}=\gbeg{2}{3}
\got{1}{X}\got{1}{Y}\gnl
\gbr\gnl
\gob{1}{Y}\gob{1}{X}
\gend
$ 
and $c^{-1}_{X, Y}=
\gbeg{2}{3}
\got{1}{Y}\got{1}{X}\gnl
\gibr\gnl
\gob{1}{X}\gob{1}{Y}
\gend
$. Recall that $c$ satisfies  
\begin{equation}\label{braiding}
c_{X, Y\ot Z}=
\gbeg{3}{4}
\got{1}{X}\got{1}{Y}\got{1}{Z}\gnl
\gbr\gcl{1}\gnl
\gcl{1}\gbr\gnl
\gob{1}{Y}\gob{1}{Z}\gob{1}{X}
\gend
\mbox{\hspace{4mm}and\hspace{4mm}}
c_{X\ot Y, Z}=
\gbeg{3}{4}
\got{1}{X}\got{1}{Y}\got{1}{Z}\gnl
\gcl{1}\gbr\gnl
\gbr\gcl{1}\gnl
\gob{1}{Z}\gob{1}{X}\gob{1}{Y}
\gend
\hspace{2mm},
\end{equation}
for all $X, Y, Z\in\Cc$. The naturality of $c$ can be expressed as follows:
\[
\gbeg{2}{4}
\got{1}{M}\got{1}{N}\gnl
\gbr\gnl
\gmp{g}\gmp{f}\gnl
\gob{1}{V}\gob{1}{U}
\gend =
\gbeg{2}{4}
\got{1}{M}\got{1}{N}\gnl
\gmp{f}\gmp{g}\gnl
\gbr\gnl
\gob{1}{V}\gob{1}{U}
\gend 
\hspace{2mm},    
\]
for any $f:M\ra U$, $g: N\ra V$ in ${\cal C}$. In particular, for a morphism
$
\gbeg{3}{3} 
\got{1}{X}\gvac{1}\got{1}{Y}\gnl
\gwmu{3}\gnl
\gvac{1}\gob{1}{Z}
\gend
$
between $X\ot Y$ and $Z$ in $\Cc$, and an object $T\in \Cc$, we have
\begin{equation}\eqlabel{nat1cup}
\gbeg{4}{5}
\got{1}{T}\got{1}{X}\gvac{1}\got{1}{Y}\gnl
\gcl{1}\gwmu{3}\gnl\
\gcl{1}\gcn{1}{1}{3}{1}\gnl
\gbr\gnl
\gob{1}{Z}\gob{1}{T}
\gend
=
\gbeg{4}{5}
\gvac{1}\got{1}{T}\got{1}{X}\got{1}{Y}\gnl
\gvac{1}\gbr\gcl{1}\gnl
\gvac{1}\gcn{1}{1}{1}{-1}\gbr\gnl
\gwmu{3}\gcl{1}\gnl
\gvac{1}\gob{1}{Z}\gvac{1}\gob{1}{T}
\gend
\hspace{2mm}\mbox{and}\hspace{2mm}
\gbeg{4}{5}
\got{1}{X}\gvac{1}\got{1}{Y}\got{1}{T}\gnl
\gwmu{3}\gcl{1}\gnl
\gvac{1}\gcn{1}{1}{1}{3}\gvac{1}\gcl{1}\gnl
\gvac{2}\gbr\gnl
\gvac{2}\gob{1}{T}\gob{1}{Z}\gnl
\gend 
=
\gbeg{4}{5}
\got{1}{X}\got{1}{Y}\got{1}{T}\gnl
\gcl{1}\gbr\gnl
\gbr\gcn{1}{1}{1}{3}\gnl
\gcl{1}\gwmu{3}\gnl
\gob{1}{T}\gvac{1}\gob{1}{Z}
\gend
\hspace{2mm}.
\end{equation}
In a similar way, we have for a morphism $
\gbeg{3}{3}
\gvac{1}\got{1}{X}\gnl
\gwcm{3}\gnl
\gob{1}{Y}\gvac{1}\gob{1}{Z}\gnl
\gend
$ 
between $X$ and $Y\ot Z$ that
\begin{equation}\eqlabel{nat2cup}
\gbeg{4}{5}
\got{1}{X}\got{1}{T}\gnl
\gbr\gnl
\gcl{1}\gcn{1}{1}{1}{3}\gnl
\gcl{1}\gwcm{3}\gnl
\gob{1}{T}\gob{1}{Y}\gvac{1}\gob{1}{Z}
\gend
=
\gbeg{4}{5}
\gvac{1}\got{1}{X}\gvac{1}\got{1}{T}\gnl
\gwcm{3}\gcl{1}\gnl
\gcn{1}{1}{1}{3}\gvac{1}\gbr\gnl
\gvac{1}\gbr\gcl{1}\gnl
\gvac{1}\gob{1}{T}\gob{1}{Y}\gob{1}{Z}
\gend
\hspace{2mm}\mbox{and}\hspace{2mm}
\gbeg{4}{4}
\gvac{1}\got{1}{T}\got{1}{X}\gnl
\gvac{1}\gbr\gnl
\gwcm{3}\gcn{1}{1}{-1}{1}\gnl
\gob{1}{Y}\gvac{1}\gob{1}{Z}\gob{1}{T}
\gend 
=
\gbeg{4}{5}
\got{1}{T}\gvac{1}\got{1}{X}\gnl
\gcl{1}\gwcm{3}\gnl
\gbr\gvac{1}\gcn{1}{1}{1}{-1}\gnl
\gcl{1}\gbr\gnl
\gob{1}{Y}\gob{1}{Z}\gob{1}{T}
\gend
\hspace{2mm}.
\end{equation}
$c_{\un{1}, X}:\ \un{1}\ot X=X\to X\ot\un{1}=X$ and $c_{X,\un{1}}:\ X\ot\un{1}=X\to \un{1}\ot X=X$
are equal to identity morphism of $
\Id_X=
\gbeg{1}{3}
\got{1}{X}\gnl
\gcl{1}\gnl
\gob{1}{X}
\gend
$,
see \cite[Prop. XIII.1.2]{kas}.\\
Let us now recall the notions of algebra and coalgebra in a monoidal category $\Cc$,
and of bialgebra and Hopf algebra in a braided monoidal category $\Cc$.
An algebra in $\Cc$ is a triple $(A, \un{m}_A, \un{\eta}_A)$, where $A\in \Cc$, and
$\un{m}_A=\gbeg{2}{3} 
\got{1}{A}\got{1}{A}\gnl
\gmu\gnl
\gob{2}{A}
\gend :\ A\ot A\ra A$
and
$\un{\eta}_A=\gbeg{1}{3}
\got{1}{\un{1}}\gnl
\gu{1}\gnl
\gob{1}{A}
\gend:\ \un{1}\ra A$
are morphisms in $\Cc$ satisfying the associativity and unit conditions
$\un{m}_A\circ (\un{m }_A\ot {\rm Id}_A)=
\un {m}_A\circ ({\rm Id}_A\ot\un{m}_A)$
and
$\un{m}_A\circ (\un{\eta }_A\ot {\rm Id}_A)=
\un {m}_A\circ ({\rm Id}_A\ot\un{\eta}_A)={\rm Id}_A$.\\
A coalgebra in $\Cc$ is a triple $(B, \un{\Delta}_B, \un{\varepsilon}_B)$, where $B\in \Cc$, and
$\un{\Delta}_B=\gbeg{2}{3}
\got{2}{B}\gnl
\gcmu\gnl
\gob{1}{B}\gob{1}{B}
\gend:\ B\ra B\ot B$
and
$\un{\varepsilon}_B=\gbeg{1}{3}
\got{1}{B}\gnl
\gcu{1}\gnl
\gob{1}{\un{1}}
\gend :\ B\ra \un{1}$, satisfying appropriate coassociativity and counit conditions.\\
A bialgebra in $\Cc$ is a fivetuple $(B, \un{m}_B, \un{\eta }_B, \un {\Delta }_B, \un {\va }_B)$, such
that $(B, \un{m}_B, \un{\eta}_B)$ is an algebra and $(B, \un{\Delta}_B, \un{\va}_B)$ is a coalgebra
such that $\un{\Delta}_B:\ B\to B\ot B$ and $\un{\va}_B: B\ra \un{1}$ are algebra morphisms.
$B\ot B$ has the tensor product algebra structure (using the braiding on $\Cc$), and
$\un{1}$ is an algebra, with both the multiplication and unit map equal to the identity on $\un{1}$.
For later reference, we give explicit formulas for the axioms of a bialgebra $B$: $\un{\va}_B\un{\eta}_B=\Id_{\un{1}}$, and
\begin{eqnarray}
&&
\gbeg{3}{5}
\got{1}{B}\got{1}{B}\got{1}{B}\gnl
\gmu\gcl{2}\gnl
\gvac{1}\gcn{1}{1}{0}{1}\gnl
\gvac{1}\gmu\gnl
\gob{4}{B}
\gend =
\gbeg{3}{5}
\got{1}{B}\got{1}{B}\got{1}{B}\gnl
\gcl{2}\gmu\gnl
\gvac{1}\gcn{1}{1}{2}{1}\gnl
\gmu\gnl
\gob{2}{B}
\gend
\hspace{2mm},\hspace{2mm}
\gbeg{2}{4}
\got{1}{B}\gnl
\gcl{1}\gu{1}\gnl
\gmu\gnl
\gob{2}{B}
\gend =
\gbeg{1}{3}
\got{1}{B}\gnl
\gcl{1}\gnl
\gob{1}{B}
\gend =
\gbeg{2}{4}
\got{3}{B}\gnl
\gu{1}\gcl{1}\gnl
\gmu\gnl
\gob{2}{B}
\gend
\hspace{2mm},\hspace{2mm}
\gbeg{3}{5}
\got{4}{B}\gnl
\gvac{1}\gcmu\gnl
\gvac{1}\gcn{1}{1}{1}{0}\gcl{1}\gnl
\gcmu\gcl{1}\gnl
\gob{1}{B}\gob{1}B\gob{1}{B}
\gend =
\gbeg{3}{5}
\got{2}{B}\gnl
\gcmu\gnl
\gcl{2}\gcn{1}{1}{1}{2}\gnl
\gvac{1}\gcmu\gnl
\gob{1}B\gob{1}{B}\gob{1}{B}
\gend 
\hspace{2mm},
\nonumber\\
&&\eqlabel{braidedbialgebra}\\
&&
\gbeg{2}{4}
\got{2}{B}\gnl
\gcmu\gnl
\gcl{1}\gcu{1}\gnl
\gob{1}{B}
\gend =
\gbeg{1}{3}
\got{1}{B}\gnl
\gcl{1}\gnl
\gob{1}{B}
\gend =
\gbeg{2}{4}
\got{2}{B}\gnl
\gcmu\gnl
\gcu{1}\gcl{1}\gnl
\gob{3}{B}
\gend
\hspace{1mm},\hspace{1mm}
\gbeg{2}{5}
\got{1}{B}\got{1}{B}\gnl
\gmu\gnl
\gcn{1}{1}{2}{1}\gnl
\gcu{1}\gnl
\gob{2}{\un{1}}
\gend
=
\gbeg{2}{3}
\got{1}{B}\got{1}{B}\gnl
\gcu{1}\gcu{1}\gnl
\gob{2}{\un{1}}
\gend
\hspace{1mm},\hspace{1mm}
\gbeg{2}{5}
\got{2}{\un{1}}\gnl
\gu{1}\gnl
\gcn{1}{1}{1}{2}\gnl
\gcmu\gnl
\gob{1}{B}\gob{1}{B}
\gend
=
\gbeg{2}{3}
\got{2}{\un{1}}\gnl
\gu{1}\gu{1}\gnl
\gob{1}{B}\gob{1}{B}
\gend
\hspace{1mm},\hspace{1mm}
\gbeg{2}{4}
\got{1}{B}\got{1}{B}\gnl
\gmu\gnl
\gcmu\gnl
\gob{1}{B}\gob{1}{B}
\gend = 
\gbeg{4}{5}
\got{2}{B}\got{2}{B}\gnl
\gcmu\gcmu\gnl
\gcl{1}\gbr\gcl{1}\gnl
\gmu\gmu\gnl
\gob{2}{B}\gob{2}{B}
\gend
\mbox{\hspace{2mm}.}
\nonumber
\end{eqnarray}
For a bialgebra $B$, we can introduce the category of left $B$-modules ${}_B\Cc$ and the category
of left $B$-comodules ${}^B\Cc$. The left $B$-action on $X\in {}_B\Cc$ is denoted by
$
\gbeg{2}{3}
\got{1}{B}\got{1}{X}\gnl
\glm\gnl
\gvac{1}\gob{1}{X}
\gend
$, and the left $B$-coaction on $X\in {}^B\Cc$ by
$
\gbeg{2}{3}
\gvac{1}\got{1}{X}\gnl
\glcm\gnl
\gob{1}{B}\gob{1}{X}
\gend
$.
${}_B\Cc$ and ${}^B\Cc$ are monoidal categories; for $X, Y\in
{}_B\Cc$ (resp. ${}^B\Cc$), then $X\ot Y$ is a left $B$-module (resp. 
left $B$-comodule) via the action (resp. coaction)
\[
\gbeg{4}{5}
\got{2}{B}\got{1}{X}\got{1}{Y}\gnl
\gcmu\gcl{1}\gcl{1}\gnl
\gcl{1}\gbr\gcl{1}\gnl
\glm\glm\gnl
\gvac{1}\gob{1}{X}\gvac{1}\gob{1}{Y}
\gend
\hspace{3mm}
\left(\mbox{resp.}
\hspace{3mm}
\gbeg{4}{5}
\gvac{1}\got{1}{X}\gvac{1}\got{1}{Y}\gnl
\glcm\glcm\gnl
\gcl{1}\gbr\gcl{1}\gnl
\gmu\gcl{1}\gcl{1}\gnl
\gob{2}{B}\gob{1}{X}\gob{1}{Y}
\gend
\right)
\hspace{2mm}.
\]
A Hopf algebra in a braided monoidal category $\Cc$ is a bialgebra 
$B$ in $\Cc$ together with a morphism $\un {S}: B\ra B$ in $\Cc$ 
(the antipode) satisfying the axioms
\begin{equation}\eqlabel{braidedantipode}
\gbeg{2}{5}
\got{2}{B}\gnl
\gcmu\gnl
\gmp{\un{S}}\gcl{1}\gnl
\gmu\gnl
\gob{2}{B}
\gend =
\gbeg{1}{4}
\got{1}{B}\gnl
\gcu{1}\gnl
\gu{1}\gnl
\gob{1}{B}
\gend =
\gbeg{2}{5}
\got{2}{B}\gnl
\gcmu\gnl
\gcl{1}\gmp{\un{S}}\gnl
\gmu\gnl
\gob{2}{B}
\gend
\hspace{2mm}.
\end{equation}

It is well-known, see \cite[Lemma 2.3]{m3}, that the antipode $\un {S}$ of a Hopf algebra 
$B$ in a braided monoidal category $\Cc$ is an anti-algebra and anti-coalgebra morphism, 
in the sense that 
\begin{equation}\eqlabel{antiac}
\mbox{(a)\hspace{3mm}}
\gbeg{2}{5}
\got{1}{B}\got{1}{B}\gnl
\gmu\gnl
\gcn{1}{1}{2}{1}\gnl
\gmp{\un{S}}\gnl
\gob{1}{B}
\gend =
\gbeg{2}{5}
\got{1}{B}\got{1}{B}\gnl
\gbr\gnl
\gmp{\un{S}}\gmp{\un{S}}\gnl
\gmu\gnl
\gob{2}{B}
\gend
\hspace{2mm},\hspace{2mm}
\gbeg{1}{4}
\got{1}{\un{1}}\gnl
\gu{1}\gnl
\gmp{\un{S}}\gnl
\gob{1}{B}
\gend =
\gbeg{1}{3}
\got{1}{\un{1}}\gnl
\gu{1}\gnl
\gob{1}{B}
\gend
\mbox{\hspace{5mm}and\hspace{5mm}(b)\hspace{3mm}}
\gbeg{2}{5}
\got{1}{B}\gnl
\gmp{\un{S}}\gnl
\gcn{1}{1}{1}{2}\gnl
\gcmu\gnl
\gob{1}{B}\gob{1}{B}
\gend =
\gbeg{2}{5}
\got{2}{B}\gnl
\gcmu\gnl
\gmp{\un{S}}\gmp{\un{S}}\gnl
\gbr\gnl
\gob{1}{B}\gob{1}{B}
\gend
\hspace{2mm},\hspace{2mm}
\gbeg{1}{4}
\got{1}{B}\gnl
\gmp{\un{S}}\gnl
\gcu{1}\gnl
\gob{1}{\un{1}}
\gend = 
\gbeg{1}{3}
\got{1}{B}\gnl
\gcu{1}\gnl
\gob{1}{\un{1}}
\gend
\hspace{2mm}.
\end{equation}

\section{Cross product algebras and coalgebras}\selabel{crossprodalgandcoalg}
\setcounter{equation}{0}
Let $A$ and $B$ be algebras and coalgebras in $\Cc$, but not necessarily bialgebras. Consider morphisms
$\psi=
\gbeg{2}{3}
\got{1}{B}\got{1}{A}\gnl
\gbrc\gnl
\gob{1}{A}\gob{1}{B}
\gend:\ B\ot A\to A\ot B
$
and
$\phi=
\gbeg{2}{3}
\got{1}{A}\got{1}{B}\gnl
\gbrbox\gnl
\gob{1}{B}\gob{1}{A}
\gend:\ A\ot B\to B\ot A
$ 
in $\Cc$. $\psi$ and $\phi$ can be used to define a multiplication and a comultiplication on $A\ot B$:
$$\gbeg{4}{4}
\got{1}{A}\got{1}{B}\got{1}{A}\got{1}{B}\gnl
\gcl{1}\gbrc\gcl{1}\gnl
\gmu\gmu\gnl
\gob{2}{A}\gob{2}{B}
\gend\hspace*{1cm};\hspace*{1cm}
\gbeg{4}{4}
\got{2}{A}\got{2}{B}\gnl
\gcmu\gcmu\gnl
\gcl{1}\gbrbox\gvac{2}\gcl{1}\gnl
\gob{1}{A}\gob{1}{B}\gob{1}{A}\gob{1}{B}
\gend
\hspace{2mm}.$$
$A\#_\psi B$ is $A\ot B$ together with the multiplication induced by $\psi$, and with unit map
$\un{\eta}_A\ot\un{\eta}_B$; $A\#^\phi B$ is  $A\ot B$ together with the comultiplication
induced by $\phi$, and with counit map
$\un{\va}_A\ot\un{\va}_B$. If $A\#_\psi B$ is an algebra in $\Cc$, then we say that
$A\#_\psi B$ is a cross product algebra of $A$ and $B$; if  $A\#^\phi B$ is a coalgebra in $\Cc$, then we say that
$A\#^\phi B$ is a cross product coalgebra of $A$ and $B$. $A\ot B$ together with the multiplication
induced by $\psi$, the comultiplication induced by $\phi$, unit $\un{\eta}_A\ot\un{\eta}_B$ and counit
$\un{\va}_A\ot\un{\va}_B$ will be denoted by $A\#_\psi^\phi B$.
If $A\#_\psi^\phi B$ is a bialgebra in $\Cc$, then we 
call it a cross product bialgebra.\\
It is known, see for example \cite[Theorem 2.5]{cimz} in the case where $\Cc$ is the category
of vector spaces, that $A\#_\psi B$ is a cross product algebra if and only
the four following relations hold: 
\begin{eqnarray}
&&\nonumber 
\hspace{2mm}\mbox{(a)}\hspace{2mm}
\gbeg{3}{5}
\got{1}{B}\got{1}{B}\got{1}{A}\gnl
\gmu\gcl{1}\gnl
\gcn{1}{1}{2}{3}\gvac{1}\gcl{1}\gnl
\gvac{1}\gbrc\gnl
\gvac{1}\gob{1}{A}\gob{1}{B}
\gend = 
\gbeg{3}{5}
\got{1}{B}\got{1}{B}\got{1}{A}\gnl
\gcl{1}\gbrc\gnl
\gbrc\gcl{1}\gnl
\gcl{1}\gmu\gnl
\gob{1}{A}\gob{2}{B}
\gend
\hspace{2mm},\hspace{2mm}\mbox{(b)}\hspace{2mm}
\gbeg{3}{5}
\got{1}{B}\got{1}{A}\got{1}{A}\gnl
\gcl{1}\gmu\gnl
\gcl{1}\gcn{1}{1}{2}{1}\gnl
\gbrc\gnl
\gob{1}{A}\gob{1}{B}
\gend =
\gbeg{3}{5}
\got{1}{B}\got{1}{A}\got{1}{A}\gnl
\gbrc\gcl{1}\gnl
\gcl{1}\gbrc\gnl
\gmu\gcl{1}\gnl
\gob{2}{A}\gob{1}{B}
\gend \\
&&\eqlabel{crossprodalg}\\
&&\nonumber
\hspace{2mm}\mbox{(c)}\hspace{2mm}
\gbeg{2}{4}
\got{1}{B}\gnl
\gcl{1}\gu{1}\gnl
\gbrc\gnl
\gob{1}{A}\gob{1}{B}
\gend =
\gbeg{2}{3}
\got{3}{B}\gnl
\gu{1}\gcl{1}\gnl
\gob{1}{A}\gob{1}{B}
\gend
\hspace{2mm},\hspace{2mm}\mbox{(d)}\hspace{2mm}
\gbeg{2}{4}
\got{3}{A}\gnl
\gu{1}\gcl{1}\gnl
\gbrc\gnl
\gob{1}{A}\gob{1}{B}
\gend =
\gbeg{2}{3}
\got{1}{A}\gnl
\gcl{1}\gu{1}\gnl
\gob{1}{A}\gob{1}{B} 
\gend
\hspace{1mm}.
\end{eqnarray}
This can be restated in the language of monoidal categories. For an algebra $A$ in $\Cc$, we
consider the category ${\mf T}_A$ of right transfer morphisms through $A$.
The objects are pairs $(X, \psi_{X,A})$ with $X\in \Cc$ 
and $\psi_{X,A}:\ X\ot A\ra A\ot X$ a morphism in $\Cc$ such that 
\[
\gbeg{3}{5}
\got{1}{X}\got{1}{A}\got{1}{A}\gnl
\gbrc\gcl{1}\gnl
\gcl{1}\gbrc\gnl
\gmu\gcl{1}\gnl
\gob{2}{A}\gob{1}{X}
\gend =
\gbeg{3}{5}
\got{1}{X}\got{1}{A}\got{1}{A}\gnl
\gcl{2}\gmu\gnl
\gvac{1}\gcn{1}{1}{2}{1}\gnl
\gbrc\gnl
\gob{1}{A}\gob{1}{X}
\gend
\mbox{~~and~~}
\gbeg{2}{4}
\got{1}{X}\gnl
\gcl{1}\gu{1}\gnl
\gbrc\gnl
\gob{1}{A}\gob{2}{X}
\gend =
\gbeg{2}{3}
\got{3}{X}\gnl
\gu{1}\gcl{1}\gnl
\gob{1}{A}\gob{1}{X}
\gend
\hspace{2mm}.
\]
A morphism in ${\mf T}_A$ between $(X, \psi_{X, A})$ and $(Y, \psi_{Y, A})$ 
is a morphism $\mu:\ X\ra Y$ in $\Cc$ such that 
$(\Id_A\ot \mu)\circ \psi_{X, A}=\psi_{Y, A}\circ (\mu\ot\Id_A)$.
${\mf T}_A$ is a strict monoidal category, with unit object $(\un{1}, \Id_A)$ and tensor product
\[
(X, \psi_{X, A})\un{\ot}(Y, \psi_{Y, A})=(X\ot Y, \psi_{X\ot Y, A}),~~{\rm with}~~ 
\psi_{X\ot Y, A}:=
\gbeg{3}{4}
\got{1}{X}\got{1}{Y}\got{1}{A}\gnl
\gcl{1}\gbrc\gnl
\gbrc\gcl{1}\gnl
\gob{1}{A}\gob{1}{X}\gob{1}{Y}
\gend\hspace{2mm}.
\]
The category ${}_A{\mf T}$ of left transfer morphisms throught $A$ is defined in a similar way,
and is also a strict monoidal category. Then we have the following result, going back to \cite{Tambara},
see also \cite[Sec. 4]{Schauenburg2003}.

\begin{proposition}\prlabel{spa}
Let $A$ and $B$ be algebras in a strict monoidal category $\Cc$, and
$\psi: B\ot A\ra A\ot B$ a morphism in $\Cc$. Then the following assertions are equivalent:
\begin{itemize}
\item[(i)] $A\#_\psi B$ is a cross product algebra;
\item[(ii)] $(B, \psi)$ is an algebra in ${\mf T}_A$;
\item[(iii)] $(A, \psi)$ is an algebra in ${}_B{\mf T}$.
\end{itemize}
\end{proposition}

\begin{proof}
Observe $(B, \psi)\in {\mf T}_A$ is equivalent to (\ref{eq:crossprodalg}.b-c);
if these hold, then (\ref{eq:crossprodalg}.a) and (\ref{eq:crossprodalg}.d)
mean precisely that $(B, \psi)$ is an algebra 
in ${\mf T}_A$. This proves the equivalence of (i) and (ii). The equivalence of (i)
and (iii) can be proved in a similar way: (\ref{eq:crossprodalg}.a) and (\ref{eq:crossprodalg}.d)
 are equivalent to $(A, \psi)\in {}_B{\mf T}$, and then the
two other conditions mean that $(A, \psi)$ is an algebra in ${}_B{\mf T}$.
\end{proof}

Recall that an augmented algebra is a pair $(B,\un{\varepsilon}_B)$, where $B$ is an algebra,
and $\un{\varepsilon}_B:\ B\to \un{1}$ is an algebra morphism.

\begin{lemma}\lelabel{action}
Let $A\#_\psi B$ be a cross product algebra. If $(B,\un{\varepsilon}_B)$ is an augmented
algebra then $A\in {}_B\Cc$ via \equuref{action1}{a}.
 If $(A,\un{\varepsilon}_A)$ is an augmented
algebra then $B\in\Cc_A$ via \equuref{action1}{b}.
\begin{equation}\eqlabel{action1}
(a)~~~\gbeg{2}{3}
\got{1}{B}\got{1}{A}\gnl
\glm\gnl
\gvac{1}\gob{1}{A}
\gend
:=
\gbeg{2}{4}
\got{1}{B}\got{1}{A}\gnl
\gbrc\gnl
\gcl{1}\gcu{1}\gnl
\gob{1}{A}
\gend ~~~;~~~(b)~~~
\gbeg{2}{3}
\got{1}{B}\got{1}{A}\gnl
\grm\gnl
\gob{1}{B}
\gend
:=
\gbeg{2}{4}
\got{1}{B}\got{1}{A}\gnl
\gbrc\gnl
\gcu{1}\gcl{1}\gnl
\gvac{1}\gob{1}{B}
\gend 
\end{equation}
\end{lemma}

\begin{proof}
Composing (\ref{eq:crossprodalg}.a) and (\ref{eq:crossprodalg}.d)
to the left with $\Id_A\ot \un{\va}_B$, we find that $A$ is a left $B$-module.
\end{proof}

For further reference, we record the dual results. We leave it to the reader to introduce the monoidal categories 
${}^A{\mf T}$ and ${\mf T}^A$ of left and right transfer morphisms
through the coalgebra $A$.

\begin{proposition}\prlabel{spadual}
Let $A$ and $B$ be coalgebras, and let $\phi:\ A\ot B\to B\ot A$ be a morphism in $\Cc$.
Then
the following statements are equivalent:\\
1) $A\#^\phi B$ is a cross product coalgebra;\\
2) the following relations hold:
\begin{eqnarray}
&&\nonumber 
\hspace{2mm}\mbox{(a)}\hspace{2mm}
\gbeg{3}{5}
\gvac{1}\got{1}{A}\got{1}{B}\gnl
\gvac{1}\gbrbox\gnl
\gvac{1}\gcn{1}{1}{1}{0}\gcl{1}\gnl
\gcmu\gcl{1}\gnl
\gob{1}{B}\gob{1}{B}\gob{1}{A}
\gend =
\gbeg{3}{5}
\got{1}{A}\got{2}{B}\gnl
\gcl{1}\gcmu\gnl
\gbrbox\gvac{2}\gcl{1}\gnl
\gcl{1}\gbrbox\gnl
\gob{1}{B}\gob{1}{B}\gob{1}{A}\gnl
\gend
\hspace{2mm},\hspace{2mm}\mbox{(b)}\hspace{2mm}
\gbeg{3}{5}
\got{1}{A}\got{1}{B}\gnl
\gbrbox\gnl
\gcl{1}\gvac{1}\gcn{1}{1}{-1}{0}\gnl
\gcl{1}\gcmu\gnl
\gob{1}{B}\gob{1}{A}\gob{1}{A}
\gend =
\gbeg{3}{5}
\got{2}{A}\got{1}{B}\gnl
\gcmu\gcl{1}\gnl
\gcl{1}\gbrbox\gnl
\gbr\gvac{-1}\gnot{\hspace*{-4mm}\Box}\gvac{1}\gcl{1}\gnl
\gob{1}{B}\gob{1}{A}\gob{1}{A}
\gend\hspace{1mm},\\
&&\eqlabel{crossprodcoalg}\\
&&\nonumber
\hspace{2mm}\mbox{(c)}\hspace{2mm}
\gbeg{2}{4}
\got{1}{A}\got{1}{B}\gnl
\gbrbox\gnl
\gcl{1}\gcu{1}\gnl
\gob{1}{B}
\gend =
\gbeg{2}{3}
\got{1}{A}\got{1}{B}\gnl
\gcu{1}\gcl{1}\gnl
\gvac{1}\gob{1}{B}
\gend
\hspace{2mm},\hspace{2mm}\mbox{(d)}\hspace{2mm}
\gbeg{2}{4}
\got{1}{A}\got{1}{B}\gnl
\gbrbox\gnl
\gcu{1}\gcl{1}\gnl
\gvac{1}\gob{1}{A}
\gend =
\gbeg{2}{3}
\got{1}{A}\got{1}{B}\gnl
\gcl{1}\gcu{1}\gnl
\gob{1}{A}
\gend\hspace{1mm};
\end{eqnarray}
3) $(B, \phi)$ is a coalgebra in ${}^A{\mf T}$;\\
4) $(A, \phi)$ is a coalgebra in ${\mf T}^B$.
\end{proposition}

A coaugmented coalgebra is a pair $(B,\un{\eta}_B)$, where $B$ is a coalgebra,
and $\un{\eta}_B:\ \un{1}\to B$ is a coalgebra morphism.

\begin{lemma}\lelabel{coaction}
Assume that $A\#^\phi B$ is a cross product coalgebra. If $(B,\un{\eta}_B)$ is a
coaugmented coalgebra, then $A\in {}^B\Cc$ via \equuref{coaction1}{a}.
If $(A,\un{\eta}_A)$ is a
coaugmented coalgebra, then $B\in \Cc^A$ via \equuref{coaction1}{b}.
\begin{equation}\eqlabel{coaction1}
(a)~~~\gbeg{2}{3}
\gvac{1}\got{1}{A}\gnl
\glcm\gnl
\gob{1}{B}\gob{1}{A}
\gend
:=
\gbeg{2}{4}
\got{1}{A}\gnl
\gcl{1}\gu{1}\gnl
\gbrbox\gnl
\gob{1}{B}\gob{1}{A}
\gend~~~;~~~
(b)~~~\gbeg{2}{3}
\got{1}{B}\gnl
\grcm\gnl
\gob{1}{B}\gob{1}{A}
\gend
:=
\gbeg{2}{4}
\gvac{1}\got{1}{A}\gnl
\gu{1}\gcl{1}\gnl
\gbrbox\gnl
\gob{1}{B}\gob{1}{A}
\gend.
\end{equation}
\end{lemma}

\section{Smash product algebras and coalgebras}\selabel{smash}
\setcounter{equation}{0}
These are particular examples of cross product
algebras and coalgebras. Assume that  $B$ is a bialgebra, so that ${}_B\Cc$, the category of left 
$B$-representations, and ${}^B\Cc$, the category of left $B$-corepresentations, 
are monoidal categories.\\
For an algebra $A$ in ${}_B\Cc$, we have a cross product algebra
$A\#_\psi B$, with $
\psi =\gbeg{3}{5}
\got{2}{B}\got{1}{A}\gnl
\gcmu\gcl{1}\gnl
\gcl{1}\gbr\gnl
\glm\gcl{1}\gnl
\gvac{1}\gob{1}{A}\gob{1}{B}
\gend
$,
where the left $B$-action on $A$ is $
\gbeg{2}{3}
\got{1}{B}\got{1}{A}\gnl
\glm\gnl
\gvac{1}\gob{1}{A}
\gend
$. This algebra is called the left smash product algebra of $A$ and $B$.\\
In a similar way, for a coalgebra $A$ in ${}_B\Cc$, we have a cross product
coalgebra $A\#^\phi B$, with
$
\phi=\gbeg{3}{5}
\gvac{1}\got{1}{A}\got{1}{B}\gnl
\glcm\gcl{1}\gnl
\gcl{1}\gbr\gnl
\gmu\gcl{1}\gnl
\gob{2}{B}\gob{1}{A}
\gend
$, where the left $B$-coaction on $A$ is
$
\gbeg{2}{3}
\gvac{1}\got{1}{A}\gnl
\glcm\gnl
\gob{1}{B}\gob{1}{A}
\gend
$. This coalgebra is called a left smash product coalgebra.\\
We remark that right smash product algebras and coalgebras can be considered as well.\\
Assume that $B$ is a bialgebra, and that $A\#_\psi B$ is a cross product algebra.
In \prref{crossprodissmashprodalg}, we discuss when $A\#_\psi B$ is a
smash product algebra.

\begin{proposition}\prlabel{crossprodissmashprodalg}
Let $B$ be a bialgebra, and let $A$ be an algebra, and consider $\psi: B\ot A\to
A\ot B$ such that $A\#_\psi B$ is a cross product algebra.
$A\#_\psi B$ is a left smash product algebra if and only if 
\begin{equation}\eqlabel{psismashprod}
\psi=\gbeg{3}{6}
\got{2}{B}\got{1}{A}\gnl
\gcmu\gcl{1}\gnl
\gcl{1}\gbr\gnl
\gbrc\gcl{1}\gnl
\gcl{1}\gcu{1}\gcl{1}\gnl
\gob{1}{A}\gvac{1}\gob{1}{B}
\gend\hspace{1mm}.
\end{equation}
Moreover, the full subcategory ${}_B{\mf T}'$ of ${}_B{\mf T}$, with objects of the form $(X,\psi)$, where $\psi$
satisfies \equref{psismashprod}, with $A$ replaced by $X$, is a monoidal subcategory of ${}_B{\mf T}$ that is monoidal
isomorphic to ${}_B\Cc$.
\end{proposition}

\begin{proof}
Assume first that
$A\#_\psi B$ is a smash product algebra. Then $A$ is an algebra in ${}_B\Cc$ and 
\begin{equation}\eqlabel{x1}
\gbeg{4}{4}
\got{1}{A}\got{1}{B}\got{1}{A}\got{1}{B}\gnl
\gcl{1}\gbrc\gcl{1}\gnl
\gmu\gmu\gnl
\gob{2}{A}\gob{2}{B}
\gend
= 
\gbeg{5}{7}
\got{1}{A}\got{2}{B}\got{1}{A}\got{1}{B}\gnl
\gcl{1}\gcmu\gcl{1}\gcl{1}\gnl
\gcl{1}\gcl{1}\gbr\gcl{1}\gnl
\gcl{1}\glm\gcl{1}\gcl{1}\gnl
\gcl{1}\gvac{1}\gcn{1}{1}{1}{-1}\gcl{1}\gcl{1}\gnl
\gmu\gvac{1}\gmu\gnl
\gob{2}{A}\gvac{1}\gob{2}{B}
\gend\hspace{1mm}.
\end{equation}
Composing \equref{x1} to the right with $\un{\eta}_A\ot \Id_{B\ot A}\ot \un{\eta}_B$,
we obtain that
\[
\gbeg{2}{3}
\got{1}{B}\got{1}{A}\gnl
\gbrc\gnl
\gob{1}{A}\gob{1}{B}
\gend
= 
\gbeg{3}{5}
\got{2}{B}\got{1}{A}\gnl
\gcmu\gcl{1}\gnl
\gcl{1}\gbr\gnl
\glm\gcl{1}\gnl
\gvac{1}\gob{1}{A}\gob{1}{B}
\gend
\hspace{2mm},\mbox{ hence }\hspace{2mm}
\gbeg{2}{4}
\got{1}{B}\got{1}{A}\gnl
\gbrc\gnl
\gcl{1}\gcu{1}\gnl
\gob{1}{A}\gob{1}{B}
\gend 
=
\gbeg{2}{3}
\got{1}{B}\got{1}{A}\gnl
\glm\gnl
\gvac{1}\gob{1}{A}
\gend~,
\]
and this implies \equref{psismashprod}.\\
Conversely, assume that $A\#_\psi B$ is a cross product algebra
and that $\psi$ satisfies \equref{psismashprod}. We know that $A\in 
{}_B\Cc$, with left $B$-action (\ref{eq:action1}.a).
$A$ is an algebra in ${}_B\Cc$ since 
\[
\gbeg{4}{8}
\got{2}{B}\got{1}{A}\got{1}{A}\gnl
\gcmu\gcl{1}\gcl{1}\gnl
\gcl{1}\gbr\gcl{1}\gnl
\gbrc\gbrc\gnl
\gcl{1}\gcu{1}\gcl{1}\gcu{1}\gnl
\gcn{1}{1}{1}{3}\gvac{1}\gcl{1}\gnl
\gvac{1}\gmu\gnl
\gvac{1}\gob{2}{A}
\gend
\hspace{2mm}
\equal{\equref{psismashprod}}
\hspace{2mm}
\gbeg{3}{5}
\got{1}{B}\got{1}{A}\got{1}{A}\gnl
\gbrc\gcl{1}\gnl
\gcl{1}\gbrc\gnl
\gmu\gcu{1}\gnl
\gob{2}{A}
\gend
\hspace{2mm}
\equal{\equref{crossprodalg}}
\hspace{2mm}
\gbeg{3}{6}
\got{1}{B}\got{1}{A}\got{1}{A}\gnl
\gcl{1}\gmu\gnl
\gcl{1}\gcn{1}{1}{2}{1}\gnl
\gbrc\gnl
\gcl{1}\gcu{1}\gnl
\gob{1}{A}
\gend
\hspace{2mm}\mbox{and}\hspace{2mm}
\gbeg{2}{5}
\got{1}{B}\gnl
\gcl{1}\gu{1}\gnl
\gbrc\gnl
\gcl{1}\gcu{1}\gnl
\gob{1}{A}
\gend
\hspace{2mm}
\equal{\equref{crossprodalg}}
\hspace{2mm}
\gbeg{2}{3}
\gvac{1}\got{1}{B}\gnl
\gu{1}\gcu{1}\gnl
\gob{1}{A}
\gend
\hspace{1mm}.
\]
The multiplication on the smash product algebra is 
\[
\gbeg{5}{7}
\got{1}{A}\got{2}{B}\got{1}{A}\got{1}{B}\gnl
\gcl{1}\gcmu\gcl{1}\gcl{1}\gnl
\gcl{1}\gcl{1}\gbr\gcl{1}\gnl
\gcl{1}\glm\gmu\gnl
\gcl{1}\gvac{1}\gcn{1}{1}{1}{-1}\gcn{1}{2}{2}{2}\gnl
\gmu\gnl
\gob{2}{A}\gvac{1}\gob{2}{B}
\gend
\hspace{2mm}=\hspace{2mm}
\gbeg{5}{6}
\got{1}{A}\got{2}{B}\got{1}{A}\got{1}{B}\gnl
\gcl{1}\gcmu\gcl{1}\gcl{1}\gnl
\gcl{1}\gcl{1}\gbr\gcl{1}\gnl
\gcl{1}\gbrc\gcl{1}\gcl{1}\gnl
\gmu\gcu{1}\gmu\gnl
\gob{2}{A}\gvac{1}\gob{2}{B}
\gend
\hspace{2mm}
\equal{\equref{psismashprod}}
\hspace{2mm}
\gbeg{4}{4}
\got{1}{A}\got{1}{B}\got{1}{A}\got{1}{B}\gnl
\gcl{1}\gbrc\gcl{1}\gnl
\gmu\gmu\gnl
\gob{2}{A}\gob{2}{B}
\gend
\hspace{1mm},
\]
and coincides with the multiplication on the cross product algebra $A\#_\psi B$.
This finishes the proof of the first statement.\\
We next show that ${}_B{\mf T}'$ is closed under the tensor product:
if $(X, \psi_{B, X}),~(Y, \psi_{B, Y})\in {}_B{\mf T}'$, then
$(X, \psi_{B, X})\un{\ot}(Y, \psi_{B, Y})\in {}_B{\mf T}'$, since
\[
\gbeg{4}{8}
\got{2}{B}\got{1}{X}\got{1}{Y}\gnl
\gcmu\gcl{1}\gcl{1}\gnl
\gcl{1}\gbr\gcl{1}\gnl
\gcl{1}\gcl{1}\gbr\gnl
\gbrc\gcl{1}\gcl{1}\gnl
\gcl{1}\gbrc\gcl{1}\gnl
\gcl{1}\gcl{1}\gcu{1}\gcl{1}\gnl
\gob{1}{X}\gob{1}{Y}\gvac{1}\gob{1}{B}
\gend
\hspace{1mm}
\equal{\equref{psismashprod}}
\hspace{1mm}
\gbeg{5}{8}
\gvac{1}\got{2}{B}\got{1}{X}\got{1}{Y}\gnl
\gvac{1}\gcmu\gcl{1}\gcl{1}\gnl
\gvac{1}\gcn{1}{1}{1}{0}\gbr\gcl{1}\gnl
\gcmu\gcl{1}\gbr\gnl
\gcl{1}\gbr\gcl{1}\gcl{1}\gnl
\gbrc\gbrc\gcl{1}\gnl
\gcl{1}\gcu{1}\gcl{1}\gcu{1}\gcl{1}\gnl
\gob{1}{X}\gvac{1}\gob{1}{Y}\gvac{1}\gob{1}{B}
\gend
\hspace{1mm}
=
\hspace{1mm}
\gbeg{5}{9}
\got{2}{B}\gvac{1}\got{1}{X}\got{1}{Y}\gnl
\gcmu\gvac{1}\gcl{1}\gcl{1}\gnl
\gcl{1}\gcn{1}{1}{1}{2}\gvac{1}\gcl{1}\gcl{1}\gnl
\gcl{1}\gcmu\gcl{1}\gcl{1}\gnl
\gcl{1}\gcl{1}\gbr\gcl{1}\gnl
\gcl{1}\gbr\gbr\gnl
\gbrc\gbrc\gcl{1}\gnl
\gcl{1}\gcu{1}\gcl{1}\gcu{1}\gcl{1}\gnl
\gob{1}{X}\gvac{1}\gob{1}{Y}\gvac{1}\gob{1}{B}
\gend
\hspace{1mm}
\equal{\equref{nat2cup}}
\hspace{1mm}
\gbeg{5}{9}
\got{2}{B}\got{1}{X}\gvac{1}\got{1}{Y}\gnl
\gcmu\gcl{1}\gvac{1}\gcl{1}\gnl
\gcl{1}\gbr\gvac{1}\gcl{1}\gnl
\gcl{1}\gcl{1}\gcn{1}{1}{1}{2}\gvac{1}\gcl{1}\gnl
\gcl{1}\gcl{1}\gcmu\gcl{1}\gnl
\gbrc\gcl{1}\gbr\gnl
\gcl{1}\gcu{1}\gbrc\gcl{1}\gnl
\gcl{1}\gvac{1}\gcl{1}\gcu{1}\gcl{1}\gnl
\gob{1}{X}\gvac{1}\gob{1}{Y}\gvac{1}\gob{1}{B}
\gend
\hspace{1mm}
\equal{\equref{psismashprod}}
\hspace{1mm}
\gbeg{3}{4}
\got{1}{B}\got{1}{X}\got{1}{Y}\gnl
\gbrc\gcl{1}\gnl
\gcl{1}\gbrc\gnl
\gob{1}{X}\gob{1}{Y}\gob{1}{B}
\gend
\hspace{1mm}.
\]
Finally, we will construct a monoidal isomorphism $F:\ 
{}_B{\mf T}'\to {}_B\Cc$. Take $(X,\psi)\in {}_B{\mf T}'$.
In the first part of the proof, we have seen that $X\in {}_B\Cc$ via
the $B$-action
$
\gbeg{2}{3}
\got{1}{B}\got{1}{X}\gnl
\glm\gnl
\gvac{1}\gob{1}{X}
\gend
=
\gbeg{2}{4}
\got{1}{B}\got{1}{X}\gnl
\gbrc\gnl
\gcl{1}\gcu{1}\gnl
\gob{1}{X}\gob{1}{}
\gend 
$,
and this defines $F$ at the level of objects. At the level of morphisms, $F$ acts
as the identity.
Now we define a functor $G:\ {}_B\Cc\to {}_B{\mf T}'$. Take a left $B$-module $X$, and
let
$\psi=\gbeg{2}{3}
\got{1}{B}\got{1}{X}\gnl
\gbrc\gnl
\gob{1}{X}\gob{1}{B}
\gend 
=:
\gbeg{3}{5}
\got{2}{B}\got{1}{X}\gnl
\gcmu\gcl{1}\gnl
\gcl{1}\gbr\gnl
\glm\gcl{1}\gnl
\gvac{1}\gob{1}{X}\gob{1}{B}
\gend  
$. Then $
\gbeg{2}{4}
\got{1}{B}\got{1}{X}\gnl
\gbrc\gnl
\gcl{1}\gcu{1}\gnl
\gob{1}{X}
\gend
=
\gbeg{2}{3}
\got{1}{B}\got{1}{X}\gnl
\glm\gnl
\gvac{1}\gob{1}{X}
\gend
$, and therefore $\psi$ satisfies \equref{psismashprod}. $(X, \psi_{B, X})$ is an object of ${}_B{\mf T}$ since
\[
\gbeg{3}{6}
\got{1}{B}\got{1}{B}\got{1}{X}\gnl
\gmu\gcl{1}\gnl
\gcmu\gcl{1}\gnl
\gcl{1}\gbr\gnl
\glm\gcl{1}\gnl
\gvac{1}\gob{1}{X}\gob{1}{B}
\gend
=
\gbeg{5}{8}
\got{2}{B}\got{2}{B}\got{1}{X}\gnl
\gcmu\gcmu\gcl{1}\gnl
\gcl{1}\gbr\gcl{1}\gcl{1}\gnl
\gmu\gmu\gcl{1}\gnl
\gcn{1}{1}{2}{5}\gvac{2}\gcn{1}{1}{0}{1}\gcl{1}\gnl
\gvac{2}\gcl{1}\gbr\gnl
\gvac{2}\glm\gcl{1}\gnl
\gvac{3}\gob{1}{X}\gob{1}{B}
\gend
=
\gbeg{5}{10}
\got{2}{B}\got{2}{B}\got{1}{X}\gnl
\gcmu\gcmu\gcl{1}\gnl
\gcl{1}\gbr\gcl{1}\gcl{1}\gnl
\gcl{1}\gcl{1}\gmu\gcl{1}\gnl
\gcl{1}\gcl{1}\gvac{1}\gcn{1}{1}{0}{-1}\gcn{1}{1}{1}{-1}\gnl
\gcl{1}\gcl{1}\gbr\gnl
\gcl{1}\glm\gcl{1}\gnl
\gcl{1}\gvac{1}\gcn{1}{1}{1}{-1}\gcl{1}\gnl
\glm\gvac{1}\gcl{1}\gnl
\gvac{1}\gob{1}{X}\gvac{1}\gob{1}{B}
\gend
\equal{\equref{nat1cup}}
\gbeg{5}{8}
\got{2}{B}\got{2}{B}\got{1}{X}\gnl
\gcmu\gcmu\gcl{1}\gnl
\gcl{1}\gbr\gcl{1}\gcl{1}\gnl
\gcl{1}\gcl{1}\gcl{1}\gbr\gnl
\gcl{1}\gcl{1}\gbr\gcl{1}\gnl
\gcn{1}{1}{1}{3}\glm\gmu\gnl
\gvac{1}\glm\gcn{1}{1}{2}{2}\gnl
\gvac{2}\gob{1}{X}\gob{2}{B}
\gend
\equal{\equref{nat1cup}}
\gbeg{5}{8}
\got{2}{B}\got{2}{B}\got{1}{X}\gnl
\gcmu\gcmu\gcl{1}\gnl
\gcl{1}\gcl{1}\gcl{1}\gbr\gnl
\gcl{1}\gcl{1}\glm\gcl{1}\gnl
\gcl{1}\gcl{1}\gvac{1}\gcn{1}{1}{1}{-1}\gcn{1}{1}{1}{-1}\gnl
\gcl{1}\gbr\gcl{1}\gnl
\glm\gmu\gnl
\gvac{1}\gob{1}{X}\gob{2}{B}
\gend
\] 
and 
$$
\gbeg{3}{7}
\gvac{2}\got{1}{X}\gnl
\gvac{1}\gu{1}\gcl{2}\gnl
\gvac{1}\gcn{1}{1}{1}{0}\gnl
\gcmu\gcl{1}\gnl
\gcl{1}\gbr\gnl
\glm\gcl{1}\gnl
\gvac{1}\gob{1}{X}\gob{1}{B}
\gend  
=\gbeg{3}{4}
\gvac{1}\got{1}{X}\gnl
\gu{1}\gcl{1}\gnl
\glm\gu{1}\gnl
\gvac{1}\gob{1}{X}\gob{1}{B}
\gend
=
\gbeg{2}{3}
\got{1}{X}\gnl
\gcl{1}\gu{1}\gnl
\gob{1}{X}\gob{1}{B}
\gend~.$$
We conclude that $(X, \psi_{B, X})\in {}_B{\mf T}'$, and we define $G(X)=(X, \psi_{B, X})$.
At the level of morphisms, $G$ acts as the identity. Using \equref{psismashprod}, we can
show that $F$ and $G$ are inverses. Finally, using the coassociativity of the comultiplication
on $B$ and \equref{nat2cup}, we can prove that
\[
\gbeg{5}{8}
\got{2}{B}\got{1}{X}\got{1}{Y}\gnl
\gcmu\gcl{1}\gcl{1}\gnl
\gcl{1}\gbr\gcl{1}\gnl
\gcl{1}\gcl{1}\gcn{1}{1}{1}{2}\gcn{1}{1}{1}{3}\gnl
\glm\gcmu\gcl{1}\gnl
\gvac{1}\gcl{1}\gcl{1}\gbr\gnl
\gvac{1}\gcl{1}\glm\gcl{1}\gnl
\gvac{1}\gob{1}{X}\gvac{1}\gob{1}{Y}\gob{1}{B}
\gend
=
\gbeg{5}{7}
\gvac{1}\got{2}{B}\got{1}{X}\got{1}{Y}\gnl
\gvac{1}\gcmu\gcl{1}\gcl{1}\gnl
\gvac{1}\gcn{1}{1}{1}{0}\gbr\gcl{1}\gnl
\gcmu\gcl{1}\gbr\gnl
\gcl{1}\gbr\gcl{1}\gcl{1}\gnl
\glm\glm\gcl{1}\gnl
\gvac{1}\gob{1}{X}\gvac{1}\gob{1}{Y}\gob{1}{B}
\gend\hspace{1mm},
\]
and this implies that $F$ is a strictly monoidal functor.
\end{proof}

We end this Section with the dual version of \prref{crossprodissmashprodalg}. Verification of the details
is left to the reader.

\begin{proposition}\prlabel{crossprodissmashprodcoalg}
Let $B$ be a bialgebra, and let $A$ be a coalgebra. Assume that $\phi:\ A\ot B\to B\ot A$ is such that
$A\#^\phi B$ is a cross product coalgebra. $A\#^\phi B$ 
is a smash product coalgebra if and only if 
\begin{equation}\eqlabel{phismashprod}
\phi=
\gbeg{3}{6}
\got{1}{A}\gvac{1}\got{1}{B}\gnl
\gcl{1}\gu{1}\gcl{1}\gnl
\gbrbox\gvac{2}\gcl{1}\gnl
\gcl{1}\gbr\gnl
\gmu\gcl{1}\gnl
\gob{2}{B}\gob{1}{A}
\gend\hspace{1mm}.
\end{equation}
The full subcategory of ${\mf T}^B$ consisting of objects 
$(X, \phi)$ satisfying \equref{phismashprod}, with $A$ replaced by $X$, is strictly monoidal 
and can be identified to $\Cc^B$ as a monoidal category.
\end{proposition}  

\section{Cross product bialgebras}\selabel{crossprodbialg}
\setcounter{equation}{0}
Suppose that $A$ and $B$ are algebras and coalgebras, and that we have morphisms
$\psi:\ B\ot A\to A\ot B$ and $\phi:\ A\ot B\to B\ot A$ such that $A\#_\psi B$ is a cross product
algebra and $A\#^\phi B$ is a cross product coalgebra. Then we will call $(A,B,\psi,\phi)$
a {\sl cross product algebra-coalgebra datum}. In \cite[Sec. 2]{bespdrab1}, 
$(A,B,\psi,\phi)$ is a called a {\sl bialgebra admissible tuple}, or a BAT, if 
$A\#_\psi^\phi B$ is a cross product bialgebra. Take a cross product algebra-coalgebra datum
$(A,B,\psi,\phi)$. We will produce a list of properties that are satisfied if
$(A,B,\psi,\phi)$ is an admissible tuple; otherwise stated, we will make a list of necessary
conditions for  $A\#_\psi^\phi B$ being a cross product bialgebra.
Then we will identify subsets of this list of properties
that guarantee that $(A,B,\psi,\phi)$ is a bialgebra admissible tuple, in other words,
sets of necessary and sufficient conditions for  $A\#_\psi^\phi B$ being a cross product bialgebra.
The results will be summarized in  \thref{firstsetequivcond}.\\

$A\#_\psi^\phi B$ is a cross product bialgebra if and only if the comultiplication and counit
are algebra maps; these conditions come down to the following equalities:
\begin{eqnarray}
&&
\hspace{2mm}\mbox{(a)}\hspace{2mm}
\gbeg{4}{6}
\got{1}{A}\got{1}{B}\got{1}{A}\got{1}{B}\gnl
\gcl{1}\gbrc\gcl{1}\gnl
\gmu\gmu\gnl
\gcmu\gcmu\gnl
\gcl{1}\gbrbox\gvac{2}\gcl{1}\gnl
\gob{1}{A}\gob{1}{B}\gob{1}{A}\gob{1}{B} 
\gend
=
\gbeg{8}{9}
\got{2}{A}\got{2}{B}\got{2}{A}\got{2}{B}\gnl
\gcmu\gcmu\gcmu\gcmu\gnl
\gcl{1}\gbrbox\gvac{2}\gcl{1}\gcl{1}\gbrbox\gvac{2}\gcl{1}\gnl
\gcl{1}\gcl{1}\gcl{1}\gbr\gcl{1}\gcl{1}\gcl{1}\gnl
\gcl{1}\gcl{1}\gbr\gbr\gcl{1}\gcl{1}\gnl
\gcl{1}\gcl{1}\gcl{1}\gbr\gcl{1}\gcl{1}\gcl{1}\gnl
\gcl{1}\gbrc\gcl{1}\gcl{1}\gbrc\gcl{1}\gnl
\gmu\gmu\gmu\gmu\gnl
\gob{2}{A}\gob{2}{B}\gob{2}{A}\gob{2}{B}
\gend 
\hspace{2mm},\hspace{2mm}\mbox{(b)}\hspace{2mm}
\gbeg{4}{6}
\gvac{1}\got{2}{\un{1}}\gnl
\gvac{1}\gu{1}\gvac{1}\gu{1}\gnl
\gvac{1}\gcn{1}{1}{1}{0}\gcn{1}{1}{3}{2}\gnl
\gcmu\gcmu\gnl
\gcl{1}\gbrbox\gvac{2}\gcl{1}\gnl
\gob{1}{A}\gob{1}{B}\gob{1}{A}\gob{1}{B}
\gend
=
\gbeg{4}{3}
\gvac{1}\got{2}{\un{1}}\gnl
\gu{1}\gu{1}\gu{1}\gu{1}\gnl
\gob{1}{A}\gob{1}{B}\gob{1}{A}\gob{1}{B}
\gend
\hspace{1mm},\nonumber \\
&&
\eqlabel{crossbialgcond} \\
&&
\hspace{2mm}\mbox{(c)}\hspace{2mm}
\gbeg{4}{6}
\got{1}{A}\got{1}{B}\got{1}{A}\got{1}{B}\gnl
\gcl{1}\gbrc\gcl{1}\gnl
\gmu\gmu\gnl
\gcn{1}{1}{2}{1}\gvac{1}\gcn{1}{1}{2}{1}\gnl
\gcu{1}\gvac{1}\gcu{1}\gnl
\gvac{1}\gob{2}{\un{1}}
\gend
=
\gbeg{4}{3}
\got{1}{A}\got{1}{B}\got{1}{A}\got{1}{B}\gnl
\gcu{1}\gcu{1}\gcu{1}\gcu{1}\gnl
\gvac{1}\gob{2}{\un{1}}
\gend
\hspace{2mm}\mbox{and}\hspace{2mm}\mbox{(d)}\hspace{2mm}
\gbeg{2}{4}
\got{1}{\un{1}}\got{1}{\un{1}}\gnl
\gu{1}\gu{1}\gnl
\gcu{1}\gcu{1}\gnl
\gob{1}{\un{1}}\gob{1}{\un{1}}
\gend 
=
\gbeg{2}{3}
\got{1}{\un{1}}\got{1}{\un{1}}\gnl
\gcl{1}\gcl{1}\gnl
\gob{1}{\un{1}}\gob{1}{\un{1}}
\gend
\hspace{2mm}.
\nonumber
\end{eqnarray}
Note that the first composition at the left hand side of (\ref{eq:crossbialgcond}.d) is the composition
of the counit and the unit of $A$, and the second one is the 
composition of the counit and the unit of $B$. Using (\ref{eq:crossbialgcond}.d) and
the counit conditions (\ref{eq:crossprodcoalg}.c-d), we can see that (\ref{eq:crossbialgcond}.b)
is equivalent to
\begin{equation}\eqlabel{comultunitcomp}
\mbox{(a)}\hspace{2mm}\gbeg{2}{5}
\got{2}{\un{1}}\gnl
\gu{1}\gnl
\gcn{1}{1}{1}{2}\gnl
\gcmu\gnl
\gob{1}{A}\gob{1}{A}
\gend
=
\gbeg{2}{3}
\got{2}{\un{1}}\gnl
\gu{1}\gu{1}\gnl
\gob{1}{A}\gob{1}{A}
\gend
\hspace{2mm},\hspace{2mm}\mbox{(b)}\hspace{2mm}
\gbeg{2}{5}
\got{2}{\un{1}}\gnl
\gu{1}\gnl
\gcn{1}{1}{1}{2}\gnl
\gcmu\gnl
\gob{1}{B}\gob{1}{B}
\gend
=
\gbeg{2}{3}
\got{2}{\un{1}}\gnl
\gu{1}\gu{1}\gnl
\gob{1}{B}\gob{1}{B}
\gend
\hspace{2mm},\hspace{2mm}\mbox{(c)}\hspace{2mm}
\gbeg{2}{4}
\got{2}{\un{1}}\gnl
\gu{1}\gu{1}\gnl
\gbrbox\gnl
\gob{1}{B}\gob{1}{A}
\gend 
=
\gbeg{2}{3}
\got{2}{\un{1}}\gnl
\gu{1}\gu{1}\gnl
\gob{1}{B}\gob{1}{A}
\gend
\hspace{1mm}.
\end{equation}
In a similar way, (\ref{eq:crossbialgcond}.c) is equivalent to
\begin{equation}\eqlabel{multcounitcomp}
\mbox{(a)}\hspace{2mm}
\gbeg{2}{5}
\got{1}{A}\got{1}{A}\gnl
\gmu\gnl
\gcn{1}{1}{2}{1}\gnl
\gcu{1}\gnl
\gob{2}{\un{1}}
\gend
=
\gbeg{2}{3}
\got{1}{A}\got{1}{A}\gnl
\gcu{1}\gcu{1}\gnl
\gob{2}{\un{1}}
\gend
\hspace{2mm},\hspace{2mm}\mbox{(b)}\hspace{2mm}
\gbeg{2}{5}
\got{1}{B}\got{1}{B}\gnl
\gmu\gnl
\gcn{1}{1}{2}{1}\gnl
\gcu{1}\gnl
\gob{2}{\un{1}}
\gend
=
\gbeg{2}{3}
\got{1}{B}\got{1}{B}\gnl
\gcu{1}\gcu{1}\gnl
\gob{2}{\un{1}}
\gend
\hspace{2mm},\hspace{2mm}\mbox{(c)}\hspace{2mm}
\gbeg{3}{4}
\got{1}{B}\got{1}{A}\gnl
\gbrc\gnl
\gcu{1}\gcu{1}\gnl
\gob{2}{\un{1}}
\gend
=
\gbeg{2}{3}
\got{1}{B}\got{1}{A}\gnl
\gcu{1}\gcu{1}\gnl
\gob{2}{\un{1}}
\gend
\hspace{1mm}.
\end{equation}
(\ref{eq:crossbialgcond}.d) is equivalent to 
$\un{\va}_A\circ \un{\eta}_A=\Id_{\un{1}}=\un{\va}_B\circ \un{\eta}_B$;
this follows from the observation that both compositions are invertible 
idempotents of ${\rm End}_\Cc({\un{1}})$.\\ 

We can now formulate a first list of properties of bialgebra admissible tuples.

\begin{proposition}\prlabel{4.1}
A bialgebra admissible tuple $(A,B,\psi,\phi)$ 
satisfies the following properties.
\begin{eqnarray}
&&
\hspace{2mm}\mbox{\rm (a)}\hspace{2mm}
\gbeg{3}{5}
\got{1}{A}\got{1}{A}\gnl
\gmu\gnl
\gcmu\gu{1}\gnl
\gcl{1}\gbrbox\gnl
\gob{1}{A}\gob{1}{B}\gob{1}{A}
\gend 
=
\gbeg{6}{7}
\got{2}{A}\gvac{1}\got{2}{A}\gnl
\gcmu\gu{1}\gcmu\gu{1}\gnl
\gcl{1}\gbrbox\gvac{2}\gcl{1}\gbrbox\gnl
\gcl{1}\gcl{1}\gbr\gcl{1}\gcl{1}\gnl
\gcl{1}\gbrc\gbr\gcl{1}\gnl
\gmu\gmu\gmu\gnl
\gob{2}{A}\gob{2}{B}\gob{2}{A}
\gend
\hspace{2mm},\hspace{2mm}\mbox{\rm (b)}\hspace{2mm}
\gbeg{3}{5}
\gvac{1}\got{1}{B}\got{1}{B}\gnl
\gvac{1}\gmu\gnl
\gu{1}\gcmu\gnl
\gbrbox\gvac{2}\gcl{1}\gnl
\gob{1}{B}\gob{1}{A}\gob{1}{B}
\gend
=
\gbeg{6}{7}
\gvac{1}\got{2}{B}\gvac{1}\got{2}{B}\gnl
\gu{1}\gcmu\gu{1}\gcmu\gnl
\gbrbox\gvac{2}\gcl{1}\gbrbox\gvac{2}\gcl{3}\gnl
\gcl{1}\gcl{1}\gbr\gcl{1}\gnl
\gcl{1}\gbr\gbrc\gnl
\gmu\gmu\gmu\gnl
\gob{2}{B}\gob{2}{A}\gob{2}{B}
\gend
\hspace{2mm},\nonumber\\
&&
\hspace{2mm}\mbox{\rm (c)}\hspace{2mm}
\gbeg{2}{3}
\got{1}{A}\got{1}{B}\gnl
\gbrbox\gnl
\gob{1}{B}\gob{1}{A}
\gend
=
\gbeg{4}{6}
\got{1}{A}\gvac{2}\got{1}{B}\gnl
\gcl{1}\gu{1}\gu{1}\gcl{1}\gnl
\gbr\gvac{-1}\gnot{\hspace*{-4mm}\Box}\gvac{1}
\gbr\gvac{-1}\gnot{\hspace*{-4mm}\Box}\gnl
\gcl{1}\gbr\gcl{1}\gnl
\gmu\gmu\gnl
\gob{2}{B}\gob{2}{A}
\gend
\hspace{2mm},\hspace{2mm}\mbox{\rm (d)}\hspace{2mm}
\gbeg{4}{6}
\gvac{1}\got{1}{B}\got{1}{A}\gnl
\gvac{1}\gbrc\gnl
\gvac{1}\gcn{1}{1}{1}{0}\gcn{1}{1}{1}{2}\gnl
\gcmu\gcmu\gnl
\gcl{1}\gbrbox\gvac{2}\gcl{1}\gnl
\gob{1}{A}\gob{1}{B}\gob{1}{A}\gob{1}{B}
\gend
=
\gbeg{6}{7}
\gvac{1}\got{2}{B}\got{2}{A}\gnl
\gu{1}\gcmu\gcmu\gu{1}\gnl
\gbrbox\gvac{2}\gbr\gbrbox\gnl
\gcl{1}\gbr\gbr\gcl{1}\gnl
\gbrc\gbr\gbrc\gnl
\gcl{1}\gmu\gmu\gcl{1}\gnl
\gob{1}{A}\gob{2}{B}\gob{2}{A}\gob{1}{B}
\gend
\hspace{2mm},\nonumber\\
&&
\eqlabel{neccconds}\\
&&
\hspace{2mm}\mbox{\rm (e)}\hspace{2mm}
\gbeg{3}{5}
\got{1}{A}\got{1}{B}\got{1}{A}\gnl
\gcl{1}\gbrc\gnl
\gmu\gcu{1}\gnl
\gcmu\gnl
\gob{1}{A}\gob{1}{A}
\gend
=
\gbeg{6}{7}
\got{2}{A}\got{2}{B}\got{2}{A}\gnl
\gcmu\gcmu\gcmu\gnl
\gcl{1}\gbrbox\gvac{2}\gbr\gcl{1}\gnl
\gcl{1}\gcl{1}\gbr\gbrc\gnl
\gcl{1}\gbrc\gmu\gcu{1}\gnl
\gmu\gcu{1}\gcn{1}{1}{2}{2}\gnl
\gob{2}{A}\gvac{1}\gob{2}{A}
\gend
\hspace{2mm},\hspace{2mm}\mbox{\rm (f)}\hspace{2mm}
\gbeg{3}{5}
\got{1}{B}\got{1}{A}\got{1}{B}\gnl
\gbrc\gcl{1}\gnl
\gcu{1}\gmu\gnl
\gvac{1}\gcmu\gnl
\gvac{1}\gob{1}{B}\gob{1}{B}
\gend
=
\gbeg{6}{7}
\got{2}{B}\got{2}{A}\got{2}{B}\gnl
\gcmu\gcmu\gcmu\gnl
\gcl{1}\gbr\gbrbox\gvac{2}\gcl{1}\gnl
\gbrc\gbr\gcl{1}\gcl{1}\gnl
\gcu{1}\gmu\gbrc\gcl{1}\gnl
\gvac{1}\gcn{1}{1}{2}{2}\gvac{1}\gcu{1}\gmu\gnl
\gvac{1}\gob{2}{B}\gvac{1}\gob{2}{B}
\gend
\hspace{2mm},
\nonumber\\
&&
\hspace{2mm}\mbox{\rm (g)}\hspace{2mm}
\gbeg{2}{3}
\got{1}{B}\got{1}{A}\gnl
\gbrc\gnl
\gob{1}{A}\gob{1}{B}
\gend
=
\gbeg{4}{6}
\got{2}{B}\got{2}{A}\gnl
\gcmu\gcmu\gnl
\gcl{1}\gbr\gcl{1}\gnl
\gbrc\gbrc\gnl
\gcl{1}\gcu{1}\gcu{1}\gcl{1}\gnl
\gob{1}{A}\gvac{2}\gob{1}{B}
\gend
\hspace{2mm}\mbox{\rm and}\hspace{2mm}\mbox{\rm (h)}\hspace{2mm}
\gbeg{4}{6}
\got{1}{A}\got{1}{B}\got{1}{A}\got{1}{B}\gnl
\gcl{1}\gbrc\gcl{1}\gnl
\gmu\gmu\gnl
\gcn{1}{1}{2}{3}\gcn{1}{1}{4}{3}\gnl
\gvac{1}\gbrbox\gnl
\gvac{1}\gob{1}{B}\gob{1}{A}
\gend
=\gbeg{6}{7}
\got{1}{A}\got{2}{B}\got{2}{A}\got{1}{B}\gnl
\gcl{1}\gcmu\gcmu\gcl{1}\gnl
\gbrbox\gvac{2}\gbr\gbrbox\gnl
\gcl{1}\gbr\gbr\gcl{1}\gnl
\gbrc\gbr\gbrc\gnl
\gcu{1}\gmu\gmu\gcu{1}\gnl
\gvac{1}\gob{2}{B}\gob{2}{A}
\gend
\hspace{2mm}.
\nonumber
\end{eqnarray}
\end{proposition}

\begin{proof}
Compose (\ref{eq:crossbialgcond}.a) to the right with $\Id_A\ot \un{\eta}_B\ot \Id_{A\ot B}$ and to the left 
with $\Id_{A\ot B\ot A}\ot \un{\va}_B$. Applying (\ref{eq:crossprodalg}.d), (\ref{eq:comultunitcomp}.b)
and (\ref{eq:multcounitcomp}.b), and the fact that 
$\un{\va}_B\un{\eta}_B=\Id_{\un{1}}$ we obtain 
\begin{equation}\eqlabel{4.5}
\gbeg{3}{5}
\got{1}{A}\got{1}{A}\got{1}{B}\gnl
\gmu\gcl{2}\gnl
\gcmu\gnl
\gcl{1}\gbrbox\gnl
\gob{1}{A}\gob{1}{B}\gob{1}{A}
\gend 
=
\gbeg{6}{7}
\got{2}{A}\gvac{1}\got{2}{A}\got{1}{B}\gnl
\gcmu\gu{1}\gcmu\gcl{1}\gnl
\gcl{1}\gbrbox\gvac{2}\gcl{1}\gbrbox\gnl
\gcl{1}\gcl{1}\gbr\gcl{1}\gcl{1}\gnl
\gcl{1}\gbrc\gbr\gcl{1}\gnl
\gmu\gmu\gmu\gnl
\gob{2}{A}\gob{2}{B}\gob{2}{A}
\gend\hspace{1mm}.
\end{equation}
We find (\ref{eq:neccconds}.a) after we compose \equref{4.5} 
to the right with $\Id_{A\ot A}\ot \un{\eta}_B$. Composing \equref{4.5} 
to the right with 
$\Id_A\ot \un{\eta}_A\ot \Id_B$ and to the left with $\un{\va}_A\ot \Id_{B\ot A}$,
and with the help of (\ref{eq:comultunitcomp}.a) and (\ref{eq:crossprodalg}.c),
we deduce (\ref{eq:neccconds}.c).\\
Now compose (\ref{eq:crossbialgcond}.a) to the right with
$\Id_{A\ot B}\ot \un{\eta}_A\ot \Id_B$ and to the left with 
$\un{\va}_A\ot \Id_{B\ot A\ot B}$. Combining the resulting equation with
(\ref{eq:crossprodalg}.c), we obtain that
\begin{equation}\eqlabel{4.6}
\gbeg{3}{5}
\got{1}{A}\got{1}{B}\got{1}{B}\gnl
\gcl{1}\gmu\gnl
\gcl{1}\gcmu\gnl
\gbrbox\gvac{2}\gcl{1}\gnl
\gob{1}{B}\gob{1}{A}\gob{1}{B}
\gend
=
\gbeg{6}{7}
\got{1}{A}\got{2}{B}\gvac{1}\got{2}{B}\gnl
\gcl{1}\gcmu\gu{1}\gcmu\gnl
\gbrbox\gvac{2}\gcl{1}\gbrbox\gvac{2}\gcl{3}\gnl
\gcl{1}\gcl{1}\gbr\gcl{1}\gnl
\gcl{1}\gbr\gbrc\gnl
\gmu\gmu\gmu\gnl
\gob{2}{B}\gob{2}{A}\gob{2}{B}
\gend
\hspace{2mm}.
\end{equation} 
After we compose \equref{4.6} to the right with $\un{\eta}_A\ot \Id_{B\ot B}$,
we obtain (\ref{eq:neccconds}.b).\\
Compose (\ref{eq:crossbialgcond}.a)
to the right with $\Id_{A\ot B\ot A}\ot \un{\eta}_B$ 
and to the left with $\Id_A\ot \un{\va}_B\ot \Id_{A\ot B}$. Then by 
(\ref{eq:crossprodcoalg}.d), we obtain that
\begin{equation}\eqlabel{multcomultleftact}
\gbeg{3}{5}
\got{1}{A}\got{1}{B}\got{1}{A}\gnl
\gcl{1}\gbrc\gnl
\gmu\gcl{1}\gnl
\gcmu\gcl{1}\gnl
\gob{1}{A}\gob{1}{A}\gob{1}{B}
\gend
=
\gbeg{6}{7}
\got{2}{A}\got{2}{B}\got{2}{A}\gnl
\gcmu\gcmu\gcmu\gnl
\gcl{1}\gbrbox\gvac{2}\gbr\gcl{1}\gnl
\gcl{1}\gcl{1}\gbr\gbrc\gnl
\gcl{1}\gbrc\gcl{1}\gcl{1}\gcl{1}\gnl
\gmu\gcu{1}\gmu\gcl{1}\gnl
\gob{2}{A}\gvac{1}\gob{2}{A}\gob{1}{B}
\gend\hspace{1mm}.
\end{equation}
Composing \equref{multcomultleftact} to the right with $\un{\eta}_A\ot \Id_{B\ot A}$ and to 
the left with $\Id_A\ot \un{\va}_A\ot \Id_B$, and using (\ref{eq:crossprodcoalg}.c) and
(\ref{eq:multcounitcomp}.a), we find (\ref{eq:neccconds}.g). Composing
\equref{multcomultleftact} to the left with $\Id_{A\ot A}\ot \un{\va}_B$, we find
(\ref{eq:neccconds}.e).\\
Now compose (\ref{eq:crossbialgcond}.a) to the left with $\Id_{A\ot B}\ot \un{\va}_A\ot \Id_B$ 
and to the right with $\un{\eta}_A\ot \Id_{B\ot A\ot B}$. This gives
\begin{equation}\eqlabel{4.8}
\gbeg{3}{5}
\got{1}{B}\got{1}{A}\got{1}{B}\gnl
\gbrc\gcl{1}\gnl
\gcl{1}\gmu\gnl
\gcl{1}\gcmu\gnl
\gob{1}{A}\gob{1}{B}\gob{1}{B}
\gend
=
\gbeg{6}{7}
\got{2}{B}\got{2}{A}\got{2}{B}\gnl
\gcmu\gcmu\gcmu\gnl
\gcl{1}\gbr\gbrbox\gvac{2}\gcl{1}\gnl
\gbrc\gbr\gcl{1}\gcl{1}\gnl
\gcl{1}\gmu\gbrc\gcl{1}\gnl
\gcl{1}\gcn{1}{1}{2}{2}\gvac{1}\gcu{1}\gmu\gnl
\gob{1}{A}\gob{2}{B}\gvac{1}\gob{2}{B}
\gend
\hspace{2mm}.
\end{equation}
Composing \equref{4.8} to the left with $\un{\va}_A\ot \Id_{B\ot B}$, 
we obtain (\ref{eq:neccconds}.f).\\
(\ref{eq:neccconds}.d) follows after we compose (\ref{eq:crossbialgcond}.a)
to the right with $\un{\eta}_A\ot \Id_{B\ot A}\ot \un{\eta}_B$, and then use
(\ref{eq:comultunitcomp}.a-b). Finally, (\ref{eq:neccconds}.h) follows after we compose 
(\ref{eq:crossbialgcond}.a) to the left with 
$\un{\va}_A\ot \Id_{B\ot A}\ot \un{\va}_B$, and then use (\ref{eq:multcounitcomp}.a-b).
\end{proof}

Observe that we could have skipped half of the proof: (\ref{eq:neccconds}.e-h) follow
from (\ref{eq:neccconds}.a-d) using duality arguments.\\
Applying \prref{4.1}, we find some more properties of bialgebra admissible
tuples. They deserve a separate formulation for two reasons: they appear also in
the Bespalov-Drabant list, and they play a key role in the formulation of
\thref{firstsetequivcond}.

\begin{corollary}\colabel{BDCond}
If $A\#_\psi^\phi B$ is a cross product bialgebra then the following equalities hold:
\begin{eqnarray}
&&\hspace{2mm}\mbox{\rm (a)}\hspace{2mm}
\gbeg{2}{4}
\got{1}{A}\got{1}{A}\gnl
\gmu\gnl
\gcmu\gnl
\gob{1}{A}\gob{1}{A}
\gend
=\gbeg{5}{7}
\got{2}{A}\gvac{1}\got{2}{A}\gnl
\gcmu\gu{1}\gcmu\gnl
\gcl{1}\gbrbox\gvac{2}\gcl{1}\gcl{1}\gnl
\gcl{1}\gcl{1}\gbr\gcl{1}\gnl
\gcl{1}\gbrc\gcl{1}\gcl{1}\gnl
\gmu\gcu{1}\gmu\gnl
\gob{2}{A}\gvac{1}\gob{2}{A}
\gend
\hspace{2mm},\hspace{2mm}\mbox{\rm (b)}\hspace{2mm}
\gbeg{2}{4}
\got{1}{B}\got{1}{B}\gnl
\gmu\gnl
\gcmu\gnl
\gob{1}{B}\gob{1}{B}
\gend
=
\gbeg{5}{7}
\got{2}{B}\gvac{1}\got{2}{B}\gnl
\gcmu\gu{1}\gcmu\gnl
\gcl{1}\gcl{1}\gbrbox\gvac{2}\gcl{1}\gnl
\gcl{1}\gbr\gcl{1}\gcl{1}\gnl
\gcl{1}\gcl{1}\gbrc\gcl{1}\gnl
\gmu\gcu{1}\gmu\gnl
\gob{2}{B}\gvac{1}\gob{2}{B}
\gend
\hspace{2mm},
\nonumber\\
&&
\mbox{\rm (algebra-coalgebra compatibility)}
\nonumber\\
&&
\hspace{2mm}\mbox{\rm (c)}\hspace{2mm}
\gbeg{3}{5}
\got{1}{A}\got{1}{A}\gnl
\gmu\gnl
\gcn{1}{1}{2}{3}\gvac{1}\gu{1}\gnl
\gvac{1}\gbrbox\gnl
\gvac{1}\gob{1}{B}\gob{1}{A}
\gend
=
\gbeg{5}{7}
\got{1}{A}\gvac{1}\got{2}{A}\gnl
\gcl{1}\gu{1}\gcmu\gnl
\gbrbox\gvac{2}\gcl{1}\gcl{1}\gu{1}\gnl
\gcl{1}\gbr\gbrbox\gnl
\gbrc\gbr\gcl{1}\gnl
\gcu{1}\gmu\gmu\gnl
\gvac{1}\gob{2}{B}\gob{2}{A}
\gend
\hspace{2mm},\hspace{2mm}\mbox{\rm (d)}\hspace{2mm}
\gbeg{3}{5}
\gvac{1}\got{1}{B}\got{1}{B}\gnl
\gvac{1}\gmu\gnl
\gu{1}\gcn{1}{1}{2}{1}\gnl
\gbrbox\gnl
\gob{1}{B}\gob{1}{A}
\gend
=
\gbeg{5}{7}
\gvac{1}\got{2}{B}\gvac{1}\got{1}{B}\gnl
\gu{1}\gcmu\gu{1}\gcl{1}\gnl
\gbrbox\gvac{2}\gcl{1}\gbrbox\gnl
\gcl{1}\gcl{1}\gbr\gcl{1}\gnl
\gcl{1}\gbr\gbrc\gnl
\gmu\gmu\gcu{1}\gnl
\gob{2}{B}\gob{2}{A}
\gend
\hspace{2mm},\eqlabel{BespDrabComp}\\
&&
\mbox{\rm (comodule-algebra compatibility)}
\nonumber\\
&&
\hspace{2mm}\mbox{\rm (e)}\hspace{2mm}
\gbeg{2}{5}
\got{1}{B}\got{1}{A}\gnl
\gbrc\gnl
\gcn{1}{1}{1}{2}\gcu{1}\gnl
\gcmu\gnl
\gob{1}{A}\gob{1}{A}
\gend
=
\gbeg{5}{7}
\gvac{1}\got{2}{B}\got{2}{A}\gnl
\gu{1}\gcmu\gcmu\gnl
\gbrbox\gvac{2}\gbr\gcl{1}\gnl
\gcl{1}\gbr\gbrc\gnl
\gbrc\gcl{1}\gcl{1}\gcu{1}\gnl
\gcl{1}\gcu{1}\gmu\gnl
\gob{1}{A}\gvac{1}\gob{2}{A}
\gend
\hspace{2mm}\mbox{\rm and}\hspace{2mm}\mbox{\rm (f)}\hspace{2mm}
\gbeg{3}{5}
\got{1}{B}\got{1}{A}\gnl
\gbrc\gnl
\gcu{1}\gcn{1}{1}{1}{2}\gnl
\gvac{1}\gcmu\gnl
\gvac{1}\gob{1}{B}\gob{1}{B}
\gend
=
\gbeg{5}{7}
\got{2}{B}\got{2}{A}\gnl
\gcmu\gcmu\gu{1}\gnl
\gcl{1}\gbr\gbrbox\gnl
\gbrc\gbr\gcl{1}\gnl
\gcu{1}\gcl{1}\gcl{1}\gbrc\gnl
\gvac{1}\gmu\gcu{1}\gcl{1}\gnl
\gvac{1}\gob{2}{B}\gvac{1}\gob{1}{B}
\gend
\hspace{2mm}.
\nonumber\\
&&
\mbox{\rm (module-coalgebra compatibility)}
\nonumber
\end{eqnarray}
\end{corollary}

\begin{proof}
(\ref{eq:BespDrabComp}.a) follows after we compose (\ref{eq:neccconds}.a) to the left with
$\Id_A\ot\un{\va}_B\ot \Id_A$, and (\ref{eq:BespDrabComp}.b) follows after we compose
(\ref{eq:neccconds}.b) to the left with $\Id_B\ot \un{\va}_A\ot \Id_B$.
In a similar way,
(\ref{eq:BespDrabComp}.c) follows after we compose (\ref{eq:neccconds}.a) to the left with
$\un{\va}_A\ot \Id_{B\ot A}$, and 
(\ref{eq:BespDrabComp}.d) follows after we compose (\ref{eq:neccconds}.b) to the left with
$\Id_{B\ot A}\ot \un{\va}_B$. Finally, (\ref{eq:BespDrabComp}.e) follows after we compose 
(\ref{eq:neccconds}.d) to the left with $\Id_A\ot \un{\va}_B\ot \Id_A\ot \un{\va}_B$,
and (\ref{eq:BespDrabComp}.f) follows after we compose 
(\ref{eq:neccconds}.d) to the left with $\un{\va}_A\ot \Id_B\ot \un{\va}_A\ot \Id_B$. Note that
in all these computations, we have to use freely the relations \equref{comultunitcomp} and
\equref{multcounitcomp}, and the fact that 
$\un{\va}_A\un{\eta}_A=\Id_{\un{1}}=\un{\va}_B\un{\eta}_B$. 
\end{proof}
 
\begin{proposition}\prlabel{4.5}
Let $(A,B,\psi,\phi)$ be a crossed product algebra-coalgebra datum.
Assume that  (\ref{eq:neccconds}.a-d)
or (\ref{eq:neccconds}.e-h) holds. Then (\ref{eq:crossbialgcond}.a)
holds, that is, the comultiplication on
$A\#_\psi^\phi B$ is multiplicative.   
\end{proposition} 

\begin{proof}
We only prove the first assertion; the proof of the second one is similar, and can also be
obtained by duality arguments. The first assertion follows from
the following computation
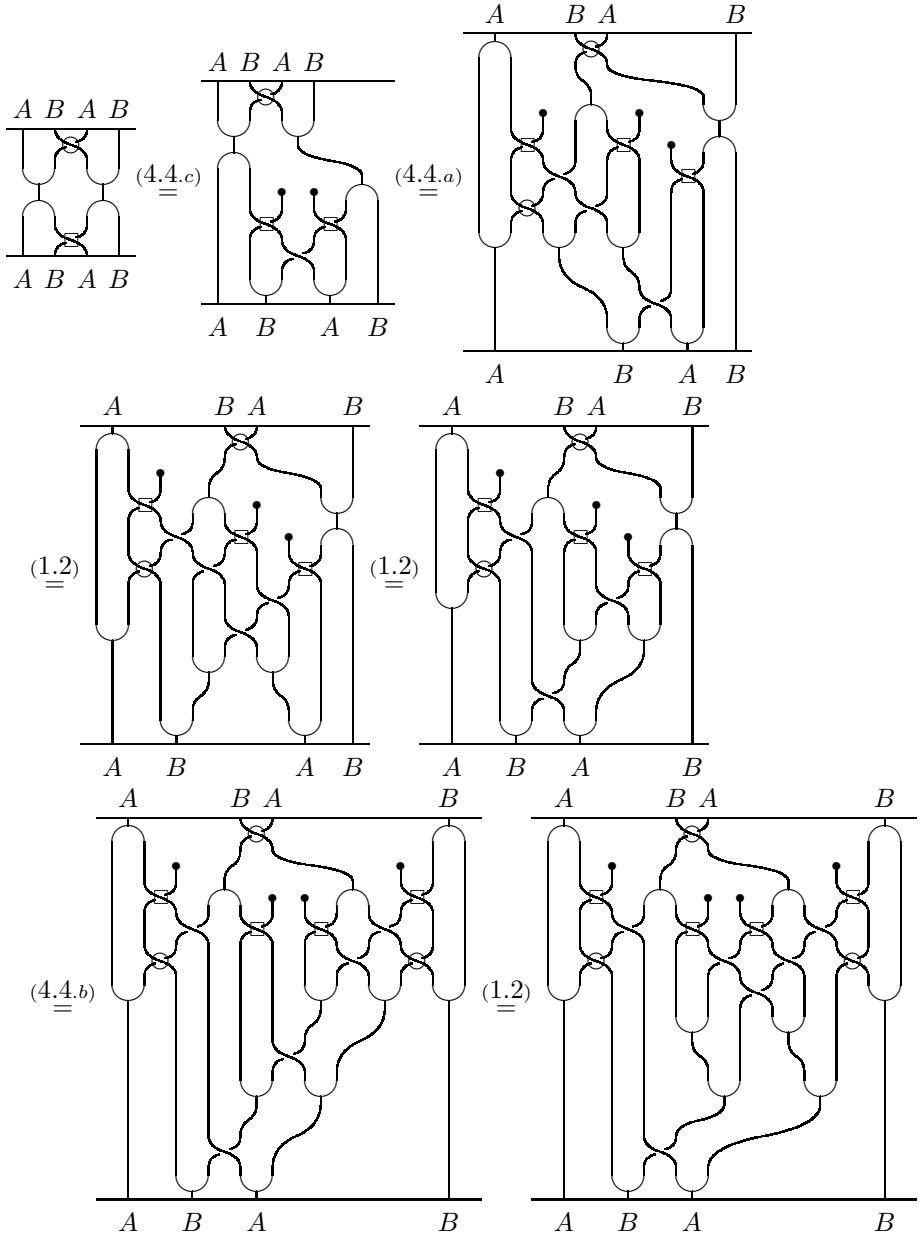
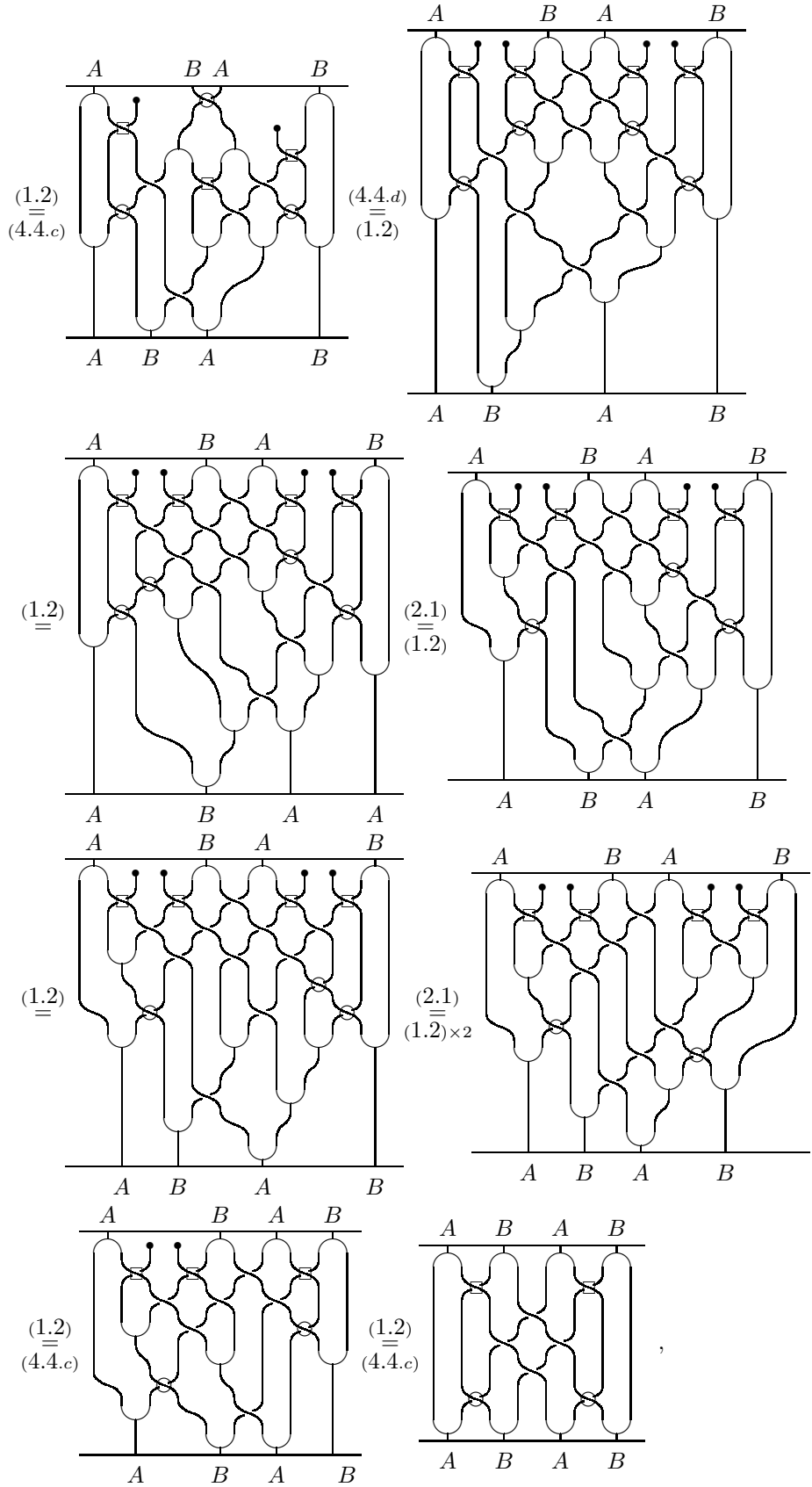
\begin{eqnarray*}
&&\hspace*{-5mm}
\gbeg{4}{6}
\got{1}{A}\got{1}{B}\got{1}{A}\got{1}{B}\gnl
\gcl{1}\gbrc\gcl{1}\gnl
\gmu\gmu\gnl
\gcmu\gcmu\gnl
\gcl{1}\gbrbox\gvac{2}\gcl{1}\gnl
\gob{1}{A}\gob{1}{B}\gob{1}{A}\gob{1}{B} 
\gend
\equal{(\ref{eq:neccconds}.c)}
\gbeg{6}{9}
\got{1}{A}\got{1}{B}\got{1}{A}\got{1}{B}\gnl
\gcl{1}\gbrc\gcl{1}\gnl
\gmu\gmu\gnl
\gcmu\gvac{1}\gcn{1}{1}{0}{4}\gnl
\gcl{1}\gcl{1}\gu{1}\gu{1}\gcmu\gnl
\gcl{1}\gbrbox\gvac{2}\gbr\gvac{-1}\gnot{\hspace*{-4mm}\Box}\gvac{1}\gcl{1}\gnl
\gcl{1}\gcl{1}\gbr\gcl{1}\gcl{1}\gnl
\gcl{1}\gmu\gmu\gcl{1}\gnl
\gob{1}{A}\gob{2}{B}\gob{2}{A}\gob{1}{B}
\gend
\equal{(\ref{eq:neccconds}.a)}
\gbeg{9}{12}
\got{2}{A}\gvac{1}\got{1}{B}\got{1}{A}\gvac{3}\got{1}{B}\gnl
\gcmu\gvac{1}\gbrc\gvac{3}\gcl{1}\gnl
\gcl{1}\gcl{1}\gvac{1}\gcn{1}{1}{1}{2}\gcn{1}{1}{1}{7}\gvac{3}\gcl{1}\gnl
\gcl{1}\gcl{1}\gu{1}\gcmu\gu{1}\gvac{1}\gmu\gnl
\gcl{1}\gbrbox\gvac{2}\gcl{1}\gbrbox\gvac{2}\gu{1}\gcmu\gnl
\gcl{1}\gcl{1}\gbr\gcl{1}\gcl{1}\gbrbox\gvac{2}\gcl{1}\gnl
\gcl{1}\gbrc\gbr\gcl{1}\gcl{1}\gcl{1}\gcl{1}\gnl
\gmu\gmu\gmu\gcl{1}\gcl{1}\gcl{1}\gnl
\gcn{1}{3}{2}{2}\gvac{1}\gcn{1}{2}{2}{5}\gvac{1}\gcn{1}{1}{2}{3}\gvac{1}\gcl{1}\gcl{1}\gcl{1}\gnl
\gvac{5}\gbr\gcl{1}\gcl{1}\gnl
\gvac{4}\gmu\gmu\gcl{1}\gnl
\gob{2}{A}\gvac{2}\gob{2}{B}\gob{2}{A}\gob{1}{B}
\gend\\
&&\hspace*{-2mm}
\equal{\equref{nat1cup}}
\gbeg{9}{12}
\got{2}{A}\gvac{2}\got{1}{B}\got{1}{A}\gvac{2}\got{1}{B}\gnl
\gcmu\gvac{2}\gbrc\gvac{2}\gcl{1}\gnl
\gcl{1}\gcl{1}\gu{1}\gvac{1}\gcn{1}{1}{1}{0}\gcn{1}{1}{1}{5}\gvac{2}\gcl{1}\gnl
\gcl{1}\gbrbox\gvac{2}\gcmu\gu{1}\gvac{1}\gmu\gnl
\gcl{1}\gcl{1}\gbr\gbrbox\gvac{2}\gu{1}\gcmu\gnl
\gcl{1}\gbrc\gbr\gcl{1}\gbrbox\gvac{2}\gcl{1}\gnl
\gcl{1}\gcl{1}\gcl{1}\gcl{1}\gcl{1}\gbr\gcl{1}\gcl{1}\gnl
\gmu\gcl{1}\gcl{1}\gbr\gcl{1}\gcl{1}\gcl{1}\gnl
\gcn{1}{3}{2}{2}\gvac{1}\gcl{1}\gmu\gmu\gcl{1}\gcl{1}\gnl
\gvac{2}\gcl{1}\gcn{1}{1}{2}{1}\gvac{1}\gcn{1}{1}{2}{3}\gvac{1}\gcl{1}\gcl{1}\gnl
\gvac{2}\gmu\gvac{2}\gmu\gcl{1}\gnl
\gob{2}{A}\gob{2}{B}\gvac{2}\gob{2}{A}\gob{1}{B}
\gend
\equal{\equref{nat1cup}}
\gbeg{9}{12}
\got{2}{A}\gvac{2}\got{1}{B}\got{1}{A}\gvac{2}\got{1}{B}\gnl
\gcmu\gvac{2}\gbrc\gvac{2}\gcl{1}\gnl
\gcl{1}\gcl{1}\gu{1}\gvac{1}\gcn{1}{1}{1}{0}\gcn{1}{1}{1}{5}\gvac{2}\gcl{1}\gnl
\gcl{1}\gbrbox\gvac{2}\gcmu\gu{1}\gvac{1}\gmu\gnl
\gcl{1}\gcl{1}\gbr\gbrbox\gvac{2}\gu{1}\gcmu\gnl
\gcl{1}\gbrc\gcl{1}\gcl{1}\gcl{1}\gbrbox\gvac{2}\gcl{1}\gnl
\gmu\gcl{1}\gcl{1}\gcl{1}\gbr\gcl{1}\gcl{1}\gnl
\gcn{1}{4}{2}{2}\gvac{1}\gcl{1}\gcl{1}\gmu\gmu\gcl{1}\gnl
\gvac{2}\gcl{1}\gcl{1}\gcn{1}{1}{2}{1}\gvac{1}\gcn{1}{2}{2}{-1}\gvac{1}\gcl{2}\gnl
\gvac{2}\gcl{1}\gbr\gnl
\gvac{2}\gmu\gmu\gvac{2}\gcl{1}\gnl
\gob{2}{A}\gob{2}{B}\gob{2}{A}\gvac{2}\gob{1}{B}
\gend\\
&&\hspace*{-2mm}
\equal{(\ref{eq:neccconds}.b)}
\gbeg{12}{14}
\got{2}{A}\gvac{2}\got{1}{B}\got{1}{A}\gvac{4}\got{2}{B}\gnl
\gcmu\gvac{2}\gbrc\gvac{4}\gcmu\gnl
\gcl{1}\gcl{1}\gu{1}\gvac{1}\gcn{1}{1}{1}{0}\gcn{1}{1}{1}{6}\gvac{3}\gu{1}\gcl{1}\gcl{1}\gnl
\gcl{1}\gbrbox\gvac{2}\gcmu\gu{1}\gu{1}\gcmu\gbrbox\gvac{2}\gcl{1}\gnl
\gcl{1}\gcl{1}\gbr\gbrbox\gvac{2}\gbr\gvac{-1}\gnot{\hspace*{-4mm}\Box}\gvac{1}\gbr\gcl{1}\gcl{1}\gnl
\gcl{1}\gbrc\gcl{1}\gcl{1}\gcl{1}\gcl{1}\gbr\gbrc\gcl{1}\gnl
\gmu\gcl{1}\gcl{1}\gcl{1}\gcl{1}\gmu\gmu\gmu\gnl
\gcn{1}{6}{2}{2}\gvac{1}\gcl{1}\gcl{1}\gcl{1}\gcl{1}\gcn{1}{1}{2}{1}\gcn{1}{2}{4}{1}\gcn{1}{6}{6}{6}\gnl
\gvac{2}\gcl{1}\gcl{1}\gcl{1}\gbr\gnl
\gvac{2}\gcl{1}\gcl{1}\gmu\gmu\gnl
\gvac{2}\gcl{1}\gcl{1}\gcn{1}{1}{2}{1}\gvac{1}\gcn{1}{2}{2}{-1}\gnl
\gvac{2}\gcl{1}\gbr\gnl
\gvac{2}\gmu\gmu\gnl
\gob{2}{A}\gob{2}{B}\gob{2}{A}\gvac{4}\gob{2}{B}
\gend
\equal{\equref{nat1cup}}
\gbeg{12}{14}
\got{2}{A}\gvac{2}\got{1}{B}\got{1}{A}\gvac{4}\got{2}{B}\gnl
\gcmu\gvac{2}\gbrc\gvac{4}\gcmu\gnl
\gcl{1}\gcl{1}\gu{1}\gvac{1}\gcn{1}{1}{1}{0}\gcn{1}{1}{1}{6}\gvac{3}\gu{1}\gcl{1}\gcl{1}\gnl
\gcl{1}\gbrbox\gvac{2}\gcmu\gu{1}\gu{1}\gcmu\gbrbox\gvac{2}\gcl{1}\gnl
\gcl{1}\gcl{1}\gbr\gbrbox\gvac{2}\gbr\gvac{-1}\gnot{\hspace*{-4mm}\Box}\gvac{1}\gbr\gcl{1}\gcl{1}\gnl
\gcl{1}\gbrc\gcl{1}\gcl{1}\gbr\gbr\gbrc\gcl{1}\gnl
\gmu\gcl{1}\gcl{1}\gcl{1}\gcl{1}\gbr\gcl{1}\gcl{1}\gmu\gnl
\gcn{1}{6}{2}{2}\gvac{1}\gcl{1}\gcl{1}\gmu\gcl{1}\gmu\gcl{1}\gcn{1}{6}{2}{2}\gnl
\gvac{2}\gcl{1}\gcl{1}\gcn{1}{1}{2}{3}\gvac{1}\gcl{1}\gcn{1}{1}{2}{3}\gvac{1}\gcl{1}\gnl
\gvac{2}\gcl{1}\gcl{1}\gvac{1}\gmu\gvac{1}\gmu\gnl
\gvac{2}\gcl{1}\gcl{1}\gcn{1}{1}{4}{1}\gvac{3}\gcn{1}{2}{2}{-5}\gnl
\gvac{2}\gcl{1}\gbr\gnl
\gvac{2}\gmu\gmu\gnl
\gob{2}{A}\gob{2}{B}\gob{2}{A}\gvac{4}\gob{2}{B}
\gend\\
&&\hspace*{-2mm}
\equalupdown{\equref{nat1cup}}{(\ref{eq:neccconds}.c)}
\gbeg{10}{11}
\got{2}{A}\gvac{2}\got{1}{B}\got{1}{A}\gvac{2}\got{2}{B}\gnl
\gcmu\gu{1}\gvac{1}\gbrc\gvac{2}\gcmu\gnl
\gcl{1}\gbrbox\gvac{3}\gcn{1}{1}{1}{0}\gcn{1}{1}{1}{2}\gvac{1}\gu{1}\gcl{1}\gcl{1}\gnl
\gcl{1}\gcl{1}\gcl{1}\gcmu\gcmu\gbrbox\gvac{2}\gcl{1}\gnl
\gcl{1}\gcl{1}\gbr\gbrbox\gvac{2}\gbr\gcl{1}\gcl{1}\gnl
\gcl{1}\gbrc\gcl{1}\gcl{1}\gbr\gbrc\gcl{1}\gnl
\gmu\gcl{1}\gcl{1}\gmu\gmu\gmu\gnl
\gcn{1}{3}{2}{2}\gvac{1}\gcl{1}\gcl{1}\gcn{1}{1}{2}{1}\gvac{1}\gcn{1}{2}{2}{-1}\gcn{1}{3}{4}{4}\gnl
\gvac{2}\gcl{1}\gbr\gnl
\gvac{2}\gmu\gmu\gnl
\gob{2}{A}\gob{2}{B}\gob{2}{A}\gvac{2}\gob{2}{B}
\gend
\equalupdown{(\ref{eq:neccconds}.d)}{\equref{nat1cup}}
\gbeg{12}{15}
\got{2}{A}\gvac{2}\got{2}{B}\got{2}{A}\gvac{2}\got{2}{B}\gnl
\gcmu\gu{1}\gu{1}\gcmu\gcmu\gu{1}\gu{1}\gcmu\gnl
\gcl{1}\gbrbox\gvac{2}\gbr\gvac{-1}\gnot{\hspace*{-4mm}\Box}\gvac{1}\gbr\gbrbox
\gvac{2}\gbr\gvac{-1}\gnot{\hspace*{-4mm}\Box}\gvac{1}\gcl{1}\gnl
\gcl{1}\gcl{1}\gcl{1}\gcl{1}\gbr\gbr\gcl{1}\gcl{1}\gcl{1}\gcl{1}\gnl
\gcl{1}\gcl{1}\gcl{1}\gbrc\gbr\gbrc\gcl{1}\gcl{1}\gcl{1}\gnl
\gcl{1}\gcl{1}\gbr\gmu\gmu\gbr\gcl{1}\gcl{1}\gnl
\gcl{1}\gbrc\gcl{1}\gcn{1}{1}{2}{1}\gvac{1}\gcn{1}{1}{2}{3}\gvac{1}\gcl{1}\gbrc\gcl{1}\gnl
\gmu\gcl{1}\gbr\gvac{2}\gbr\gcl{1}\gmu\gnl
\gcn{1}{6}{2}{2}\gvac{1}\gcl{1}\gcl{1}\gcn{1}{1}{1}{3}\gvac{2}\gcn{1}{1}{1}{-1}\gmu\gcn{1}{6}{2}{2}\gnl
\gvac{2}\gcl{1}\gcl{1}\gvac{1}\gbr\gvac{1}\gcn{1}{1}{2}{-1}\gnl
\gvac{2}\gcl{1}\gcl{1}\gvac{1}\gcn{1}{1}{1}{-1}\gmu\gnl
\gvac{2}\gcl{1}\gmu\gvac{1}\gcn{1}{3}{2}{2}\gnl
\gvac{2}\gcl{1}\gcn{1}{1}{2}{1}\gnl
\gvac{2}\gmu\gnl
\gob{2}{A}\gob{2}{B}\gvac{2}\gob{2}{A}\gvac{2}\gob{2}{B}
\gend\\
&&
\equal{\equref{nat1cup}}
\gbeg{12}{14}
\got{2}{A}\gvac{2}\got{2}{B}\got{2}{A}\gvac{2}\got{2}{B}\gnl
\gcmu\gu{1}\gu{1}\gcmu\gcmu\gu{1}\gu{1}\gcmu\gnl
\gcl{1}\gbrbox\gvac{2}\gbr\gvac{-1}\gnot{\hspace*{-4mm}\Box}\gvac{1}\gbr\gbrbox\gvac{2}
\gbr\gvac{-1}\gnot{\hspace*{-4mm}\Box}\gvac{1}\gcl{1}\gnl
\gcl{1}\gcl{1}\gbr\gbr\gbr\gcl{1}\gcl{1}\gcl{1}\gcl{1}\gnl
\gcl{1}\gcl{1}\gcl{1}\gbr\gbr\gbrc\gcl{1}\gcl{1}\gcl{1}\gnl
\gcl{1}\gcl{1}\gbrc\gbr\gmu\gbr\gcl{1}\gcl{1}\gnl
\gcl{1}\gbrc\gmu\gcl{1}\gcn{1}{1}{2}{3}\gvac{1}\gcl{1}\gbrc\gcl{1}\gnl
\gmu\gcl{1}\gcn{1}{3}{2}{5}\gvac{1}\gcl{1}\gvac{1}\gbr\gcl{1}\gcl{1}\gcl{1}\gnl
\gcn{1}{5}{2}{2}\gvac{1}\gcl{1}\gvac{2}\gcn{1}{1}{1}{3}\gvac{1}\gcl{1}\gmu\gmu\gnl
\gvac{2}\gcl{1}\gvac{3}\gbr\gvac{1}\gcn{1}{1}{0}{-1}\gcn{1}{4}{2}{2}\gnl
\gvac{2}\gcn{1}{2}{1}{5}\gvac{2}\gmu\gmu\gnl
\gvac{5}\gcn{1}{1}{2}{1}\gvac{1}\gcn{1}{2}{2}{2}\gnl
\gvac{4}\gmu\gnl
\gob{2}{A}\gvac{2}\gob{2}{B}\gvac{1}\gob{2}{A}\gvac{1}\gob{2}{A}
\gend
\equalupdown{\equref{crossprodalg}}{\equref{nat1cup}}
\gbeg{12}{13}
\got{2}{A}\gvac{2}\got{2}{B}\got{2}{A}\gvac{2}\got{2}{B}\gnl
\gcmu\gu{1}\gu{1}\gcmu\gcmu\gu{1}\gu{1}\gcmu\gnl
\gcl{1}\gbrbox\gvac{2}\gbr\gvac{-1}\gnot{\hspace*{-4mm}\Box}\gvac{1}\gbr\gbrbox\gvac{2}
\gbr\gvac{-1}\gnot{\hspace*{-4mm}\Box}\gvac{1}\gcl{1}\gnl
\gcl{1}\gcl{1}\gbr\gbr\gbr\gcl{1}\gcl{1}\gcl{1}\gcl{1}\gnl
\gcl{1}\gmu\gbr\gbr\gbrc\gcl{1}\gcl{1}\gcl{1}\gnl
\gcl{1}\gcn{1}{1}{2}{3}\gvac{1}\gcl{1}\gcl{1}\gcl{1}\gmu\gbr\gcl{1}\gcl{1}\gnl
\gcn{1}{1}{1}{3}\gvac{1}\gbrc\gcl{1}\gcl{1}\gcn{1}{1}{2}{3}\gvac{1}\gcl{1}\gbrc\gcl{1}\gnl
\gvac{1}\gmu\gcl{1}\gcl{1}\gcn{1}{1}{1}{3}\gvac{1}\gbr\gcl{1}\gcl{1}\gcl{1}\gnl
\gvac{1}\gcn{1}{4}{2}{2}\gvac{1}\gcl{1}\gcl{1}\gvac{1}\gmu\gmu\gmu\gnl
\gvac{3}\gcl{1}\gcn{1}{1}{1}{3}\gvac{1}\gcn{1}{1}{2}{1}\gvac{1}\gcn{1}{2}{2}{-1}\gvac{1}\gcn{1}{3}{2}{2}\gnl
\gvac{3}\gcn{1}{1}{1}{3}\gvac{1}\gbr\gnl
\gvac{4}\gmu\gmu\gnl
\gvac{1}\gob{2}{A}\gvac{1}\gob{2}{B}\gob{2}{A}\gvac{2}\gob{2}{B}
\gend\\
&&
\equal{\equref{nat1cup}}
\gbeg{12}{13}
\got{2}{A}\gvac{2}\got{2}{B}\got{2}{A}\gvac{2}\got{2}{B}\gnl
\gcmu\gu{1}\gu{1}\gcmu\gcmu\gu{1}\gu{1}\gcmu\gnl
\gcl{1}\gbrbox\gvac{2}\gbr\gvac{-1}\gnot{\hspace*{-4mm}\Box}\gvac{1}\gbr\gbrbox\gvac{2}
\gbr\gvac{-1}\gnot{\hspace*{-4mm}\Box}\gvac{1}\gcl{1}\gnl
\gcl{1}\gcl{1}\gbr\gbr\gbr\gbr\gcl{1}\gcl{1}\gnl
\gcl{1}\gmu\gbr\gbr\gbr\gcl{1}\gcl{1}\gcl{1}\gnl
\gcl{1}\gcn{1}{1}{2}{3}\gvac{1}\gcl{1}\gcl{1}\gcl{1}\gcl{1}\gcl{1}\gbrc\gcl{1}\gcl{1}\gnl
\gcn{1}{1}{1}{3}\gvac{1}\gbrc\gcl{1}\gcl{1}\gbr\gcl{1}\gbrc\gcl{1}\gnl
\gvac{1}\gmu\gcl{1}\gcl{1}\gmu\gcl{1}\gmu\gmu\gnl
\gvac{1}\gcn{1}{4}{2}{2}\gvac{1}\gcl{1}\gcl{1}\gcn{1}{1}{2}{1}\gvac{1}\gcl{1}\gcn{1}{1}{2}{1}\gcn{1}{4}{4}{4}\gnl
\gvac{3}\gcl{1}\gbr\gvac{1}\gmu\gnl
\gvac{3}\gmu\gcn{1}{1}{1}{3}\gvac{1}\gcn{1}{1}{2}{1}\gnl
\gvac{3}\gcn{1}{1}{2}{2}\gvac{2}\gmu\gnl
\gvac{1}\gob{2}{A}\gob{2}{B}\gvac{1}\gob{2}{A}\gvac{2}\gob{2}{B}
\gend
\equalupdown{\equref{crossprodalg}}{\equref{nat1cup}\times 2}
\gbeg{12}{12}
\got{2}{A}\gvac{2}\got{2}{B}\got{2}{A}\gvac{2}\got{2}{B}\gnl
\gcmu\gu{1}\gu{1}\gcmu\gcmu\gu{1}\gu{1}\gcmu\gnl
\gcl{1}\gbrbox\gvac{2}\gbr\gvac{-1}\gnot{\hspace*{-4mm}\Box}\gvac{1}\gbr\gbrbox\gvac{2}
\gbr\gvac{-1}\gnot{\hspace*{-4mm}\Box}\gvac{1}\gcl{1}\gnl
\gcl{1}\gcl{1}\gbr\gbr\gcl{1}\gcl{1}\gbr\gcl{1}\gcl{1}\gnl
\gcl{1}\gmu\gbr\gcl{1}\gcl{1}\gmu\gmu\gcl{1}\gnl
\gcl{1}\gcn{1}{1}{2}{3}\gvac{1}\gcl{1}\gcl{1}\gcl{1}\gcl{1}\gcn{1}{1}{2}{1}\gvac{1}\gcn{1}{2}{2}{-1}\gvac{1}\gcl{1}\gnl
\gcn{1}{1}{1}{3}\gvac{1}\gbrc\gcl{1}\gcl{1}\gbr\gvac{3}\gcn{1}{2}{1}{-3}\gnl
\gvac{1}\gmu\gcl{1}\gcl{1}\gbr\gbrc\gnl
\gvac{1}\gcn{1}{3}{2}{2}\gvac{1}\gcl{1}\gbr\gmu\gmu\gnl
\gvac{3}\gmu\gcl{1}\gcn{1}{1}{2}{1}\gvac{1}\gcn{1}{2}{2}{2}\gnl
\gvac{3}\gcn{1}{1}{2}{2}\gvac{1}\gmu\gnl
\gvac{1}\gob{2}{A}\gob{2}{B}\gob{2}{A}\gvac{1}\gob{2}{B}
\gend\\
&&
\equalupdown{\equref{nat1cup}}{(\ref{eq:neccconds}.c)}
\gbeg{10}{10}
\got{2}{A}\gvac{2}\got{2}{B}\got{2}{A}\got{2}{B}\gnl
\gcmu\gu{1}\gu{1}\gcmu\gcmu\gcmu\gnl
\gcl{1}\gbrbox\gvac{2}\gbr\gvac{-1}\gnot{\hspace*{-4mm}\Box}\gvac{1}\gbr\gbrbox\gvac{2}\gcl{1}\gnl
\gcl{1}\gcl{1}\gbr\gbr\gbr\gcl{1}\gcl{1}\gnl
\gcl{1}\gmu\gbr\gcl{1}\gcl{1}\gbrc\gcl{1}\gnl
\gcl{1}\gcn{1}{1}{2}{3}\gvac{1}\gcl{1}\gmu\gcl{1}\gcl{1}\gmu\gnl
\gcn{1}{1}{1}{3}\gvac{1}\gbrc\gcn{1}{1}{2}{3}\gvac{1}\gcl{1}\gcl{1}\gcn{1}{3}{2}{2}\gnl
\gvac{1}\gmu\gvac{1}\gcn{1}{1}{-1}{1}\gbr\gcl{1}\gnl
\gvac{1}\gcn{1}{1}{2}{2}\gvac{2}\gmu\gmu\gnl
\gvac{1}\gob{2}{A}\gvac{1}\gob{2}{B}\gob{2}{A}\gvac{1}\gob{1}{B}
\gend
\equalupdown{\equref{nat1cup}}{(\ref{eq:neccconds}.c)}
\gbeg{8}{9}
\got{2}{A}\got{2}{B}\got{2}{A}\got{2}{B}\gnl
\gcmu\gcmu\gcmu\gcmu\gnl
\gcl{1}\gbrbox\gvac{2}\gcl{1}\gcl{1}\gbrbox\gvac{2}\gcl{1}\gnl
\gcl{1}\gcl{1}\gcl{1}\gbr\gcl{1}\gcl{1}\gcl{1}\gnl
\gcl{1}\gcl{1}\gbr\gbr\gcl{1}\gcl{1}\gnl
\gcl{1}\gcl{1}\gcl{1}\gbr\gcl{1}\gcl{1}\gcl{1}\gnl
\gcl{1}\gbrc\gcl{1}\gcl{1}\gbrc\gcl{1}\gnl
\gmu\gmu\gmu\gmu\gnl
\gob{2}{A}\gob{2}{B}\gob{2}{A}\gob{2}{B}
\gend 
\hspace{2mm},
\end{eqnarray*}
as desired. At some steps, we used the associativity of 
$\un{m}_A$ and $\un{m}_B$, and the naturality of 
the braiding.
\end{proof}

\begin{corollary}\colabel{firstequivcond}
Let $(A,B,\psi,\phi)$ be a cross product algebra-coalgebra datum.
 Then the following assertions are equivalent:
\begin{itemize}
\item[(i)] $(A,B,\psi,\phi)$ is a bialgebra admissible tuple, that is,
$A\#_\psi^\phi B$ is a cross product bialgebra;
\item[(ii)] $\un{\va}_X\un{\eta}_X=\Id_{\un{1}}$ for $X\in \{A, B\}$, 
and \equref{comultunitcomp}, \equref{multcounitcomp} 
and (\ref{eq:neccconds}.a-d) hold;
\item[(iii)] $\un{\va}_X\un{\eta}_X=\Id_{\un{1}}$ for $X\in \{A, B\}$, 
and \equref{comultunitcomp}, \equref{multcounitcomp} 
and (\ref{eq:neccconds}.e-h) hold.
\end{itemize}
\end{corollary}

\coref{firstequivcond} is a first list of sets of necessary and sufficient conditions
for a cross product algebra-coalgebra datum being a bialgebra admissible tuple.
Before we can extend this list, we need another Lemma.

\begin{lemma}\lelabel{someechivcond}
Let $(A, B, \psi, \phi)$ be a cross product algebra-coalgebra datum and assume that 
$\un{\va}_X\un{\eta}_X=\Id_{\un{1}}$, for $X\in \{A, B\}$, and that 
(\ref{eq:comultunitcomp}-\ref{eq:multcounitcomp}) hold.\\
(i) If (\ref{eq:neccconds}.g) holds then 
(\ref{eq:neccconds}.a) is equivalent to 
(\ref{eq:BespDrabComp}.a,c), and (\ref{eq:neccconds}.b) is equivalent to 
(\ref{eq:BespDrabComp}.b,d).\\
(ii) If (\ref{eq:neccconds}.c) holds then (\ref{eq:neccconds}.e) is equivalent to 
(\ref{eq:BespDrabComp}.a,e), and (\ref{eq:neccconds}.f) is equivalent to 
(\ref{eq:BespDrabComp}.b,f). 
\end{lemma} 

\begin{proof}
We will prove the first statement of (i), the proof of all the other assertions is similar.
If (\ref{eq:neccconds}.a) holds, then (\ref{eq:BespDrabComp}.a), resp. (\ref{eq:BespDrabComp}.c)
 follows after we compose
(\ref{eq:neccconds}.a) to the left with $\Id_A\ot \un{\va}_B\ot \Id_A$, resp. 
$\un{\va}_A\ot \Id_{B\ot A}$, see the proof of \coref{BDCond}. The proof of the converse implication
follows from our next computation:
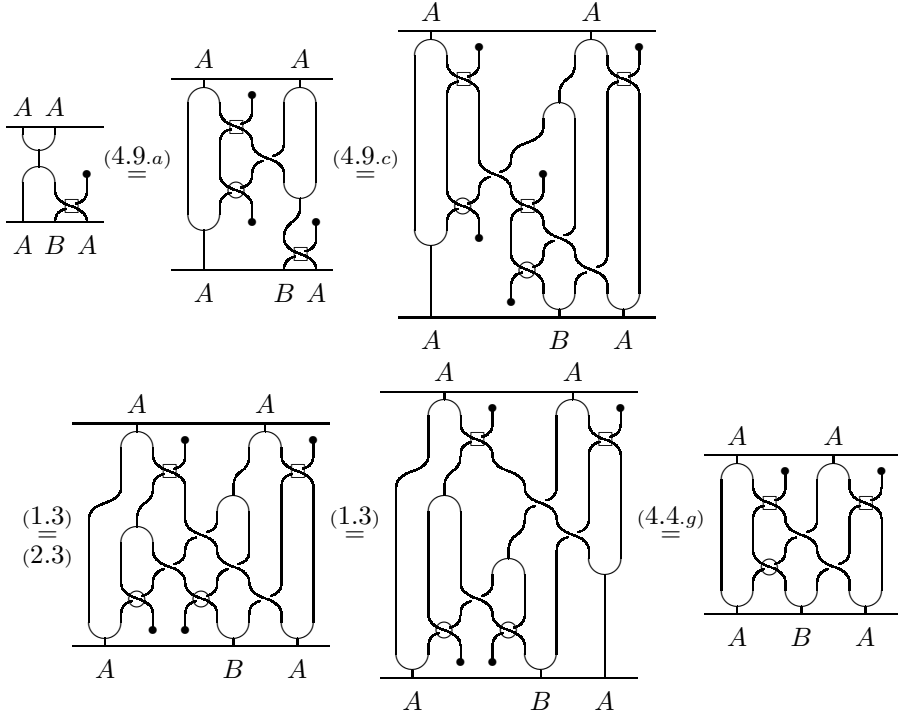
\begin{eqnarray*}
&&\hspace*{-2mm}
\gbeg{3}{5}
\got{1}{A}\got{1}{A}\gnl
\gmu\gnl
\gcmu\gu{1}\gnl
\gcl{1}\gbrbox\gnl
\gob{1}{A}\gob{1}{B}\gob{1}{A}
\gend
\equal{(\ref{eq:BespDrabComp}.a)}
\gbeg{5}{8}
\got{2}{A}\gvac{1}\got{2}{A}\gnl
\gcmu\gu{1}\gcmu\gnl
\gcl{1}\gbrbox\gvac{2}\gcl{1}\gcl{1}\gnl
\gcl{1}\gcl{1}\gbr\gcl{1}\gnl
\gcl{1}\gbrc\gmu\gnl
\gmu\gcu{1}\gcn{1}{1}{2}{1}\gu{1}\gnl
\gcn{1}{1}{2}{2}\gvac{2}\gbrbox\gnl
\gob{2}{A}\gvac{1}\gob{1}{B}\gob{1}{A}
\gend
\equal{(\ref{eq:BespDrabComp}.c)}
\gbeg{8}{11}
\got{2}{A}\gvac{3}\got{2}{A}\gnl
\gcmu\gu{1}\gvac{2}\gcmu\gu{1}\gnl
\gcl{1}\gbrbox\gvac{4}\gcn{1}{1}{1}{0}\gbr\gvac{-1}\gnot{\hspace*{-4mm}\Box}\gnl
\gcl{1}\gcl{1}\gcl{1}\gvac{1}\gcmu\gcl{1}\gcl{1}\gnl
\gcl{1}\gcl{1}\gcl{1}\gvac{1}\gcn{1}{1}{1}{-1}\gcl{1}\gcl{1}\gcl{1}\gnl
\gcl{1}\gcl{1}\gbr\gu{1}\gcl{1}\gcl{1}\gcl{1}\gnl
\gcl{1}\gbrc\gbrbox\gvac{2}\gcl{1}\gcl{1}\gcl{1}\gnl
\gmu\gcu{1}\gcl{1}\gbr\gcl{1}\gcl{1}\gnl
\gcn{1}{2}{2}{2}\gvac{2}\gbrc\gbr\gcl{1}\gnl
\gvac{3}\gcu{1}\gmu\gmu\gnl
\gob{2}{A}\gvac{2}\gob{2}{B}\gob{2}{A}
\gend\\
&&
\equalupdown{\equref{nat2cup}}{\equref{crossprodcoalg}}
\gbeg{8}{9}
\gvac{1}\got{2}{A}\gvac{2}\got{2}{A}\gnl
\gvac{1}\gcmu\gu{1}\gvac{1}\gcmu\gu{1}\gnl
\gvac{1}\gcl{1}\gbrbox\gvac{3}\gcn{1}{1}{1}{0}\gbrbox\gnl
\gvac{1}\gcn{1}{1}{1}{-1}\gcn{1}{1}{1}{0}\gcl{1}\gcmu\gcl{1}\gcl{1}\gnl
\gcl{1}\gcmu\gbr\gcl{1}\gcl{1}\gcl{1}\gnl
\gcl{1}\gcl{1}\gbr\gbr\gcl{1}\gcl{1}\gnl
\gcl{1}\gbrc\gbrc\gbr\gcl{1}\gnl
\gmu\gcu{1}\gcu{1}\gmu\gmu\gnl
\gob{2}{A}\gvac{2}\gob{2}{B}\gob{2}{A}
\gend
\equal{\equref{nat2cup}}
\gbeg{8}{11}
\gvac{1}\got{2}{A}\gvac{2}\got{2}{A}\gnl
\gvac{1}\gcmu\gu{1}\gvac{1}\gcmu\gu{1}\gnl
\gvac{1}\gcl{1}\gbrbox\gvac{3}\gcl{1}\gbr\gvac{-1}\gnot{\hspace*{-4mm}\Box}\gnl
\gvac{1}\gcn{1}{1}{1}{-1}\gcn{1}{1}{1}{0}\gcn{1}{1}{1}{3}\gvac{1}\gcl{1}\gcl{1}\gcl{1}\gnl
\gcl{1}\gcmu\gvac{1}\gbr\gcl{1}\gcl{1}\gnl
\gcl{1}\gcl{1}\gcl{1}\gvac{1}\gcn{1}{1}{1}{0}\gbr\gcl{1}\gnl
\gcl{1}\gcl{1}\gcl{1}\gcmu\gcl{1}\gmu\gnl
\gcl{1}\gcl{1}\gbr\gcl{1}\gcl{1}\gcn{1}{3}{2}{2}\gnl
\gcl{1}\gbrc\gbrc\gcl{1}\gnl
\gmu\gcu{1}\gcu{1}\gmu\gnl
\gob{2}{A}\gvac{2}\gob{2}{B}\gob{2}{A}
\gend
\equal{(\ref{eq:neccconds}.g)}
\gbeg{6}{7}
\got{2}{A}\gvac{1}\got{2}{A}\gnl
\gcmu\gu{1}\gcmu\gu{1}\gnl
\gcl{1}\gbrbox\gvac{2}\gcl{1}\gbrbox\gnl
\gcl{1}\gcl{1}\gbr\gcl{1}\gcl{1}\gnl
\gcl{1}\gbrc\gbr\gcl{1}\gnl
\gmu\gmu\gmu\gnl
\gob{2}{A}\gob{2}{B}\gob{2}{A}
\gend
\hspace{2mm}.
\end{eqnarray*}
\end{proof}

\thref{firstsetequivcond} is the main result of this Section. Two new 
conditions will appear, namely
\begin{equation}\eqlabel{twoanothYDconds}
\hspace{1mm}\mbox{\rm (a)}\hspace{1mm}
\gbeg{4}{6}
\got{1}{A}\got{1}{B}\got{1}{A}\gnl
\gcl{1}\gbrc\gnl
\gmu\gcn{1}{1}{1}{2}\gnl
\gcn{1}{1}{2}{3}\gvac{1}\gcmu\gnl
\gvac{1}\gbrbox\gvac{2}\gcl{1}\gnl
\gvac{1}\gob{1}{B}\gob{1}{A}\gob{1}{B}
\gend
=
\gbeg{6}{7}
\got{1}{A}\got{2}{B}\got{2}{A}\gnl
\gcl{1}\gcmu\gcmu\gu{1}\gnl
\gbrbox\gvac{2}\gbr\gbrbox\gnl
\gcl{1}\gbr\gbr\gcl{1}\gnl
\gbrc\gbr\gbrc\gnl
\gcu{1}\gmu\gmu\gcl{1}\gnl
\gvac{1}\gob{2}{B}\gob{2}{A}\gob{1}{B}
\gend
\hspace{1mm}\mbox{\rm or (b)}\hspace{1mm}
\gbeg{4}{6}
\gvac{1}\got{1}{B}\got{1}{A}\got{1}{B}\gnl
\gvac{1}\gbrc\gcl{1}\gnl
\gvac{1}\gcn{1}{1}{1}{0}\gmu\gnl
\gcmu\gcn{1}{1}{2}{1}\gnl
\gcl{1}\gbrbox\gnl
\gob{1}{A}\gob{1}{B}\gob{1}{A}
\gend
=
\gbeg{6}{7}
\gvac{1}\got{2}{B}\got{2}{A}\got{1}{B}\gnl
\gu{1}\gcmu\gcmu\gcl{1}\gnl
\gbrbox\gvac{2}\gbr\gbrbox\gnl
\gcl{1}\gbr\gbr\gcl{1}\gnl
\gbrc\gbr\gbrc\gnl
\gcl{1}\gmu\gmu\gcu{1}\gnl
\gob{1}{A}\gob{2}{B}\gob{2}{A}
\gend
\hspace{1mm}.
\end{equation}

\begin{theorem}\thlabel{firstsetequivcond}
Let $(A, B, \psi, \phi)$ be a cross product algebra-coalgebra datum. 
Then the following assertions are equivalent,
\begin{itemize}
\item[(i)] $A\#_\psi^\phi B$ is a cross product bialgebra;
\item[(ii)] $\un{\va}_X\un{\eta}_X=\Id_{\un{1}}$ for $X\in \{A, B\}$, 
and  \equref{comultunitcomp}, \equref{multcounitcomp} 
and (\ref{eq:neccconds}.a-d) hold;
\item[(iii)] $\un{\va}_X\un{\eta}_X=\Id_{\un{1}}$ for $X\in \{A, B\}$, 
and  \equref{comultunitcomp}, \equref{multcounitcomp} 
and (\ref{eq:neccconds}.e-h) hold;
\item[(iv)] $\un{\va}_X\un{\eta}_X=\Id_{\un{1}}$ for $X\in \{A, B\}$, 
and  \equref{comultunitcomp}, \equref{multcounitcomp}, 
(\ref{eq:neccconds}.c,d,g) and (\ref{eq:BespDrabComp}.a-d) hold;
\item[(v)] $\un{\va}_X\un{\eta}_X=\Id_{\un{1}}$ for $X\in \{A, B\}$, 
and \equref{comultunitcomp}, \equref{multcounitcomp}, 
(\ref{eq:neccconds}.c,g,h) and (\ref{eq:BespDrabComp}.a,b,e,f) hold;
\item[(vi)] $\un{\va}_X\un{\eta}_X=\Id_{\un{1}}$ for $X\in \{A, B\}$, 
and \equref{comultunitcomp}, \equref{multcounitcomp}, 
(\ref{eq:neccconds}.c,g), (\ref{eq:twoanothYDconds}.a) and 
(\ref{eq:BespDrabComp}.a,b,d,e) hold;
\item[(vii)] $\un{\va}_X\un{\eta}_X=\Id_{\un{1}}$ for $X\in \{A, B\}$, 
and \equref{comultunitcomp}, \equref{multcounitcomp}, 
(\ref{eq:neccconds}.c,g), (\ref{eq:twoanothYDconds}.b) and 
(\ref{eq:BespDrabComp}.a,b,c,f) hold.
\end{itemize}
\end{theorem}

\begin{proof}
We have already seen in \coref{firstequivcond} that (i), (ii) and (iii) are equivalent.
The equivalences $(ii)\Leftrightarrow (iv)$ and $(iii)\Leftrightarrow (v)$ follow 
from \leref{someechivcond}.\\
$\un{(ii)\Rightarrow (vi)}$. We have seen in the proof of 
\coref{BDCond} 
that (\ref{eq:neccconds}.a) implies (\ref{eq:BespDrabComp}.a,c), and that 
(\ref{eq:neccconds}.b) implies (\ref{eq:BespDrabComp}.b,d). 
 (\ref{eq:neccconds}.g) follows after we compose 
 (\ref{eq:neccconds}.d) to the left with $\un{\va}_A\ot \un{\va}_B\ot \Id_{A\ot B}$.
 It remains to be shown that (\ref{eq:twoanothYDconds}.a) holds. To this end,
 we compute that
 \begin{eqnarray*}
&&\hspace*{-2mm}
\gbeg{4}{6}
\got{1}{A}\got{1}{B}\got{1}{A}\gnl
\gcl{1}\gbrc\gnl
\gmu\gcn{1}{1}{1}{2}\gnl
\gcn{1}{1}{2}{3}\gvac{1}\gcmu\gnl
\gvac{1}\gbrbox\gvac{2}\gcl{1}\gnl
\gvac{1}\gob{1}{B}\gob{1}{A}\gob{1}{B}
\gend
\equal{(\ref{eq:neccconds}.c)}
\gbeg{5}{8}
\got{1}{A}\got{1}{B}\got{1}{A}\gnl
\gcl{1}\gbrc\gnl
\gmu\gcn{1}{1}{1}{4}\gnl
\gcn{1}{1}{2}{1}\gu{1}\gu{1}\gcmu\gnl
\gbrbox\gvac{2}\gbr\gvac{-1}\gnot{\hspace*{-4mm}\Box}\gvac{1}\gcl{1}\gnl
\gcl{1}\gbr\gcl{1}\gcl{1}\gnl
\gmu\gmu\gcl{1}\gnl
\gob{2}{B}\gob{2}{A}\gob{1}{B}
\gend
\equal{(\ref{eq:BespDrabComp}.c)}
\gbeg{8}{13}
\got{1}{A}\gvac{3}\got{1}{B}\got{1}{A}\gnl
\gcl{1}\gvac{3}\gbrc\gnl
\gcl{1}\gu{1}\gvac{2}\gcn{1}{1}{1}{-2}\gcn{1}{1}{1}{4}\gnl
\gbrbox\gvac{2}\gcmu\gu{1}\gu{1}\gcmu\gnl
\gcl{1}\gbr\gbrbox\gvac{2}\gbr\gvac{-1}\gnot{\hspace*{-4mm}\Box}\gvac{1}\gcl{1}\gnl
\gbrc\gbr\gcl{1}\gcl{1}\gcl{1}\gcl{1}\gnl
\gcu{1}\gcl{1}\gcl{1}\gmu\gcl{1}\gcl{1}\gcl{1}\gnl
\gvac{1}\gcl{1}\gcl{1}\gcn{1}{1}{2}{3}\gvac{1}\gcl{1}\gcl{1}\gcl{1}\gnl
\gvac{1}\gcl{1}\gcn{1}{1}{1}{3}\gvac{1}\gbr\gcl{1}\gcl{1}\gnl
\gvac{1}\gcn{1}{2}{1}{3}\gvac{1}\gmu\gmu\gcl{3}\gnl
\gvac{3}\gcn{1}{1}{2}{1}\gvac{1}\gcn{1}{2}{2}{2}\gnl
\gvac{2}\gmu\gnl
\gvac{2}\gob{2}{B}\gvac{1}\gob{2}{A}\gob{1}{B}
\gend\\
&&
\equal{\equref{nat1cup}}
\gbeg{8}{11}
\got{1}{A}\gvac{3}\got{1}{B}\got{1}{A}\gnl
\gcl{1}\gvac{3}\gbrc\gnl
\gcl{1}\gu{1}\gvac{2}\gcn{1}{1}{1}{-2}\gcn{1}{1}{1}{4}\gnl
\gbrbox\gvac{2}\gcmu\gu{1}\gu{1}\gcmu\gnl
\gcl{1}\gbr\gbrbox\gvac{2}\gbr\gvac{-1}\gnot{\hspace*{-4mm}\Box}\gvac{1}\gcl{1}\gnl
\gbrc\gbr\gbr\gcl{1}\gcl{1}\gnl
\gcu{1}\gcl{1}\gcl{1}\gbr\gmu\gcl{1}\gnl
\gvac{1}\gcl{1}\gmu\gcl{1}\gcn{1}{1}{2}{1}\gvac{1}\gcl{3}\gnl
\gvac{1}\gcl{1}\gcn{1}{1}{2}{1}\gvac{1}\gmu\gnl
\gvac{1}\gmu\gvac{1}\gcn{1}{1}{2}{2}\gnl
\gvac{1}\gob{2}{B}\gvac{1}\gob{2}{A}\gvac{1}\gob{1}{B}
\gend
\equalupdown{\equref{nat1cup}}{(\ref{eq:neccconds}.c)}
\gbeg{6}{8}
\got{1}{A}\gvac{2}\got{1}{B}\got{1}{A}\gnl
\gcl{1}\gvac{2}\gbrc\gnl
\gcl{1}\gu{1}\gvac{1}\gcn{1}{1}{1}{0}\gcn{1}{1}{1}{2}\gnl
\gbrbox\gvac{2}\gcmu\gcmu\gnl
\gcl{1}\gbr\gbrbox\gvac{2}\gcl{1}\gnl
\gbrc\gbr\gcl{1}\gcl{1}\gnl
\gcu{1}\gmu\gmu\gcl{1}\gnl
\gvac{1}\gob{2}{B}\gob{2}{A}\gob{1}{B}
\gend
\equal{(\ref{eq:neccconds}.d)}
\gbeg{8}{12}
\got{1}{A}\gvac{2}\got{2}{B}\got{2}{A}\gnl
\gcl{1}\gvac{1}\gu{1}\gcmu\gcmu\gu{1}\gnl
\gcl{1}\gvac{1}\gbrbox\gvac{2}\gbr\gbrbox\gnl
\gcl{1}\gu{1}\gcl{1}\gbr\gbr\gcl{1}\gnl
\gbrbox\gvac{2}\gbrc\gbr\gbrc\gnl
\gcl{1}\gbr\gmu\gcl{1}\gcl{1}\gcl{1}\gnl
\gbrc\gcl{1}\gcn{1}{1}{2}{1}\gvac{1}\gcl{1}\gcl{1}\gcl{1}\gnl
\gcu{1}\gcl{1}\gbr\gvac{1}\gcn{1}{1}{1}{-1}\gcl{1}\gcl{1}\gnl
\gvac{1}\gmu\gmu\gvac{1}\gcl{1}\gcl{1}\gnl
\gvac{1}\gcn{1}{2}{2}{2}\gvac{1}\gcn{1}{1}{2}{3}\gvac{2}\gcn{1}{1}{1}{-1}\gcl{2}\gnl
\gvac{4}\gmu\gnl
\gvac{1}\gob{2}{B}\gvac{1}\gob{2}{A}\gvac{1}\gob{1}{A}
\gend\\
&&
\equal{\equref{nat1cup}\times 2}
\gbeg{8}{12}
\got{1}{A}\gvac{2}\got{2}{B}\got{2}{A}\gnl
\gcl{1}\gvac{1}\gu{1}\gcmu\gcmu\gu{1}\gnl
\gcl{1}\gvac{1}\gbrbox\gvac{2}\gbr\gbrbox\gnl
\gcl{1}\gu{1}\gcl{1}\gbr\gbr\gcl{1}\gnl
\gbrbox\gvac{2}\gbrc\gcl{1}\gcl{1}\gbrc\gnl
\gcl{1}\gbr\gcl{1}\gcl{1}\gcl{1}\gcl{1}\gcl{1}\gnl
\gbrc\gbr\gcl{1}\gcl{1}\gcl{1}\gcl{1}\gnl
\gcu{1}\gmu\gmu\gcl{1}\gcl{1}\gcl{1}\gnl
\gvac{1}\gcn{1}{1}{2}{5}\gvac{1}\gcn{1}{1}{2}{3}\gvac{1}\gcl{1}\gcl{1}\gcl{1}\gnl
\gvac{3}\gcl{1}\gbr\gcl{1}\gcl{1}\gnl
\gvac{3}\gmu\gmu\gcl{1}\gnl
\gvac{3}\gob{2}{B}\gob{2}{A}\gob{1}{B}
\gend
\equal{\equref{crossprodalg}}
\gbeg{8}{10}
\got{1}{A}\gvac{2}\got{2}{B}\got{2}{A}\gnl
\gcl{1}\gu{1}\gu{1}\gcmu\gcmu\gu{1}\gnl
\gbrbox\gvac{2}\gbr\gvac{-1}\gnot{\hspace*{-4mm}\Box}\gvac{1}\gbr\gbrbox\gnl
\gcl{1}\gbr\gbr\gbr\gcl{1}\gnl
\gmu\gbr\gcl{1}\gcl{1}\gbrc\gnl
\gcn{1}{1}{2}{3}\gvac{1}\gcl{1}\gmu\gcl{1}\gcl{1}\gcl{1}\gnl
\gvac{1}\gbrc\gcn{1}{1}{2}{3}\gvac{1}\gcl{1}\gcl{1}\gcl{1}\gnl
\gvac{1}\gcu{1}\gcn{1}{1}{1}{3}\gvac{1}\gbr\gcl{1}\gcl{1}\gnl
\gvac{3}\gmu\gmu\gcl{1}\gnl
\gvac{3}\gob{2}{B}\gob{2}{A}\gob{1}{B}
\gend
\equalupdown{\equref{nat1cup}}{(\ref{eq:neccconds}.c)}
\gbeg{6}{7}
\got{1}{A}\got{2}{B}\got{2}{A}\gnl
\gcl{1}\gcmu\gcmu\gu{1}\gnl
\gbrbox\gvac{2}\gbr\gbrbox\gnl
\gcl{1}\gbr\gbr\gcl{1}\gnl
\gbrc\gbr\gbrc\gnl
\gcu{1}\gmu\gmu\gcl{1}\gnl
\gvac{1}\gob{2}{B}\gob{2}{A}\gob{1}{B}
\gend
\hspace{1mm},
\end{eqnarray*}
as needed. In the last but one equality we also applied the 
naturality of the braiding to the morphism $\psi$.\\
$\un{(vi)\Rightarrow (ii)}$. It is easy to see that (\ref{eq:twoanothYDconds}.a) implies 
(\ref{eq:BespDrabComp}.c,f). We know from \leref{someechivcond} that
(\ref{eq:BespDrabComp}.a,c) 
imply (\ref{eq:neccconds}.a), and that (\ref{eq:BespDrabComp}.b,d)  imply (\ref{eq:neccconds}.b).
(\ref{eq:neccconds}.d) can be proved as follows:
\begin{eqnarray*}
&&\hspace*{-2mm}
\gbeg{4}{6}
\gvac{1}\got{1}{B}\got{1}{A}\gnl
\gvac{1}\gbrc\gnl
\gvac{1}\gcn{1}{1}{1}{0}\gcn{1}{1}{1}{2}\gnl
\gcmu\gcmu\gnl
\gcl{1}\gbrbox\gvac{2}\gcl{1}\gnl
\gob{1}{A}\gob{1}{B}\gob{1}{A}\gob{1}{B}
\gend
\equal{(\ref{eq:neccconds}.g)}
\gbeg{4}{9}
\got{2}{B}\got{2}{A}\gnl
\gcmu\gcmu\gnl
\gcl{1}\gbr\gcl{1}\gnl
\gbrc\gbrc\gnl
\gcl{1}\gcu{1}\gcu{1}\gcl{1}\gnl
\gcn{1}{1}{1}{2}\gvac{2}\gcn{1}{1}{1}{0}\gnl
\gcmu\gcmu\gnl
\gcl{1}\gbrbox\gvac{2}\gcl{1}\gnl
\gob{1}{A}\gob{1}{B}\gob{1}{A}\gob{1}{B}
\gend
\equal{(\ref{eq:BespDrabComp}.e)}
\gbeg{7}{11}
\gvac{3}\got{2}{B}\got{2}{A}\gnl
\gvac{3}\gcmu\gcmu\gnl
\gvac{3}\gcn{1}{1}{1}{-2}\gbr\gcl{1}\gnl
\gu{1}\gcmu\gvac{1}\gcn{1}{1}{1}{0}\gbrc\gnl
\gbrbox\gvac{2}\gcl{1}\gcmu\gcu{1}\gcn{1}{1}{1}{0}\gnl
\gcl{1}\gcl{1}\gbr\gcl{1}\gcmu\gnl
\gcl{1}\gbr\gbrc\gcl{1}\gcl{1}\gnl
\gbrc\gmu\gcu{1}\gcl{1}\gcl{1}\gnl
\gcl{2}\gcu{1}\gcn{1}{1}{2}{3}\gvac{2}\gcn{1}{1}{1}{-1}\gcl{2}\gnl
\gvac{3}\gbrbox\gnl
\gob{1}{A}\gvac{2}\gob{1}{B}\gob{1}{A}\gvac{1}\gob{1}{B}
\gend
\equal{\equref{nat2cup}}
\gbeg{7}{11}
\gvac{1}\got{2}{B}\gvac{1}\got{2}{A}\gnl
\gvac{1}\gcmu\gvac{1}\gcmu\gnl
\gvac{1}\gcl{1}\gcn{1}{1}{1}{2}\gvac{1}\gcl{1}\gcn{1}{1}{1}{2}\gnl
\gu{1}\gcl{1}\gcmu\gcl{1}\gcmu\gnl
\gbrbox\gvac{2}\gcl{1}\gbr\gcl{1}\gcl{1}\gnl
\gcl{1}\gcl{1}\gbr\gbr\gcl{1}\gnl
\gcl{1}\gbr\gbrc\gbrc\gnl
\gbrc\gmu\gcu{1}\gcu{1}\gcn{1}{1}{1}{0}\gnl
\gcl{1}\gcu{1}\gcn{1}{1}{2}{5}\gvac{2}\gcmu\gnl
\gcl{1}\gvac{3}\gbrbox\gvac{2}\gcl{1}\gnl
\gob{1}{A}\gvac{3}\gob{1}{B}\gob{1}{A}\gob{1}{B}
\gend\\
&&
\equalupdown{\equref{nat2cup}}{(\ref{eq:neccconds}.g)}
\gbeg{6}{9}
\gvac{1}\got{2}{B}\got{2}{A}\gnl
\gu{1}\gcmu\gcmu\gnl
\gbrbox\gvac{2}\gbr\gcl{1}\gnl
\gcl{1}\gbr\gbrc\gnl
\gbrc\gmu\gcl{1}\gnl
\gcl{1}\gcu{1}\gcn{1}{2}{2}{3}\gvac{1}\gcn{1}{1}{1}{2}\gnl
\gcl{1}\gvac{3}\gcmu\gnl
\gcl{1}\gvac{2}\gbrbox\gvac{2}\gcl{1}\gnl
\gob{1}{A}\gvac{2}\gob{1}{B}\gob{1}{A}\gob{1}{B}
\gend
\equal{(\ref{eq:twoanothYDconds}.a)}
\gbeg{8}{10}
\gvac{1}\got{2}{B}\got{2}{A}\gnl
\gu{1}\gcmu\gcmu\gnl
\gbrbox\gvac{2}\gbr\gcn{1}{1}{1}{4}\gnl
\gcl{1}\gbr\gcn{1}{1}{1}{2}\gvac{1}\gcmu\gu{1}\gnl
\gbrc\gcl{1}\gcmu\gcl{1}\gbrbox\gnl
\gcl{1}\gcu{1}\gbrbox\gvac{2}\gbr\gcl{1}\gcl{1}\gnl
\gcl{1}\gvac{1}\gcl{1}\gbr\gbr\gcl{1}\gnl
\gcl{1}\gvac{1}\gbrc\gbr\gbrc\gnl
\gcl{1}\gvac{1}\gcu{1}\gmu\gmu\gcl{1}\gnl
\gob{1}{A}\gvac{2}\gob{2}{B}\gob{2}{A}\gob{1}{B}
\gend
\equal{\equref{nat2cup}}
\gbeg{8}{11}
\gvac{2}\got{2}{B}\gvac{1}\got{2}{A}\gnl
\gvac{2}\gcmu\gvac{1}\gcmu\gnl
\gvac{2}\gcn{1}{1}{1}{0}\gcl{1}\gvac{1}\gcn{1}{1}{1}{0}\gcl{1}\gnl
\gu{1}\gcmu\gcl{1}\gcmu\gcl{1}\gu{1}\gnl
\gbrbox\gvac{2}\gcl{1}\gbr\gcl{1}\gbrbox\gnl
\gcl{1}\gbrbox\gvac{2}\gcl{1}\gcl{1}\gcl{1}\gcl{1}\gcl{1}\gnl
\gcl{1}\gcl{1}\gbr\gbr\gcl{1}\gcl{1}\gnl
\gcl{1}\gbr\gbr\gbr\gcl{1}\gnl
\gbrc\gbrc\gbr\gbrc\gnl
\gcl{1}\gcu{1}\gcu{1}\gmu\gmu\gcl{1}\gnl
\gob{1}{A}\gvac{2}\gob{2}{B}\gob{2}{A}\gob{1}{B}
\gend\\
&&
\equalupdown{\equref{crossprodcoalg}}{\equref{nat2cup}}
\gbeg{8}{11}
\gvac{2}\got{2}{B}\gvac{1}\got{2}{A}\gnl
\gvac{1}\gu{1}\gcmu\gvac{1}\gcmu\gu{1}\gnl
\gvac{1}\gbrbox\gvac{2}\gcn{1}{1}{1}{3}\gvac{1}\gcl{1}\gbrbox\gnl
\gvac{1}\gcn{1}{1}{1}{0}\gcl{1}\gvac{1}\gbr\gcl{1}\gcl{1}\gnl
\gcmu\gcl{1}\gvac{1}\gcn{1}{1}{1}{0}\gbr\gcl{1}\gnl
\gcl{1}\gcl{1}\gcl{1}\gcmu\gcl{1}\gbrc\gnl
\gcl{1}\gcl{1}\gbr\gcl{1}\gcl{1}\gcl{1}\gcl{1}\gnl
\gcl{1}\gbr\gbr\gcl{1}\gcl{1}\gcl{1}\gnl
\gbrc\gbrc\gbr\gcl{1}\gcl{1}\gnl
\gcl{1}\gcu{1}\gcu{1}\gmu\gmu\gcl{1}\gnl
\gob{1}{A}\gvac{2}\gob{2}{B}\gob{2}{A}\gob{1}{B}
\gend
\equal{\equref{nat2cup}}
\gbeg{8}{12}
\gvac{3}\got{2}{B}\got{2}{A}\gnl
\gvac{3}\gcmu\gcmu\gnl
\gvac{1}\gu{1}\gcn{1}{1}{3}{1}\gvac{1}\gcl{1}\gcl{1}\gcl{1}\gu{1}\gnl
\gvac{1}\gbrbox\gvac{3}\gbr\gbrbox\gnl
\gvac{1}\gcn{1}{1}{1}{0}\gcn{1}{1}{1}{3}\gvac{1}\gcl{1}\gcl{1}\gcl{1}\gcl{1}\gnl
\gcmu\gvac{1}\gbr\gbr\gcl{1}\gnl
\gcl{1}\gcl{1}\gvac{1}\gcn{1}{1}{1}{0}\gcl{1}\gcl{1}\gcl{1}\gcl{1}\gnl
\gcl{1}\gcl{1}\gcmu\gbr\gbrc\gnl
\gcl{1}\gbr\gcl{1}\gcl{1}\gmu\gcl{3}\gnl
\gbrc\gbrc\gcl{1}\gcn{1}{2}{2}{2}\gnl
\gcl{1}\gcu{1}\gcu{1}\gmu\gnl
\gob{1}{A}\gvac{2}\gob{2}{B}\gob{2}{A}\gob{1}{B}
\gend
\equal{(\ref{eq:neccconds}.g)}
\gbeg{6}{7}
\gvac{1}\got{2}{B}\got{2}{A}\gnl
\gu{1}\gcmu\gcmu\gu{1}\gnl
\gbrbox\gvac{2}\gbr\gbrbox\gnl
\gcl{1}\gbr\gbr\gcl{1}\gnl
\gbrc\gbr\gbrc\gnl
\gcl{1}\gmu\gmu\gcl{1}\gnl
\gob{1}{A}\gob{2}{B}\gob{2}{A}\gob{1}{B}
\gend
\hspace{2mm}.
\end{eqnarray*}
Finally, we observe that the proof of $(iii)\Leftrightarrow (vii)$ is similar to the 
proof of $(ii)\Leftrightarrow (vi)$.
 \end{proof} 
 
 If $C$ is a coalgebra and $\mathbb{A}$ is an algebra,
 then ${\rm Hom}_\Cc(C, \mathbb{A})$ is a monoid, with the convolution
 $f * g =\un{m}_{\mathbb{A}}(f\ot g)\un{\Delta}_C$ as multiplication, and unit $\un{\eta}_{\mathbb{A}}\un{\va}_C$.
 We will now discuss some sufficient conditions for a cross product bialgebra to be a Hopf
 algebra.
 
 \begin{proposition}\prlabel{whenacrossprodisHA}
 Let $A\times_\psi^\phi B$ be a cross product bialgebra, and assume that $\Id_A$
 and $\Id_B$ are convolution invertible. Then $A\times_\psi^\phi B$ is a Hopf algebra.
 \end{proposition}
 
 \begin{proof}
Let $\un{S}$ be the convolution inverse of $\Id_A$, and $\un{s}$ the convolution inverse of $\Id_B$.
We then claim that 
$
\gbeg{2}{5}
\got{1}{A}\got{1}{B}\gnl
\gbrbox\gnl
\gmp{\un{s}}\gmp{\un{S}}\gnl
\gbrc\gnl
\gob{1}{A}\gob{1}{B}
\gend
$ 
is the antipode for $A\times_\psi^\phi B$. Indeed, we have
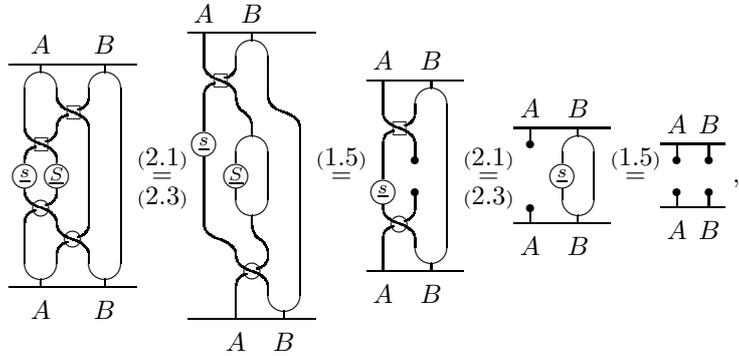
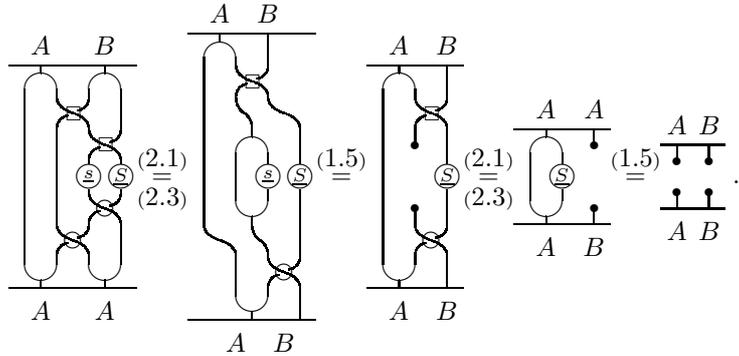
\[
\gbeg{4}{9}
\got{2}{A}\got{2}{B}\gnl
\gcmu\gcmu\gnl
\gcl{1}\gbrbox\gvac{2}\gcl{1}\gnl
\gbrbox\gvac{2}\gcl{1}\gcl{1}\gnl
\gmp{\un{s}}\gmp{\un{S}}\gcl{1}\gcl{1}\gnl
\gbrc\gcl{1}\gcl{1}\gnl
\gcl{1}\gbrc\gcl{1}\gnl
\gmu\gmu\gnl
\gob{2}{A}\gob{2}{B}
\gend
\equalupdown{\equref{crossprodalg}}{\equref{crossprodcoalg}}
\gbeg{4}{11}
\got{1}{A}\got{2}{B}\gnl
\gcl{1}\gcmu\gnl
\gbrbox\gvac{2}\gcl{1}\gnl
\gcl{1}\gcn{1}{1}{1}{2}\gcn{1}{1}{1}{3}\gnl
\gmp{\un{s}}\gcmu\gcl{1}\gnl
\gcl{1}\gmp{\un{S}}\gcl{1}\gcl{1}\gnl
\gcl{1}\gmu\gcl{1}\gnl
\gcn{1}{1}{1}{3}\gcn{1}{1}{2}{3}\gvac{1}\gcl{1}\gnl
\gvac{1}\gbrc\gcl{1}\gnl
\gvac{1}\gcl{1}\gmu\gnl
\gvac{1}\gob{1}{A}\gob{2}{B}
\gend
\equal{\equref{braidedantipode}}
\gbeg{3}{8}
\got{1}{A}\got{2}{B}\gnl
\gcl{1}\gcmu\gnl
\gbrbox\gvac{2}\gcl{1}\gnl
\gcl{1}\gcu{1}\gcl{1}\gnl
\gmp{\un{s}}\gu{1}\gcl{1}\gnl
\gbrc\gcl{1}\gnl
\gcl{1}\gmu\gnl
\gob{1}{A}\gob{2}{B}
\gend
\equalupdown{\equref{crossprodalg}}{\equref{crossprodcoalg}}
\gbeg{3}{5}
\got{1}{A}\got{2}{B}\gnl
\gcu{1}\gcmu\gnl
\gvac{1}\gmp{\un{s}}\gcl{1}\gnl
\gu{1}\gmu\gnl
\gob{1}{A}\gob{2}{B}
\gend
\equal{\equref{braidedantipode}}
\gbeg{2}{4}
\got{1}{A}\got{1}{B}\gnl
\gcu{1}\gcu{1}\gnl
\gu{1}\gu{1}\gnl
\gob{1}{A}\gob{1}{B}
\gend\hspace{1mm},
\]
and
\[
\gbeg{4}{9}
\got{2}{A}\got{2}{B}\gnl
\gcmu\gcmu\gnl
\gcl{1}\gbrbox\gvac{2}\gcl{1}\gnl
\gcl{1}\gcl{1}\gbrbox\gnl
\gcl{1}\gcl{1}\gmp{\un{s}}\gmp{\un{S}}\gnl
\gcl{1}\gcl{1}\gbrc\gnl
\gcl{1}\gbrc\gcl{1}\gnl
\gmu\gmu\gnl
\gob{2}{A}\gob{2}{A}
\gend
\equalupdown{\equref{crossprodalg}}{\equref{crossprodcoalg}}
\gbeg{4}{11}
\got{2}{A}\got{1}{B}\gnl
\gcmu\gcl{1}\gnl
\gcl{1}\gbrbox\gnl
\gcl{1}\gcn{1}{1}{1}{2}\gcn{1}{1}{1}{3}\gnl
\gcl{1}\gcmu\gcl{1}\gnl
\gcl{1}\gcl{1}\gmp{\un{s}}\gmp{\un{S}}\gnl
\gcl{1}\gmu\gcl{1}\gnl
\gcn{1}{1}{1}{3}\gcn{1}{1}{2}{3}\gvac{1}\gcl{1}\gnl
\gvac{1}\gcl{1}\gbrc\gnl
\gvac{1}\gmu\gcl{1}\gnl
\gvac{1}\gob{1}{A}\gob{2}{B}
\gend
\equal{\equref{braidedantipode}}
\gbeg{3}{9}
\got{2}{A}\got{1}{B}\gnl
\gcmu\gcl{1}\gnl
\gcl{1}\gbrbox\gnl
\gcl{1}\gcu{1}\gcl{1}\gnl
\gcl{1}\gvac{1}\gmp{\un{S}}\gnl
\gcl{1}\gu{1}\gcl{1}\gnl
\gcl{1}\gbrc\gnl
\gmu\gcl{1}\gnl
\gob{2}{A}\gob{1}{B}
\gend
\equalupdown{\equref{crossprodalg}}{\equref{crossprodcoalg}}
\gbeg{3}{5}
\got{2}{A}\got{1}{A}\gnl
\gcmu\gcu{1}\gnl
\gcl{1}\gmp{\un{S}}\gnl
\gmu\gu{1}\gnl
\gob{2}{A}\gob{1}{B}
\gend
\equal{\equref{braidedantipode}}
\gbeg{2}{4}
\got{1}{A}\got{1}{B}\gnl
\gcu{1}\gcu{1}\gnl
\gu{1}\gu{1}\gnl
\gob{1}{A}\gob{1}{B}
\gend\hspace{1mm}.
\]
\end{proof}

\begin{remark}\relabel{antpartneccconds}
Some of the sufficient conditions in \prref{whenacrossprodisHA} are also necessary. More precisely, if 
$A\times_\psi^\phi B$ admits 
$
\gbeg{2}{5}
\got{1}{A}\got{1}{B}\gnl
\gcl{1}\gcl{1}\gnl
\gsbox{2}\gnl
\gcl{1}\gcl{1}\gnl
\gob{1}{A}\gob{1}{B}
\gend
$ 
as antipode, then \equref{braidedantipode} specializes to 
\begin{equation}\eqlabel{defantcpHa}
\gbeg{4}{7}
\got{2}{A}\got{2}{B}\gnl
\gcmu\gcmu\gnl
\gcl{1}\gbrbox\gvac{2}\gcl{1}\gnl
\gsbox{2}\gvac{2}\gcl{1}\gcl{1}\gnl
\gcl{1}\gbrc\gcl{1}\gnl
\gmu\gmu\gnl
\gob{2}{A}\gob{2}{B}
\gend
=
\gbeg{2}{4}
\got{1}{A}\got{1}{B}\gnl
\gcu{1}\gcu{1}\gnl
\gu{1}\gu{1}\gnl
\gob{1}{A}\gob{1}{B}
\gend
=
\gbeg{4}{7}
\got{2}{A}\got{2}{B}\gnl
\gcmu\gcmu\gnl
\gcl{1}\gbrbox\gvac{2}\gcl{1}\gnl
\gcl{1}\gcl{1}\gsbox{2}\gnl
\gcl{1}\gbrc\gcl{1}\gnl
\gmu\gmu\gnl
\gob{2}{A}\gob{2}{B}
\gend
\hspace{1mm}.
\end{equation}
If we compose the first equality to the left with $\un{\va}_A\ot \Id_B$ and to the right with $\un{\eta}_A\ot \Id_B$, 
and the second equality to the left with $\Id_A\ot \un{\va}_B$ and to the right with $\Id_A\ot \un{\eta}_B$, 
we deduce that 
\begin{equation}\eqlabel{deriveddefantcpHa}
\gbeg{4}{7}
\gvac{2}\got{2}{B}\gnl
\gvac{1}\gu{1}\gcmu\gnl
\gu{1}\gbrbox\gvac{2}\gcl{1}\gnl
\gsbox{2}\gvac{2}\gcl{1}\gcl{1}\gnl
\gcu{1}\gbrc\gcl{1}\gnl
\gvac{1}\gcu{1}\gmu\gnl
\gvac{2}\gob{2}{B}
\gend
=
\gbeg{1}{4}
\got{1}{B}\gnl
\gcu{1}\gnl
\gu{1}\gnl
\gob{1}{B}
\gend
\hspace{2mm}{\rm and}\hspace{2mm}
\gbeg{4}{7}
\got{2}{A}\gnl
\gcmu\gu{1}\gnl
\gcl{1}\gbrbox\gvac{2}\gu{1}\gnl
\gcl{1}\gcl{1}\gsbox{2}\gnl
\gcl{1}\gbrc\gcu{1}\gnl
\gmu\gcu{1}\gnl
\gob{2}{A}
\gend
=
\gbeg{1}{4}
\got{1}{A}\gnl
\gcu{1}\gnl
\gu{1}\gnl
\gob{1}{A}
\gend\hspace{1mm}.
\end{equation}
This means that $\Id_B$ has a left inverse in ${\rm Hom}(B, B)$ and that
$\Id_A$ has a right inverse in ${\rm Hom}(A, A)$. At this moment, it remains unclear to us
whether these one-sided inverses are inverses. We will see in \seref{strHopfalgwithprof} that 
this is true in the case of a smash (co)product Hopf algebra. 
\end{remark}
\section{Cross product bialgebras and Hopf data}\selabel{crossprodvsHopfdatum}
\setcounter{equation}{0}
If $(A,B,\psi,\phi)$ is a cross product algebra-coalgebra datum, then $A$ is a left $B$-module and a left
$B$-comodule, and $B$ is a right $A$-module and a right $A$-comodule, see Lemmas \ref{le:action}
and \ref{le:coaction}. Now we can ask the following question: suppose that $A$ and $B$ are algebras
and coalgebras, that $A$ is a left $B$-module and a left
$B$-comodule, and $B$ is a right $A$-module and a right $A$-comodule. Is there a list of necessary
and sufficient conditions that these actions and coactions need to satisfy, so that they give rise
to bialgebra admissible tuple?\\
This question was partially answered in \cite{bespdrab1}. In \cite[Def. 2.5]{bespdrab1}, a list of axioms
is proposed, we call this list the Bespalov-Drabant list. If these axioms are satisfied, then $(A,B)$
is a Hopf pair. If $(A,B,\psi,\phi)$ is a bialgebra admissible tuple, then $(A,B)$, with actions and
coactions given by (\ref{eq:action1}-\ref{eq:coaction1}), is a Hopf pair, see \cite[Prop. 2.7]{bespdrab1}.
Moreover, the - crucial - conditions (\ref{eq:neccconds}.g,c) show that $\psi$ and $\phi$ can be
recovered from the actions and coactions.\\
Conversely, given a Hopf pair, we can produce a cross product algebra-coalgebra datum
$(A,B,\psi,\phi)$, but we don't know whether it is a bialgebra admissible tuple, see \cite[Prop. 2.6]{bespdrab1}.
Otherwise stated, we obtain a cross product algebra and coalgebra, but we don't know whether it is a bialgebra.
We could also say the following: the Bespalov-Drabant list is necessary, but not sufficient.
Using the results of \seref{crossprodbialg}, we are able to present an alternative list of necessary and
sufficient conditions. Basically, this is a - technical - restatement of \thref{firstsetequivcond}.
The computations will turned out to be quite lengthy, and this
is why we decided to divide them over several Lemmas.

\begin{lemma}\lelabel{implic1}
Let $A, B$ be algebras and coalgebras such that 
$\un{\va}_X\circ \un{\eta}_X=\Id_{\un{1}}$, 
$\un{\va}_X\circ \un{m}_X=\un{\va}_X\ot \un{\va}_X$ and $\un{\Delta}_X\circ \un{\eta}_X=\un{\eta}_X\ot \un{\eta}_X$, 
for all $X\in \{A, B\}$. Furthermore, assume that 
$\psi: B\ot A\ra A\ot B$ and $\phi : A\ot B\ra B\ot A$ are morphisms in $\Cc$ satisfying 
(\ref{eq:crossprodalg}.c-d) and (\ref{eq:crossprodcoalg}.c-d). Then\\
(i) 
$
\gbeg{3}{5}
\gvac{1}\got{1}{B}\got{1}{A}\gnl
\gvac{1}\gbrc\gnl
\gvac{1}\gcn{1}{1}{1}{0}\gcl{1}\gnl
\gcmu\gcl{1}\gnl
\gob{1}{A}\gob{1}{A}\gob{1}{B}
\gend
=
\gbeg{5}{7}
\gvac{1}\got{2}{B}\got{2}{A}\gnl
\gu{1}\gcmu\gcmu\gnl
\gbrbox\gvac{2}\gbr\gcl{1}\gnl
\gcl{1}\gbr\gbrc\gnl
\gbrc\gmu\gcl{1}\gnl
\gcl{1}\gcu{1}\gcn{1}{1}{2}{2}\gvac{1}\gcl{1}\gnl
\gob{1}{A}\gvac{1}\gob{2}{A}\gob{1}{B}
\gend
$ 
if and only if (\ref{eq:neccconds}.g) and (\ref{eq:BespDrabComp}.e) hold;\\
(ii) 
$
\gbeg{3}{5}
\got{1}{B}\got{1}{A}\gnl
\gbrc\gnl
\gcl{1}\gcn{1}{1}{1}{2}\gnl
\gcl{1}\gcmu\gnl
\gob{1}{A}\gob{1}{B}\gob{1}{B}
\gend
=
\gbeg{5}{7}
\got{2}{B}\got{2}{A}\gnl
\gcmu\gcmu\gu{1}\gnl
\gcl{1}\gbr\gbrbox\gnl
\gbrc\gbr\gcl{1}\gnl
\gcl{1}\gmu\gbrc\gnl
\gcl{1}\gcn{1}{1}{2}{2}\gvac{1}\gcu{1}\gcl{1}\gnl
\gob{1}{A}\gob{2}{B}\gvac{1}\gob{1}{B}
\gend
$ 
if and only if (\ref{eq:neccconds}.g) and (\ref{eq:BespDrabComp}.f) hold;\\
(iii) 
$
\gbeg{3}{5}
\got{1}{A}\got{1}{A}\got{1}{B}\gnl
\gmu\gcl{1}\gnl
\gcn{1}{1}{2}{3}\gvac{1}\gcl{1}\gnl
\gvac{1}\gbrbox\gnl
\gvac{1}\gob{1}{B}\gob{1}{A}
\gend
=
\gbeg{5}{7}
\got{1}{A}\gvac{1}\got{2}{A}\got{1}{B}\gnl
\gcl{1}\gu{1}\gcmu\gcl{1}\gnl
\gbrbox\gvac{2}\gcl{1}\gbrbox\gnl
\gcl{1}\gbr\gcl{1}\gcl{1}\gnl
\gbrc\gbr\gcl{1}\gnl
\gcu{1}\gmu\gmu\gnl
\gvac{1}\gob{2}{B}\gob{2}{A}
\gend
$ 
if and only if (\ref{eq:neccconds}.c) and (\ref{eq:BespDrabComp}.c) hold;\\
(iv) 
$
\gbeg{3}{5}
\got{1}{A}\got{1}{B}\got{1}{B}\gnl
\gcl{1}\gmu\gnl
\gcl{1}\gcn{1}{1}{2}{1}\gnl
\gbrbox\gnl
\gob{1}{B}\gob{1}{A}
\gend
=
\gbeg{5}{7}
\got{1}{A}\got{2}{B}\gvac{1}\got{1}{B}\gnl
\gcl{1}\gcmu\gu{1}\gcl{1}\gnl
\gbrbox\gvac{2}\gcl{1}\gbrbox\gnl
\gcl{1}\gcl{1}\gbr\gcl{1}\gnl
\gcl{1}\gbr\gbrc\gnl
\gmu\gmu\gcu{1}\gnl
\gob{2}{B}\gob{2}{A}
\gend
$ 
if and only if (\ref{eq:neccconds}.c) and (\ref{eq:BespDrabComp}.d) hold. 
\end{lemma}

\begin{proof}
We only prove (i). The proof of (ii), (iii) and (iv) is similar. Actually (iii) and (iv) follow
from (i) and (ii) by duality arguments.\\
The direct implication in (i) follows easily by composing the given equality to the left
with $\un{\va}_A\ot \Id_{A\ot B}$, to obtain (\ref{eq:neccconds}.g), and
with $\Id_{A\ot A}\ot \un{\va}_B$, to obtain (\ref{eq:BespDrabComp}.e).\\
To prove the converse, we compute
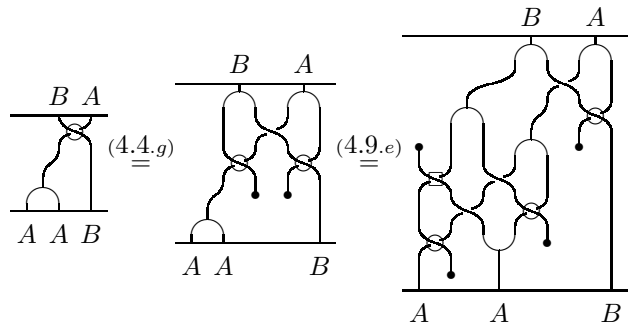
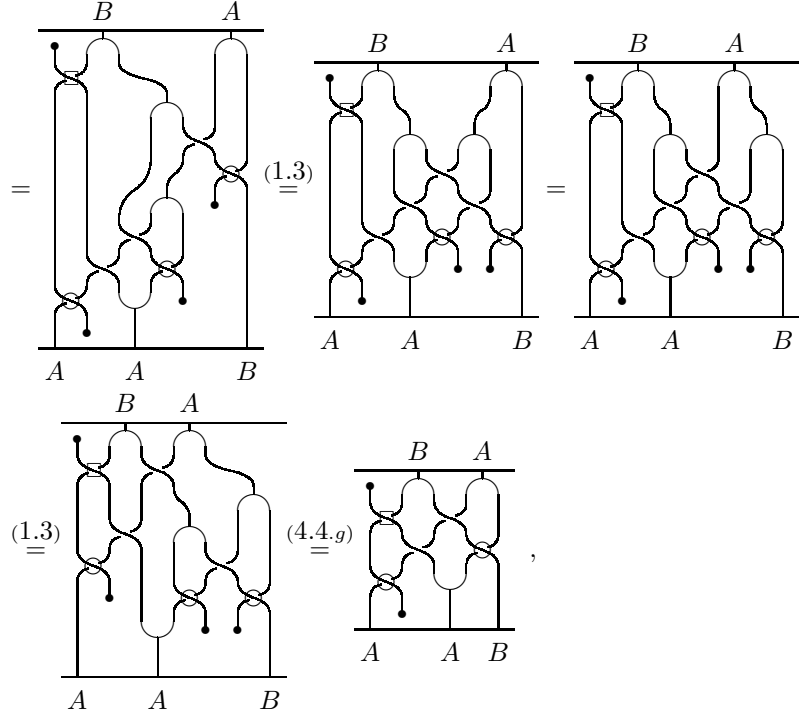
\begin{eqnarray*}
&&\hspace*{-3mm}
\gbeg{3}{5}
\gvac{1}\got{1}{B}\got{1}{A}\gnl
\gvac{1}\gbrc\gnl
\gvac{1}\gcn{1}{1}{1}{0}\gcl{1}\gnl
\gcmu\gcl{1}\gnl
\gob{1}{A}\gob{1}{A}\gob{1}{B}
\gend
\equal{(\ref{eq:neccconds}.g)}
\gbeg{5}{7}
\gvac{1}\got{2}{B}\got{2}{A}\gnl
\gvac{1}\gcmu\gcmu\gnl
\gvac{1}\gcl{1}\gbr\gcl{1}\gnl
\gvac{1}\gbrc\gbrc\gnl
\gvac{1}\gcn{1}{1}{1}{0}\gcu{1}\gcu{1}\gcl{1}\gnl
\gcmu\gvac{2}\gcl{1}\gnl
\gob{1}{A}\gob{1}{A}\gvac{2}\gob{1}{B}
\gend
\equal{(\ref{eq:BespDrabComp}.e)}
\gbeg{7}{10}
\gvac{3}\got{2}{B}\got{2}{A}\gnl
\gvac{3}\gcmu\gcmu\gnl
\gvac{3}\gcn{1}{1}{1}{-2}\gbr\gcl{1}\gnl
\gvac{1}\gcmu\gvac{1}\gcn{1}{1}{1}{0}\gbrc\gnl
\gu{1}\gcl{1}\gcl{1}\gcmu\gcu{1}\gcl{1}\gnl
\gbrbox\gvac{2}\gbr\gcl{1}\gvac{1}\gcl{1}\gnl
\gcl{1}\gbr\gbrc\gvac{1}\gcl{1}\gnl
\gbrc\gmu\gcu{1}\gvac{1}\gcl{1}\gnl
\gcl{1}\gcu{1}\gcn{1}{1}{2}{2}\gvac{3}\gcl{1}\gnl
\gob{1}{A}\gvac{1}\gob{2}{A}\gvac{2}\gob{1}{B}
\gend\\
&&\hspace*{3mm}
=
\gbeg{7}{12}
\gvac{1}\got{2}{B}\gvac{2}\got{2}{A}\gnl
\gu{1}\gcmu\gvac{2}\gcmu\gnl
\gbrbox\gvac{2}\gcn{1}{1}{1}{4}\gvac{2}\gcl{1}\gcl{1}\gnl
\gcl{1}\gcl{1}\gvac{1}\gcmu\gcl{1}\gcl{1}\gnl
\gcl{1}\gcl{1}\gvac{1}\gcl{1}\gbr\gcl{1}\gnl
\gcl{1}\gcl{1}\gvac{1}\gcn{1}{2}{1}{-1}\gcn{1}{1}{1}{0}\gbrc\gnl
\gcl{1}\gcl{1}\gvac{1}\gcmu\gcu{1}\gcl{5}\gnl
\gcl{1}\gcl{1}\gbr\gcl{1}\gnl
\gcl{1}\gbr\gbrc\gnl
\gbrc\gmu\gcu{1}\gnl
\gcl{1}\gcu{1}\gcn{1}{1}{2}{2}\gnl
\gob{1}{A}\gvac{1}\gob{2}{A}\gvac{2}\gob{1}{B}
\gend
\equal{\equref{nat2cup}}
\gbeg{7}{10}
\gvac{1}\got{2}{B}\gvac{2}\got{2}{A}\gnl
\gu{1}\gcmu\gvac{2}\gcmu\gnl
\gbrbox\gvac{2}\gcn{1}{1}{1}{2}\gvac{2}\gcn{1}{1}{1}{0}\gcl{1}\gnl
\gcl{1}\gcl{1}\gcmu\gcmu\gcl{1}\gnl
\gcl{1}\gcl{1}\gcl{1}\gbr\gcl{1}\gcl{1}\gnl
\gcl{1}\gcl{1}\gbr\gbr\gcl{1}\gnl
\gcl{1}\gbr\gbrc\gbrc\gnl
\gbrc\gmu\gcu{1}\gcu{1}\gcl{1}\gnl
\gcl{1}\gcu{1}\gcn{1}{1}{2}{2}\gvac{3}\gcl{1}\gnl
\gob{1}{A}\gvac{1}\gob{2}{A}\gvac{2}\gob{1}{B}
\gend
=
\gbeg{7}{10}
\gvac{1}\got{2}{B}\gvac{1}\got{2}{A}\gnl
\gu{1}\gcmu\gvac{1}\gcmu\gnl
\gbrbox\gvac{2}\gcn{1}{1}{1}{2}\gvac{1}\gcl{1}\gcn{1}{1}{1}{2}\gnl
\gcl{1}\gcl{1}\gcmu\gcl{1}\gcmu\gnl
\gcl{1}\gcl{1}\gcl{1}\gbr\gcl{1}\gcl{1}\gnl
\gcl{1}\gcl{1}\gbr\gbr\gcl{1}\gnl
\gcl{1}\gbr\gbrc\gbrc\gnl
\gbrc\gmu\gcu{1}\gcu{1}\gcl{1}\gnl
\gcl{1}\gcu{1}\gcn{1}{1}{2}{2}\gvac{3}\gcl{1}\gnl
\gob{1}{A}\gvac{1}\gob{2}{A}\gvac{2}\gob{1}{B}
\gend\\
&&\hspace{3mm}
\equal{\equref{nat2cup}}
\gbeg{7}{10}
\gvac{1}\got{2}{B}\got{2}{A}\gnl
\gu{1}\gcmu\gcmu\gnl
\gbrbox\gvac{2}\gbr\gcn{1}{1}{1}{4}\gnl
\gcl{1}\gcl{1}\gcl{1}\gcn{1}{1}{1}{2}\gvac{1}\gcmu\gnl
\gcl{1}\gbr\gcmu\gcl{1}\gcl{1}\gnl
\gbrc\gcl{1}\gcl{1}\gbr\gcl{1}\gnl
\gcl{1}\gcu{1}\gcl{1}\gbrc\gbrc\gnl
\gcl{1}\gvac{1}\gmu\gcu{1}\gcu{1}\gcl{1}\gnl
\gcl{1}\gvac{1}\gcn{1}{1}{2}{2}\gvac{3}\gcl{1}\gnl
\gob{1}{A}\gvac{1}\gob{2}{A}\gvac{2}\gob{1}{B}
\gend
\equal{(\ref{eq:neccconds}.g)}
\gbeg{5}{7}
\gvac{1}\got{2}{B}\got{2}{A}\gnl
\gu{1}\gcmu\gcmu\gnl
\gbrbox\gvac{2}\gbr\gcl{1}\gnl
\gcl{1}\gbr\gbrc\gnl
\gbrc\gmu\gcl{1}\gnl
\gcl{1}\gcu{1}\gcn{1}{1}{2}{2}\gvac{1}\gcl{1}\gnl
\gob{1}{A}\gvac{1}\gob{2}{A}\gob{1}{B}
\gend\hspace{2mm},
\end{eqnarray*}  
as required. Note that we used 
the coassociativity of $\un{\Delta}_B$ and $\un{\Delta}_A$ in the third and the fifth equality.
\end{proof}

Our next aim is to show that (\ref{eq:crossprodalg}.a-b) are satisfied if
(\ref{eq:neccconds}.g), (\ref{eq:BespDrabComp}.a,b,e,f), and
\begin{eqnarray}
&&\hspace{2mm}\mbox{\rm (a)}\hspace{2mm}
\gbeg{3}{6}
\got{1}{B}\got{1}{B}\got{1}{A}\gnl
\gmu\gcl{1}\gnl
\gcn{1}{1}{2}{3}\gvac{1}\gcl{1}\gnl
\gvac{1}\gbrc\gnl
\gvac{1}\gcl{1}\gcu{1}\gnl
\gvac{1}\gob{1}{A}
\gend
=
\gbeg{3}{5}
\got{1}{B}\got{1}{B}\got{1}{A}\gnl
\gcl{1}\gbrc\gnl
\gbrc\gcu{1}\gnl
\gcl{1}\gcu{1}\gnl
\gob{1}{A}
\gend
\hspace{2mm},\hspace{2mm}\mbox{\rm (b)}\hspace{2mm}
\gbeg{3}{6}
\got{1}{B}\got{1}{B}\got{1}{A}\gnl
\gmu\gcl{1}\gnl
\gcn{1}{1}{2}{3}\gvac{1}\gcl{1}\gnl
\gvac{1}\gbrc\gnl
\gvac{1}\gcu{1}\gcl{1}\gnl
\gvac{2}\gob{1}{B}
\gend
=
\gbeg{3}{5}
\got{1}{B}\got{1}{B}\got{1}{A}\gnl
\gcl{1}\gbrc\gnl
\gbrc\gcl{1}\gnl
\gcu{1}\gmu\gnl
\gvac{1}\gob{2}{B}
\gend
\hspace{2mm},\nonumber\\
&&\eqlabel{crossprodalg2}\\
&&\hspace{2mm}\mbox{\rm (c)}\hspace{2mm}
\gbeg{3}{6}
\got{1}{B}\got{1}{A}\got{1}{A}\gnl
\gcl{1}\gmu\gnl
\gcl{1}\gcn{1}{1}{2}{1}\gnl
\gbrc\gnl
\gcu{1}\gcl{1}\gnl
\gvac{1}\gob{1}{B}
\gend
=
\gbeg{3}{5}
\got{1}{B}\got{1}{A}\got{1}{A}\gnl
\gbrc\gcl{1}\gnl
\gcu{1}\gbrc\gnl
\gvac{1}\gcu{1}\gcl{1}\gnl
\gvac{2}\gob{1}{B}
\gend
\hspace{2mm},\hspace{2mm}\mbox{\rm (d)}\hspace{2mm}
\gbeg{3}{6}
\got{1}{B}\got{1}{A}\got{1}{A}\gnl
\gcl{1}\gmu\gnl
\gcl{1}\gcn{1}{1}{2}{1}\gnl
\gbrc\gnl
\gcl{1}\gcu{1}\gnl
\gob{1}{A}
\gend
=
\gbeg{3}{5}
\got{1}{B}\got{1}{A}\got{1}{A}\gnl
\gbrc\gcl{1}\gnl
\gcl{1}\gbrc\gnl
\gmu\gcu{1}\gnl
\gob{2}{A}
\gend
\hspace{2mm}\nonumber
\end{eqnarray}
are satisfied. More precisely, we have the following result.

\begin{lemma}\lelabel{implic2}
Under the same hypotheses as in \leref{implic1}, we have\\
(i) (\ref{eq:neccconds}.g), (\ref{eq:BespDrabComp}.b,e) and (\ref{eq:crossprodalg2}.a-c) imply
(\ref{eq:crossprodalg}.a);\\
(i) (\ref{eq:neccconds}.g), (\ref{eq:BespDrabComp}.a,f) and (\ref{eq:crossprodalg2}.a,c,d) imply
(\ref{eq:crossprodalg}.b).
\end{lemma}

\begin{proof}
We prove (ii), the proof of (i) is similar. The proof of (ii) works as follows.
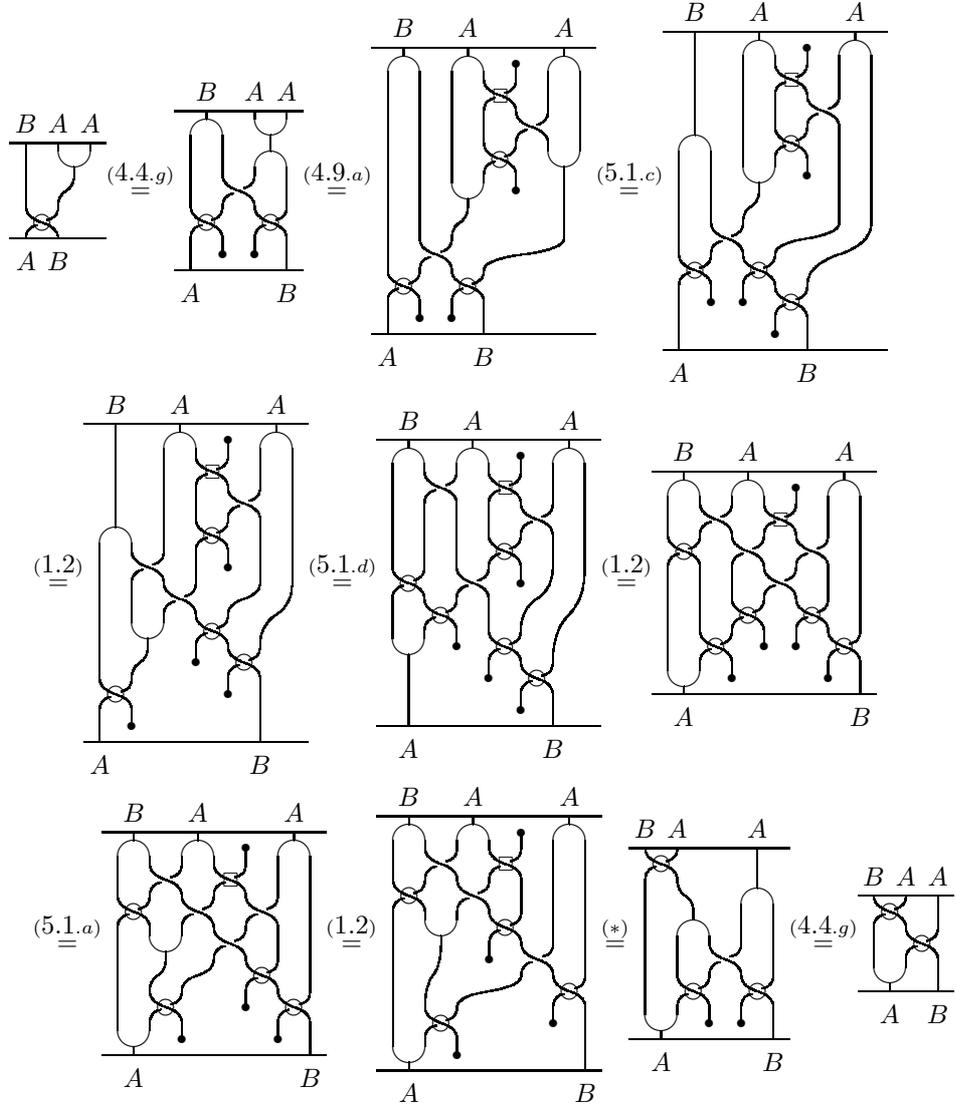
\begin{eqnarray*}
&&\hspace*{-3mm}
\gbeg{3}{5}
\got{1}{B}\got{1}{A}\got{1}{A}\gnl
\gcl{1}\gmu\gnl
\gcl{1}\gcn{1}{1}{2}{1}\gnl
\gbrc\gnl
\gob{1}{A}\gob{1}{B}
\gend
\equal{(\ref{eq:neccconds}.g)}
\gbeg{4}{7}
\got{2}{B}\got{1}{A}\got{1}{A}\gnl
\gcmu\gmu\gnl
\gcl{1}\gcl{1}\gcmu\gnl
\gcl{1}\gbr\gcl{1}\gnl
\gbrc\gbrc\gnl
\gcl{1}\gcu{1}\gcu{1}\gcl{1}\gnl
\gob{1}{A}\gvac{2}\gob{1}{B}
\gend
\equal{(\ref{eq:BespDrabComp}.a)}
\gbeg{7}{11}
\got{2}{B}\got{2}{A}\gvac{1}\got{2}{A}\gnl
\gcmu\gcmu\gu{1}\gcmu\gnl
\gcl{1}\gcl{1}\gcl{1}\gbrbox\gvac{2}\gcl{1}\gcl{1}\gnl
\gcl{1}\gcl{1}\gcl{1}\gcl{1}\gbr\gcl{1}\gnl
\gcl{1}\gcl{1}\gcl{1}\gbrc\gmu\gnl
\gcl{1}\gcl{1}\gmu\gcu{1}\gcn{1}{2}{2}{2}\gnl
\gcl{1}\gcl{1}\gcn{1}{1}{2}{1}\gnl
\gcl{1}\gbr\gcn{1}{1}{6}{1}\gnl
\gbrc\gbrc\gnl
\gcl{1}\gcu{1}\gcu{1}\gcl{1}\gnl
\gob{1}{A}\gvac{2}\gob{1}{B}
\gend
\equal{(\ref{eq:crossprodalg2}.c)}
\gbeg{7}{12}
\got{2}{B}\got{2}{A}\gvac{1}\got{2}{A}\gnl
\gcn{1}{3}{2}{2}\gvac{1}\gcmu\gu{1}\gcmu\gnl
\gvac{2}\gcl{1}\gbrbox\gvac{2}\gcl{1}\gcl{1}\gnl
\gvac{2}\gcl{1}\gcl{1}\gbr\gcl{1}\gnl
\gcmu\gcl{1}\gbrc\gcl{1}\gcl{1}\gnl
\gcl{1}\gcl{1}\gmu\gcu{1}\gcl{1}\gcl{1}\gnl
\gcl{1}\gcl{1}\gcn{1}{1}{2}{1}\gvac{2}\gcl{1}\gcl{1}\gnl
\gcl{1}\gbr\gvac{2}\gcn{1}{1}{1}{-3}\gcn{1}{2}{1}{-3}\gnl
\gbrc\gbrc\gnl
\gcl{1}\gcu{1}\gcu{1}\gbrc\gnl
\gcl{1}\gvac{2}\gcu{1}\gcl{1}\gnl
\gob{1}{A}\gvac{3}\gob{1}{B}
\gend\\
&&
\equal{\equref{nat1cup}}
\gbeg{7}{12}
\got{2}{B}\got{2}{A}\gvac{1}\got{2}{A}\gnl
\gcn{1}{3}{2}{2}\gvac{1}\gcmu\gu{1}\gcmu\gnl
\gvac{2}\gcl{1}\gbrbox\gvac{2}\gcl{1}\gcl{1}\gnl
\gvac{2}\gcl{1}\gcl{1}\gbr\gcl{1}\gnl
\gcmu\gcl{1}\gbrc\gcl{1}\gcl{1}\gnl
\gcl{1}\gbr\gcl{1}\gcu{1}\gcl{1}\gcl{1}\gnl
\gcl{1}\gcl{1}\gbr\gvac{1}\gcn{1}{1}{1}{-1}\gcn{1}{2}{1}{-1}\gnl
\gcl{1}\gmu\gbrc\gnl
\gcl{1}\gcn{1}{1}{2}{1}\gvac{1}\gcu{1}\gbrc\gnl
\gbrc\gvac{2}\gcu{1}\gcl{2}\gnl
\gcl{1}\gcu{1}\gnl
\gob{1}{A}\gvac{4}\gob{1}{B}
\gend
\equal{(\ref{eq:crossprodalg2}.d)}
\gbeg{7}{11}
\got{2}{B}\got{2}{A}\gvac{1}\got{2}{A}\gnl
\gcmu\gcmu\gu{1}\gcmu\gnl
\gcl{1}\gbr\gbrbox\gvac{2}\gcl{1}\gcl{1}\gnl
\gcl{1}\gcl{1}\gcl{1}\gcl{1}\gbr\gcl{1}\gnl
\gcl{1}\gcl{1}\gcl{1}\gbrc\gcl{1}\gcl{1}\gnl
\gbrc\gbr\gcu{1}\gcn{1}{2}{1}{-1}\gcn{1}{3}{1}{-1}\gnl
\gcl{1}\gbrc\gcl{1}\gnl
\gmu\gcu{1}\gbrc\gnl
\gcn{1}{2}{2}{2}\gvac{2}\gcu{1}\gbrc\gnl
\gvac{4}\gcu{1}\gcl{1}\gnl
\gob{2}{A}\gvac{3}\gob{1}{B}
\gend
\equal{\equref{nat1cup}}
\gbeg{7}{9}
\got{2}{B}\got{2}{A}\gvac{1}\got{2}{A}\gnl
\gcmu\gcmu\gu{1}\gcmu\gnl
\gcl{1}\gbr\gbrbox\gvac{2}\gcl{1}\gcl{1}\gnl
\gbrc\gbr\gbr\gcl{1}\gnl
\gcl{1}\gcl{1}\gcl{1}\gbr\gcl{1}\gcl{1}\gnl
\gcl{1}\gcl{1}\gbrc\gbrc\gcl{1}\gnl
\gcl{1}\gbrc\gcu{1}\gcu{1}\gbrc\gnl
\gmu\gcu{1}\gvac{2}\gcu{1}\gcl{1}\gnl
\gob{2}{A}\gvac{4}\gob{1}{B}
\gend\\
&&
\equal{(\ref{eq:crossprodalg2}.a)}
\gbeg{7}{9}
\got{2}{B}\got{2}{A}\gvac{1}\got{2}{A}\gnl
\gcmu\gcmu\gu{1}\gcmu\gnl
\gcl{1}\gbr\gbrbox\gvac{2}\gcl{1}\gcl{1}\gnl
\gbrc\gbr\gbr\gcl{1}\gnl
\gcl{1}\gmu\gbr\gcl{1}\gcl{1}\gnl
\gcl{1}\gcn{1}{1}{2}{1}\gvac{1}\gcn{1}{1}{1}{-1}\gbrc\gcl{1}\gnl
\gcl{1}\gbrc\gvac{1}\gcu{1}\gbrc\gnl
\gmu\gcu{1}\gvac{2}\gcu{1}\gcl{1}\gnl
\gob{2}{A}\gvac{4}\gob{1}{B}
\gend
\equal{\equref{nat1cup}}
\gbeg{7}{10}
\got{2}{B}\got{2}{A}\gvac{1}\got{2}{A}\gnl
\gcmu\gcmu\gu{1}\gcmu\gnl
\gcl{1}\gbr\gbrbox\gvac{2}\gcl{1}\gcl{1}\gnl
\gbrc\gbr\gcl{1}\gcl{1}\gcl{1}\gnl
\gcl{1}\gmu\gbrc\gcl{1}\gcl{1}\gnl
\gcl{2}\gcn{1}{2}{2}{1}\gvac{1}\gcu{1}\gbr\gcl{1}\gnl
\gvac{4}\gcn{1}{1}{1}{-3}\gbrc\gnl
\gcl{1}\gbrc\gvac{2}\gcu{1}\gcl{2}\gnl
\gmu\gcu{1}\gnl
\gob{2}{A}\gvac{4}\gob{1}{B}
\gend
\equal{(*)}
\gbeg{5}{8}
\got{1}{B}\got{1}{A}\gvac{1}\got{2}{A}\gnl
\gbrc\gvac{1}\gcn{1}{1}{2}{2}\gnl
\gcl{1}\gcn{1}{1}{1}{2}\gvac{1}\gcmu\gnl
\gcl{1}\gcmu\gcl{1}\gcl{1}\gnl
\gcl{1}\gcl{1}\gbr\gcl{1}\gnl
\gcl{1}\gbrc\gbrc\gnl
\gmu\gcu{1}\gcu{1}\gcl{1}\gnl
\gob{2}{A}\gvac{2}\gob{1}{B}
\gend
\equal{(\ref{eq:neccconds}.g)}
\gbeg{3}{5}
\got{1}{B}\got{1}{A}\got{1}{A}\gnl
\gbrc\gcl{1}\gnl
\gcl{1}\gbrc\gnl
\gmu\gcl{1}\gnl
\gob{2}{A}\gob{1}{B}
\gend
\hspace{1mm}.
\end{eqnarray*}
(*): we used \leref{implic1} (ii). 
\end{proof}

As the reader might expect, we have a dual version of \leref{implic2}. To this end, 
we need 
the dual versions of the equations in \equref{crossprodalg2}, namely,
\begin{eqnarray}
&&\hspace{2mm}\mbox{\rm (a)}\hspace{2mm}
\gbeg{3}{6}
\gvac{2}\got{1}{B}\gnl
\gvac{1}\gu{1}\gcl{1}\gnl
\gvac{1}\gbrbox\gnl
\gvac{1}\gcn{1}{1}{1}{0}\gcl{1}\gnl
\gcmu\gcl{1}\gnl
\gob{1}{B}\gob{1}{B}\gob{1}{A}
\gend
=
\gbeg{3}{5}
\gvac{1}\got{2}{B}\gnl
\gu{1}\gcmu\gnl
\gbrbox\gvac{2}\gcl{1}\gnl
\gcl{1}\gbrbox\gnl
\gob{1}{B}\gob{1}{B}\gob{1}{A}
\gend
\hspace{2mm},\hspace{2mm}\mbox{\rm (b)}\hspace{2mm}
\gbeg{3}{6}
\gvac{1}\got{1}{A}\gnl
\gvac{1}\gcl{1}\gu{1}\gnl
\gvac{1}\gbrbox\gnl
\gcn{1}{1}{3}{2}\gvac{1}\gcl{1}\gnl
\gcmu\gcl{1}\gnl
\gob{1}{B}\gob{1}{B}\gob{1}{A}
\gend
=
\gbeg{3}{5}
\got{1}{A}\gnl
\gcl{1}\gu{1}\gnl
\gbrbox\gvac{2}\gu{1}\gnl
\gcl{1}\gbrbox\gnl
\gob{1}{B}\gob{1}{B}\gob{1}{A}
\gend\hspace{2mm},\nonumber\\
&&\eqlabel{crossprodcoalg2}\\
&&\hspace{2mm}\mbox{\rm (c)}\hspace{2mm}
\gbeg{3}{6}
\gvac{1}\got{1}{B}\gnl
\gu{1}\gcl{1}\gnl
\gbrbox\gnl
\gcl{1}\gcn{1}{1}{1}{2}\gnl
\gcl{1}\gcmu\gnl
\gob{1}{B}\gob{1}{A}\gob{1}{A}
\gend
=
\gbeg{3}{5}
\gvac{2}\got{1}{B}\gnl
\gvac{1}\gu{1}\gcl{1}\gnl
\gu{1}\gbrbox\gnl
\gbr\gvac{-1}\gnot{\hspace*{-4mm}\Box}\gvac{1}\gcl{1}\gnl
\gob{1}{B}\gob{1}{A}\gob{1}{A}
\gend
\hspace{2mm},\hspace{2mm}\mbox{\rm (d)}\hspace{2mm}
\gbeg{3}{6}
\got{1}{A}\gnl
\gcl{1}\gu{1}\gnl
\gbrbox\gnl
\gcl{1}\gcn{1}{1}{1}{2}\gnl
\gcl{1}\gcmu\gnl
\gob{1}{B}\gob{1}{A}\gob{1}{A}
\gend
=
\gbeg{3}{5}
\got{2}{A}\gnl
\gcmu\gu{1}\gnl
\gcl{1}\gbrbox\gnl
\gbr\gvac{-1}\gnot{\hspace*{-4mm}\Box}\gvac{1}\gcl{1}\gnl
\gob{1}{B}\gob{1}{A}\gob{1}{A}
\gend
\hspace{2mm}.\nonumber
\end{eqnarray}
The proof of \leref{implic2p} is omitted, as it can be obtained from the proof
of \leref{implic2} using duality arguments.

\begin{lemma}\lelabel{implic2p}
Under the same hypotheses as in \leref{implic1}, we have\\
(i) (\ref{eq:neccconds} .c), (\ref{eq:BespDrabComp}.b,c), and 
(\ref{eq:crossprodcoalg2}.a-c) imply (\ref{eq:crossprodcoalg}.a);\\
(ii) (\ref{eq:neccconds} .c), (\ref{eq:BespDrabComp}.a,d), and 
(\ref{eq:crossprodcoalg2}.b-d) imply (\ref{eq:crossprodcoalg}.b).
\end{lemma}

\begin{theorem}\thlabel{crossprobialasactandcoact}
Let $A$ and $B$ be algebras and coalgebras.
There exist $\psi: B\ot A\ra A\ot B$ and $\phi: A\ot B\ra B\ot A$
such that $(A,B,\psi,\phi)$ is a bialgebra admissible tuple, that is,
$A\#_\psi^\phi B$ is a cross product bialgebra
if and only if the following assertions
hold.\\
(i) $\un{\va}_X\circ \un{\eta}_X=\Id_{\un{1}}$, $\un{\va}_X\circ \un{m}_X=\un{\va}_X\ot \un{\va}_X$ and $\un{\Delta}_X\circ \un{\eta}_X=
\un{\eta}_X\ot \un{\eta}_X$, for all $X\in \{A, B\}$;\\
(ii) $A\in {}_B\Cc$ via 
$
\gbeg{2}{3}
\got{1}{B}\got{1}{A}\gnl
\glm\gnl
\gvac{1}\gob{1}{A}
\gend
$ 
satisfying 
$
\gbeg{2}{4}
\got{1}{B}\gnl
\gcl{1}\gu{1}\gnl
\glm\gnl
\gvac{1}\gob{1}{A}
\gend
=
\gbeg{2}{3}
\got{1}{B}\gnl
\gcu{1}\gu{1}\gnl
\gvac{1}\gob{1}{A}
\gend
$ 
and 
$
\gbeg{2}{4}
\got{1}{B}\got{1}{A}\gnl
\glm\gnl
\gvac{1}\gcu{1}\gnl
\gob{2}{\un{1}}
\gend
=
\gbeg{2}{3}
\got{1}{B}\got{1}{A}\gnl
\gcu{1}\gcu{1}\gnl
\gob{2}{\un{1}}
\gend
$;\\
(iii) $A\in {}^B\Cc$ via 
$
\gbeg{2}{3}
\gvac{1}\got{1}{A}\gnl
\glcm\gnl
\gob{1}{B}\gob{1}{A}
\gend 
$ 
obeying 
$
\gbeg{2}{4}
\gvac{1}\got{1}{\un{1}}\gnl
\gvac{1}\gu{1}\gnl
\glcm\gnl
\gob{1}{B}\gob{1}{A}
\gend
=
\gbeg{2}{3}
\got{2}{\un{1}}\gnl
\gu{1}\gu{1}\gnl
\gob{1}{B}\gob{1}{A}
\gend
$ 
and 
$
\gbeg{2}{4}
\gvac{1}\got{1}{A}\gnl
\glcm\gnl
\gcl{1}\gcu{1}\gnl
\gob{1}{B}
\gend
=
\gbeg{2}{3}
\gvac{1}\got{1}{A}\gnl
\gu{1}\gcu{1}\gnl
\gob{1}{B}
\gend
$;\\
(iv) $B\in \Cc_A$ via 
$
\gbeg{2}{3}
\got{1}{B}\got{1}{A}\gnl
\grm\gnl
\gob{1}{B}
\gend
$ 
such that 
$
\gbeg{2}{4}
\gvac{1}\got{1}{A}\gnl
\gu{1}\gcl{1}\gnl
\grm\gnl
\gob{1}{B}
\gend
=
\gbeg{2}{3}
\gvac{1}\got{1}{A}\gnl
\gu{1}\gcu{1}\gnl
\gob{1}{B}
\gend
$ 
and 
$\gbeg{2}{4}
\got{1}{B}\got{1}{A}\gnl
\grm\gnl
\gcu{1}\gnl
\gob{2}{\un{1}}
\gend 
=
\gbeg{2}{3}
\got{1}{B}\got{1}{A}\gnl
\gcu{1}\gcu{1}\gnl
\gob{2}{\un{1}}
\gend$;\\
(v) $B\in \Cc^A$ via 
$
\gbeg{2}{3}
\got{1}{B}\gnl
\grcm\gnl
\gob{1}{B}\gob{1}{A}
\gend
$ 
such that 
$
\gbeg{2}{4}
\got{2}{\un{1}}\gnl
\gu{1}\gnl
\grcm\gnl
\gob{1}{B}\gob{1}{A}
\gend
=
\gbeg{2}{3}
\got{2}{\un{1}}\gnl
\gcu{1}\gcu{1}\gnl
\gob{1}{B}\gob{1}{A}
\gend
$ 
and 
$
\gbeg{2}{4}
\got{1}{B}\gnl
\grcm\gnl
\gcu{1}\gcl{1}\gnl
\gvac{1}\gob{1}{A}
\gend 
=
\gbeg{2}{3}
\got{1}{B}\gnl
\gcu{1}\gu{1}\gnl
\gvac{1}\gob{1}{A}
\gend
$;\\
(vi) these actions and coactions are compatible in the sense that
\begin{eqnarray*}
&&\hspace*{-2mm}
\gbeg{3}{5}
\got{1}{B}\got{1}{A}\got{1}{A}\gnl
\gcl{1}\gmu\gnl
\gcl{1}\gcn{1}{1}{2}{1}\gnl
\glm\gnl
\gvac{1}\gob{1}{A}
\gend
=
\gbeg{5}{8}
\got{2}{B}\got{2}{A}\got{1}{A}\gnl
\gcmu\gcmu\gcl{1}\gnl
\gcl{1}\gbr\gcl{1}\gcl{1}\gnl
\glm\grm\gcn{1}{1}{1}{-1}\gnl
\gvac{1}\gcl{1}\glm\gnl
\gvac{1}\gcn{1}{1}{1}{3}\gvac{1}\gcl{1}\gnl
\gvac{2}\gmu\gnl
\gvac{2}\gob{2}{A}
\gend
\hspace{1mm},\hspace{1mm}
\gbeg{3}{5}
\gvac{1}\got{1}{A}\gnl
\glcm\gnl
\gcl{1}\gcn{1}{1}{1}{2}\gnl
\gcl{1}\gcmu\gnl
\gob{1}{B}\gob{1}{A}\gob{1}{A}
\gend
=
\gbeg{5}{8}
\gvac{1}\got{2}{A}\gnl
\gvac{1}\gcmu\gnl
\glcm\gcn{1}{1}{1}{3}\gnl
\gcl{1}\gcl{1}\glcm\gnl
\gcl{1}\gcl{1}\grcm\gcn{1}{1}{-1}{1}\gnl
\gcl{1}\gbr\gcl{1}\gcl{1}\gnl
\gmu\gmu\gcl{1}\gnl
\gob{2}{B}\gob{2}{A}\gob{1}{A}
\gend
\hspace{1mm},\hspace{1mm}
\gbeg{2}{4}
\got{1}{A}\got{1}{A}\gnl
\gmu\gnl
\gcmu\gnl
\gob{1}{A}\gob{1}{A}
\gend
=
\gbeg{5}{9}
\got{2}{A}\gvac{1}\got{2}{A}\gnl
\gcmu\gvac{1}\gcmu\gnl
\gcl{1}\gcn{1}{1}{1}{3}\gvac{1}\gcl{1}\gcl{1}\gnl
\gcl{1}\glcm\gcl{1}\gcl{1}\gnl
\gcl{1}\gcl{1}\gbr\gcl{1}\gnl
\gcl{1}\glm\gmu\gnl
\gcl{1}\gvac{1}\gcn{1}{1}{1}{-1}\gcn{1}{2}{2}{2}\gnl
\gmu\gnl
\gob{2}{A}\gvac{1}\gob{2}{A}
\gend 
\hspace{1mm},\\
&&
\hspace*{-2mm}
\gbeg{3}{5}
\got{1}{B}\got{1}{B}\got{1}{A}\gnl
\gmu\gcl{1}\gnl
\gcn{1}{1}{2}{3}\gvac{1}\gcl{1}\gnl
\gvac{1}\grm\gnl
\gvac{1}\gob{1}{B}
\gend
=
\gbeg{5}{7}
\got{1}{B}\got{2}{B}\got{2}{A}\gnl
\gcl{1}\gcmu\gcmu\gnl
\gcl{1}\gcl{1}\gbr\gcl{1}\gnl
\gcn{1}{1}{1}{3}\glm\grm\gnl
\gvac{1}\grm\gcn{1}{1}{1}{-1}\gnl
\gvac{1}\gmu\gnl
\gvac{1}\gob{2}{B}
\gend
\hspace{1mm},\hspace{1mm}
\gbeg{3}{5}
\gvac{1}\got{1}{B}\gnl
\gvac{1}\grcm\gnl
\gvac{1}\gcn{1}{1}{1}{0}\gcl{1}\gnl
\gcmu\gcl{1}\gnl
\gob{1}{B}\gob{1}{B}\gob{1}{A}
\gend
=
\gbeg{5}{8}
\gvac{2}\got{2}{B}\gnl
\gvac{2}\gcmu\gnl
\gvac{2}\gcn{1}{1}{1}{-1}\gcl{1}\gnl
\gvac{1}\grcm\grcm\gnl
\gvac{1}\gcn{1}{1}{1}{-1}\gvac{-1}\glcm\gcl{1}\gcl{1}\gnl
\gcl{1}\gcl{1}\gbr\gcl{1}\gnl
\gcl{1}\gmu\gmu\gnl
\gob{1}{B}\gob{2}{A}\gob{2}{A}
\gend
\hspace{1mm},\hspace{1mm}
\gbeg{2}{4}
\got{1}{B}\got{1}{B}\gnl
\gmu\gnl
\gcmu\gnl
\gob{1}{B}\gob{1}{B}
\gend
=
\gbeg{5}{8}
\got{2}{B}\got{2}{B}
\gnl
\gcmu\gcmu\gnl
\gcl{1}\gcl{1}\grcm\gcn{1}{1}{-1}{1}\gnl
\gcl{1}\gbr\gcl{1}\gcl{1}\gnl
\gmu\grm\gcl{1}\gnl
\gcn{1}{2}{2}{2}\gvac{1}\gcl{1}\gcn{1}{1}{3}{1}\gnl
\gvac{2}\gmu\gnl
\gob{2}{B}\gob{2}{B}
\gend
\hspace{1mm};
\end{eqnarray*}
(vii) one of the four following sets of three equations is satisfied:
$(vii.1)=\{\equref{modulecomodule1}, \equref{additional1}\}$, 
$(vii.2)=\{\equref{modulecomodule2},\equref{additional2}\}$, 
$(vii.3)=\{\equref{modulecomodule3}, \equuref{additional1}{b},\equuref{additional2}{a}\}$, 
$(vii.4)=\{\equref{modulecomodule4},\equuref{additional1}{a},\equuref{additional2}{b}\}$.
These four sets are equivalent if (i-vi) are satisfied.
\begin{eqnarray}
\eqlabel{modulecomodule1}&&\hspace*{6mm}
\gbeg{6}{10}
\gvac{1}\got{2}{B}\got{2}{A}\gnl
\gvac{1}\gcmu\gcmu\gnl
\gvac{1}\gcl{1}\gbr\gcl{1}\gnl
\gvac{1}\glm\grm\gnl
\gvac{1}\gcn{1}{1}{3}{2}\gvac{1}\gcn{1}{1}{1}{2}\gnl
\gvac{1}\gcmu\gcmu\gnl
\gcn{1}{1}{3}{1}\glcm\grcm\gcn{1}{1}{-1}{1}\gnl
\gcl{1}\gcl{1}\gbr\gcl{1}\gcl{1}\gnl
\gcl{1}\gmu\gmu\gcl{1}\gnl
\gob{1}{A}\gob{2}{B}\gob{2}{A}\gob{1}{B}
\gend
=
\gbeg{10}{12}
\gvac{1}\got{2}{B}\gvac{4}\got{2}{A}\gnl
\gvac{1}\gcmu\gvac{4}\gcmu\gnl
\gvac{1}\gcl{1}\gcn{1}{1}{1}{5}\gvac{3}\gcn{1}{1}{3}{-1}\gvac{1}\gcl{1}\gnl
\gvac{1}\grcm\gvac{1}\gbr\gvac{1}\glcm\gnl
\gvac{1}\gcn{1}{2}{1}{0}\gcn{1}{1}{1}{3}\gvac{1}\gcl{1}\gcl{1}\gcn{1}{1}{3}{1}\gvac{1}\gcn{1}{2}{1}{2}\gnl
\gvac{3}\gbr\gbr\gnl
\gcmu\gvac{1}\gcn{1}{1}{1}{0}\gbr\gcn{1}{1}{1}{2}\gvac{1}\gcmu\gnl
\gcl{1}\gcl{1}\gcmu\gcl{1}\gcl{1}\gcmu\gcl{1}\gcl{1}\gnl
\gcl{1}\gbr\gcl{1}\gcl{1}\gcl{1}\gcl{1}\gbr\gcl{1}\gnl
\glm\grm\gcn{1}{1}{1}{-1}\gcn{1}{1}{1}{3}\glm\grm\gnl
\gvac{1}\gcl{1}\gmu\gvac{2}\gmu\gcl{1}\gnl
\gvac{1}\gob{1}{A}\gob{2}{B}\gvac{2}\gob{2}{A}\gob{1}{B}
\gend
\hspace{1mm};\\
\eqlabel{modulecomodule2}&&\hspace*{6mm}
\gbeg{6}{10}
\got{1}{A}\got{2}{B}\got{2}{A}\got{1}{B}\gnl
\gcl{1}\gcmu\gcmu\gcl{1}\gnl
\gcl{1}\gcl{1}\gbr\gcl{1}\gcl{1}\gnl
\gcn{1}{1}{1}{3}\glm\grm\gcn{1}{1}{1}{-1}\gnl
\gvac{1}\gmu\gmu\gnl
\gvac{1}\gcn{1}{1}{2}{3}\gvac{1}\gcn{1}{1}{2}{1}\gnl
\gvac{1}\glcm\grcm\gnl
\gvac{1}\gcl{1}\gbr\gcl{1}\gnl
\gvac{1}\gmu\gmu\gnl
\gvac{1}\gob{2}{B}\gob{2}{A}
\gend
=
\gbeg{10}{12}
\gvac{1}\got{1}{A}\gvac{1}\got{2}{B}\got{2}{A}\gvac{1}\got{1}{B}\gnl
\glcm\gvac{1}\gcmu\gcmu\gvac{1}\grcm\gnl
\gcl{1}\gcl{1}\gvac{1}\gcn{1}{1}{1}{-1}\gcl{1}\gcl{1}\gcn{1}{1}{1}{3}\gvac{1}\gcl{1}\gcl{1}\gnl
\gcl{1}\gcl{1}\grcm\gcl{1}\gcl{1}\glcm\gcl{1}\gcl{1}\gnl
\gcl{1}\gbr\gcl{1}\gcl{1}\gcl{1}\gcl{1}\gbr\gcl{1}\gnl
\gmu\gmu\gcl{1}\gcl{1}\gmu\gmu\gnl
\gcn{1}{1}{2}{5}\gvac{1}\gcn{1}{1}{2}{3}\gvac{1}\gbr\gcn{1}{1}{2}{1}\gcn{1}{1}{4}{1}\gnl
\gvac{2}\gcl{1}\gbr\gbr\gcl{1}\gnl
\gvac{2}\grm\gbr\glm\gnl
\gvac{2}\gcn{1}{1}{1}{3}\gvac{1}\gcl{1}\gcl{1}\gvac{1}\gcn{1}{1}{1}{-1}\gnl
\gvac{3}\gmu\gmu\gnl
\gvac{3}\gob{2}{B}\gob{2}{A}
\gend
\hspace{1mm};\\
\eqlabel{modulecomodule3}&&\hspace*{6mm}
\gbeg{5}{11}
\got{1}{A}\got{2}{B}\got{2}{A}\gnl
\gcl{1}\gcmu\gcmu\gnl
\gcl{1}\gcl{1}\gbr\gcl{1}\gnl
\gcn{1}{1}{1}{3}\glm\grm\gnl
\gvac{1}\gmu\gcn{1}{1}{1}{0}\gnl
\gvac{1}\gcn{1}{1}{2}{1}\gcmu\gnl
\glcm\gcl{1}\gcn{1}{1}{1}{3}\gnl
\gcl{1}\gcl{1}\grcm\gcl{1}\gnl
\gcl{1}\gbr\gcl{1}\gcl{1}\gnl
\gmu\gmu\gcl{1}\gnl
\gob{2}{B}\gob{2}{A}\gob{1}{B}
\gend
=
\gbeg{10}{12}
\gvac{1}\got{1}{A}\got{2}{B}\gvac{1}\got{2}{A}\gnl
\gvac{1}\gcl{1}\gcmu\gvac{1}\gcmu\gnl
\glcm\gcl{1}\gcn{1}{1}{1}{3}\gvac{1}\gcl{1}\gcn{1}{1}{1}{3}\gnl
\gcl{1}\gcl{1}\grcm\gbr\glcm\gnl
\gcl{1}\gbr\gcl{1}\gcl{1}\gbr\gcl{1}\gnl
\gmu\gmu\gcl{1}\gcl{1}\gcn{1}{1}{1}{2}\gcn{1}{1}{1}{4}\gnl
\gcn{1}{2}{2}{5}\gvac{1}\gcn{1}{1}{2}{3}\gvac{1}\gcl{1}\gcl{1}\gcmu\gcmu\gnl
\gvac{3}\gbr\gcl{1}\gcl{1}\gbr\gcl{1}\gnl
\gvac{2}\grm\gbr\glm\grm\gnl
\gvac{2}\gcn{1}{1}{1}{3}\gvac{1}\gcl{1}\gcl{1}\gvac{1}\gcn{1}{1}{1}{-1}\gcl{2}\gnl
\gvac{3}\gmu\gmu\gnl
\gvac{3}\gob{2}{B}\gob{2}{A}\gvac{1}\gob{1}{B}
\gend
\hspace{1mm};\\
\eqlabel{modulecomodule4}&&\hspace*{6mm}
\gbeg{6}{11}
\gvac{1}\got{2}{B}\got{2}{A}\got{1}{B}\gnl
\gvac{1}\gcmu\gcmu\gcl{1}\gnl
\gvac{1}\gcl{1}\gbr\gcl{1}\gcl{1}\gnl
\gvac{1}\glm\grm\gcn{1}{1}{1}{-1}\gnl
\gvac{2}\gcn{1}{1}{1}{0}\gmu\gnl
\gvac{1}\gcmu\gcn{1}{1}{2}{1}\gnl
\gvac{1}\gcn{1}{1}{1}{-1}\gcl{1}\grcm\gnl
\gcl{1}\glcm\gcl{1}\gcl{1}\gnl
\gcl{1}\gcl{1}\gbr\gcl{1}\gnl
\gcl{1}\gmu\gmu\gnl
\gob{1}{A}\gob{2}{B}\gob{2}{A}
\gend
=
\gbeg{10}{12}
\gvac{2}\got{2}{B}\gvac{2}\got{2}{A}\got{1}{B}\gnl
\gvac{2}\gcmu\gvac{2}\gcmu\gcl{1}\gnl
\gvac{2}\gcl{1}\gcn{1}{1}{1}{3}\gvac{2}\gcn{1}{1}{1}{-1}\gcl{1}\gcl{1}\gnl
\gvac{2}\grcm\gbr\glcm\grcm\gnl
\gvac{2}\gcl{1}\gbr\gcl{1}\gcl{1}\gbr\gcl{1}\gnl
\gvac{2}\gcn{1}{1}{1}{-2}\gcn{1}{1}{1}{0}\gcl{3}\gcl{1}\gmu\gmu\gnl
\gcmu\gcmu\gvac{1}\gcl{1}\gcn{1}{1}{2}{1}\gvac{1}\gcn{1}{1}{2}{-1}\gnl
\gcl{1}\gbr\gcl{1}\gvac{1}\gbr\gcl{1}\gnl
\glm\grm\gbr\glm\gnl
\gvac{1}\gcl{1}\gcl{1}\gvac{1}\gcn{1}{1}{1}{-1}\gcn{1}{1}{1}{3}\gvac{1}\gcl{1}\gnl
\gvac{1}\gcl{1}\gmu\gvac{2}\gmu\gnl
\gvac{1}\gob{1}{A}\gob{2}{B}\gvac{2}\gob{2}{A}
\gend
\hspace{1mm};\\
\eqlabel{additional1}
&&(a)~~~~
\gbeg{2}{5}
\got{1}{A}\got{1}{A}\gnl
\gmu\gnl
\gcn{1}{1}{2}{3}\gnl
\glcm\gnl
\gob{1}{B}\gob{1}{A}
\gend
=
\gbeg{5}{8}
\gvac{1}\got{1}{A}\got{2}{A}\gnl
\glcm\gcmu\gnl
\gcl{1}\gcl{1}\gcl{1}\gcn{1}{1}{1}{3}\gnl
\gcl{1}\gbr\glcm\gnl
\grm\gbr\gcl{1}\gnl
\gcl{1}\gvac{1}\gcn{1}{1}{1}{-1}\gmu\gnl
\gmu\gvac{1}\gcn{1}{1}{2}{2}\gnl
\gob{2}{B}\gvac{1}\gob{2}{A}
\gend
~~;~~(b)~~~~
\gbeg{2}{5}
\got{1}{B}\got{1}{B}\gnl
\gmu\gnl
\gcn{1}{1}{2}{1}\gnl
\grcm\gnl
\gob{1}{B}\gob{1}{A}
\gend
=
\gbeg{5}{8}
\got{2}{B}\gvac{1}\got{1}{B}\gnl
\gcmu\gvac{1}\grcm\gnl
\grcm\gcn{1}{1}{-1}{1}\gcl{1}\gcl{1}\gnl
\gcl{1}\gcl{1}\gbr\gcl{1}\gnl
\gcl{1}\gbr\glm\gnl
\gmu\gcl{1}\gvac{1}\gcn{1}{1}{1}{-1}\gnl
\gcn{1}{1}{2}{2}\gvac{1}\gmu\gnl
\gob{2}{B}\gob{2}{A}
\gend
\hspace{1mm};\\
\eqlabel{additional2}
&&(a)~~~~
\gbeg{2}{5}
\got{1}{B}\got{1}{A}\gnl
\glm\gnl
\gvac{1}\gcn{1}{1}{1}{0}\gnl
\gcmu\gnl
\gob{1}{A}\gob{1}{A}
\gend
=
\gbeg{5}{8}
\got{2}{B}\gvac{1}\got{2}{A}\gnl
\gcmu\gvac{1}\gcmu\gnl
\grcm\gcn{1}{1}{-1}{1}\gcl{1}\gcl{1}\gnl
\gcl{1}\gcl{1}\gbr\gcl{1}\gnl
\gcl{1}\gbr\glm\gnl
\glm\gcl{1}\gvac{1}\gcn{1}{1}{1}{-1}\gnl
\gvac{1}\gcl{1}\gmu\gnl
\gvac{1}\gob{1}{A}\gob{2}{A}
\gend
~~;~~(b)~~~~
\gbeg{2}{5}
\got{1}{B}\got{1}{A}\gnl
\grm\gnl
\gcn{1}{1}{1}{2}\gnl
\gcmu\gnl
\gob{1}{B}\gob{1}{B}
\gend
=
\gbeg{5}{8}
\got{2}{B}\got{2}{A}\gnl
\gcmu\gcmu\gnl
\gcl{1}\gbr\gcn{1}{1}{1}{3}\gnl
\grm\gcl{1}\glcm\gnl
\gcl{1}\gvac{1}\gbr\gcl{1}\gnl
\gcl{1}\gvac{1}\gcn{1}{1}{1}{-1}\grm\gnl
\gmu\gvac{1}\gcl{1}\gnl
\gob{2}{B}\gvac{1}\gob{1}{B}
\gend
\hspace{1mm}.
\end{eqnarray}
\end{theorem}

\begin{proof}
Suppose that there exist $\psi$ and $\phi$ such that 
$A\#_\psi^\phi B$ is a cross product bialgebra. 
Then $A$ is a left $B$-module and $B$-comodule, and $B$ is a right $A$-module and $A$-comodule,
see (\ref{eq:action1},\ref{eq:coaction1}). 
(ii)-(v) follow from 
\equref{comultunitcomp}, \equref{multcounitcomp}, (\ref{eq:crossprodalg}c,d), (\ref{eq:crossprodcoalg}c,d),
and the fact 
that $\un{\va}_X\circ \un{\eta}_X=\Id_{\un{1}}$, for all $X=A,B$.\\
Now observe that it follows from (\ref{eq:neccconds}.g,c) that $\psi$ and $\phi$ can be recovered from
the actions and coactions
\begin{equation}\eqlabel{psiphiasactcoact}
\psi=
\gbeg{4}{5}
\got{2}{B}\got{2}{A}\gnl
\gcmu\gcmu\gnl
\gcl{1}\gbr\gcl{1}\gnl
\glm\grm\gnl
\gvac{1}\gob{1}{A}\gob{1}{B}
\gend
~~\mbox{and}~~
\phi
=
\gbeg{4}{5}
\gvac{1}\got{1}{A}\got{1}{B}\gnl
\glcm\grcm\gnl
\gcl{1}\gbr\gcl{1}\gnl
\gmu\gmu\gnl
\gob{2}{B}\gob{2}{A}
\gend
\hspace{1mm}.
\end{equation}
Then the six formulas in (vi) are reformulations of
\equuref{crossprodalg2}{d}, \equuref{crossprodcoalg2}{d}, \equuref{BespDrabComp}{a}, 
\equuref{crossprodalg2}{b}, \equuref{crossprodcoalg2}{a} and \equuref{BespDrabComp}{b}.
In a similar fashion, we have that
\begin{itemize}
\item (vii.1) is a reformulation of \equuref{BespDrabComp}{c,d} and \equuref{neccconds}{d};
\item (vii.2) is the reformulation of \equuref{BespDrabComp}{e,f} and \equuref{neccconds}{h}
in terms of actions and coactions
\item (vii.3) follows from \equuref{BespDrabComp}{d}, \equuref{BespDrabComp}{e} and 
\equuref{twoanothYDconds}{a};
\item (vii.4) is a reformulation of \equuref{BespDrabComp}{c}, \equuref{BespDrabComp}{f} and 
\equuref{twoanothYDconds}{b}.
\end{itemize} 
It follows from \thref{firstsetequivcond} that these sets of conditions are equivalent.\\
Conversely, assume that $A$ is a left $B$-module and $B$-comodule, and that $B$ is
a right $A$-module and $A$-comodule, satisfying all the conditions of the Theorem. Then we
define $\psi$ and $\phi$ using \equref{psiphiasactcoact}. The actions and coactions are then
given by (\ref{eq:action1},\ref{eq:coaction1}) because of the unit-counit conditions in (ii-v).
A simple verification tells us that $\psi$ and $\phi$ satisfy \equuref{neccconds}{g,c}.
As in the proof of the direct implication, we show that
 (i-vi) imply that $\un{\va}_X\circ \un{\eta}_X=\Id_{\un{1}}$ for $X=A,B$, and
 \equref{comultunitcomp}, \equref{multcounitcomp}, \equuref{crossprodalg}{c,d}, \equuref{crossprodcoalg}{c,d},
\equref{crossprodalg2} and \equref{crossprodcoalg2}.
 In addition, the last equalities in (vii.1-vii.4) turn out to be \equuref{neccconds}{d}, \equuref{neccconds}{h}, 
\equuref{twoanothYDconds}{a} and \equuref{twoanothYDconds}{b}. \equref{BespDrabComp} follows
immediately from (vi), using \equref{psiphiasactcoact}.
It then follows from
Lemmas \ref{le:implic2} and \ref{le:implic2p} that $(A, B, \psi, \phi)$ is a cross product algebra-coalgebra datum. 
The result now follows from the equivalence of the conditions (iv-vii) in \thref{firstsetequivcond},
verification of the details is left to the reader.
\end{proof}

Let us compare the conditions in \thref{crossprobialasactandcoact} with the Bespalov-Drabant list.
Conditions (i-v) appear in the Bespalov-Drabant list. Conditions (vi) are also in the Bespalov-Drabant
list, namely they are the module-algebra, the comodule-coalgebra, and the algebra-coalgebra compatibility.
The remaining conditions in the Bespalov-Drabant list are the module-comodule, module-coalgebra
and comodule-algebra compatibility. In order to obtain sufficient conditions, these
three conditons have to be replaced
by our condition (vii), which appears in four equivalent sets of three equations. Each of the
four equations (\ref{eq:modulecomodule1}-\ref{eq:modulecomodule2})
can be regarded as the appropriate substitute of the module-comodule compatibility.\\

We end this Section with a reformulation of \prref{whenacrossprodisHA} in terms of actions and coactions.
The proof is left to the reader.

\begin{proposition}\prlabel{5.5}
Let $A\times_\psi^\phi B$ be a cross product bialgebra. If $\Id_A$ and $\Id_B$ have convolution inverses
$\un{S}$ and $\un{s}$, 
then $A\times_\psi^\phi B$ is a Hopf algebra with antipode      
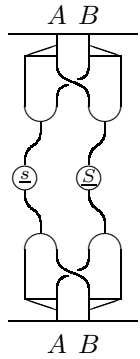
\[
\gbeg{4}{11}
\gvac{1}\got{1}{A}\got{1}{B}\gnl
\glcm\grcm\gnl
\gcl{1}\gbr\gcl{1}\gnl
\gmu\gmu\gnl
\gcn{1}{1}{2}{1}\gvac{1}\gcn{1}{1}{2}{1}\gnl
\gmp{\un{s}}\gvac{1}\gmp{\un{S}}\gnl
\gcn{1}{1}{1}{2}\gvac{1}\gcn{1}{1}{1}{2}\gnl
\gcmu\gcmu\gnl
\gcl{1}\gbr\gcl{1}\gnl
\glm\grm\gnl
\gvac{1}\gob{1}{A}\gob{1}{B}
\gend
\hspace{1mm}.
\]
\end{proposition}

\section{Smash cross (co)product bialgebras}\selabel{smashcross(co)prodbialgs}
\setcounter{equation}{0}
As a general conclusion so far, we can conclude that there are essentially three ways to
describe cross product bialgebras:
\begin{enumerate}
\item by bialgebra admissible tuples, these are characterized in \thref{firstsetequivcond};
\item by actions and coactions, this is discussed in \thref{crossprobialasactandcoact};
\item by injections and projections, this result will be recalled in \prref{DrBespExtVers}.
\end{enumerate}
The second and third description are not entirely satisfactory in the following sense. As we have
remarked above, the substitute of the module-comodule compatibility in
\thref{crossprobialasactandcoact} appears in four different forms, which are equivalent if
some other conditions are satisfied. What is missing is a
kind of  unified module-comodule
compatibility. The objection to the injection/projection description is that
we need two algebras/coalgebras and two projections. In some classical results, see
a brief survey in the introduction,
one projection is sufficient.\\
In this Section, we will characterize smash product bialgebras and smash coproduct bialgebras,
and we will see that the four module-comodule 
compatibility relations unify in this case.\\
As applications, we will see that if a cross product bialgebra comes with a tensor product (co)algebra 
structure then it is necessarily a double cross (co)product bialgebra 
in the sense of Majid \cite{majbip}. When we apply this result to the category of sets, then
we obtain that the only cross product Hopf algebra structure
is the bicross product of groups introduced by Takeuchi in \cite{tak}. 
We will also describe the cross product bialgebras that are a biproduct in 
the sense of Radford \cite{rad}.\\
The second objection can be overcome if we restrict attention to smash (co)product Hopf algebras;
then it turns out that one projection suffices, the
other one can be recovered from it. This will be the topic of \seref{strHopfalgwithprof}.\\

First we will establish that smash product bialgebras and smash coproduct bialgebras
are completely determined by normality properties of the morphisms $\psi$ and $\phi$.  
This is mainly due to the crucial relations \equuref{BespDrabComp}{c,g}.  

\begin{definition}\delabel{normal}
Let $A, B$ be algebras and coalgebras and $\psi: B\ot A\ra A\ot B$, 
$\phi: A\ot B\ra B\ot A$ morphisms in $\Cc$. 
\begin{itemize}
\item[(i)] $\psi$ is called left (right) conormal if 
$
\gbeg{2}{4}
\got{1}{B}\got{1}{A}\gnl
\gbrc\gnl
\gcu{1}\gcl{1}\gnl
\gvac{1}\gob{1}{B}
\gend
=
\gbeg{2}{3}
\got{1}{B}\got{1}{A}\gnl
\gcl{1}\gcu{1}\gnl
\gob{1}{B}
\gend
$ 
$\left(
\gbeg{2}{4}
\got{1}{B}\got{1}{A}\gnl
\gbrc\gnl
\gcl{1}\gcu{1}\gnl
\gob{1}{A}
\gend
=
\gbeg{2}{3}
\got{1}{B}\got{1}{A}\gnl
\gcu{1}\gcl{1}\gnl
\gvac{1}\gob{1}{A}
\gend
\right)
$. 
\item[(ii)] $\phi$ is called left (right) normal if 
$
\gbeg{2}{4}
\gvac{1}\got{1}{B}\gnl
\gu{1}\gcl{1}\gnl
\gbrbox\gnl
\gob{1}{B}\gob{1}{A}\gnl
\gend
=
\gbeg{2}{3}
\got{1}{B}\gnl
\gcl{1}\gu{1}\gnl
\gob{1}{B}\gob{1}{A}
\gend
$ 
$
\left(
\gbeg{2}{4}
\got{1}{A}\gnl
\gcl{1}\gu{1}\gnl
\gbrbox\gnl
\gob{1}{B}\gob{1}{A}\gnl
\gend
=
\gbeg{2}{3}
\gvac{1}\got{1}{A}\gnl
\gu{1}\gcl{1}\gnl
\gob{1}{B}\gob{1}{A}
\gend
\right)
$.  
\end{itemize}
\end{definition}

\begin{lemma}\lelabel{normalsmash}
Let $A\times _\psi^\phi B$ be a cross product bialgebra. 
$\psi$ is left (right) conormal if and only if $A\# _\psi B$ is a left (right) 
smash product algebra.
$\phi$ is left (right) normal if and only if $A\#^\phi B$ is a 
left (right) smash coproduct coalgebra.  
\end{lemma}

\begin{proof}
Since $A\times_\psi^\phi B$ is a cross product bialgebra the equalities \equuref{neccconds}{g,c} and
\equuref{BespDrabComp}{b} hold.
Thus if $\psi$ is left conormal then $B$ is a bialgebra in $\Cc$ and $\psi$ satisfies \equref{psismashprod}. It then follows from 
\prref{crossprodissmashprodalg} that $A\times_\psi B$ is a smash product. 
Conversely, if $A\times_\psi B$ is  a left smash product algebra, then $B$ is a bialgebra in $\Cc$ and 
$\psi$ satisfies \equref{psismashprod}, see \prref{crossprodissmashprodalg}. Compose 
\equref{psismashprod} to the left 
with $\un{\va}_A\ot \Id_B$; using \equuref{multcounitcomp}{c}, it follows that $\psi$ is left conormal.   
The proof of the right handed version is similar, and the second assertion
is the dual of the first one. 
\end{proof}

\begin{corollary}\colabel{Radford}
A cross product bialgebra $A\#_\psi^\phi B$ is a left (right) Radford biproduct (this
means that $A\# _\psi B$ is a left (right) smash product algebra and $A\#^\phi B$ is a 
left (right) smash coproduct coalgebra)
 if and only if $\psi$ is left (right) conormal and 
$\phi$ is left (right) normal. If, moreover, $B$ is a Hopf algebra and $\Id_A$ is convolution invertible, then 
$A\#_\psi^\phi B$ is a Hopf algebra.   
\end{corollary}

Our next aim is to describe smash cross product bialgebras, these are
cross product bialgebras with a smash product algebra as underlying algebra. Obviously
Radford biproducts are special cases, and this is why we did not provide an explicit construction
of the Radford biproduct. \thref{strsmashcrossprodHa} is a generalization of \cite[Theorem 4.5]{cimz},
where the special case where $A$ and $B$ are bialgebras is discussed.

\begin{theorem}\thlabel{strsmashcrossprodHa}
Let $A, B$ be algebras and coalgebras, and $\psi: B\ot A\ra A\ot B$ and $\phi: A\ot B\ra B\ot A$ 
morphisms in $\Cc$ such that $\psi$ is left normal. The following assertions are equivalent:\\
(i) $A\#_\psi^\phi B$ is a cross product bialgebra (and therefore a smash cross product bialgebra,
by \leref{normalsmash}).\\
(ii) $B$ is a bialgebra in $\Cc$, $A$ is a left $B$-module algebra and a left $B$-comodule algebra, $B$ is a right 
$A$-module and comodule and the following relations hold:
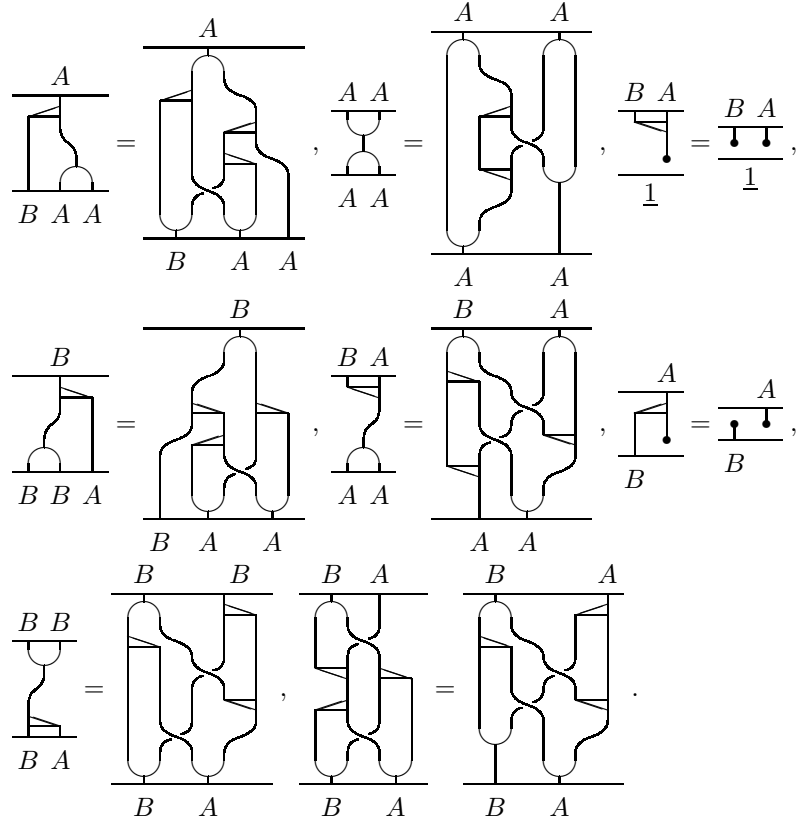
\begin{eqnarray*}
&&
\gbeg{3}{5}
\gvac{1}\got{1}{A}\gnl
\glcm\gnl
\gcl{1}\gcn{1}{1}{1}{2}\gnl
\gcl{1}\gcmu\gnl
\gob{1}{B}\gob{1}{A}\gob{1}{A}
\gend
=
\gbeg{5}{8}
\gvac{1}\got{2}{A}\gnl
\gvac{1}\gcmu\gnl
\glcm\gcn{1}{1}{1}{3}\gnl
\gcl{1}\gcl{1}\glcm\gnl
\gcl{1}\gcl{1}\grcm\gcn{1}{1}{-1}{1}\gnl
\gcl{1}\gbr\gcl{1}\gcl{1}\gnl
\gmu\gmu\gcl{1}\gnl
\gob{2}{B}\gob{2}{A}\gob{1}{A}
\gend
\hspace{1mm},\hspace{1mm}
\gbeg{2}{4}
\got{1}{A}\got{1}{A}\gnl
\gmu\gnl
\gcmu\gnl
\gob{1}{A}\gob{1}{A}
\gend
=
\gbeg{5}{9}
\got{2}{A}\gvac{1}\got{2}{A}\gnl
\gcmu\gvac{1}\gcmu\gnl
\gcl{1}\gcn{1}{1}{1}{3}\gvac{1}\gcl{1}\gcl{1}\gnl
\gcl{1}\glcm\gcl{1}\gcl{1}\gnl
\gcl{1}\gcl{1}\gbr\gcl{1}\gnl
\gcl{1}\glm\gmu\gnl
\gcl{1}\gvac{1}\gcn{1}{1}{1}{-1}\gcn{1}{2}{2}{2}\gnl
\gmu\gnl
\gob{2}{A}\gvac{1}\gob{2}{A}
\gend 
\hspace{1mm},\hspace{1mm}
\gbeg{2}{4}
\got{1}{B}\got{1}{A}\gnl
\glm\gnl
\gvac{1}\gcu{1}\gnl
\gob{2}{\un{1}}
\gend
=
\gbeg{2}{3}
\got{1}{B}\got{1}{A}\gnl
\gcu{1}\gcu{1}\gnl
\gob{2}{\un{1}}
\gend
\hspace{1mm},\\
&&
\gbeg{3}{5}
\gvac{1}\got{1}{B}\gnl
\gvac{1}\grcm\gnl
\gvac{1}\gcn{1}{1}{1}{0}\gcl{1}\gnl
\gcmu\gcl{1}\gnl
\gob{1}{B}\gob{1}{B}\gob{1}{A}
\gend
=
\gbeg{5}{8}
\gvac{2}\got{2}{B}\gnl
\gvac{2}\gcmu\gnl
\gvac{2}\gcn{1}{1}{1}{-1}\gcl{1}\gnl
\gvac{1}\grcm\grcm\gnl
\gvac{1}\gcn{1}{1}{1}{-1}\gvac{-1}\glcm\gcl{1}\gcl{1}\gnl
\gcl{1}\gcl{1}\gbr\gcl{1}\gnl
\gcl{1}\gmu\gmu\gnl
\gob{1}{B}\gob{2}{A}\gob{2}{A}
\gend
\hspace{1mm},\hspace{1mm}
\gbeg{2}{5}
\got{1}{B}\got{1}{A}\gnl
\glm\gnl
\gvac{1}\gcn{1}{1}{1}{0}\gnl
\gcmu\gnl
\gob{1}{A}\gob{1}{A}
\gend
=
\gbeg{5}{8}
\got{2}{B}\gvac{1}\got{2}{A}\gnl
\gcmu\gvac{1}\gcmu\gnl
\grcm\gcn{1}{1}{-1}{1}\gcl{1}\gcl{1}\gnl
\gcl{1}\gcl{1}\gbr\gcl{1}\gnl
\gcl{1}\gbr\glm\gnl
\glm\gcl{1}\gvac{1}\gcn{1}{1}{1}{-1}\gnl
\gvac{1}\gcl{1}\gmu\gnl
\gvac{1}\gob{1}{A}\gob{2}{A}
\gend
\hspace{1mm},\hspace{1mm}
\gbeg{2}{4}
\gvac{1}\got{1}{A}\gnl
\glcm\gnl
\gcl{1}\gcu{1}\gnl
\gob{1}{B}
\gend
=
\gbeg{2}{3}
\gvac{1}\got{1}{A}\gnl
\gu{1}\gcu{1}\gnl
\gob{1}{B}
\gend
\hspace{1mm},\\
&&
\gbeg{2}{5}
\got{1}{B}\got{1}{B}\gnl
\gmu\gnl
\gcn{1}{1}{2}{1}\gnl
\grcm\gnl
\gob{1}{B}\gob{1}{A}
\gend
=
\gbeg{5}{8}
\got{2}{B}\gvac{1}\got{2}{B}\gnl
\gcmu\gvac{1}\grcm\gnl
\grcm\gcn{1}{1}{-1}{1}\gcl{1}\gcl{1}\gnl
\gcl{1}\gcl{1}\gbr\gcl{1}\gnl
\gcl{1}\gcl{1}\gcl{1}\glm\gnl
\gcl{1}\gbr\gvac{1}\gcn{1}{1}{1}{-1}\gnl
\gmu\gmu\gnl
\gob{2}{B}\gob{2}{A}
\gend
\hspace{1mm},\hspace{1mm}
\gbeg{4}{8}
\got{2}{B}\got{1}{A}\gnl
\gcmu\gcl{1}\gnl
\gcl{1}\gbr\gnl
\glm\grcm\gnl
\glcm\gcl{1}\gcl{1}\gnl
\gcl{1}\gbr\gcl{1}\gnl
\gmu\gmu\gnl
\gob{2}{B}\gob{2}{A}
\gend
=
\gbeg{5}{8}
\got{2}{B}\gvac{2}\got{1}{A}\gnl
\gcmu\gvac{1}\glcm\gnl
\grcm\gcn{1}{1}{-1}{1}\gcl{1}\gcl{1}\gnl
\gcl{1}\gcl{1}\gbr\gcl{1}\gnl
\gcl{1}\gbr\glm\gnl
\gmu\gcl{1}\gcn{1}{1}{3}{1}\gnl
\gcn{1}{1}{2}{2}\gvac{1}\gmu\gnl
\gob{2}{B}\gob{2}{A}
\gend
\hspace{1mm}.
\end{eqnarray*}    
If $\Id_A$
has a convolution inverse $\un{S}$ and $B$ is a 
Hopf algebra with antipode $\un{s}$, then $A\times_\psi^\phi B$ is a Hopf algebra 
in $\Cc$ with antipode
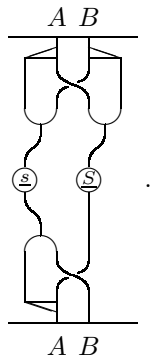
\[
\gbeg{4}{11}
\gvac{1}\got{1}{A}\got{1}{B}\gnl
\glcm\grcm\gnl
\gcl{1}\gbr\gcl{1}\gnl
\gmu\gmu\gnl
\gcn{1}{1}{2}{1}\gvac{1}\gcn{1}{1}{2}{1}\gnl
\gmp{\un{s}}\gvac{1}\gmp{\un{S}}\gnl
\gcn{1}{1}{1}{2}\gvac{1}\gcl{1}\gnl
\gcmu\gcl{1}\gnl
\gcl{1}\gbr\gnl
\glm\gcl{1}\gnl
\gvac{1}\gob{1}{A}\gob{1}{B}
\gend
\hspace{1mm}.
\] 
\end{theorem}

\begin{proof} 
$A\#_\psi^\phi B$ is a cross product bialgebra if and only if
conditions (i-vi) and (vii.2) from \thref{crossprobialasactandcoact} are fulfilled.
Using the left normality of $\psi$, it follows easily that these conditions reduce to
condition (ii) in \thref{strsmashcrossprodHa}, with one exception: we will show that
the third equality in (vii.2) is equivalent to the seventh and eighth compatibility condition in
\thref{strsmashcrossprodHa} and the fact that $A$ is a left $B$-comodule algebra.
Indeed, using the left normality of $\psi$, the third equality in (vii.2) takes the form
\[
\gbeg{5}{11}
\got{1}{A}\got{2}{B}\got{1}{A}\got{1}{B}\gnl
\gcl{1}\gcmu\gcl{1}\gcl{1}\gnl
\gcl{1}\gcl{1}\gbr\gcl{1}\gnl
\gcl{1}\glm\gmu\gnl
\gcl{1}\gcn{1}{1}{3}{1}\gvac{1}\gcn{1}{1}{2}{-1}\gnl
\gmu\grcm\gnl
\gcn{1}{1}{2}{3}\gvac{1}\gcl{1}\gcl{1}\gnl
\glcm\gcl{1}\gcl{1}\gnl
\gcl{1}\gbr\gcl{1}\gnl
\gmu\gmu\gnl
\gob{2}{B}\gob{2}{A}
\gend
=
\gbeg{9}{10}
\gvac{1}\got{1}{A}\got{2}{B}\gvac{2}\got{1}{A}\got{1}{B}\gnl
\glcm\gcmu\gvac{1}\glcm\grcm\gnl
\gcl{1}\gcl{1}\grcm\gcn{1}{1}{-1}{1}\gcl{1}\gbr\gcl{1}\gnl
\gcl{1}\gbr\gcl{1}\gcl{1}\gmu\gmu\gnl
\gmu\gmu\gcl{1}\gcn{1}{1}{2}{1}\gcn{1}{1}{4}{1}\gnl
\gcn{1}{2}{2}{5}\gvac{1}\gcn{1}{1}{2}{3}\gvac{1}\gbr\gcl{1}\gnl
\gvac{3}\gbr\glm\gnl
\gvac{2}\gcl{1}\gcl{1}\gcl{1}\gcn{1}{1}{3}{1}\gnl
\gvac{2}\gmu\gmu\gnl
\gvac{2}\gob{2}{B}\gob{2}{A}
\gend
\hspace{1mm}.
\] 
Composing this equality to the right with $\un{\eta}_A\ot \Id_B\ot \un{\eta}_A\ot \Id_B$,
we obtain the seventh compatibility condition. Composing it to the right with
$\Id_A\ot \un{\eta}_B\ot \Id_A\ot \un{\eta}_B$, we find that the left $B$-coaction on
$A$ is a morphism in ${}^B\Cc$. Together with
$
\gbeg{2}{4}
\gvac{1}\got{1}{\un{1}}\gnl
\gvac{1}\gu{1}\gnl
\glcm\gnl
\gob{1}{B}\gob{1}{A}
\gend
=
\gbeg{2}{3}
\got{2}{\un{1}}\gnl
\gu{1}\gu{1}\gnl
\gob{1}{B}\gob{1}{A}
\gend~,
$ 
this tells us that $A$ is a left $B$-comodule algebra. Finally, composition to the right
with $\un{\eta}_A\ot \Id_{B\ot A}\ot \un{\eta}_B$ gives the eighth compatibility condition.
The proof on the converse implication is based on a direct computation:
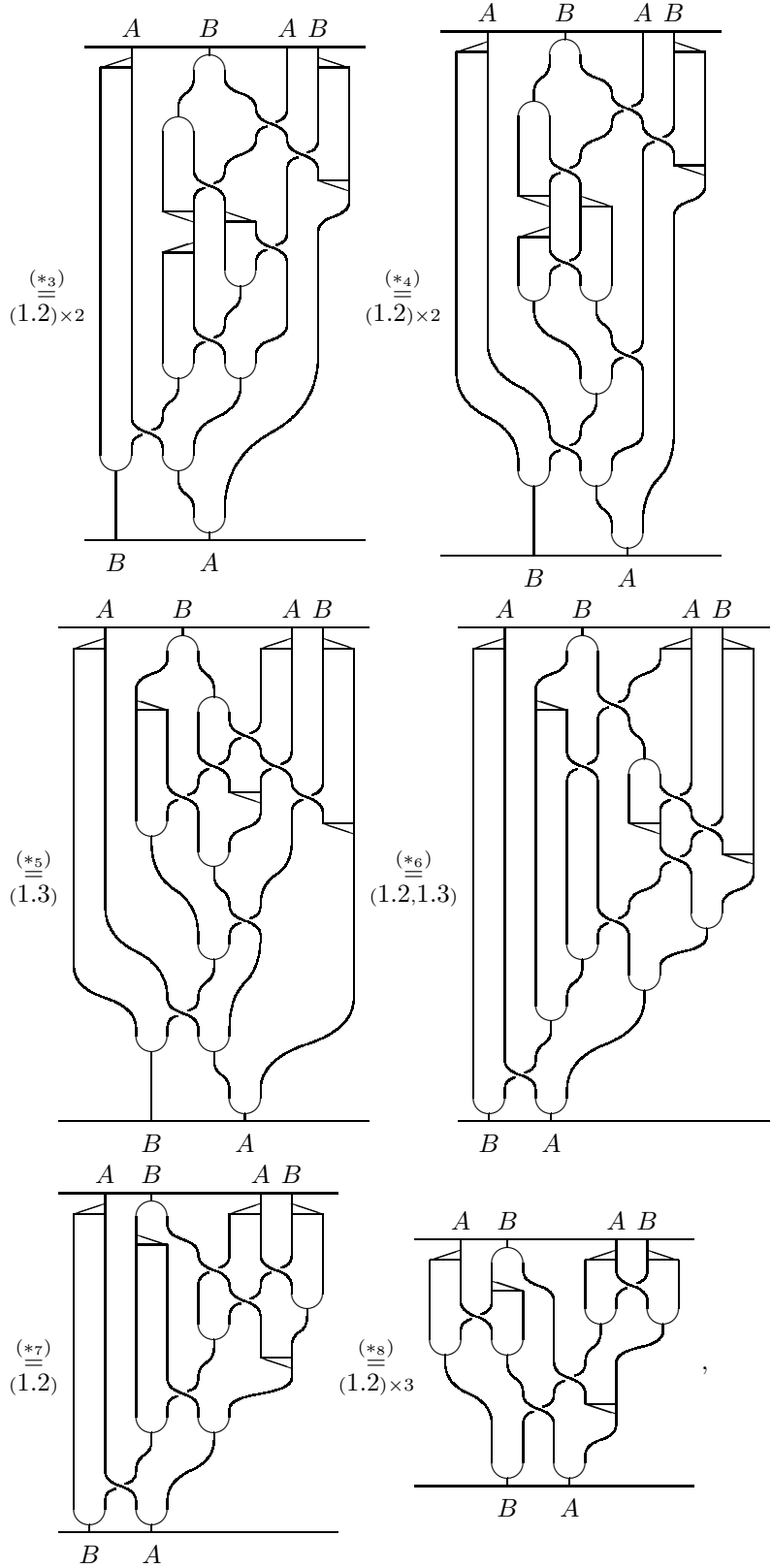
\begin{eqnarray*}
&&\hspace*{-1cm}
\gbeg{5}{11}
\got{1}{A}\got{2}{B}\got{1}{A}\got{1}{B}\gnl
\gcl{1}\gcmu\gcl{1}\gcl{1}\gnl
\gcl{1}\gcl{1}\gbr\gcl{1}\gnl
\gcl{1}\glm\gmu\gnl
\gcl{1}\gcn{1}{1}{3}{1}\gvac{1}\gcn{1}{1}{2}{-1}\gnl
\gmu\grcm\gnl
\gcn{1}{1}{2}{3}\gvac{1}\gcl{1}\gcl{1}\gnl
\glcm\gcl{1}\gcl{1}\gnl
\gcl{1}\gbr\gcl{1}\gnl
\gmu\gmu\gnl
\gob{2}{B}\gob{2}{A}
\gend
\equal{(*_1)}
\gbeg{9}{13}
\gvac{1}\got{1}{A}\got{2}{B}\got{1}{A}\gvac{2}\got{1}{B}\gnl
\glcm\gcmu\gcl{1}\gvac{2}\grcm\gnl
\gcl{1}\gcl{1}\gcl{1}\gbr\gvac{2}\gcl{1}\gcl{1}\gnl
\gcl{1}\gcl{1}\glm\gcn{1}{1}{1}{2}\gvac{2}\gcl{1}\gcl{1}\gnl
\gcl{1}\gcl{1}\glcm\gcmu\gvac{1}\gcl{1}\gcl{1}\gnl
\gcl{1}\gbr\gcl{1}\grcm\gcn{1}{1}{-1}{1}\gcl{1}\gcl{1}\gnl
\gmu\gmu\gcl{1}\gcl{1}\gbr\gcl{1}\gnl
\gcn{1}{1}{2}{5}\gvac{1}\gcn{1}{1}{2}{3}\gvac{1}\gcl{1}\gbr\glm\gnl
\gvac{2}\gcl{1}\gcl{1}\gmu\gcl{1}\gvac{1}\gcn{1}{1}{1}{-1}\gnl
\gvac{2}\gcl{1}\gcl{1}\gcn{1}{1}{2}{1}\gvac{1}\gmu\gnl
\gvac{2}\gcl{1}\gbr\gvac{1}\gcn{1}{1}{2}{-1}\gnl
\gvac{2}\gmu\gmu\gnl
\gvac{2}\gob{2}{B}\gob{2}{A}
\gend
\equalupdown{(*_2)}{\equref{nat2cup}}
\gbeg{9}{16}
\gvac{1}\got{1}{A}\gvac{1}\got{2}{B}\gvac{1}\got{1}{A}\got{1}{B}\gnl
\glcm\gvac{1}\gcmu\gvac{1}\gcl{1}\grcm\gnl
\gcl{1}\gcl{1}\gvac{1}\gcn{1}{1}{1}{0}\gcn{1}{1}{1}{3}\gvac{1}\gcl{1}\gcl{1}\gcl{1}\gnl
\gcl{1}\gcl{1}\gcmu\gvac{1}\gbr\gcl{1}\gcl{1}\gnl
\gcl{1}\gcl{1}\gcl{1}\gcl{1}\gvac{1}\gcn{1}{1}{1}{-1}\gbr\gcl{1}\gnl
\gcl{1}\gcl{1}\gcl{1}\gbr\gvac{1}\gcl{1}\glm\gnl
\gcl{1}\gcl{1}\glm\grcm\gcl{1}\gcn{1}{1}{3}{1}\gnl
\gcl{1}\gcl{1}\glcm\gcl{1}\gbr\gcl{1}\gnl
\gcl{1}\gbr\gcl{1}\gmu\gcl{1}\gcl{1}\gnl
\gmu\gmu\gcn{1}{2}{2}{1}\gvac{1}\gcl{1}\gcl{1}\gnl
\gcn{1}{2}{2}{5}\gvac{1}\gcn{1}{1}{2}{3}\gvac{3}\gcn{1}{2}{1}{-1}\gcn{1}{4}{1}{-1}\gnl
\gvac{3}\gbr\gvac{1}\gnl
\gvac{2}\gmu\gmu\gnl
\gvac{2}\gcn{1}{2}{2}{2}\gvac{1}\gcn{1}{1}{2}{3}\gnl
\gvac{5}\gmu\gnl
\gvac{2}\gob{2}{B}\gvac{1}\gob{2}{A}
\gend
\\
&&
\equalupdown{\rm (*_3)}{\equref{nat1cup}\times 2}
\gbeg{9}{18}
\gvac{1}\got{1}{A}\gvac{1}\got{2}{B}\gvac{1}\got{1}{A}\got{1}{B}\gnl
\glcm\gvac{1}\gcmu\gvac{1}\gcl{1}\grcm\gnl
\gcl{1}\gcl{1}\gvac{1}\gcn{1}{1}{1}{0}\gcn{1}{1}{1}{3}\gvac{1}\gcl{1}\gcl{1}\gcl{1}\gnl
\gcl{1}\gcl{1}\gcmu\gvac{1}\gbr\gcl{1}\gcl{1}\gnl
\gcl{1}\gcl{1}\gcl{1}\gcl{1}\gvac{1}\gcn{1}{1}{1}{-1}\gbr\gcl{1}\gnl
\gcl{1}\gcl{1}\gcl{1}\gbr\gvac{1}\gcl{1}\glm\gnl
\gcl{1}\gcl{1}\glm\grcm\gcl{1}\gcn{1}{1}{3}{1}\gnl
\gcl{1}\gcl{1}\glcm\gcl{1}\gbr\gcl{1}\gnl
\gcl{1}\gcl{1}\gcl{1}\gcl{1}\gmu\gcl{1}\gcl{1}\gnl
\gcl{1}\gcl{1}\gcl{1}\gcl{1}\gcn{1}{1}{2}{1}\gvac{1}\gcl{1}\gcl{1}\gnl
\gcl{1}\gcl{1}\gcl{1}\gbr\gvac{1}\gcn{1}{1}{1}{-1}\gcl{1}\gnl
\gcl{1}\gcl{1}\gmu\gmu\gvac{1}\gcn{1}{5}{1}{-5}\gnl
\gcl{1}\gcl{1}\gcn{1}{1}{2}{1}\gvac{1}\gcn{1}{2}{2}{-1}\gnl
\gcl{1}\gbr\gnl
\gmu\gmu\gnl
\gcn{1}{2}{2}{2}\gcn{1}{1}{4}{5}\gnl
\gvac{3}\gmu\gnl
\gob{2}{B}\gvac{1}\gob{2}{A}
\gend
\equalupdown{\rm (*_4)}{\equref{nat1cup}\times 2}
\gbeg{9}{19}
\gvac{1}\got{1}{A}\gvac{1}\got{2}{B}\gvac{1}\got{1}{A}\got{1}{B}\gnl
\glcm\gvac{1}\gcmu\gvac{1}\gcl{1}\grcm\gnl
\gcl{1}\gcl{1}\gvac{1}\gcn{1}{1}{1}{0}\gcn{1}{1}{1}{3}\gvac{1}\gcl{1}\gcl{1}\gcl{1}\gnl
\gcl{1}\gcl{1}\gcmu\gvac{1}\gbr\gcl{1}\gcl{1}\gnl
\gcl{1}\gcl{1}\gcl{1}\gcl{1}\gvac{1}\gcn{1}{1}{1}{-1}\gbr\gcl{1}\gnl
\gcl{1}\gcl{1}\gcl{1}\gbr\gvac{1}\gcl{1}\glm\gnl
\gcl{1}\gcl{1}\glm\grcm\gcl{1}\gcn{1}{1}{3}{1}\gnl
\gcl{1}\gcl{1}\glcm\gcl{1}\gcl{1}\gcl{1}\gcl{1}\gnl
\gcl{1}\gcl{1}\gcl{1}\gbr\gcl{1}\gcl{1}\gcl{1}\gnl
\gcl{1}\gcl{1}\gmu\gmu\gcl{1}\gcl{1}\gnl
\gcl{1}\gcl{1}\gcn{1}{2}{2}{5}\gvac{1}\gcn{1}{1}{2}{3}\gvac{1}\gcl{1}\gcl{1}\gnl
\gcl{1}\gcn{1}{3}{1}{5}\gvac{1}\gvac{2}\gbr\gcl{1}\gnl
\gcn{1}{3}{1}{5}\gvac{3}\gmu\gcl{1}\gcl{1}\gnl
\gvac{4}\gcn{1}{1}{2}{1}\gvac{1}\gcl{1}\gcl{1}\gnl
\gvac{3}\gbr\gvac{1}\gcn{1}{1}{1}{-1}\gcn{1}{3}{1}{-1}\gnl
\gvac{2}\gmu\gmu\gnl
\gvac{2}\gcn{1}{2}{2}{2}\gvac{1}\gcn{1}{1}{2}{3}\gnl
\gvac{5}\gmu\gnl
\gvac{2}\gob{2}{B}\gvac{1}\gob{2}{A}
\gend\\
&&
\equalupdown{(*_5)}{\equref{nat2cup}}
\gbeg{10}{18}
\gvac{1}\got{1}{A}\gvac{1}\got{2}{B}\gvac{2}\got{1}{A}\got{1}{B}\gnl
\glcm\gvac{1}\gcmu\gvac{1}\glcm\grcm\gnl
\gcl{1}\gcl{1}\gvac{1}\gcn{1}{1}{1}{-1}\gcn{1}{1}{1}{2}\gvac{1}\gcl{1}\gcl{1}\gcl{1}\gcl{1}\gnl
\gcl{1}\gcl{1}\grcm\gcmu\gcl{1}\gcl{1}\gcl{1}\gcl{1}\gnl
\gcl{1}\gcl{1}\gcl{1}\gcl{1}\gcl{1}\gbr\gcl{1}\gcl{1}\gcl{1}\gnl
\gcl{1}\gcl{1}\gcl{1}\gcl{1}\gbr\gbr\gcl{1}\gcl{1}\gnl
\gcl{1}\gcl{1}\gcl{1}\gbr\glm\gbr\gcl{1}\gnl
\gcl{1}\gcl{1}\gmu\gcl{1}\gvac{1}\gcn{1}{1}{1}{-1}\gcl{1}\glm\gnl
\gcl{1}\gcl{1}\gcn{1}{3}{2}{5}\gvac{1}\gmu\gvac{1}\gcn{1}{2}{1}{-1}\gvac{1}\gcl{5}\gnl
\gcl{1}\gcl{1}\gvac{2}\gcn{1}{1}{2}{3}\gnl
\gcl{1}\gcn{1}{3}{1}{5}\gvac{3}\gbr\gnl
\gcl{1}\gvac{3}\gmu\gcn{1}{3}{1}{-1}\gnl
\gcn{1}{2}{1}{5}\gvac{3}\gcn{1}{1}{2}{1}\gnl
\gvac{3}\gbr\gvac{4}\gcn{1}{3}{1}{-5}\gnl
\gvac{2}\gmu\gmu\gnl
\gvac{2}\gcn{1}{2}{2}{2}\gvac{1}\gcn{1}{1}{2}{3}\gnl
\gvac{5}\gmu\gnl
\gvac{2}\gob{2}{B}\gvac{1}\gob{2}{A}
\gend
\equalupdown{(*_6)}{(\ref{eq:nat1cup},\ref{eq:nat2cup})}
\gbeg{10}{18}
\gvac{1}\got{1}{A}\gvac{1}\got{2}{B}\gvac{2}\got{1}{A}\got{1}{B}\gnl
\glcm\gvac{1}\gcmu\gvac{1}\glcm\grcm\gnl
\gcl{1}\gcl{1}\gcn{1}{1}{3}{1}\gvac{1}\gcl{1}\gcn{1}{1}{3}{1}\gvac{1}\gcl{1}\gcl{1}\gcl{1}\gnl
\gcl{1}\gcl{1}\grcm\gbr\gvac{1}\gcl{1}\gcl{1}\gcl{1}\gnl
\gcl{1}\gcl{1}\gcl{1}\gcl{1}\gcl{1}\gcn{1}{1}{1}{2}\gvac{1}\gcl{1}\gcl{1}\gcl{1}\gnl
\gcl{1}\gcl{1}\gcl{1}\gbr\gcmu\gcl{1}\gcl{1}\gcl{1}\gnl
\gcl{1}\gcl{1}\gcl{1}\gcl{1}\gcl{1}\gcl{1}\gbr\gcl{1}\gcl{1}\gnl
\gcl{1}\gcl{1}\gcl{1}\gcl{1}\gcl{1}\glm\gbr\gcl{1}\gnl
\gcl{1}\gcl{1}\gcl{1}\gcl{1}\gcl{1}\gvac{1}\gbr\glm\gnl
\gcl{1}\gcl{1}\gcl{1}\gcl{1}\gcl{1}\gvac{1}\gcn{1}{1}{1}{-1}\gcl{1}\gcn{1}{1}{3}{1}\gnl
\gcl{1}\gcl{1}\gcl{1}\gcl{1}\gbr\gvac{1}\gmu\gnl
\gcl{1}\gcl{1}\gcl{1}\gmu\gcl{1}\gcn{1}{1}{4}{1}\gnl
\gcl{1}\gcl{1}\gcl{1}\gcn{1}{1}{2}{1}\gvac{1}\gmu\gnl
\gcl{1}\gcl{1}\gmu\gvac{1}\gcn{1}{3}{2}{-3}\gnl
\gcl{1}\gcl{1}\gcn{1}{1}{2}{1}\gnl
\gcl{1}\gbr\gnl
\gmu\gmu\gnl
\gob{2}{B}\gob{2}{A}
\gend\\
&&
\equalupdown{(*_7)}{\equref{nat1cup}}
\gbeg{9}{13}
\gvac{1}\got{1}{A}\got{2}{B}\gvac{2}\got{1}{A}\got{1}{B}\gnl
\glcm\gcmu\gvac{1}\glcm\grcm\gnl
\gcl{1}\gcl{1}\grcm\gcn{1}{1}{-1}{1}\gcl{1}\gcl{1}\gcl{1}\gcl{1}\gnl
\gcl{1}\gcl{1}\gcl{1}\gcl{1}\gbr\gbr\gcl{1}\gnl
\gcl{1}\gcl{1}\gcl{1}\gcl{1}\gcl{1}\gbr\gmu\gnl
\gcl{1}\gcl{1}\gcl{1}\gcl{1}\gmu\gcl{1}\gcn{1}{1}{2}{1}\gnl
\gcl{1}\gcl{1}\gcl{1}\gcl{1}\gcn{1}{1}{2}{1}\gvac{1}\glm\gnl
\gcl{1}\gcl{1}\gcl{1}\gbr\gvac{2}\gcn{1}{1}{1}{-3}\gnl
\gcl{1}\gcl{1}\gmu\gmu\gnl
\gcl{1}\gcl{1}\gcn{1}{1}{2}{1}\gcn{1}{2}{4}{1}\gnl
\gcl{1}\gbr\gnl
\gmu\gmu\gnl
\gob{2}{B}\gob{2}{A}
\gend
\equalupdown{(*_8)}{\equref{nat1cup}\times 3}
\gbeg{9}{10}
\gvac{1}\got{1}{A}\got{2}{B}\gvac{2}\got{1}{A}\got{1}{B}\gnl
\glcm\gcmu\gvac{1}\glcm\grcm\gnl
\gcl{1}\gcl{1}\grcm\gcn{1}{1}{-1}{1}\gcl{1}\gbr\gcl{1}\gnl
\gcl{1}\gbr\gcl{1}\gcl{1}\gmu\gmu\gnl
\gmu\gmu\gcl{1}\gcn{1}{1}{2}{1}\gcn{1}{1}{4}{1}\gnl
\gcn{1}{2}{2}{5}\gvac{1}\gcn{1}{1}{2}{3}\gvac{1}\gbr\gcl{1}\gnl
\gvac{3}\gbr\glm\gnl
\gvac{2}\gcl{1}\gcl{1}\gcl{1}\gcn{1}{1}{3}{1}\gnl
\gvac{2}\gmu\gmu\gnl
\gvac{2}\gob{2}{B}\gob{2}{A}
\gend
\hspace{1mm},
\end{eqnarray*}
We used the following properties:
At $(*_1)$: $A$ is a left $B$-comodule algebra, and the seventh compatibility condition;
at $(*_2)$: $\un{\Delta}_B$ is coassociative and $\un{m}_A$ is associative; 
at $(*_3)$ and $(*_8)$: $\un{m}_A$ and $\un{m}_B$ are associative;
at $(*_4)$ and $(*_6)$: $\un{m}_B$ is associative;
at $(*_5)$: $\un{\Delta}_B$ coassociative, and the eigth compatibility condition; 
at $(*_7)$: naturality of the braiding, $A$ is a left $B$-comodule algebra, and the fact that 
$
\gbeg{3}{5}
\got{2}{B}\got{1}{A}\gnl
\gcmu\gcl{1}\gnl
\gcl{1}\gbr\gnl
\glm\gcl{1}\gnl
\gvac{1}\gob{1}{A}\gob{1}{B}
\gend 
$ 
is a morphism in $\Cc$. The assertion concerning the antipode of $A\times_\psi^\phi B$ follows easily from \prref{whenacrossprodisHA}.
\end{proof}

Obviously we also have a right handed version of \thref{strsmashcrossprodHa}.

\begin{corollary}\colabel{doublecrosscoprodbialg}
If $A\times_\psi^\phi B$ is a cross product bialgebra, and
$\psi$ is left and right conormal 
then $A\times_\psi^\phi B=A\blacktriangleright\hspace*{-1mm}\blacktriangleleft B$ is a double cross coproduct bialgebra. 
If $A$ and $B$ are Hopf algebras, then $A\blacktriangleright\hspace*{-1mm}\blacktriangleleft B$ is also a
Hopf algebra.
\end{corollary}

\begin{proof}
First observe that $\psi$ is left and right conormal if and only if $\psi$ is equal to the braiding of $B$
and $A$ in $\Cc$: one implication follows from \equuref{neccconds}{g}, and the other one is immediate.
Then (ii) in \thref{strsmashcrossprodHa} takes the form:\\
1) $A, B$ are bialgebras such that $A$ is a left $B$-comodule algebra and $B$ is a right 
$A$-comodule algebra;\\
2) The following equalities hold:
\[
\gbeg{3}{5}
\gvac{1}\got{1}{A}\gnl
\glcm\gnl
\gcl{1}\gcn{1}{1}{1}{2}\gnl
\gcl{1}\gcmu\gnl
\gob{1}{B}\gob{1}{A}\gob{1}{A}
\gend
=
\gbeg{5}{8}
\gvac{1}\got{2}{A}\gnl
\gvac{1}\gcmu\gnl
\glcm\gcn{1}{1}{1}{3}\gnl
\gcl{1}\gcl{1}\glcm\gnl
\gcl{1}\gcl{1}\grcm\gcn{1}{1}{-1}{1}\gnl
\gcl{1}\gbr\gcl{1}\gcl{1}\gnl
\gmu\gmu\gcl{1}\gnl
\gob{2}{B}\gob{2}{A}\gob{1}{A}
\gend
\hspace{1mm},\hspace{1mm}
\gbeg{3}{5}
\gvac{1}\got{1}{B}\gnl
\gvac{1}\grcm\gnl
\gvac{1}\gcn{1}{1}{1}{0}\gcl{1}\gnl
\gcmu\gcl{1}\gnl
\gob{1}{B}\gob{1}{B}\gob{1}{A}
\gend
=
\gbeg{5}{8}
\gvac{2}\got{2}{B}\gnl
\gvac{2}\gcmu\gnl
\gvac{2}\gcn{1}{1}{1}{-1}\gcl{1}\gnl
\gvac{1}\grcm\grcm\gnl
\gvac{1}\gcn{1}{1}{1}{-1}\gvac{-1}\glcm\gcl{1}\gcl{1}\gnl
\gcl{1}\gcl{1}\gbr\gcl{1}\gnl
\gcl{1}\gmu\gmu\gnl
\gob{1}{B}\gob{2}{A}\gob{2}{A}
\gend
\hspace{1mm},\hspace{1mm}
\gbeg{4}{6}
\gvac{1}\got{1}{B}\got{1}{A}\gnl
\gvac{1}\gbr\gnl
\glcm\grcm\gnl
\gcl{1}\gbr\gcl{1}\gnl
\gmu\gmu\gnl
\gob{2}{B}\gob{2}{A}
\gend
=
\gbeg{4}{5}
\got{1}{B}\gvac{2}\got{1}{A}\gnl
\grcm\glcm\gnl
\gcl{1}\gbr\gcl{1}\gnl
\gmu\gmu\gnl
\gob{2}{B}\gob{2}{A}
\gend
\hspace{1mm}.
\]
In this situation, 
the algebra structure of $A\times_\psi^\phi B$ is the tensor product algebra 
of $A$ and $B$, while the coalgebra structure is given by 
$$\un{\Delta}_{A\times _\psi^\phi B}=
\gbeg{6}{6}
\gvac{1}\got{2}{A}\got{2}{B}\gnl
\gvac{1}\gcmu\gcmu\gnl
\gcn{1}{1}{3}{1}\glcm\grcm\gcn{1}{1}{-1}{1}\gnl
\gcl{1}\gcl{1}\gbr\gcl{1}\gcl{1}\gnl
\gcl{1}\gmu\gmu\gcl{1}\gnl
\gob{1}{A}\gob{2}{B}\gob{2}{A}\gob{1}{B}
\gend
$$
and $\un{\va}_{A\times_\psi^\phi B}=\un{\va}_A\ot \un{\va}_B$. This tells us that
$A\times_\psi^\phi B=A\blacktriangleright\hspace*{-1mm}\blacktriangleleft B$ is a double cross
 coproduct bialgebra.\\
 Finally, if $A$ and $B$ are Hopf algebras with antipodes $\un{S}$ and $\un{s}$, 
then 
$A\blacktriangleright\hspace*{-1mm}\blacktriangleleft B$ is a Hopf algebra with antipode
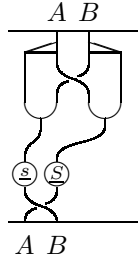
\[
\gbeg{4}{8}
\gvac{1}\got{1}{A}\got{1}{B}\gnl
\glcm\grcm\gnl
\gcl{1}\gbr\gcl{1}\gnl
\gmu\gmu\gnl
\gcn{1}{1}{2}{1}\gvac{1}\gcn{1}{1}{2}{-1}\gnl
\gmp{\un{s}}\gmp{\un{S}}\gnl
\gbr\gnl
\gob{1}{A}\gob{1}{B}
\gend
\hspace{1mm}.
\]
\end{proof}

Now we investigate the dual situation. A cross product bialgebra $A\#_\psi^\phi B$ is called
a smash cross coproduct bialgebra if $A\#^\phi B$ is a smash product coalgebra.

\begin{theorem}\thlabel{strsmashcrosscoprHa}
Let $A, B$ be algebras and coalgebras, and $\psi: B\ot A\ra A\ot B$ and 
$\phi: A\ot B\ra B\ot A$  morphisms in $\Cc$ such that $\phi$ is left normal. Then the following assertions are 
equivalent.\\
(i) $A\#_\psi^\phi B$ is a cross product bialgebra (and therefore a smash cross coproduct bialgebra,
by \leref{normalsmash}).\\
(ii) $B$ is a bialgebra, $A$ is a left $B$-comodule coalgebra, a left $B$-module coalgebra,  a right $B$-module 
and a right $B$-comodule and the following compatibility relations hold:
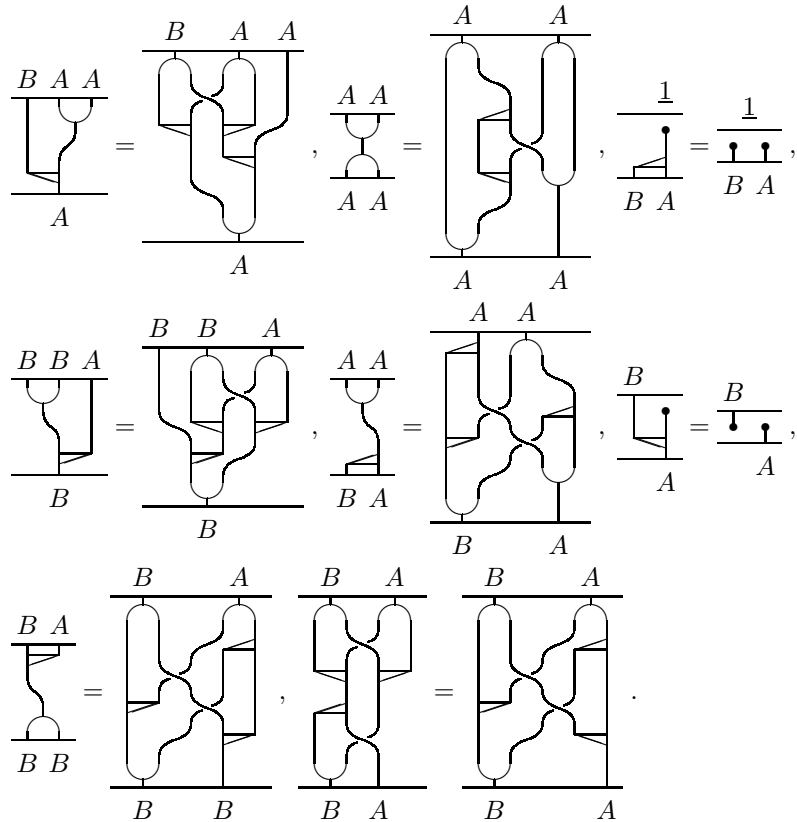
\begin{eqnarray*}
&&
\gbeg{3}{5}
\got{1}{B}\got{1}{A}\got{1}{A}\gnl
\gcl{1}\gmu\gnl
\gcl{1}\gcn{1}{1}{2}{1}\gnl
\glm\gnl
\gvac{1}\gob{1}{A}
\gend
=
\gbeg{5}{8}
\got{2}{B}\got{2}{A}\got{1}{A}\gnl
\gcmu\gcmu\gcl{1}\gnl
\gcl{1}\gbr\gcl{1}\gcl{1}\gnl
\glm\grm\gcn{1}{1}{1}{-1}\gnl
\gvac{1}\gcl{1}\glm\gnl
\gvac{1}\gcn{1}{1}{1}{3}\gvac{1}\gcl{1}\gnl
\gvac{2}\gmu\gnl
\gvac{2}\gob{2}{A}
\gend
\hspace{1mm},\hspace{1mm}
\gbeg{2}{4}
\got{1}{A}\got{1}{A}\gnl
\gmu\gnl
\gcmu\gnl
\gob{1}{A}\gob{1}{A}
\gend
=
\gbeg{5}{9}
\got{2}{A}\gvac{1}\got{2}{A}\gnl
\gcmu\gvac{1}\gcmu\gnl
\gcl{1}\gcn{1}{1}{1}{3}\gvac{1}\gcl{1}\gcl{1}\gnl
\gcl{1}\glcm\gcl{1}\gcl{1}\gnl
\gcl{1}\gcl{1}\gbr\gcl{1}\gnl
\gcl{1}\glm\gmu\gnl
\gcl{1}\gvac{1}\gcn{1}{1}{1}{-1}\gcn{1}{2}{2}{2}\gnl
\gmu\gnl
\gob{2}{A}\gvac{1}\gob{2}{A}
\gend 
\hspace{1mm},\hspace{1mm}
\gbeg{2}{4}
\gvac{1}\got{1}{\un{1}}\gnl
\gvac{1}\gu{1}\gnl
\glcm\gnl
\gob{1}{B}\gob{1}{A}
\gend
=
\gbeg{2}{3}
\got{2}{\un{1}}\gnl
\gu{1}\gu{1}\gnl
\gob{1}{B}\gob{1}{A}
\gend\hspace*{1mm},\\
&&
\gbeg{3}{5}
\got{1}{B}\got{1}{B}\got{1}{A}\gnl
\gmu\gcl{1}\gnl
\gcn{1}{1}{2}{3}\gvac{1}\gcl{1}\gnl
\gvac{1}\grm\gnl
\gvac{1}\gob{1}{B}
\gend
=
\gbeg{5}{7}
\got{1}{B}\got{2}{B}\got{2}{A}\gnl
\gcl{1}\gcmu\gcmu\gnl
\gcl{1}\gcl{1}\gbr\gcl{1}\gnl
\gcn{1}{1}{1}{3}\glm\grm\gnl
\gvac{1}\grm\gcn{1}{1}{1}{-1}\gnl
\gvac{1}\gmu\gnl
\gvac{1}\gob{2}{B}
\gend
\hspace{1mm},\hspace{1mm}
\gbeg{2}{5}
\got{1}{A}\got{1}{A}\gnl
\gmu\gnl
\gcn{1}{1}{2}{3}\gnl
\glcm\gnl
\gob{1}{B}\gob{1}{A}
\gend
=
\gbeg{5}{8}
\gvac{1}\got{1}{A}\got{2}{A}\gnl
\glcm\gcmu\gnl
\gcl{1}\gcl{1}\gcl{1}\gcn{1}{1}{1}{3}\gnl
\gcl{1}\gbr\glcm\gnl
\grm\gbr\gcl{1}\gnl
\gcl{1}\gvac{1}\gcn{1}{1}{1}{-1}\gmu\gnl
\gmu\gvac{1}\gcn{1}{1}{2}{2}\gnl
\gob{2}{B}\gvac{1}\gob{2}{A}
\gend
\hspace{1mm},\hspace{1mm}
\gbeg{2}{4}
\got{1}{B}\gnl
\gcl{1}\gu{1}\gnl
\glm\gnl
\gvac{1}\gob{1}{A}
\gend
=
\gbeg{2}{3}
\got{1}{B}\gnl
\gcu{1}\gu{1}\gnl
\gvac{1}\gob{1}{A}
\gend
\hspace{1mm},\\
&&
\gbeg{2}{5}
\got{1}{B}\got{1}{A}\gnl
\grm\gnl
\gcn{1}{1}{1}{2}\gnl
\gcmu\gnl
\gob{1}{B}\gob{1}{B}
\gend
=
\gbeg{5}{8}
\got{2}{B}\gvac{1}\got{2}{A}\gnl
\gcmu\gvac{1}\gcmu\gnl
\gcl{1}\gcl{1}\gcn{1}{1}{3}{1}\glcm\gnl
\gcl{1}\gbr\gcl{1}\gcl{1}\gnl
\grm\gbr\gcl{1}\gnl
\gcl{1}\gvac{1}\gcn{1}{1}{1}{-1}\grm\gnl
\gmu\gvac{1}\gcl{1}\gnl
\gob{2}{B}\gvac{1}\gob{1}{B}
\gend
\hspace{1mm},\hspace{1mm}
\gbeg{4}{8}
\got{2}{B}\got{2}{A}\gnl
\gcmu\gcmu\gnl
\gcl{1}\gbr\gcl{1}\gnl
\glm\grm\gnl
\glcm\gcl{1}\gnl
\gcl{1}\gbr\gnl
\gmu\gcl{1}\gnl
\gob{2}{B}\gob{1}{A}
\gend
=
\gbeg{5}{8}
\got{2}{B}\gvac{1}\got{2}{A}\gnl
\gcmu\gvac{1}\gcmu\gnl
\gcl{1}\gcl{1}\gcn{1}{1}{3}{1}\glcm\gnl
\gcl{1}\gbr\gcl{1}\gcl{1}\gnl
\grm\gbr\gcl{1}\gnl
\gcl{1}\gvac{1}\gcn{1}{1}{1}{-1}\glm\gnl
\gmu\gvac{2}\gcl{1}\gnl
\gob{2}{B}\gvac{2}\gob{1}{A}
\gend
\hspace{1mm}.
\end{eqnarray*}
If $B$ is a Hopf algebra with
antipode $\un{s}$ and $\Id_A$ has a convolution inverse $\un{S}$, then
$A\#_\psi^\phi B$ is a Hopf algebra with antipode
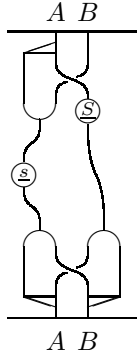
\[
\gbeg{4}{11}
\gvac{1}\got{1}{A}\got{1}{B}\gnl
\glcm\gcl{1}\gnl
\gcl{1}\gbr\gnl
\gmu\gmp{\un{S}}\gnl
\gcn{1}{1}{2}{1}\gvac{1}\gcn{1}{3}{1}{2}\gnl
\gmp{\un{s}}\gnl
\gcn{1}{1}{1}{2}\gnl
\gcmu\gcmu\gnl
\gcl{1}\gbr\gcl{1}\gnl
\glm\grm\gnl
\gvac{1}\gob{1}{A}\gob{1}{B}
\gend
\hspace{1mm}.
\]
makes $A\times_\psi^\phi B$ a Hopf algebra in $\Cc$.     
\end{theorem}

\begin{proof}
We omit the proof, as it is merely a dual version of the proof of \thref{strsmashcrossprodHa}.
Let us just mention that the left normality of $\phi$ implies that the conditions
(i-vi) and (vii.1) 
in \thref{crossprobialasactandcoact} are equivalent to the eight compatibility conditions
in the present Theorem.
\end{proof}

We invite the reader to state the right handed version of \thref{strsmashcrosscoprHa}. Combining
the left and right handed versions of \thref{strsmashcrosscoprHa}, we can characterize cross
product bialgebras having the property that $\phi$ is left and right normal.

\begin{corollary}\colabel{6.7}
Let $A\# _\psi^\phi B$ be a cross product bialgebra such that $\phi$ is left and 
right normal. Then $(A, B)$ is a right-left matched pair and 
$A\#_\psi^\phi B =A\Join B$, the double cross product bialgebra associated to $(A, B)$. 
If $A$ and $B$ are Hopf algebras, then $A\Join B$ is also a Hopf algebra, with antipode   
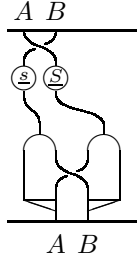
\[
\gbeg{4}{8}
\got{1}{A}\got{1}{B}\gnl
\gbr\gnl
\gmp{\un{s}}\gmp{\un{S}}\gnl
\gcn{1}{1}{1}{2}\gcn{1}{1}{1}{4}\gnl
\gcmu\gcmu\gnl
\gcl{1}\gbr\gcl{1}\gnl
\glm\grm\gnl
\gvac{1}\gob{1}{A}\gob{1}{B}
\gend
\hspace{1mm}.
\]
\end{corollary} 

\begin{proof}
It can be easily seen from \equuref{neccconds}{c} that $\phi$ is left and right normal if and only
if it is equal to the braiding of $A$ and $B$. The rest of the proof is then similar to the proof
of \coref{doublecrosscoprodbialg}. We obtain relations that tell us that $(A, B)$ is a 
right-left matched pair. Moreover, $A\#^\phi B$ is the tensor product coalgebra, and
$A\#_\psi^\phi B$ is a double cross product bialgebra.
\end{proof}

We refer to \cite{kas,tak} for detail on the bicross product of two groups.

\begin{corollary}\colabel{6.8}
A cross product Hopf algebra in the category of sets is a bicross product of two groups.
\end{corollary}

\begin{proof}
It is well-known that an algebra in $\un{\rm Sets}$ is a monoid, and that 
any set $X$ has a unique coalgebra structure given by the comultiplication 
$\un{\Delta}_X(x)=(x, x)$, for all $x\in X$, and the counit $\un{\va}_X=*$, where the singleton 
$\{*\}$ is the unit object of the monoidal category $\un{\rm Sets}$. In this way any monoid $M$ is a bialgebra 
in $\un{\rm Sets}$ and it is, moreover, a Hopf algebra if and only if $M$ is a group. Consequently, 
the only cross coproduct in $\un{\rm Sets}$ is the tensor product coalgebra, and the statement
then follows from \coref{6.7}.  
\end{proof}

\section{The structure of a Hopf algebra with an appropriate projection}\selabel{strHopfalgwithprof}
\setcounter{equation}{0}
As we have already mentioned several times, cross product bialgebras can be characterized using
injections and projections. We now recall this classical result, see \cite[Prop. 2.2]{bespdrab1}, \cite[Theorem 4.3]{cimz}, 
with a sketch of proof.

\begin{proposition}\prlabel{DrBespExtVers}
For  a bialgebra $H$, the following statements are equivalent:\\
(i) $H$ is isomorphic to a cross product bialgebra;\\
(ii) There exist algebras and coalgebras $A, B$ and 
morphisms $i:\ B\to H$, $\pi:\ H\to B$, $j:\ A\to H$, $p:\ H\to A$
such that 
\begin{itemize}
\item $i, j$ are algebra morphisms, $p, \pi$ are coalgebra morphisms 
and $pj=\Id_A$ and $\pi i=\Id_B$;
\item $\zeta= \un{m}_H(j\ot i): A\ot B\ra H$ is an isomorphism in $\Cc$ with inverse 
$\zeta^{-1}=(p\ot \pi)\un{\Delta}_H: H\ra A\ot B$.  
\end{itemize}
\end{proposition}

\begin{proof}
For the complete proof, we refer to \cite{bespdrab1}. For later reference, we give a brief
sketch of the proof of $(ii)\Rightarrow (i)$. $\psi$ and $\phi$ are defined by the formulas
\begin{equation}\eqlabel{psiphifromexts}
\psi=
\gbeg{2}{6}
\got{1}{B}\got{1}{A}\gnl
\gmp{i}\gmp{j}\gnl
\gmu\gnl
\gcmu\gnl
\gmp{p}\gmp{\pi}\gnl
\gob{1}{A}\gob{1}{B}
\gend
\hspace{2mm}{\rm and}\hspace{2mm}
\phi=
\gbeg{2}{6}
\got{1}{A}\got{1}{B}\gnl
\gmp{j}\gmp{i}\gnl
\gmu\gnl
\gcmu\gnl
\gmp{\pi}\gmp{p}\gnl
\gob{1}{B}\gob{1}{A}
\gend
\hspace{1mm}.
\end{equation}
Then we show that $A\#_\psi^\phi B$ is a cross product bialgebra, and that $\zeta$ is an isomorphism
of bialgebras.
\end{proof}

In \prref{DrBespExtVers}, we need two data, namely $(A,p,j)$ and $(B,\pi,i)$. We will see that
one of the two data can be recovered from the other one if some additional conditions are
satisfied.

\begin{lemma}\lelabel{scpbextprop}
Let $H=A\#_\psi^\phi B$ be a (left) smash cross product bialgebra and 
$\pi=
\gbeg{2}{3}
\got{2}{H}\gnl
\gcu{1}\gcl{1}\gnl
\gvac{1}\gob{1}{B}
\gend
$ 
and  
$
i=
\gbeg{2}{3}
\gvac{1}\got{1}{B}\gnl
\gu{1}\gcl{1}\gnl
\gob{2}{H}
\gend
$ 
the canonical morphisms. Then the following assertions hold. 

(i) $\pi$ is a bialgebra morphism, $i$ is an algebra morphism, $\pi i=\Id_B$ and 
\begin{equation}\eqlabel{piiscpb}
(a)\hspace{1mm}
\gbeg{2}{6}
\got{1}{B}\gnl
\gmp{i}\gnl
\gcn{1}{1}{1}{2}\gnl
\gcmu\gnl
\gcl{1}\gmp{\pi}\gnl
\gob{1}{H}\gob{1}{B}
\gend
=
\gbeg{2}{4}
\got{2}{B}\gnl
\gcmu\gnl
\gmp{i}\gcl{1}\gnl
\gob{1}{H}\gob{1}{B}\gnl
\gend
\hspace{2mm}{\rm and}\hspace{2mm}(b)\hspace{1mm}
\gbeg{3}{7}
\gvac{1}\got{2}{B}\gnl
\gvac{1}\gcmu\gnl
\gvac{1}\gmp{i}\gcl{1}\gnl
\gvac{1}\gcn{1}{1}{1}{0}\gmp{\un{s}}\gnl
\gcmu\gmp{i}\gnl
\gcl{1}\gmu\gnl
\gob{1}{H}\gob{2}{H}
\gend
=
\gbeg{3}{11}
\got{1}{B}\gnl
\gmp{i}\gnl
\gcn{1}{1}{1}{2}\gnl
\gcmu\gnl
\gcl{1}\gcn{1}{1}{1}{2}\gnl
\gmp{\pi}\gcmu\gnl
\gmp{i}\gcl{1}\gmp{\pi}\gnl
\gcl{1}\gcl{1}\gmp{\un{s}}\gnl
\gcl{1}\gcl{1}\gmp{i}\gnl
\gcl{1}\gmu\gnl
\gob{1}{H}\gob{2}{H}
\gend
\hspace{1mm}.
\end{equation}
For \equuref{piiscpb}{b}, we need the additional assumption that $B$
is a Hopf algebra, with antipode $\un{s}$.\\
(ii) If $H$ is a Hopf algebra with antipode 
$\un{\mathcal{S}}=
\gbeg{2}{5}
\got{1}{A}\got{1}{B}\gnl
\gcl{1}\gcl{1}\gnl
\gsbox{2}\gnl
\gcl{1}\gcl{1}\gnl
\got{1}{A}\got{1}{B}
\gend
$ 
then $B$ is a Hopf algebra with antipode $\un{s}$ defined in \equuref{antscpb}{a},
satisfying \equuref{antscpb}{b}
\begin{equation}\eqlabel{antscpb}
(a)\hspace*{1cm}
\un{s}=
\gbeg{2}{5}
\gvac{1}\got{1}{B}\gnl
\gu{1}\gcl{1}\gnl
\gsbox{2}\gnl
\gcu{1}\gcl{1}\gnl
\gvac{1}\gob{1}{B}
\gend
~~;\hspace*{1cm}(b)\hspace*{1cm}
\gbeg{2}{8}
\got{1}{B}\gnl
\gmp{i}\gnl
\gcn{1}{1}{1}{2}\gnl
\gcmu\gnl
\gmp{\pi}\gmp{\un{\mathcal{S}}}\gnl
\gmp{i}\gcl{1}\gnl
\gmu\gnl
\gob{2}{H}
\gend
=
\gbeg{1}{4}
\got{1}{B}\gnl
\gcu{1}\gnl
\gu{1}\gnl
\gob{1}{B}
\gend
\hspace{1mm}.
\end{equation}
\end{lemma}

\begin{proof}
The proof of (i) is straightforward, and is left to the reader. Observe that the conormality of
$\psi$ is needed in order to show that $\pi$ is a bialgebra morphism, but is not needed in the proof of
\equref{piiscpb}. \equuref{piiscpb}{a} tells us that $i:\ B\to H$ is right $B$-colinear.
Here $H\in \Cc^B$ via $\pi\circ \un{\Delta}_H$ and $B\in \Cc^B$ via $\un{\Delta}_B$.\\

We will only prove that the morphism $\un{s}$ as defined in \equref{antscpb} is antipode for $B$.
We have seen in \reref{antpartneccconds} that $\Id_B$ 
has always a left convolution inverse.
We prove that it also has a right inverse.
Compose \equuref{defantcpHa}{a} to the left with $\un{\va}_A\ot \Id_B$ and to the right with 
$\Id_A\ot \un{\eta}_B$. Using the left conormality of $\psi$ we obtain that 
$
\gbeg{2}{5}
\got{1}{A}\gnl
\gcl{1}\gu{1}\gnl
\gsbox{2}\gnl
\gcu{1}\gcl{1}\gnl
\gvac{1}\gob{1}{B}
\gend
=
\gbeg{1}{4}
\got{1}{A}\gnl
\gcu{1}\gnl
\gu{1}\gnl
\gob{1}{B}
\gend
$. 
Now compose \equuref{defantcpHa}{b} to the left with $\un{\va}_A\ot \Id_B$ and to the right
with $\un{\eta}_A\ot \Id_B$. Again using the left conormality of $\psi$, we now find that
\[
\gbeg{1}{4}
\got{1}{B}\gnl
\gcu{1}\gnl
\gu{1}\gnl
\gob{1}{B}
\gend
=
\gbeg{3}{8}
\gvac{1}\got{2}{B}\gnl
\gu{1}\gcmu\gnl
\gbrbox\gvac{2}\gcl{1}\gnl
\gcl{1}\gsbox{2}\gnl
\gcl{1}\gcu{1}\gcl{1}\gnl
\gcl{1}\gvac{1}\gcn{1}{1}{1}{-1}\gnl
\gmu\gnl
\gob{2}{B}
\gend
=
\gbeg{5}{12}
\gvac{1}\got{2}{B}\gnl
\gu{1}\gcmu\gnl
\gbrbox\gvac{2}\gcn{1}{1}{1}{5}\gnl
\gcl{1}\gcl{1}\gu{1}\gu{1}\gcl{1}\gnl
\gcl{1}\gcl{1}\gbrc\gcl{1}\gnl
\gcl{1}\gmu\gmu\gnl
\gcl{1}\gcn{1}{1}{2}{1}\gvac{1}\gcn{1}{1}{2}{-1}\gnl
\gcl{1}\gsbox{2}\gnl
\gcl{1}\gcu{1}\gcl{1}\gnl
\gcl{1}\gcn{1}{1}{3}{1}\gnl
\gmu\gnl
\gob{2}{B}
\gend
\equal{\equref{antiac}}
\gbeg{6}{12}
\gvac{2}\got{2}{B}\gnl
\gvac{1}\gu{1}\gcmu\gnl
\gvac{1}\gbrbox\gvac{2}\gcl{1}\gnl
\gvac{1}\gcn{1}{1}{1}{-1}\gbr\gnl
\gcl{1}\gu{1}\gcl{1}\gvac{1}\gcn{1}{1}{-1}{1}\gu{1}\gnl
\gcl{1}\gsbox{2}\gvac{3}\gsbox{2}\gnl
\gcl{1}\gcu{1}\gcl{1}\gvac{1}\gcu{1}\gcl{1}\gnl
\gcn{1}{3}{1}{3}\gvac{1}\gcl{1}\gvac{1}\gcn{1}{1}{3}{-1}\gnl
\gvac{2}\gmu\gnl
\gvac{2}\gcn{1}{1}{2}{1}\gnl
\gvac{1}\gmu\gnl
\gvac{1}\gob{2}{B}
\gend
=
\gbeg{3}{9}
\got{2}{B}\gnl
\gcmu\gnl
\gcl{1}\gcn{1}{1}{1}{3}\gnl
\gcl{1}\gu{1}\gcl{1}\gnl
\gcl{1}\gsbox{2}\gnl
\gcl{1}\gcu{1}\gcl{1}\gnl
\gcl{1}\gvac{1}\gcn{1}{1}{1}{-1}\gnl
\gmu\gnl
\gob{2}{B}
\gend
\hspace{2mm}.
\] 
This shows that $\un{s}$, as defined in \equuref{antscpb}{a}, is a right inverse for
$\Id_B$ in ${\rm Hom}(B, B)$.
\end{proof}

In \thref{strofHopfwithcertproj} we will show that \leref{scpbextprop} has a converse,
at least if some additional technical
assumptions are satisfied. In the sequel, we assume that 
$B$ is a Hopf algebra, $H$ is a bialgebra and    
$
\xymatrix{
B \ar[r]<2pt>^i &\ar[l]<2pt>^{\pi} H
}
$ 
are morphisms in $\Cc$ such that $\pi$ is a bialgebra morphism, $i$ is an algebra morphism, $\pi i=\Id_B$ and
\equuref{piiscpb}{a,b} hold. At some places, we will consider the situation where $H$ is also a Hopf algebra, and then
we will assume that \equuref{antscpb}{b} holds as well. 
In addition, we assume that $(\Id_H\ot \pi)\un{\Delta}_H,~\Id_H\ot \un{\eta}_B:~H\to H\ot B$
have a equalizer in $\Cc$. This means that there exists $A\in \Cc$ and $j:\ A\to H$ such that
\begin{equation}\eqlabel{defofA}
\gbeg{2}{6}
\got{1}{A}\gnl
\gmp{j}\gnl
\gcn{1}{1}{1}{2}\gnl
\gcmu\gnl
\gcl{1}\gmp{\pi}\gnl
\gob{1}{H}\gob{1}{B}
\gend
=
\gbeg{2}{3}
\got{1}{A}\gnl
\gmp{j}\gu{1}\gnl
\gob{1}{H}\gob{1}{B}
\gend
\hspace{1mm}.
\end{equation}
$(A,j)$ is universal in the following sense: if $f: X\ra H$ is such that 
$(\Id_H\ot \pi)\un{\Delta}_Hf=(\Id_H\ot \un{\eta}_B)f$, there is a unique morphism $\tilde{f}: X\ra A$ such that 
$j\tilde{f}=f$. Under these assumptions, we can show that $H$ is (isomorphic to) a smash cross
product bialgebra. The proof of \thref{strofHopfwithcertproj} consists of several steps, and we have
divided them over the subsequent Lemmas.

\begin{lemma}\lelabel{algstrA}
Let $B, H, \pi, i$ be as above. Then $A$ has an algebra structure such that $j: A\ra H$ is an 
algebra morphism. 
\end{lemma}

\begin{proof}
Applying the universal property of the equalizer $(A,j)$, we find unique morphisms
$\un{m}_A: A\ot A\ra A$ and $\un{\eta}: \un{1}\ra A$ morphisms in $\Cc$ making the
diagrams
\begin{equation}\eqlabel{algstrA}
\xymatrix{
A\ar[r]^j &H\ar[rr]<2pt>^{(\Id_H\ot \pi)\un{\Delta}_H} 
   \ar[rr]<-2pt>_{\Id_H\ot \un{\eta}_B} &&H\ot B \\
   & A\ot A \ar[u]_-{\un{m}_H(j\ot j)} \ar@{.>}[ul]^-{\un{m}_A} && 
}
~~{\rm and}~~
\xymatrix{
A\ar[r]^j &H\ar[rr]<2pt>^{(\Id_H\ot \pi)\un{\Delta}_H} 
   \ar[rr]<-2pt>_{\Id_H\ot \un{\eta}_B} && H\ot B \\
   & \un{1} \ar[u]_-{\un{\eta}_H} \ar@{.>}[ul]^-{\un{\eta}_A} &&
}
~.
\end{equation}
commutative, which means that $j\un{m}_A=\un{m}_H(j\ot j)$ and $j\un{\eta}_A=\un{\eta}_H$. Furthermore, a simple inspection shows that 
$j\un{m}_A(\un{m}_A\ot \Id_A)=j\un{m}_A(\Id_A\ot \un{m}_A)$ and $j\un{m}_A(\Id_A\ot \un{\eta}_A)=j=j(\un{\eta}_A\ot \Id_A)$. 
Since $j$ is a monomorphism in $\Cc$ we deduce that $\un{m}_A$ is associative and $\un{\eta}_A$ has the unit 
property. This shows that
$A$ is an algebra and $j$ is an algebra morphism.  
\end{proof}

The next step is more complicated, and consists in proving that $A$ also has a coalgebra structure. 
We will need an extra assumption, namely that $A\in \Cc$ is flat: $-\ot A$ and $A\ot-$ preserve
equalizers. Actually, we need that $\Id_A\ot j$ and 
$j\ot \Id_A$ are monomorphisms, in order to obtain that $j\ot j$ is a monomorphism.

\begin{lemma}\lelabel{coalgstrA}
Let $A, B, H, \pi, i, j$ be as above, and consider  $\tilde{p}=\un{m}_H(\Id_H\ot i\un{s}\pi)\un{\Delta}_H:\ H\to H$.
Then there exists a morphism $p: H\ra A$ such that $jp=\tilde{p}$ and $pj=\Id_A$. 
Furthermore, $A$ is a coalgebra, and $p$ is a coalgebra morphism.  
\end{lemma}

\begin{proof}
We compute that
\[
\gbeg{2}{9}
\got{2}{H}\gnl
\gcmu\gnl
\gcl{1}\gmp{\pi}\gnl
\gcl{1}\gmp{\un{s}}\gnl
\gcl{1}\gmp{i}\gnl
\gmu\gnl
\gcmu\gnl
\gcl{1}\gmp{\pi}\gnl
\gob{1}{H}\gob{1}{B}
\gend
\equal{\equref{braidedbialgebra}}
\gbeg{3}{13}
\got{2}{H}\gnl
\gcmu\gnl
\gcl{1}\gcn{1}{1}{1}{3}\gnl
\gcl{1}\gvac{1}\gmp{\pi}\gnl
\gcl{1}\gvac{1}\gmp{\un{s}}\gnl
\gcn{1}{1}{1}{2}\gvac{1}\gmp{i}\gnl
\gcmu\gcn{1}{1}{1}{2}\gnl
\gcl{1}\gcl{1}\gcmu\gnl
\gcl{1}\gbr\gcl{1}\gnl
\gmu\gmu\gnl
\gcn{1}{2}{2}{2}\gvac{1}\gcn{1}{1}{2}{1}\gnl
\gvac{2}\gmp{\pi}\gnl
\gob{2}{H}\gob{1}{B}
\gend
\equal{\equuref{piiscpb}{a}}
\gbeg{4}{11}
\got{2}{H}\gnl
\gcmu\gnl
\gcl{1}\gcn{1}{1}{1}{3}\gnl
\gcl{1}\gvac{1}\gmp{\pi}\gnl
\gcn{1}{1}{1}{2}\gvac{1}\gmp{\un{s}}\gnl
\gcmu\gcn{1}{1}{1}{2}\gnl
\gcl{1}\gcl{1}\gcmu\gnl
\gcl{1}\gbr\gcl{1}\gnl
\gcl{1}\gmp{i}\gmp{\pi}\gcl{1}\gnl
\gmu\gmu\gnl
\gob{2}{H}\gob{2}{B}
\gend
\equal{\equref{antiac}}
\gbeg{4}{13}
\got{2}{H}\gnl
\gcmu\gnl
\gcl{1}\gcn{1}{1}{1}{2}\gnl
\gcl{1}\gcmu\gnl
\gcl{1}\gmp{\pi}\gmp{\pi}\gnl
\gcl{1}\gcl{1}\gcn{1}{1}{1}{2}\gnl
\gcl{1}\gcl{1}\gcmu\gnl
\gcl{1}\gcl{1}\gmp{\un{s}}\gmp{\un{s}}\gnl
\gcl{1}\gcl{1}\gbr\gnl
\gcl{1}\gbr\gcl{1}\gnl
\gcl{1}\gmp{i}\gcl{1}\gcl{1}\gnl
\gmu\gmu\gnl
\gob{2}{H}\gob{2}{B}
\gend
\equal{\equref{nat1cup}}
\gbeg{4}{14}
\gvac{1}\got{2}{H}\gnl
\gvac{1}\gcmu\gnl
\gvac{1}\gcn{1}{1}{1}{-1}\gmp{\pi}\gnl
\gcl{1}\gvac{1}\gcn{1}{1}{1}{2}\gnl
\gcl{1}\gvac{1}\gcmu\gnl
\gcl{1}\gvac{1}\gcn{1}{1}{1}{0}\gcl{1}\gnl
\gcl{1}\gcmu\gmp{\un{s}}\gnl
\gcl{1}\gcl{1}\gmp{\un{s}}\gcl{1}\gnl
\gcl{1}\gmu\gcl{1}\gnl
\gcl{1}\gcn{1}{1}{2}{3}\gvac{1}\gcl{1}\gnl
\gcn{1}{1}{1}{3}\gvac{1}\gbr\gnl
\gvac{1}\gcl{1}\gmp{i}\gcl{1}\gnl
\gvac{1}\gmu\gcl{1}\gnl
\gvac{1}\gob{2}{H}\gob{1}{B}
\gend
\equal{\equref{braidedantipode}}
\gbeg{3}{7}
\got{2}{H}\gnl
\gcmu\gnl
\gcl{1}\gmp{\pi}\gnl
\gcl{1}\gmp{\un{s}}\gnl
\gcl{1}\gmp{i}\gnl
\gmu\gu{1}\gnl
\gob{2}{H}\gob{1}{B}
\gend
\hspace{1mm},
\]
It follows from the universal property of the equalizer $(A,j)$ that there exists a unique morphism
$p:\ H\to A$ such that $jp=\tilde{p}$. Then $jpj=\tilde{p}j$. From \equref{defofA}, we deduce
that $\tilde{p}j=j$, so $jpj=j$, hence $pj=\Id_A$.\\
Now we construct the coalgebra structure on $A$. We claim that $(A,p)$ is the coequalizer
of $\un{m}_H(\Id_H\ot i),~\Id_H\ot \un{\va}_B:~H\ot B\to B$. First of all, we have that
\[
jp(\un{m}_H(\Id_H\ot i))=\tilde{p}(\un{m}_H(\Id_H\ot i))=
\tilde{p}(\Id_H\ot \un{\va}_B)=jp(\Id_H\ot \un{\va}_B),
\]
since
\begin{equation}\eqlabel{p1oftildep}
\gbeg{2}{9}
\got{1}{H}\got{1}{B}\gnl
\gcl{1}\gmp{i}\gnl
\gmu\gnl
\gcmu\gnl
\gcl{1}\gmp{\pi}\gnl
\gcl{1}\gmp{\un{s}}\gnl
\gcl{1}\gmp{i}\gnl
\gmu\gnl
\gob{2}{H}
\gend
~\equalupdown{\equref{braidedbialgebra}}{\equuref{piiscpb}{a}}~
\gbeg{4}{11}
\got{2}{H}\got{2}{B}\gnl
\gcmu\gcmu\gnl
\gcl{1}\gcl{1}\gmp{i}\gcl{1}\gnl
\gcl{1}\gbr\gcl{1}\gnl
\gmu\gmp{\pi}\gcl{1}\gnl
\gcn{1}{1}{2}{3}\gvac{1}\gmu\gnl
\gvac{1}\gcl{1}\gcn{1}{1}{2}{1}\gnl
\gvac{1}\gcl{1}\gmp{\un{s}}\gnl
\gvac{1}\gcl{1}\gmp{i}\gnl
\gvac{1}\gmu\gnl
\gvac{1}\gob{2}{H}
\gend
~\equalupdown{\equref{antiac}}{\equref{nat2cup}}~
\gbeg{4}{11}
\gvac{1}\got{2}{H}\got{1}{B}\gnl
\gvac{1}\gcmu\gcl{1}\gnl
\gvac{1}\gcn{1}{1}{1}{-1}\gbr\gnl
\gcl{1}\gvac{1}\gcn{1}{1}{1}{0}\gmp{\pi}\gnl
\gcl{1}\gcmu\gmp{\un{s}}\gnl
\gcl{1}\gmp{i}\gmp{\un{s}}\gcl{1}\gnl
\gmu\gmu\gnl
\gcn{1}{1}{2}{3}\gvac{1}\gcn{1}{1}{2}{1}\gnl
\gvac{1}\gcl{1}\gmp{i}\gnl
\gvac{1}\gmu\gnl
\gvac{1}\gob{2}{H}
\gend
=
\gbeg{4}{13}
\gvac{1}\got{2}{H}\got{1}{B}\gnl
\gvac{1}\gcmu\gcl{1}\gnl
\gvac{1}\gcn{1}{1}{1}{-1}\gbr\gnl
\gcl{1}\gvac{1}\gcn{1}{1}{1}{0}\gmp{\pi}\gnl
\gcl{1}\gcmu\gmp{\un{s}}\gnl
\gcl{1}\gcl{1}\gmp{\un{s}}\gmp{i}\gnl
\gcl{1}\gmu\gcl{1}\gnl
\gcl{1}\gcn{1}{1}{2}{3}\gvac{1}\gcl{1}\gnl
\gcn{1}{3}{1}{3}\gvac{1}\gmp{i}\gcl{1}\gnl
\gvac{2}\gmu\gnl
\gvac{2}\gcn{1}{1}{2}{1}\gnl
\gvac{1}\gmu\gnl
\gvac{1}\gob{2}{H}
\gend
~\equal{\equref{braidedantipode}}~
\gbeg{3}{7}
\got{2}{H}\got{1}{B}\gnl
\gcmu\gcu{1}\gnl
\gcl{1}\gmp{\pi}\gnl
\gcl{1}\gmp{\un{s}}\gnl
\gcl{1}\gmp{i}\gnl
\gmu\gnl
\gob{2}{H}
\gend
\hspace{1mm}.
\end{equation}
Secondly, we need to prove the universal property. Assume that $\tilde{f}: H\ra X$
is such that $\tilde{f}\un{m}_H(\Id_H\ot i)=\tilde{f}(\Id_H\ot \un{\va}_B)$. We have to
show that there is a unique $f:\ A\to X$ such that $fp=\tilde{f}$. If $f$ exists, then
$f=f\Id_A=fpj=\tilde{f}j$, hence $f$ is unique. To prove the existence, let $f=\tilde{f}j$,
then
\[
fp=\tilde{f}jp=\tilde{f}\tilde{p}=\tilde{f}\un{m}_H(\Id_H\ot i)
(\Id_H\ot \un{s}\pi)\un{\Delta}_H=\tilde{f}(\Id_H\ot \un{\va}_B\un{s}\pi)\un{\Delta}_H=\tilde{f},
\] 
as required.\\
Now we use the universal property of the coequalizer to construct the comultiplication on
$A$. First observe that
\begin{equation}\eqlabel{p2oftildep}
\tilde{p}i=
\gbeg{2}{9}
\got{1}{B}\gnl
\gmp{i}\gnl
\gcn{1}{1}{1}{2}\gnl
\gcmu\gnl
\gcl{1}\gmp{\pi}\gnl
\gcl{1}\gmp{\un{s}}\gnl
\gcl{1}\gmp{i}\gnl
\gmu\gnl
\gob{2}{H}
\gend
~\equal{\equuref{piiscpb}{a}}~
\gbeg{2}{6}
\got{2}{B}\gnl
\gcmu\gnl
\gmp{i}\gmp{\un{s}}\gnl
\gcl{1}\gmp{i}\gnl
\gmu\gnl
\gob{2}{H}
\gend
=
\gbeg{2}{7}
\got{2}{B}\gnl
\gcmu\gnl
\gcl{1}\gmp{\un{s}}\gnl
\gmu\gnl
\gcn{1}{1}{2}{1}\gnl
\gmp{i}\gnl
\gob{1}{H}
\gend
~\equal{\equref{braidedantipode}}~
\gbeg{1}{4}
\got{1}{B}\gnl
\gcu{1}\gnl
\gu{1}\gnl
\gob{1}{H}
\gend
\hspace{1mm}.
\end{equation}
$\tilde{f}=(p\ot p)\un{\Delta}_H: H\ra A\ot A$ satisfies the equality
\begin{eqnarray*}
&&\hspace*{-2cm}
(j\ot j)\tilde{f}\un{m}_H(\Id_H\ot i)=(\tilde{p}\ot \tilde{p})\un{\Delta}_H\un{m}_H(\Id_H\ot i)\\
&\equal{\equref{braidedbialgebra}}&
\gbeg{5}{16}
\got{2}{H}\gvac{1}\got{1}{B}\gnl
\gcmu\gvac{1}\gmp{i}\gnl
\gcl{1}\gcl{1}\gvac{1}\gcn{1}{1}{1}{2}\gnl
\gcl{1}\gcl{1}\gvac{1}\gcmu\gnl
\gcl{1}\gcl{1}\gvac{1}\gcn{1}{1}{1}{-1}\gcn{1}{3}{1}{0}\gnl
\gcl{1}\gbr\gnl
\gmu\gcn{1}{1}{1}{0}\gnl
\gcn{1}{1}{2}{1}\gcmu\gcmu\gnl
\gmp{\tilde{p}}\gcl{1}\gbr\gcl{1}\gnl
\gcl{1}\gmu\gmu\gnl
\gcl{1}\gcn{1}{1}{2}{3}\gvac{1}\gcn{1}{1}{2}{1}\gnl
\gcl{1}\gvac{1}\gcl{1}\gmp{\pi}\gnl
\gcl{1}\gvac{1}\gcl{1}\gmp{\un{s}}\gnl
\gcl{1}\gvac{1}\gcl{1}\gmp{i}\gnl
\gcl{1}\gvac{1}\gmu\gnl
\gob{1}{H}\gvac{1}\gob{2}{H}
\gend
~\equal{\equref{nat2cup}\times 2}
\gbeg{6}{12}
\gvac{1}\got{2}{H}\gvac{1}\got{1}{B}\gnl
\gvac{1}\gcmu\gvac{1}\gmp{i}\gnl
\gvac{1}\gcn{1}{1}{1}{0}\gcn{1}{1}{1}{3}\gvac{1}\gcn{1}{1}{1}{2}\gnl
\gcmu\gvac{1}\gcl{1}\gcmu\gnl
\gcl{1}\gcl{1}\gvac{1}\gbr\gcl{1}\gnl
\gcl{1}\gcl{1}\gvac{1}\gcn{1}{1}{1}{0}\gmu\gnl
\gcl{1}\gcl{1}\gcmu\gcn{1}{1}{2}{1}\gnl
\gcl{1}\gbr\gcl{1}\gmp{\pi}\gnl
\gmu\gmu\gmp{\un{s}}\gnl
\gcn{1}{1}{2}{3}\gvac{1}\gcn{1}{1}{2}{3}\gvac{1}\gmp{i}\gnl
\gvac{1}\gmp{\tilde{p}}\gvac{1}\gmu\gnl
\gvac{1}\gob{1}{H}\gvac{1}\gob{2}{H}
\gend
~\equalupdown{\equuref{piiscpb}{a}}{\equref{antiac}}~
\gbeg{6}{12}
\gvac{1}\got{2}{H}\gvac{1}\got{2}{B}\gnl
\gvac{1}\gcmu\gvac{1}\gcmu\gnl
\gvac{1}\gcn{1}{1}{1}{0}\gcn{1}{1}{1}{3}\gvac{1}\gcl{1}\gcl{1}\gnl
\gcmu\gvac{1}\gbr\gcl{1}\gnl
\gcl{1}\gcl{1}\gvac{1}\gmp{i}\gbr\gnl
\gcl{1}\gcl{1}\gvac{1}\gcn{1}{1}{1}{0}\gmp{\un{s}}\gmp{\pi}\gnl
\gcl{1}\gcl{1}\gcmu\gmp{i}\gmp{\un{s}}\gnl
\gcl{1}\gbr\gcl{1}\gcl{1}\gmp{i}\gnl
\gmu\gmu\gmu\gnl
\gcn{1}{1}{2}{1}\gvac{1}\gcn{1}{1}{2}{3}\gvac{1}\gcn{1}{1}{2}{1}\gnl
\gmp{\tilde{p}}\gvac{2}\gmu\gnl
\gob{1}{H}\gvac{2}\gob{2}{H}
\gend\\
&
~\equalupdown{\equref{nat2cup}}{\equuref{piiscpb}{b}}~&
\gbeg{6}{18}
\gvac{1}\got{2}{H}\gvac{1}\got{1}{B}\gnl
\gvac{1}\gcmu\gvac{1}\gcl{1}\gnl
\gvac{1}\gcn{1}{1}{1}{0}\gcn{1}{1}{1}{3}\gvac{1}\gcl{1}\gnl
\gcmu\gvac{1}\gbr\gnl
\gcl{1}\gcl{1}\gvac{1}\gmp{i}\gmp{\pi}\gnl
\gcl{1}\gcl{1}\gvac{1}\gcn{1}{1}{1}{0}\gmp{\un{s}}\gnl
\gcl{1}\gcl{1}\gcmu\gmp{i}\gnl
\gcl{1}\gcl{1}\gmp{\pi}\gcn{1}{1}{1}{2}\gcn{1}{1}{1}{3}\gnl
\gcl{1}\gbr\gcmu\gcl{1}\gnl
\gcl{1}\gmp{i}\gcl{1}\gcl{1}\gmp{\pi}\gcl{1}\gnl
\gmu\gcl{1}\gcl{1}\gmp{\un{s}}\gcl{1}\gnl
\gcn{1}{1}{2}{1}\gvac{1}\gcl{1}\gcl{1}\gmp{i}\gcl{1}\gnl
\gmp{\tilde{p}}\gvac{1}\gcl{1}\gmu\gcl{1}\gnl
\gcl{1}\gvac{1}\gcn{1}{1}{1}{3}\gcn{1}{1}{2}{3}\gvac{1}\gcl{1}\gnl
\gcl{1}\gvac{2}\gmu\gcl{1}\gnl
\gcl{1}\gvac{2}\gcn{1}{1}{2}{3}\gvac{1}\gcl{1}\gnl
\gcl{1}\gvac{3}\gmu\gnl
\gob{1}{H}\gvac{3}\gob{2}{H}
\gend
~\equal{\equref{p1oftildep}}~
\gbeg{4}{9}
\gvac{1}\got{2}{H}\got{1}{B}\gnl
\gvac{1}\gcmu\gcl{1}\gnl
\gvac{1}\gcn{1}{1}{1}{0}\gbr\gnl
\gcmu\gmp{i}\gmp{\pi}\gnl
\gcl{1}\gcl{1}\gmp{\tilde{p}}\gmp{\un{s}}\gnl
\gcl{1}\gmu\gmp{i}\gnl
\gmp{\tilde{p}}\gcn{1}{1}{2}{3}\gvac{1}\gcl{1}\gnl
\gcl{1}\gvac{1}\gmu\gnl
\gob{1}{H}\gvac{1}\gob{2}{H}
\gend
~\equal{\equref{p2oftildep}}~
\gbeg{4}{9}
\got{2}{H}\gvac{1}\got{1}{B}\gnl
\gcmu\gvac{1}\gcu{1}\gnl
\gcl{1}\gcn{1}{1}{1}{2}\gnl
\gcl{1}\gcmu\gnl
\gcl{1}\gcl{1}\gmp{\pi}\gnl
\gmp{\tilde{p}}\gcl{1}\gmp{\un{s}}\gnl
\gcl{1}\gcl{1}\gmp{i}\gnl
\gcl{1}\gmu\gnl
\gob{1}{H}\gob{2}{H}
\gend~,
\end{eqnarray*}
where we freely used associativity and coassociativity of the multiplications and comultiplications
that are involved, and the fact the $\pi$ is a bialgebra morphism and $i$ is an algebra morphism.
So we have shown that
\[
(j\ot j)\tilde{f}\un{m}_H(\Id_H\ot i)
=(\tilde{p}\ot \tilde{p})\un{\Delta}_H(\Id_H\ot \un{\va}_B)=
(j\ot j)\tilde{f}(\Id_H\ot \un{\va}_B).
\]
Now $j\ot j$ is a monomorphism in $\Cc$, see the notes preceding the Lemma, and it follows that
$\tilde{f}\un{m}_H(\Id_H\ot i)=\tilde{f}(\Id_H\ot \un{\va}_B)$.
Applying the universal property of the coaequalizer $(A, p)$, we find a unique morphism
$\un{\Delta}_A:\ A\ra A\ot A$ such that  $\un{\Delta}_A p=(p\ot p)\un{\Delta}_H$.
Arguments dual to those 
presented in the proof of \leref{algstrA} show that $\un{\Delta}_A$ is coassociative.\\
Finally, $\tilde{f}=\un{\va}_H: H\ra\un{1}$ satisfies $\tilde{f}\un{m}_H(\Id_H\ot i)=\tilde{f}(\Id_H\ot \un{\va}_B)$. 
Applying the universal property again, we find a unique morphism
$\un{\va}_A: A\ra \un{1}$ such that $\un{\va}_Ap=\un{\va}_H$. It is immediate 
that $\un{\va}_A$ is a counit for $\un{\Delta}_A$, and hence $(A, \un{\Delta}_A, \un{\va}_A)$ 
is a coalgebra in $\Cc$. The construction of $\un{\Delta}_A$ and $\un{\va}_A$ is such that
$p$ is a coalgebra morphism, and this finishes the proof.
\end{proof}

Applying the formulas that we obtained above, we easily see that
$\un{\Delta}_A=\un{\Delta}_Apj=(p\ot p)\un{\Delta}_Hj$. Furthermore
\begin{equation}\eqlabel{deltasimplified}
\gbeg{2}{6}
\got{1}{A}\gnl
\gmp{j}\gnl
\gcn{1}{1}{1}{2}\gnl
\gcmu\gnl
\gmp{p}\gmp{\tilde{p}}\gnl
\gob{1}{A}\gob{1}{H}
\gend
=
\gbeg{3}{11}
\got{1}{A}\gnl
\gmp{j}\gnl
\gcn{1}{1}{1}{2}\gnl
\gcmu\gnl
\gcl{1}\gcn{1}{1}{1}{2}\gnl
\gcl{1}\gcmu\gnl
\gcl{1}\gcl{1}\gmp{\pi}\gnl
\gmp{p}\gcl{1}\gmp{\un{s}}\gnl
\gcl{1}\gcl{1}\gmp{i}\gnl
\gcl{1}\gmu\gnl
\gob{1}{A}\gob{2}{H}
\gend
=
\gbeg{3}{9}
\gvac{1}\got{1}{A}\gnl
\gvac{1}\gmp{j}\gnl
\gvac{1}\gcn{1}{1}{1}{2}\gnl
\gvac{1}\gcmu\gnl
\gvac{1}\gcn{1}{1}{1}{0}\gmp{\pi}\gnl
\gcmu\gmp{\un{s}}\gnl
\gmp{p}\gcl{1}\gmp{i}\gnl
\gcl{1}\gmu\gnl
\gob{1}{A}\gob{2}{H}
\gend
\equal{\equref{defofA}}
\gbeg{2}{6}
\got{1}{A}\gnl
\gmp{j}\gnl
\gcn{1}{1}{1}{2}\gnl
\gcmu\gnl
\gmp{p}\gcl{1}\gnl
\gob{1}{A}\gob{1}{H}
\gend
\hspace{2mm}.
\end{equation}
These formulas will be used in \leref{antipforA}.

\begin{lemma}\lelabel{antipforA}
Let $A, B, H, \pi, i, \tilde{p}, j$ be as in \leref{coalgstrA}. Assume that $H$ is a Hopf algebra
with antipode $\un{\mathcal{S}}$, and that
\equuref{antscpb}{b} is fulfilled. Then $\Id_A$ is convolution invertible.
\end{lemma}

\begin{proof}
$\tilde{f}=\un{m}_H(i\pi\ot \un{\mathcal{S}})\un{\Delta}_H:\ H\ra H$
satisfies the equation
\begin{eqnarray*}
&&\hspace*{-1cm}(\Id_H\ot \pi)\un{\Delta}_H\tilde{f}\\
&=&
\gbeg{2}{8}
\got{2}{H}\gnl
\gcmu\gnl
\gmp{\pi}\gmp{\un{\mathcal{S}}}\gnl
\gmp{i}\gcl{1}\gnl
\gmu\gnl
\gcmu\gnl
\gcl{1}\gmp{\pi}\gnl
\gob{1}{H}\gob{1}{B}
\gend
\equal{\equref{braidedbialgebra}}
\gbeg{4}{10}
\gvac{1}\got{2}{H}\gnl
\gvac{1}\gcmu\gnl
\gvac{1}\gmp{\pi}\gmp{\un{\mathcal{S}}}\gnl
\gvac{1}\gmp{i}\gcn{1}{1}{1}{2}\gnl
\gvac{1}\gcn{1}{1}{1}{0}\gcmu\gnl
\gcmu\gcl{1}\gmp{\pi}\gnl
\gcl{1}\gmp{\pi}\gcl{1}\gcl{1}\gnl
\gcl{1}\gbr\gcl{1}\gnl
\gmu\gmu\gnl
\gob{2}{H}\gob{2}{B}
\gend
\equal{\equref{piiscpb}(a)}
\gbeg{4}{10}
\got{2}{H}\gnl
\gcmu\gnl
\gcl{1}\gcn{1}{1}{1}{2}\gnl
\gcl{1}\gcmu\gnl
\gmp{\pi}\gmp{\pi}\gmp{\un{\mathcal{S}}}\gnl
\gmp{i}\gcl{1}\gcn{1}{1}{1}{2}\gnl
\gcl{1}\gcl{1}\gcmu\gnl
\gcl{1}\gbr\gmp{\pi}\gnl
\gmu\gmu\gnl
\gob{2}{H}\gob{2}{B}
\gend
\equalupdown{\equref{antiac}}{\equref{nat1cup}}
\gbeg{4}{12}
\gvac{1}\got{2}{H}\gnl
\gvac{1}\gcmu\gnl
\gvac{1}\gcn{1}{1}{1}{-1}\gcn{1}{1}{1}{2}\gnl
\gmp{\pi}\gvac{1}\gcmu\gnl
\gcl{1}\gvac{1}\gcn{1}{1}{1}{0}\gcl{1}\gnl
\gmp{i}\gcmu\gmp{\un{\mathcal{S}}}\gnl
\gcl{1}\gcl{1}\gmp{\un{\mathcal{S}}}\gcl{1}\gnl
\gcl{1}\gmu\gcl{1}\gnl
\gcn{1}{2}{1}{3}\gcn{1}{1}{2}{3}\gvac{1}\gcl{1}\gnl
\gvac{2}\gbr\gnl
\gvac{1}\gmu\gmp{\pi}\gnl
\gvac{1}\gob{2}{H}\gob{1}{B}
\gend
\equal{\equref{braidedantipode}}
\gbeg{3}{7}
\got{2}{H}\gnl
\gcmu\gnl
\gmp{\pi}\gmp{\un{\mathcal{S}}}\gnl
\gmp{i}\gcl{1}\gnl
\gmu\gnl
\gcn{1}{1}{2}{2}\gvac{1}\gu{1}\gnl
\gob{2}{H}\gob{1}{B}
\gend\\
&=&(\Id_H\ot \un{\eta}_B)\tilde{f}.
\end{eqnarray*}
Applying the universal property of the equalizer $(A,j)$, we obtain a unique morphism
$\tilde{\un{S}}_A: H\ra A$ such that $j\tilde{\un{S}}_A=\tilde{f}$.
We will show that 
$\un{S}=\tilde{\un{S}}_A j$ is the convolution inverse of $\Id_A$.
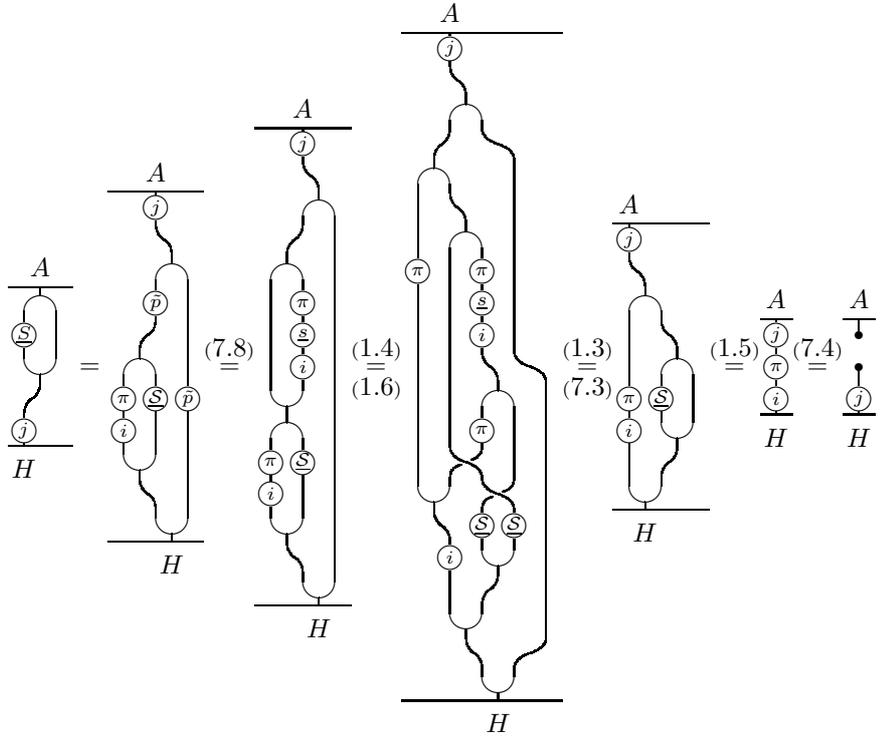
\begin{eqnarray*}
\gbeg{2}{7}
\got{2}{A}\gnl
\gcmu\gnl
\gmp{\un{S}}\gcl{1}\gnl
\gmu\gnl
\gcn{1}{1}{2}{1}\gnl
\gmp{j}\gnl
\gob{1}{H}
\gend
=
\gbeg{3}{13}
\gvac{1}\got{1}{A}\gnl
\gvac{1}\gmp{j}\gnl
\gvac{1}\gcn{1}{1}{1}{2}\gnl
\gvac{1}\gcmu\gnl
\gvac{1}\gmp{\tilde{p}}\gcl{1}\gnl
\gvac{1}\gcn{1}{1}{1}{0}\gcl{1}\gnl
\gcmu\gcl{1}\gnl
\gmp{\pi}\gmp{\un{\mathcal{S}}}\gmp{\tilde{p}}\gnl
\gmp{i}\gcl{1}\gcl{1}\gnl
\gmu\gcl{1}\gnl
\gcn{1}{1}{2}{3}\gvac{1}\gcl{1}\gnl
\gvac{1}\gmu\gnl
\gvac{1}\gob{2}{H}
\gend
\equal{\equref{deltasimplified}}
\gbeg{3}{17}
\gvac{1}\got{1}{A}\gnl
\gvac{1}\gmp{j}\gnl
\gvac{1}\gcn{1}{1}{1}{2}\gnl
\gvac{1}\gcmu\gnl
\gvac{1}\gcn{1}{1}{1}{0}\gcl{1}\gnl
\gcmu\gcl{1}\gnl
\gcl{1}\gmp{\pi}\gcl{1}\gnl
\gcl{1}\gmp{\un{s}}\gcl{1}\gnl
\gcl{1}\gmp{i}\gcl{1}\gnl
\gmu\gcl{1}\gnl
\gcmu\gcl{1}\gnl
\gmp{\pi}\gmp{\un{\mathcal{S}}}\gcl{1}\gnl
\gmp{i}\gcl{1}\gcl{1}\gnl
\gmu\gcl{1}\gnl
\gcn{1}{1}{2}{3}\gvac{1}\gcl{1}\gnl
\gvac{1}\gmu\gnl
\gvac{1}\gob{2}{H}
\gend
\equalupdown{\equref{braidedbialgebra}}{\equref{antiac}}
\gbeg{5}{23}
\gvac{1}\got{1}{A}\gnl
\gvac{1}\gmp{j}\gnl
\gvac{1}\gcn{1}{1}{1}{2}\gnl
\gvac{1}\gcmu\gnl
\gvac{1}\gcn{1}{1}{1}{0}\gcn{1}{1}{1}{3}\gnl
\gcmu\gvac{1}\gcl{1}\gnl
\gcl{1}\gcn{1}{1}{1}{2}\gvac{1}\gcl{1}\gnl
\gcl{1}\gcmu\gcl{1}\gnl
\gmp{\pi}\gcl{1}\gmp{\pi}\gcl{1}\gnl
\gcl{1}\gcl{1}\gmp{\un{s}}\gcl{1}\gnl
\gcl{1}\gcl{1}\gmp{i}\gcl{1}\gnl
\gcl{1}\gcl{1}\gcn{1}{1}{1}{2}\gcn{1}{1}{1}{3}\gnl
\gcl{1}\gcl{1}\gcmu\gcl{1}\gnl
\gcl{1}\gcl{1}\gmp{\pi}\gcl{1}\gcl{1}\gnl
\gcl{1}\gbr\gcl{1}\gcl{1}\gnl
\gmu\gbr\gcl{1}\gnl
\gcn{1}{1}{2}{3}\gvac{1}\gmp{\un{\mathcal{S}}}\gmp{\un{\mathcal{S}}}\gcl{1}\gnl
\gvac{1}\gmp{i}\gmu\gcl{1}\gnl
\gvac{1}\gcl{1}\gcn{1}{1}{2}{1}\gvac{1}\gcl{1}\gnl
\gvac{1}\gmu\gvac{1}\gcl{1}\gnl
\gvac{1}\gcn{1}{1}{2}{3}\gvac{2}\gcn{1}{1}{1}{-1}\gnl
\gvac{2}\gmu\gnl
\gvac{2}\gob{2}{H}
\gend
\equalupdown{\equref{nat2cup}}{\equref{antscpb}}
\gbeg{3}{11}
\got{1}{A}\gnl
\gmp{j}\gnl
\gcn{1}{1}{1}{2}\gnl
\gcmu\gnl
\gcl{1}\gcn{1}{1}{1}{2}\gnl
\gcl{1}\gcmu\gnl
\gmp{\pi}\gmp{\un{\mathcal{S}}}\gcl{1}\gnl
\gmp{i}\gmu\gnl
\gcl{1}\gcn{1}{1}{2}{1}\gnl
\gmu\gnl
\gob{2}{H}
\gend
\equal{\equref{braidedantipode}}
\gbeg{1}{5}
\got{1}{A}\gnl
\gmp{j}\gnl
\gmp{\pi}\gnl
\gmp{i}\gnl
\gob{1}{H}
\gend
\equal{\equref{defofA}}
\gbeg{1}{5}
\got{1}{A}\gnl
\gcu{1}\gnl
\gu{1}\gnl
\gmp{j}\gnl
\gob{1}{H}
\gend~.
\end{eqnarray*}
The first equality in \equref{braidedantipode}  now follows from the fact that
$j$ is a monomorphism. The second one follows in a similar way from the following
computation:
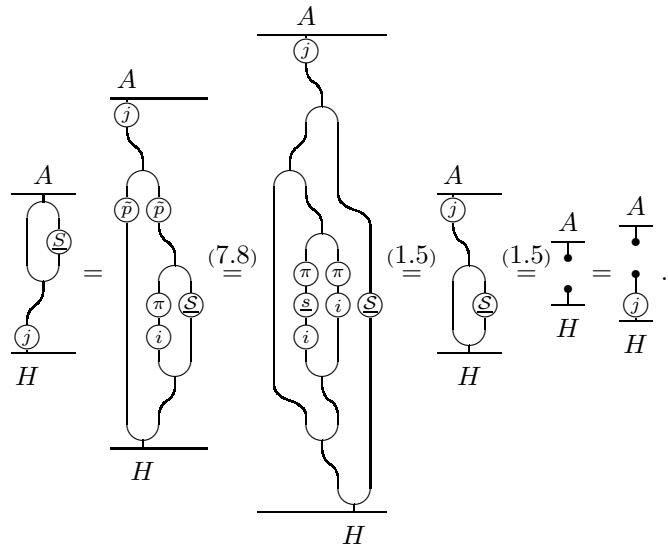
\begin{eqnarray*}
\gbeg{2}{7}
\got{2}{A}\gnl
\gcmu\gnl
\gcl{1}\gmp{\un{S}}\gnl
\gmu\gnl
\gcn{1}{1}{2}{1}\gnl
\gmp{j}\gnl
\gob{1}{H}
\gend
=
\gbeg{3}{13}
\got{1}{A}\gnl
\gmp{j}\gnl
\gcn{1}{1}{1}{2}\gnl
\gcmu\gnl
\gmp{\tilde{p}}\gmp{\tilde{p}}\gnl
\gcl{1}\gcn{1}{1}{1}{2}\gnl
\gcl{1}\gcmu\gnl
\gcl{1}\gmp{\pi}\gmp{\un{\mathcal{S}}}\gnl
\gcl{1}\gmp{i}\gcl{1}\gnl
\gcl{1}\gmu\gnl
\gcl{1}\gcn{1}{1}{2}{1}\gnl
\gmu\gnl
\gob{2}{H}
\gend
\equal{\equref{deltasimplified}}
\gbeg{4}{17}
\gvac{1}\got{1}{A}\gnl
\gvac{1}\gmp{j}\gnl
\gvac{1}\gcn{1}{1}{1}{2}\gnl
\gvac{1}\gcmu\gnl
\gvac{1}\gcn{1}{1}{1}{0}\gcl{1}\gnl
\gcmu\gcl{1}\gnl
\gcl{1}\gcn{1}{1}{1}{2}\gcn{1}{1}{1}{3}\gnl
\gcl{1}\gcmu\gcl{1}\gnl
\gcl{1}\gmp{\pi}\gmp{\pi}\gcl{1}\gnl
\gcl{1}\gmp{\un{s}}\gmp{i}\gmp{\un{\mathcal{S}}}\gnl
\gcl{1}\gmp{i}\gcl{1}\gcl{1}\gnl
\gcl{1}\gmu\gcl{1}\gnl
\gcn{1}{1}{1}{3}\gcn{1}{1}{2}{3}\gvac{1}\gcl{1}\gnl
\gvac{1}\gmu\gcl{1}\gnl
\gvac{1}\gcn{1}{1}{2}{3}\gvac{1}\gcl{1}\gnl
\gvac{2}\gmu\gnl
\gvac{2}\gob{2}{H}
\gend
\equal{\equref{braidedantipode}}
\gbeg{2}{7}
\got{1}{A}\gnl
\gmp{j}\gnl
\gcn{1}{1}{1}{2}\gnl
\gcmu\gnl
\gcl{1}\gmp{\un{\mathcal{S}}}\gnl
\gmu\gnl
\gob{2}{H}
\gend
\equal{\equref{braidedantipode}}
\gbeg{1}{4}
\got{1}{A}\gnl
\gcu{1}\gnl
\gu{1}\gnl
\gob{1}{H}
\gend
=
\gbeg{1}{5}
\got{1}{A}\gnl
\gcu{1}\gnl
\gu{1}\gnl
\gmp{j}\gnl
\gob{1}{H}
\gend~.
\end{eqnarray*}
\end{proof}

\begin{theorem}\thlabel{strofHopfwithcertproj}
Let $\Cc$ be a braided monoidal category with equalizers in which every object is flat. 
For a Hopf algebra $H$, the following assertions are equivalent.
\begin{itemize}
\item[(i)] $H$ is isomorphic to a smash cross product Hopf algebra;
\item[(ii)] there exist a Hopf algebra $B$, an algebra morphism $i:\ B\to H$
and a Hopf algebra morphism $\pi:\ H\to B$
such that $\pi i=\Id_B$ 
and the conditions \equref{piiscpb} and \equuref{antscpb}{b} are fulfilled.
\end{itemize}  
\end{theorem}  

\begin{proof}
$\un{(i)\Rightarrow (ii)}$ follows from \leref{scpbextprop}.\\
$\un{(ii)\Rightarrow (i)}$. Let $(A,j)$ be the equalizer of
$(\Id_H\ot \pi)\un{\Delta}_H,~\Id_H\ot \un{\eta}_B:\ H\to H\ot B$. It follows from Lemmas
\ref{le:algstrA} and \ref{le:coalgstrA} that $A$ is an algebra and a coalgebra, and that
there is an algebra morphism $j: A\ra H$ and a coalgebra morphism $p: H\ra A$ such that
 $pj=\Id_A$. We claim that
 $\zeta=\un{m}_H(j\ot i): A\ot B\ra H$ and $\zeta^{-1}=(p\ot \pi)\un{\Delta}_H: H\ra A\ot B$
 are inverses. Indeed, 
 $\zeta^{-1}\zeta(j\ot \Id_B)$ is equal to 
\[
\gbeg{2}{6}
\got{1}{A}\got{1}{B}\gnl
\gmp{j}\gmp{i}\gnl
\gmu\gnl
\gcmu\gnl
\gmp{\tilde{p}}\gmp{\pi}\gnl
\gob{1}{H}\gob{1}{B}
\gend
\equal{\equref{braidedantipode}}
\gbeg{4}{10}
\got{1}{A}\gvac{1}\got{1}{B}\gnl
\gmp{j}\gvac{1}\gmp{i}\gnl
\gcn{1}{1}{1}{2}\gvac{1}\gcn{1}{1}{1}{2}\gnl
\gcmu\gcmu\gnl
\gcl{1}\gmp{\pi}\gcl{1}\gmp{\pi}\gnl
\gcl{1}\gbr\gcl{1}\gnl
\gmu\gmu\gnl
\gcn{1}{1}{2}{1}\gvac{1}\gcn{1}{2}{2}{2}\gnl
\gmp{\tilde{p}}\gnl
\gob{1}{H}\gvac{1}\gob{2}{B}
\gend
\equalupdown{\equref{defofA}}{\equuref{piiscpb}{a}}
\gbeg{3}{12}
\got{1}{A}\got{1}{B}\gnl
\gmp{j}\gcl{1}\gnl
\gcl{1}\gcn{1}{1}{1}{2}\gnl
\gcl{1}\gcmu\gnl
\gcl{1}\gmp{i}\gcl{1}\gnl
\gmu\gcl{1}\gnl
\gcmu\gcl{1}\gnl
\gcl{1}\gmp{\pi}\gcl{1}\gnl
\gcl{1}\gmp{\un{s}}\gcl{1}\gnl
\gcl{1}\gmp{i}\gcl{1}\gnl
\gmu\gcl{1}\gnl
\gob{2}{H}\gob{1}{B}
\gend
\equal{\equref{braidedbialgebra}}
\gbeg{5}{13}
\gvac{1}\got{1}{A}\gvac{1}\got{2}{B}\gnl
\gvac{1}\gmp{j}\gvac{1}\gcmu\gnl
\gvac{1}\gcn{1}{1}{1}{0}\gvac{1}\gmp{i}\gcl{1}\gnl
\gcmu\gcn{1}{1}{3}{2}\gvac{1}\gcl{1}\gnl
\gcl{1}\gcl{1}\gcmu\gcl{1}\gnl
\gcl{1}\gmp{\pi}\gcl{1}\gmp{\pi}\gcl{2}\gnl
\gcl{1}\gbr\gcl{1}\gnl
\gmu\gmu\gcl{1}\gnl
\gcn{1}{1}{2}{3}\gvac{1}\gcn{1}{1}{2}{1}\gvac{1}\gcl{4}\gnl
\gvac{1}\gcl{1}\gmp{\un{s}}\gnl
\gvac{1}\gcl{1}\gmp{i}\gnl
\gvac{1}\gmu\gnl
\gvac{1}\gob{2}{H}\gvac{1}\gob{1}{B}
\gend
\equalupdown{\equref{defofA}}{\equuref{piiscpb}{a}}
\gbeg{4}{10}
\got{1}{A}\gvac{1}\got{2}{B}\gnl
\gcl{1}\gvac{1}\gcmu\gnl
\gcl{1}\gvac{1}\gcn{1}{1}{1}{0}\gcl{1}\gnl
\gmp{j}\gcmu\gcl{1}\gnl
\gcl{1}\gcl{1}\gmp{\un{s}}\gcl{1}\gnl
\gcl{1}\gmu\gcl{4}\gnl
\gcl{1}\gcn{1}{1}{2}{1}\gnl
\gcl{1}\gmp{i}\gnl
\gmu\gnl
\gob{2}{H}\gvac{1}\gob{1}{B}
\gend
\equal{\equref{braidedantipode}}
j\ot\Id_B,
\]
and this implies that $\zeta^{-1}\zeta=\Id_{A\ot B}$. We also have that
\[
\zeta\zeta^{-1}=
\gbeg{2}{6}
\got{2}{H}\gnl
\gcmu\gnl
\gmp{p}\gmp{\pi}\gnl
\gmp{j}\gmp{i}\gnl
\gmu\gnl
\gob{2}{H}
\gend
=
\gbeg{3}{11}
\got{2}{H}\gnl
\gcmu\gnl
\gcl{1}\gmp{\pi}\gnl
\gcl{1}\gcn{1}{1}{1}{2}\gnl
\gcl{1}\gcmu\gnl
\gcl{1}\gmp{\un{s}}\gcl{1}\gnl
\gcl{1}\gmu\gnl
\gcl{1}\gcn{1}{1}{2}{1}\gnl
\gcl{1}\gmp{i}\gnl
\gmu\gnl
\gob{2}{H}
\gend
\equal{\equref{braidedantipode}}
\Id_H.
\]
Now we apply the implication $(ii)\Rightarrow (i)$ in \prref{DrBespExtVers},
and obtain that $\zeta^{-1}:\ H\to A\times_\psi^\phi B$ is a bialgebra isomorphism,
with $\psi$ and $\phi$ given by \equref{psiphifromexts}. Now
$
\gbeg{2}{4}
\got{1}{B}\got{1}{A}\gnl
\gbrc\gnl
\gcu{1}\gcl{1}\gnl
\gvac{1}\gob{1}{B}
\gend 
=
\gbeg{2}{6}
\got{1}{B}\got{1}{A}\gnl
\gmp{i}\gmp{j}\gnl
\gmu\gnl
\gcn{1}{1}{2}{1}\gnl
\gmp{\pi}\gnl
\gob{1}{B}
\gend
=
\gbeg{2}{5}
\got{1}{B}\got{1}{A}\gnl
\gmp{i}\gmp{j}\gnl
\gmp{\pi}\gmp{\pi}\gnl
\gmu\gnl
\gob{2}{B}
\gend
\equal{\equref{defofA}}
\gbeg{2}{3}
\got{1}{B}\got{1}{A}\gnl
\gcl{1}\gcu{1}\gnl
\gob{1}{B}
\gend 
$,
and we conclude that $\psi$ is left conormal, and that $H$ is isomorphic to a smash cross product bialgebra.\\
Finally, according to \leref{antipforA} $\Id_A$ is convolution invertible. Together with the fact that
$B$ is a Hopf algebra, this implies  that $A\times_\psi^\phi B$ is a Hopf algebra, see \prref{whenacrossprodisHA}. 
Then $\zeta^{-1}$ is a Hopf algebra isomorphism, completing the proof.
 \end{proof}
 
We leave it to the reader to formulate the right handed version of \thref{strofHopfwithcertproj}.
 
\begin{corollary}\colabel{biprodfromproj} {\bf \cite{rad}}
Let $\Cc$ be a braided monoidal category with equalizers and assume that 
every object of $\Cc$ is flat. Then a Hopf algebra $H$ is isomorphic to a biproduct Hopf algebra
if and only if there exist a Hopf algebra $B$ and Hopf algebra morphisms $i:\ B\to H$,
$\pi:\ H\to B$ such that $\pi i=\Id_B$. 
\end{corollary}

\begin{proof}
A biproduct Hopf algebra $A\times_\psi^\phi B$ is a smash cross product Hopf algebra 
for which $\phi$ is left normal. In this case it is easy to see that the canonical morphism $i: B\ra H$ 
is a coalgebra morphism, and therefore a Hopf algebra morphism. Then it can be easily checked that  
\equref{piiscpb} and \equuref{antscpb}{b} are automatically satisfied.\\
Conversely, if $i$, $\pi$ and $B$ are given as in the Theorem, then
condition (ii) of \thref{strofHopfwithcertproj} is fulfilled since $i$ is a Hopf algebra
morphism. Hence $H$ is isomorphic to a smash cross product Hopf algebra. 
It follows from \equref{psiphifromexts} that
\[
\phi=
\gbeg{2}{6}
\got{1}{A}\got{1}{B}\gnl
\gmp{j}\gmp{i}\gnl
\gmu\gnl
\gcmu\gnl
\gmp{\pi}\gmp{p}\gnl
\gob{1}{B}\gob{1}{A}
\gend
~~\mbox{, hence}~~
\gbeg{2}{5}
\gvac{1}\got{1}{B}\gnl
\gu{1}\gcl{1}\gnl
\gbrbox\gnl
\gcl{1}\gmp{j}\gnl
\gob{1}{B}\gob{1}{H}
\gend
=
\gbeg{2}{6}
\got{1}{B}\gnl
\gmp{i}\gnl
\gcn{1}{1}{1}{2}\gnl
\gcmu\gnl
\gmp{\pi}\gmp{\tilde{p}}\gnl
\gob{1}{B}\gob{1}{H}
\gend
=
\gbeg{3}{10}
\got{2}{B}\gnl
\gcmu\gnl
\gmp{i}\gmp{i}\gnl
\gmp{\pi}\gcn{1}{1}{1}{2}\gnl
\gcl{1}\gcmu\gnl
\gcl{1}\gcl{1}\gmp{\pi}\gnl
\gcl{1}\gcl{1}\gmp{\un{s}}\gnl
\gcl{1}\gcl{1}\gmp{i}\gnl
\gcl{1}\gmu\gnl
\gob{1}{B}\gob{2}{H}
\gend
=
\gbeg{3}{9}
\got{2}{B}\gnl
\gcmu\gnl
\gcl{1}\gcn{1}{1}{1}{2}\gnl
\gcl{1}\gcmu\gnl
\gcl{1}\gcl{1}\gmp{\un{s}}\gnl
\gcl{1}\gmu\gnl
\gcl{1}\gcn{1}{1}{2}{1}\gnl
\gcl{1}\gmp{i}\gnl
\gob{1}{B}\gob{1}{H}
\gend
\equal{\equref{braidedantipode}}
\gbeg{2}{4}
\got{1}{B}\gnl
\gcl{1}\gu{1}\gnl
\gcl{1}\gmp{j}\gnl
\gob{1}{B}\gob{1}{H}
\gend~.
\]
We conclude that $\phi$ is left normal, and therefore $H$ is isomorphic to a biproduct 
Hopf algebra.    
\end{proof}

Now we focus attention to double cross coproduct Hopf algebras. Recall that $X\in \Cc$
is called right (left) coflat if $X\ot -$ (resp. $-\ot X$) preserves coequalizers. $X$ is coflat if it is left and 
right coflat.

\begin{corollary}\colabel{dccopasproj} \cite{majbip}
Let $\Cc$ be a braided monoidal category with equalizers and assume that 
every object in $\Cc$ is flat and right coflat. A Hopf algebra
$H$ is isomorphic to a double cross coproduct Hopf algebra if and only if
there exist a Hopf algebra $B$ and  Hopf algebra morphisms $i:\ B\to H$ and $\pi:\ H\to B$
such that $\pi i=\Id_B$ and \equref{doblecrosscoprproj} holds, i.e., the left adjoint action 
of $B$ on $H$ induced by $i$ is trivial on the image of $\tilde{p}$.
\begin{equation}\eqlabel{doblecrosscoprproj}
\gbeg{4}{12}
\got{2}{B}\got{2}{H}\gnl
\gcmu\gcmu\gnl
\gcl{1}\gcl{1}\gcl{1}\gmp{\pi}\gnl
\gcl{1}\gmp{\un{s}}\gcl{1}\gmp{\un{s}}\gnl
\gmp{i}\gmp{i}\gcl{1}\gmp{i}\gnl
\gcl{1}\gcl{1}\gmu\gnl
\gcl{1}\gcl{1}\gcn{1}{1}{2}{1}\gnl
\gcl{1}\gbr\gnl
\gmu\gcl{1}\gnl
\gcn{1}{1}{2}{3}\gvac{1}\gcl{1}\gnl
\gvac{1}\gmu\gnl
\gvac{1}\gob{2}{H}
\gend
=
\gbeg{3}{7}
\got{1}{B}\got{2}{H}\gnl
\gcu{1}\gcmu\gnl
\gvac{1}\gcl{1}\gmp{\pi}\gnl
\gvac{1}\gcl{1}\gmp{\un{s}}\gnl
\gvac{1}\gcl{1}\gmp{i}\gnl
\gvac{1}\gmu\gnl
\gvac{1}\gob{2}{H}
\gend~.
\end{equation} 
\end{corollary}

\begin{proof}
A double cross coproduct Hopf algebra is a 
biproduct Hopf algebra for which $\psi$ is right conormal. By \coref{biprodfromproj},
it suffices to verify  \equref{doblecrosscoprproj}. This follows directly from the 
definitions, we leave the details to the reader.\\
Conversely, assume that $i$ and $\pi$ are given.
According to \coref{biprodfromproj}, $H$ is isomorphic to a biproduct Hopf algebra 
$A\times_\psi^\phi B$. The biproduct Hopf algebra is actually a double cross coproduct Hopf algebra 
since $\psi$ is right conormal. Indeed, we have
\[
\gbeg{2}{5}
\got{1}{B}\got{1}{H}\gnl
\gcl{1}\gmp{p}\gnl
\gbrc\gnl
\gmp{j}\gcu{1}\gnl
\gob{1}{H}
\gend
=
\gbeg{2}{9}
\got{1}{B}\got{1}{H}\gnl
\gmp{i}\gmp{\tilde{p}}\gnl
\gmu\gnl
\gcmu\gnl
\gcl{1}\gmp{\pi}\gnl
\gcl{1}\gmp{\un{s}}\gnl
\gcl{1}\gmp{i}\gnl
\gmu\gnl
\gob{2}{H}
\gend
\equal{\equref{braidedbialgebra}}
\gbeg{4}{12}
\got{1}{B}\gvac{1}\got{1}{H}\gnl
\gmp{i}\gvac{1}\gmp{\tilde{p}}\gnl
\gcn{1}{1}{1}{2}\gvac{1}\gcn{1}{1}{1}{2}\gnl
\gcmu\gcmu\gnl
\gcl{1}\gmp{\pi}\gcl{1}\gmp{\pi}\gnl
\gcl{1}\gbr\gcl{1}\gnl
\gmu\gmu\gnl
\gcn{1}{1}{2}{1}\gcn{1}{1}{4}{1}\gnl
\gcl{1}\gmp{\un{s}}\gnl
\gcl{1}\gmp{i}\gnl
\gmu\gnl
\gob{2}{H}
\gend
\equalupdown{\equref{defofA}}{\equuref{piiscpb}{a}}
\gbeg{4}{12}
\got{2}{B}\got{2}{H}\gnl
\gcmu\gcmu\gnl
\gmp{i}\gmp{\un{s}}\gcl{1}\gmp{\pi}\gnl
\gcl{1}\gmp{i}\gcl{1}\gmp{\un{s}}\gnl
\gcl{1}\gcl{1}\gcl{1}\gmp{i}\gnl
\gcl{1}\gcl{1}\gmu\gnl
\gcl{1}\gcl{1}\gcn{1}{1}{2}{1}\gnl
\gcl{1}\gbr\gnl
\gmu\gcl{1}\gnl
\gcn{1}{1}{2}{3}\gvac{1}\gcl{1}\gnl
\gvac{1}\gmu\gnl
\gvac{1}\gob{2}{H}
\gend
\equal{\equref{doblecrosscoprproj}}
\gbeg{3}{7}
\got{1}{B}\got{2}{H}\gnl
\gcu{1}\gcmu\gnl
\gvac{1}\gcl{1}\gmp{\pi}\gnl
\gvac{1}\gcl{1}\gmp{\un{s}}\gnl
\gvac{1}\gcl{1}\gmp{i}\gnl
\gvac{1}\gmu\gnl
\gvac{1}\gob{2}{H}
\gend
=
\gbeg{2}{4}
\got{1}{B}\got{1}{H}\gnl
\gcu{1}\gmp{p}\gnl
\gvac{1}\gmp{j}\gnl
\gvac{1}\gob{1}{H}
\gend. 
\]    
The right coflatness of $B$ implies that $\Id_B\ot p$ is an epimorphism.
Then the right conormality of $\psi$ follows from the fact that $j$ is a monomorphism.
\end{proof}

To make our story complete, we present the dual version of \thref{strofHopfwithcertproj}.
We need some Lemmas first. Most of the proofs are omitted, as they are dual versions of
proofs that we presented above.

\begin{lemma}\lelabel{dual1}
Let $H=A\times_\psi^\phi B$ be a (left) smash cross coproduct bialgebra and let $\pi : H\ra B$ and $i: B\ra H$ 
be the canonical morphisms.

(i) $\pi$ is a coalgebra morphism, $i$ is a bialgebra morphism, $\pi i=\Id_B$ 
and
\begin{equation}\eqlabel{carprofcccHa}
\gbeg{2}{6}
\got{1}{H}\got{1}{B}\gnl
\gcl{1}\gmp{i}\gnl
\gmu\gnl
\gcn{1}{1}{2}{1}\gnl
\gmp{\pi}\gnl
\gob{1}{B}
\gend
=
\gbeg{2}{4}
\got{1}{H}\got{1}{B}\gnl
\gmp{\pi}\gcl{1}\gnl
\gmu\gnl
\gob{2}{B}
\gend
\hspace{2mm}\mbox{\rm and}\hspace{2mm}
\gbeg{3}{7}
\got{1}{H}\got{2}{H}\gnl
\gcl{1}\gcmu\gnl
\gmu\gmp{\pi}\gnl
\gcn{1}{1}{2}{1}\gvac{1}\gmp{\un{s}}\gnl
\gmp{\pi}\gvac{1}\gcn{1}{1}{1}{-1}\gnl
\gmu\gnl
\gob{2}{B}
\gend
=
\gbeg{3}{11}
\got{1}{H}\got{2}{H}\gnl
\gcl{1}\gcmu\gnl
\gmp{\pi}\gcl{1}\gmp{\pi}\gnl
\gmp{i}\gcl{1}\gmp{\un{s}}\gnl
\gcl{1}\gcl{1}\gmp{i}\gnl
\gcl{1}\gmu\gnl
\gcl{1}\gcn{1}{1}{2}{1}\gnl
\gmu\gnl
\gcn{1}{1}{2}{1}\gnl
\gmp{\pi}\gnl
\gob{1}{B}
\gend~~.
\end{equation}
For the second equality, we need the additional assumption that $B$ is a Hopf algebra with antipode $\un{s}$.\\
(ii) If $H$ is a Hopf algebra with antipode $\un{\mathcal{S}}$ then $B$ is also a Hopf algebra and 
\begin{equation}\eqlabel{carprofcccHaAnt}
\gbeg{2}{8}
\got{2}{H}\gnl
\gcmu\gnl
\gmp{\pi}\gmp{\un{\mathcal{S}}}\gnl
\gmp{i}\gcl{1}\gnl
\gmu\gnl
\gcn{1}{1}{2}{1}\gnl
\gmp{\pi}\gnl
\gob{1}{B}
\gend
=
\gbeg{1}{4}
\got{1}{H}\gnl
\gcu{1}\gnl
\gu{1}\gnl
\gob{1}{B}
\gend~~.
\end{equation}
\end{lemma} 

For the converse of \leref{dual1}, we need additional assumptions: $\Cc$ has coequalizers,
and every object of $\Cc$ is coflat, that is, it is left and right coflat.

\begin{lemma}\lelabel{dualconstrA}
Let $H$ be a bialgebra and let $B$ be a Hopf algebra with antipode $\un{s}$ and 
suppose that we have a bialgebra morphism $i:\ B\to H$ and a coalgebra morphism
$\pi:\ H\to B$ such that $\pi i=\Id_B$, such that \equref{carprofcccHa} holds.
Let $(A, p)$ be the coequalizer of $\un{m}_H(\Id_H\ot i),~\Id_H\ot \un{\va}_B:\
H\ot B\to B$. Then we have the following results.\\
(i) $A$ has a coalgebra structure such that $p: H\ra A$ is a coalgebra morphism in $\Cc$.\\
(ii) $(A, j)$ is the equalizer of $(\Id_H\ot \pi)\un{\Delta}_H,~\Id_H\ot \un{\eta}_H:\
H\to H\ot B$, where $\tilde{j}=\un{m}_H(\Id_H\ot i\un{s}\pi)\un{\Delta}_H$ and
$j$ is defined by commutativity of the diagram
\[
\xymatrix{
H\ot B \ar[rr]<3pt>^{\un{m}_H(\Id_H\ot i)} 
   \ar[rr]<-3pt>_{\Id_H\ot \un{\va}_B} && 
   ~~H\ar[d]_{\tilde{j}}\ar[r]^p & A \ar @{.>}[dl]^j .\\
   && H  & 
}
\]
Consequently $A$ is an algebra and $j: A\ra H$ 
is an algebra morphism.\\
(iii) If $H$ is a Hopf algebra with antipode $\un{\mathcal{S}}$ satisfying \equref{carprofcccHaAnt}, then 
$\Id_A$ is convolution invertible.  
\end{lemma}

\begin{proof} 
(i) follows from \leref{algstrA} by duality arguments. (ii) We just mention that the algebra structure on $A$
is obtained using  \equref{algstrA}.\\
(iii) Applying the universal property of the coequalizer $(A, p)$, we find $\widetilde{\un{S}}_A: A\ra H$ such that 
the diagram 
\[
\xymatrix{
H\ot B \ar[rr]<3pt>^{\un{m}_H(\Id_H\ot i)} 
  \ar[rr]<-3pt>_{\Id_H\ot \un{\va}_B} && ~~H\ar@<1ex>[d]_{\un{m}_H(\pi i\ot \un{\mathcal{S}})\un{\Delta}_H}\ar[r]^p & A \ar @{.>}[dl]^{\tilde{\un{S}}_A} \\
   && H  & 
}
\]
commutes. A straightforward computation shows that $\un{S}:=p\tilde{\un{S}}_A$ is the convolution inverse
of $\Id_A$.
\end{proof}

\thref{strofHopfwithcertproj2} is the dual version of \thref{strofHopfwithcertproj}. The proof is based on
\leref{dualconstrA} and is omitted.

\begin{theorem}\thlabel{strofHopfwithcertproj2}
Let $\Cc$ be a braided monoidal category with coequalizers, in which 
every object of $\Cc$ is coflat. A Hopf algebra $H$
is isomorphic to a smash cross coproduct Hopf algebra if and only if
there exists a Hopf algebra $B$ and morphisms $i:\ B\to H$ and $\pi:\ H\to B$ in $\Cc$ such that $i$ is a Hopf algebra morphism, $\pi$ is a coalgebra morphism, $\pi i=\Id_B$ 
and (\ref{eq:carprofcccHa}-\ref{eq:carprofcccHaAnt}) are satisfied.
\end{theorem}

If we specialize \thref{strofHopfwithcertproj2} to the case where $\psi$ is left conormal, then we
recover \coref{biprodfromproj}. This is due to the fact that $A$ can be defined as an equalizer or 
as a coequalizer. If we add the condition that $\phi$ is right conormal, then we obtain \coref{wpdoublecrossprod}.
The proof is similar to the proof of \coref{dccopasproj}. We just mention that the conditions in (ii) tell us 
that the left coadjoint coaction of $B$ on $H$ induced by $\pi$ is trivial on the 
image of $\tilde{j}=jp$.

\begin{corollary}\colabel{wpdoublecrossprod}
Let $\Cc$ be a braided monoidal category with coequalizers, in which 
every object  coflat and left flat. A Hopf algebra $H$ is isomorphic to a double cross product Hopf algebra
if and only if
there exist a Hopf algebra $B$ and Hopf algebra morphisms $i:\ B\to H$, $\pi:\ H\to B$
 such that $\pi i=\Id_B$ and 
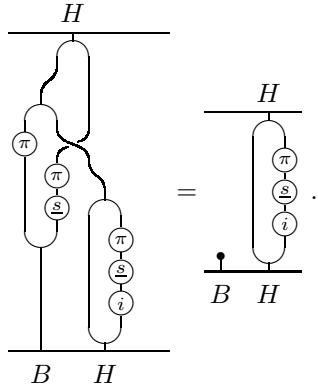
\[
\gbeg{5}{12}
\gvac{1}\got{2}{H}\gnl
\gvac{1}\gcmu\gnl
\gvac{1}\gcn{1}{1}{1}{0}\gcl{1}\gnl
\gcmu\gcl{1}\gnl
\gmp{\pi}\gbr\gnl
\gcl{1}\gmp{\pi}\gcn{1}{1}{1}{2}\gnl
\gcl{1}\gmp{\un{s}}\gcmu\gnl
\gmu\gcl{1}\gmp{\pi}\gnl
\gcn{1}{3}{2}{2}\gvac{1}\gcl{1}\gmp{\un{s}}\gnl
\gvac{2}\gcl{1}\gmp{i}\gnl
\gvac{2}\gmu\gnl
\gob{2}{B}\gob{2}{H}
\gend
=
\gbeg{3}{7}
\gvac{1}\got{2}{H}\gnl
\gvac{1}\gcmu\gnl
\gvac{1}\gcl{1}\gmp{\pi}\gnl
\gvac{1}\gcl{1}\gmp{\un{s}}\gnl
\gvac{1}\gcl{1}\gmp{i}\gnl
\gu{1}\gmu\gnl
\gob{1}{B}\gob{2}{H}
\gend~.
\]
\end{corollary}



\begin{thebibliography}{99}
\bibitem{amstAMS}
A. Ardizzoni, C. Menini, D. \c Stefan, {\sl A monoidal approach to splitting morphisms of bialgebras}, 
Trans. Amer. Math. Soc. {\bf 359} (2007), 991--1044.
\bibitem{amst}
A. Ardizzoni, C. Menini, D. \c Stefan, {\sl Weak projections onto a braided Hopf algebra}, 
J. Algebra {\bf 318} (2007), 180--201.
\bibitem{bespdrab1}
Y. Bespalov, B. Drabant, Cross product bialgebras I, J. Algebra {\bf 219} (1999), 466--505.
\bibitem{bcm}
R.J. Blattner, M. Cohen, S. Montgomery, {\sl Cross products and inner actions of Hopf algebras}, 
Trans. Amer. Math. Soc. {\bf 298} (1986), 671--711. 
\bibitem{bulacu}
D. Bulacu, ``Algebras and coalgebras in braided monoidal categories", Editura 
Universit\u a\c tii Bucure\c sti, 2009.
\bibitem{cimz}
S. Caenepeel, B. Ion, G. Militaru, S. Zhu, {\sl The factorization problem and the smash biproduct of algebras 
and coalgebras}, Algebr. Represent. Theory {\bf 3} (2000), 19--42.  
\bibitem{kas}
C. Kassel, ``Quantum Groups", Graduate Texts in Mathematics 155, Berlin:
Springer-Verlag, 1995.
\bibitem{majbip}
S. Majid, {\sl Physics for algebraists: Non-commutative and non-cocommutative Hopf algebras by a bicrossproduct 
construction}, J. Algebra {\bf 130} (1990), 17-64.   
\bibitem{m3}
S. Majid, {\sl Algebras and Hopf algebras in braided categories}, in
``Advances in Hopf Algebras", {\sl Lect. Notes Pure Appl. Math.} {\bf 158},
Dekker, New York, 1994, 55--105.
\bibitem{maj}
S. Majid, ``Foundations of quantum group theory", Cambridge University
Press, 1995.
\bibitem{2chg}
J. Park, {\sl Generalized biproduct Hopf algebras}, J. Chungcheong Math. Soc. {\bf 21} (2008), 301--320.
\bibitem{rad}
D.E. Radford, {\sl The structure of Hopf algebras with a projection}, J. Algebra 
{\bf 92} (1985), 322–-347.
\bibitem{schwproj}
P. Schauenburg, {\sl The structure of Hopf algebras with a weak projection}, Algebr. Represent. Theory {\bf 3} (2000), 
187--211.  
\bibitem{Schauenburg2003}
P. Schauenburg, {\sl Actions of monoidal categories and generalize smash products}, J. Algebra
{\bf 270} (2003), 521--563.
\bibitem{tak}
M. Takeuchi, {\sl Matched pairs of groups and bismash products of Hopf algebras}, Comm. Alg. {\bf 9} (1981), 841--822.
\bibitem{Tambara}
D. Tambara, {\sl The coendomorphism bialgebra of an algebra}, J. Fac. Sci. Univ. Tokyo Sect. IA Math. {\bf 37 (2)}
(1990), 425Ð456.
\bibitem{zhcn}
Shouchuan Zhang, Hui-Xiang Chen, {\sl The double bicrossproducts in braided tensor categories}, 
Comm. Alg. {\bf 29} (2001), 31-–66.

\end{thebibliography}
\end{document}